\documentclass{amsart}
\usepackage{amsmath}
\usepackage{amscd}
\usepackage{amssymb}
\usepackage[mathscr]{eucal}
\usepackage{multirow}
\usepackage{bm}
\usepackage[pdftex]{graphicx}
\input xy
\xyoption{all}

\newcommand{\bZ}{{\mathbb Z}}

\newcommand{\s}{\mathop{\mathrm{S}}\nolimits}
\newtheorem{thm}{Theorem}[section]
\newtheorem{prop}[thm]{Proposition}
\newtheorem{lemma}[thm]{Lemma}
\newtheorem{cor}[thm]{Corollary}

\newtheorem{thm'}{Theorem}[subsection]
\newtheorem{prop'}[thm']{Proposition}
\newtheorem{lemma'}[thm']{Lemma}
\newtheorem{cor'}[thm']{Corollary}

\theoremstyle{definition}
\newtheorem{defi}[thm]{Definition}
\newtheorem{defi'}[thm']{Definition}

\theoremstyle{remark}
\newtheorem{prob}[thm]{Problem}
\newtheorem{rem}[thm]{Remark}
\newtheorem{rem'}[thm']{Remark}

\newtheorem{exam'}[thm']{Example}
\newtheorem{prob'}[thm']{Problem}

\newtheorem{quest'}[thm']{Question}
\numberwithin{equation}{section}
\author{Hideaki \=Oshima}
\address[H. \=Oshima]{Professor Emeritus, Ibaraki University, Mito, Ibaraki 310-8512, Japan}
\email{hideaki.ooshima.mito@vc.ibaraki.ac.jp}

\author{Katsumi \=Oshima}
\address[K. \=Oshima]{1-4001-5 Ishikawa, Mito, Ibaraki 310-0905, Japan}
\email{k-oshima@mbr.nifty.com}
\begin{document}
\allowdisplaybreaks
\title{Unstable higher Toda brackets II}
\date{18 May, 2017; revised 4 October, 2020}
\maketitle
\begin{abstract}
We define and study the subscripted $n$-fold Toda bracket 
$\{\vec{\bm f}\}^{(\star)}_{\vec{\bm m}}$ which is a subset of $[\Sigma^{n-2}\Sigma^{|\vec{\bm m}|}X_1,X_{n+1}]$ and possibly empty, 
where $n$ is an integer $\ge 3$, $\vec{\bm m}=(m_n,\dots,m_1)$ is 
a sequence of non negative integers, $|\vec{\bm m}|=m_n+\cdots+m_1$, $\Sigma^i$ is the $i$-fold suspension, 
$\vec{\bm f}=(f_n,\dots,f_1)$ is a sequence of pointed maps 
$f_i:\Sigma^{m_i}X_i\to X_{i+1}$ between well pointed spaces, 
and $\star$ is any one of twelve symbols defined in our previous paper. 
We define also the subscripted stable $n$-fold Toda bracket. 
\end{abstract}

\tableofcontents

\section{Introduction}
This is the continuation of our previous paper \cite{OO}. 

Let $n\ge 3$ be an integer, $(X_{n+1},\dots,X_1)$ a sequence of well-pointed spaces, and $\vec{\bm m}=(m_n,\dots,m_1)$ a sequence of non negative integers. 
We denote by $(\vec{\bm f};\vec{\bm m})$ or simply $\vec{\bm f}$
($=(f_n,\dots,f_1)$) a sequence 
of pointed maps $f_i:\Sigma^{m_i}X_i\to X_{i+1}\,(1\le i\le n)$, where $\Sigma^kX=X\wedge \s^k$ is the $k$-fold suspension of a pointed space $X$. 
We will define the subscripted $n$-fold Toda bracket 
$\{\vec{\bm f}\,\}^{(\star)}_{\vec{\bm m}}\subset[\Sigma^{n-2}\Sigma^{|\vec{\bm m}|}X_1,X_{n+1}]$, 
where $|\vec{\bm m}|=m_n+\cdots+m_1$, $[X,Y]$ is the set of homotopy classes of pointed 
maps $X\to Y$, and $\star$ is any one of twelve symbols $q$, $aq$, $q_2$, $qs_2$, 
$q\dot{s}_2$, $q\ddot{s}_2$, $aqs_2$, $aq\dot{s}_2$, $aq\ddot{s}_2$, $s_t$, $\dot{s}_t$, 
$\ddot{s}_t$ which were defined in \cite{OO} and three of them $aq$, $aq\ddot{s}_2$, 
$\ddot{s}_t$ are essential. 
When $\vec{\bm m}=(0,\dots,0)$, $\{\vec{\bm f}\,\}^{(\star)}_{\vec{\bm m}}$ 
is equal to $\{\vec{\bm f}\,\}^{(\star)}$ which is the object defined and studied in \cite{OO}. 
We will use notations of \cite{OO} freely with the exception: 
we denote the group of all homotopy classes of stable maps from $X$ to $Y$ by 
$\langle X,Y\rangle$ instead of $\{X,Y\}$ i.e. $\langle X,Y\rangle=\lim_{k\to\infty}[\Sigma^kX,\Sigma^kY]$. 
In addition, we will use the following notations: 
$\Sigma^{m_{[k,\ell]}}=\Sigma^{m_k}\cdots \Sigma^{m_\ell}\ i.e.\ \Sigma^{m_{[k,\ell]}}X=X\wedge\s^{m_\ell}\wedge\cdots\wedge\s^{m_k}$ and 
$\s^{m_{[k,\ell]}}=\s^{m_\ell}\wedge\cdots\wedge\s^{m_k}$ for $n\ge k\ge\ell\ge 1$, 
and for convenience $\Sigma^{m_{[n,n+1]}}=\Sigma^0$ i.e. $\Sigma^{m_{[n,n+1]}}X=X$. 
Our main results are (1.1)--(1.16) below; (1.2)--(1.15) are direct generalizations of results in \cite[Section 6 (except \S 6.5)]{OO}; (1.16) relates with \cite[Example 6.2.5]{OO}. 
\begin{enumerate}
\item[(1.1)] 
$\{\vec{\bm f}\,\}^{(\star)}_{\vec{\bm m}}\subset\{f_n,\Sigma^{m_n}f_{n-1},\Sigma^{m_{[n,n-1]}}f_{n-2},\dots,\Sigma^{m_{[n,2]}}f_1\}^{(\star)}$, where the inclusion is the identity when $\vec{\bm m}=(0,\dots,0)$.
\item[(1.2)]
$\{\vec{\bm f}\,\}^{(\ddot{s}_t)}_{\vec{\bm m}}=\{\vec{\bm f}\,\}^{(\dot{s}_t)}_{\vec{\bm m}}=\{\vec{\bm f}\,\}^{(s_t)}_{\vec{\bm m}}$. 
\item[(1.3)] We have
\begin{align*}\{\vec{\bm f}\,\}^{(aqs_2)}_{\vec{\bm m}}&=\{\vec{\bm f}\,\}^{(aq\dot{s}_2)}_{\vec{\bm m}}\\
&=\{\vec{\bm f}\,\}^{(aq\ddot{s}_2)}_{\vec{\bm m}}\circ 
\Sigma^{n-3}\big((1_{\Sigma^{m_1}X_1}\wedge\tau(\s^1,\s^{m_{[n,2]}}))
\circ \Sigma^{m_{[n,2]}}\mathscr{E}(\Sigma \Sigma^{m_1}X_1)\\
&\hspace{5cm}\circ(1_{\Sigma^{m_1}X_1}\wedge\tau(\s^{m_{[n,2]}},\s^1))\big)\\
&\subset \{\vec{\bm f}\,\}^{(aq\ddot{s}_2)}_{\vec{\bm m}}\circ 
(1_{\Sigma^{m_{[n-2,1]}}X_1}\wedge\tau(\s^{n-2},\s^{m_{[n,n-1]}}))\\
&\hspace{1.5cm}\circ \Sigma^{m_{[n,n-1]}}\mathscr{E}(\Sigma^{n-2}\Sigma^{m_{[n-2,1]}}X_1)\circ (1_{\Sigma^{m_{[n-2,1]}}X_1}\wedge\tau(\s^{m_{[n,n-1]}},\s^{n-2}))\\
&=\{\vec{\bm f}\,\}^{(qs_2)}_{\vec{\bm m}}=\{\vec{\bm f}\,\}^{(q\dot{s}_2)}_{\vec{\bm m}};\\
\{\vec{\bm f}\,\}^{(q\dot{s}_2)}_{\vec{\bm m}}&=\{\vec{\bm f}\,\}^{(q\ddot{s}_2)}_{\vec{\bm m}}\ \text{for}\ n\ge 4;\  
\{\vec{\bm f}\,\}^{(q\ddot{s}_2)}_{\vec{\bm m}}=\{\vec{\bm f}\,\}^{(aq\ddot{s}_2)}_{\vec{\bm m}}\ \text{for}\ n=3.
\end{align*}
\item[(1.4)] We have
\begin{align*}
\{\vec{\bm f}\,\}^{(q)}_{\vec{\bm m}}&=\{\vec{\bm f}\,\}^{(aq)}_{\vec{\bm m}}\circ(1_{\Sigma^{m_{[n-2,1]}}X_1}\wedge\tau(\s^{n-2},\s^{m_{[n,n-1]}}))\\
&\hspace{1cm}\circ \Sigma^{m_{[n,n-1]}}\mathscr{E}(\Sigma^{n-2}\Sigma^{m_{[n-2,1]}}X_1)\circ (1_{\Sigma^{m_{[n-2,1]}}X_1}\wedge\tau(\s^{m_{[n,n-1]}},\s^{n-2})).
\end{align*}
\item[(1.5)] $\{\vec{\bm f}\,\}^{(\ddot{s}_t)}_{\vec{\bm m}}\circ\Gamma=
\{\vec{\bm f}\,\}^{(aq\ddot{s}_2)}_{\vec{\bm m}}\circ\Gamma$, where $\Gamma$ is 
the subgroup of $\mathscr{E}(\Sigma^{n-2}\Sigma^{|\vec{\bm m}|}X_1)$ defined by 
\begin{align*}
\Gamma=\big(1_{\Sigma^{m_{[n-1,1]}}X_1}\wedge\tau(\s^{n-2},\s^{m_n})\big)
&\circ \Sigma^{m_n}\mathscr{E}(\Sigma^{n-2}\Sigma^{m_{[n-1,1]}}X_1)\\
&\circ \big(1_{\Sigma^{m_{[n-1,1]}}X_1}\wedge\tau(\s^{m_n},\s^{n-2})\big).
\end{align*} 
\item[(1.6)] If $\alpha\in\{\vec{\bm f}\,\}^{(q)}_{\vec{\bm m}}$, 
then there are $\theta,\theta'\in\Gamma'$ such that 
$\alpha\circ\theta\in\{\vec{\bm f}\,\}^{(aq\ddot{s}_2)}_{\vec{\bm m}}$ and 
$\alpha\circ\theta'\in\{\vec{\bm f}\,\}^{(\ddot{s}_t)}_{\vec{\bm m}}$, 
where $\Gamma'$ is the subgroup of $[\Sigma^{n-2}\Sigma^{|\vec{\bm m}|}X_1, 
\Sigma^{n-2}\Sigma^{|\vec{\bm m}|}X_1]$ defined by 
\begin{align*}
\Gamma'=\big(1_{\Sigma^{m_{[n-1,1]}}X_1}\wedge\tau(\s^{n-2},\s^{m_n})\big)
&\circ \Sigma^{m_n}[\Sigma^{n-2}\Sigma^{m_{[n-1,1]}}X_1, \Sigma^{n-2}\Sigma^{m_{[n-1,1]}}X_1]\\
&\circ \big(1_{\Sigma^{m_{[n-1,1]}}X_1}\wedge\tau(\s^{m_n},\s^{n-2})\big).
\end{align*}  
\item[(1.7)] If $\{\vec{\bm f}\,\}^{(\star)}_{\vec{\bm m}}$ is not empty for some $\star$, then $\{\vec{\bm f}\,\}^{(\star)}_{\vec{\bm m}}$ is not empty for all $\star$. 
\item[(1.8)] If $\{\vec{\bm f}\,\}^{(\star)}_{\vec{\bm m}}$ contains $0$ for some $\star$, then $\{\vec{\bm f}\,\}^{(\star)}_{\vec{\bm m}}$ contains $0$ for all $\star$.
\item[(1.9)] Given maps $f_{n+1}:\Sigma^{m_{n+1}}X_{n+1}\to X_{n+2}$ and 
$f_0:\Sigma^{m_0}X_0\to X_1$, where $X_{n+2}, X_0$ are well-pointed spaces and $m_{n+1}, m_0$ are non negative integers, we have
\begin{align*}
&f_{n+1}\circ \Sigma^{m_{n+1}}\{f_n,\dots,f_1\}^{(\star)}_{\vec{\bm m}}\\
&\qquad \subset(-1)^{n\cdot m_{n+1}}\{f_{n+1}\circ \Sigma^{m_{n+1}}f_n,f_{n-1},\dots,f_1\}^{(\star)}_{(m_{n+1}+m_n,m_{n-1},\dots,m_1)};\\
&\{f_{n+1}\circ \Sigma^{m_{n+1}}f_n,f_{n-1},\dots,f_1\}^{(\star)}_{(m_{n+1}+m_n,m_{n-1},\dots,m_1)}\\
&\qquad \subset\{f_{n+1},f_n\circ \Sigma^{m_n}f_{n-1},f_{n-2},\dots,f_1\}^{(\star)}_{(m_{n+1},m_n+m_{n-1},m_{n-2},\dots,m_1)};\\
&\{f_n,\dots,f_2,f_1\circ \Sigma^{m_1}f_0\}^{(\star)}_{\vec{\bm m}}\\
&\qquad \subset\{f_n,\dots,f_3,f_2\circ \Sigma^{m_2}f_1,f_0\}^{(\star)}_{(m_n,\dots,m_3,m_2+m_1,m_0)}\ (\star=aq\ddot{s}_2,\ddot{s}_t);\\
&\{f_n,\dots,f_1\}^{(\star)}_{\vec{\bm m}}\circ \Sigma^{|\vec{\bm m}|+n-2}f_0\subset\{f_n,\dots,f_2,f_1\circ \Sigma^{m_1}f_0\}^{(\star)}_{\vec{\bm m}}\ (\star=aq\ddot{s}_2, \ddot{s}_t).
\end{align*}
\item[(1.10)] For a non negative integer $\ell$, we have 
$$
\Sigma^\ell \{\vec{\bm f}\,\}^{(\star)}_{\vec{\bm m}}\subset
\{\widetilde{\Sigma}^\ell\vec{\bm f}\,\}^{(\star)}_{\vec{\bm m}}
\circ(1_{X_1}\wedge\tau(\s^{|\vec{\bm m}|+n-2},\s^\ell))
=(-1)^{(|\vec{\bm m}|+n)\ell}\{\widetilde{\Sigma}^\ell\vec{\bm f}\,\}^{(\star)}_{\vec{\bm m}},
$$
where 
\begin{align*} 
\widetilde{\Sigma}^\ell f_k&=\Sigma^\ell f_k\circ(1_{X_k}\wedge\tau(\s^\ell,\s^{m_k})):\Sigma^{m_k}\Sigma^\ell X_k
\to \Sigma^\ell X_{k+1},\\
\widetilde{\Sigma}^\ell\vec{\bm f}&=(\widetilde{\Sigma}^\ell f_n,\dots,\widetilde{\Sigma}^\ell f_1). 
\end{align*}
\item[(1.11)] $\{\vec{\bm f}\,\}^{(\star)}_{\vec{\bm m}}$ depends only on the homotopy classes of $f_i\ (1\le i\le n)$ for all $\star$.
\item[(1.12)] When $n=3$, we have $\{\vec{\bm f}\,\}^{(aq\ddot{s}_2)}_{\vec{\bm m}}=\{\vec{\bm f}\,\}^{(\ddot{s}_t)}_{\vec{\bm m}}=\{f_3,f_2,\Sigma^{m_2}f_1\}_{m_3}$, where 
the last bracket is the classical Toda bracket which was denoted by 
$\{f_3,\Sigma^{m_3}f_2,\Sigma^{m_3}\Sigma^{m_2}f_1\}_{m_3}$ in \cite{T} (see \cite[Remark 3.1]{OO1}).
\item[(1.13)] When $n=4$, we have 
\begin{align*}
\{f_4,f_3,f_2,f_1\}^{(\ddot{s}_t)}_{\vec{\bm m}}&=\bigcup\{f_4,[f_3,A_2,\Sigma^{m_3}f_2],(\Sigma^{m_3}f_2,\widetilde{\Sigma}^{m_3}A_1,\Sigma^{m_{[3,2]}}f_1)\}_{m_4}\\
&\hspace{3cm}\circ \Sigma(1_{\Sigma^{m_{[3,1]}}X_1}\wedge\tau(\s^{m_4},\s^1))\\
&=(-1)^{m_4}\bigcup \{f_4,[f_3,A_2,\Sigma^{m_3}f_2],(\Sigma^{m_3}f_2,\widetilde{\Sigma}^{m_3}A_1,\Sigma^{m_{[3,2]}}f_1)\}_{m_4},
\end{align*}
where the unions $\bigcup$ are taken over all triples $\vec{\bm A}=(A_3,A_2,A_1)$ of null homotopies $A_i:f_{i+1}\circ \Sigma^{m_{i+1}}f_i\simeq *\ (i=1,2,3)$ such that 
$$[f_{i+1},A_i, \Sigma^{m_{i+1}}f_i]\circ (\Sigma^{m_{i+1}}f_i,\widetilde{\Sigma}^{m_{i+1}}A_{i-1},\Sigma^{m_{[i+1,i]}}f_{i-1})\simeq *\ (i=2,3).$$
\item[(1.14)] \begin{enumerate}
\item If $\{f_{n-1},\dots,f_1\}^{(q)}_{(m_{n-1},\dots,m_1)}\ni 0$ and $\{f_n,\dots,f_k\}^{(aq\ddot{s}_2)}_{(m_n,\dots,m_k)}=\{0\}$ for all $k$ with $2\le k<n$, then 
$\{\vec{\bm f}\,\}^{(\star)}_{\vec{\bm m}}$ is not empty for all $\star$.
\item If $\{f_n,\dots,f_2\}^{(q)}_{(m_n,\dots,m_2)}\ni 0$ and $\{f_k,\dots,f_1\}^{(aq\ddot{s}_2)}_{(m_k,\dots,m_1)}=\{0\}$ for all $k$ with $2\le k<n$, then 
$\{\vec{\bm f}\,\}^{(\star)}_{\vec{\bm m}}$ is not empty for all $\star$.
\end{enumerate}
\item[(1.15)] Given a sequence $\vec{\bm \theta}=(\theta_n,\dots,\theta_1)$ 
of stable elements $\theta_i\in\langle \Sigma^{m_i}X_i,X_{i+1}\rangle$, 
we can define the stable higher Toda bracket 
$\langle \Sigma^{\vec{\bm m}}\vec{\bm \theta}\rangle^{(\star)}$ and 
the subscripted stable higher Toda bracket 
$\langle \vec{\bm \theta}\rangle_{\vec{\bm m}}^{(\star)}$ such that 
$
\langle \Sigma^{|\vec{\bm m}|+n-2}X_1,X_{n+1}\rangle\supset  
\langle \Sigma^{\vec{\bm m}}\vec{\bm \theta}\rangle^{(\star)}\supset 
\langle \vec{\bm \theta}\rangle_{\vec{\bm m}}^{(\star)}
$, where the second containment is a stable analogy to (1.1). 
\item[(1.16)] 
Suppose that $p$ is an odd prime, $\alpha_1\in \langle \Sigma^{2p-3}\s^3,\s^3\rangle$ is of order $p$, $\overrightarrow{\bm \alpha_1}=(\overbrace{\alpha_1,\dots,\alpha_1}^p)$, and $m_i=2p-3$ for $1\le i\le p$. 
Then, for all $\star$, the stable bracket 
$\langle\overrightarrow{\bm \alpha_1}\rangle^{(\star)}_{\vec{\bm m}}$ 
contains an element of order $p$ and the order of every element of 
$\langle \Sigma^{\vec{\bm m}}\overrightarrow{\bm \alpha_1}\rangle^{(\star)}$ 
is a multiple of $p$. 
\end{enumerate}

In Section 2, we state a known result on loop spaces of a well-pointed space. 
In Section 3, we define $\{\vec{\bm f}\,\}^{(\star)}_{\vec{\bm m}}$ and prove (1.1). 
In Section 4, we prove (1.2)--(1.9). 
For k=5,6,7,8,9,10,11 we prove (1.k+5) in Section k. 
Appendix A is an addendum to \cite{OO}, that is, we give two lemmas which are used in 
Sections 10 and 11, and we define inductively a system of unstable higher Toda brackets in $\mathrm{TOP}^*$. 

In Section 3--Section 11, we work in $\mathrm{TOP}^w$, the category of well-pointed spaces ($w$-spaces for short). We can develop our consideration similarly 
in $\mathrm{TOP}^{clw}$, the category of $w$-spaces with closed base point ($clw$-spaces for short).

\section{Preliminaries}
For spaces $X$ and $Y$, let $\mathrm{Map}(Y,X)$ denote the space of all maps from $Y$ to $X$ having the compact-open topology, 
 that is, the topology (i.e. the set of open sets) 
 is generated by the set of all subsets of $\mathrm{Map}(Y,X)$ of the form 
 $W(K,U)=\{f\in\mathrm{Map}(Y,X)\,|\, f(K)\subset U\}$, where $K$ is a compact subset of $Y$ and $U$ is an open subset of $X$.  
If $X,Y$ are pointed, then $\mathrm{Map}^*(Y,X)$ denotes the subspace of $\mathrm{Map}(Y,X)$ consisting of pointed maps. 

\begin{prop}[cf. Lemma 4 of \cite{S3}]
Let $j:A\to X$ be a free cofibration and $Y$ a compact space. 
Then the following hold. 
\begin{enumerate}
\item[\rm(1)]  The map $j_* : \mathrm{Map}(Y,A)\to \mathrm{Map}(Y,X),\ f\mapsto j\circ f $, is a free cofibration. 
If $X$ is a $w$-space (resp.\,$clw$-space), then $\mathrm{Map}(Y,X)$ is a $w$-space 
(resp.\,$clw$-space) provided we assign to $\mathrm{Map}(Y,X)$ 
the constant map $c_{x_0} : Y\to \{x_0\}\subset X$ as the base point. 
\item[\rm(2)] If $j$ and $Y$ are pointed, then the map $j_*:\mathrm{Map}^*(Y,A)\to \mathrm{Map}^*(Y,X)$ 
is a free cofibration. 
If in addition $X$ is a $w$-space (resp.\,$clw$-space), then $\mathrm{Map}^*(Y,X)$ is a $w$-space 
(resp.\,$clw$-space) with $c_{x_0}$ the base point. 
\end{enumerate}
 \end{prop}
 \begin{proof}
We can assume that $j:A\subset X$ by \cite[Theorem 1]{S1}. 

(1) Note that $j_*:\mathrm{Map}(Y,A)\to \mathrm{Im}(j_*)$ is a homeomorphism.  
 In fact for a compact subset $K$ of $Y$ and an open subset $U$ of $X$ 
 we have $j_*W(K,A\cap U)=W(K,U)\cap\mathrm{Im}(j_*)$. 
 For simplicity we assume that $j_*$ is an inclusion. 
 Let $\alpha : X\to I$ and $h:X\times I\to X$ be a Str{\o}m structure \cite[Lemma 4]{S2} on $(X,A)$, that is, 
$A\subset\alpha^{-1}(0)$, $h(x,0)=x\ (x\in X)$, $h(a,t)=a\ (a\in A, t\in I)$, and $h(x,t)\in A$ whenever $t>\alpha(x)$. 
 Then the pair $(\hat{\alpha},\hat{h})$ defined by 
 \begin{gather*}
 \hat{\alpha}:\mathrm{Map}(Y,X)\to I,\quad \hat{\alpha}(f)=\sup \{\alpha\circ f(y)\,|\,y\in Y\},\\
 \hat{h}: \mathrm{Map}(Y,X)\times I\to\mathrm{Map}(Y,X),\quad \hat{h}(f,t)(y)=h(f(y),t)
 \end{gather*}
 is a Str{\o}m structure on $(\mathrm{Map}(Y,X),\mathrm{Map}(Y,A))$ (here we have used the hypothesis that $Y$ is  compact). 
Hence $j_*$ is a free cofibration. 
In particular, if $X$ is a $w$-space, that is, $i:\{x_0\}\subset X$ is a free cofibration, 
then $i_*:\mathrm{Map}(Y,\{x_0\})=\{c_{x_0}\}\subset\mathrm{Map}(Y,X)$ 
is a free cofibration so that $\mathrm{Map}(Y,X)$ is a $w$-space. 

Suppose that $X$ is a $clw$-space. 
To prove $\{c_{x_0}\}$ is closed in $\mathrm{Map}(Y,X)$, let $f\in\mathrm{Map}(Y,X)-\{c_{x_0}\}$. 
Take $y\in Y$ with $f(y)\ne x_0$. 
Then $f\in W(\{y\},X-\{x_0\})\subset \mathrm{Map}(Y,X)-\{c_{x_0}\}$. 
Hence $\mathrm{Map}(Y,X)-\{c_{x_0}\}$ is open so that $\{c_{x_0}\}$ is closed in $\mathrm{Map}(Y,X)$. 
Therefore $\mathrm{Map}(Y,X)$ is a $clw$-space. 
 
(2) As is easily seen, the Str{\o}m structure $(\hat{\alpha},\hat{h})$ on $(\mathrm{Map}(Y,X),\mathrm{Map}(Y,A))$ gives a Str{\o}m structure on $(\mathrm{Map}^*(Y,X),\mathrm{Map}^*(Y,A))$. 
\end{proof}

\begin{cor}
If $X$ is a $w$-space (resp.\,$clw$-space), then $\Omega^\ell X=\mathrm{Map}^*(\s^\ell,X)$, 
the $\ell$-fold loop space of $X$, is a $w$-space (resp.\,$clw$-space) (cf. \cite[p.180]{DKP}) for any non negative integer $\ell$. 
(Note that $\Omega^0X=X$.)
\end{cor}
\begin{proof}
By taking $j:\{x_0\}\subset X$ and $Y=\s^\ell$ in Proposition 2.1(2), we have the assertion. 
\end{proof}

Given a pointed homeomorphism $h:X\approx Y$, we define
$$
h_\#:\mathscr{E}(X)\to\mathscr{E}(Y),\ \varepsilon\mapsto h\circ\varepsilon\circ h^{-1}.
$$
Then, as is easily seen, $h_\#$ is an isomorphism with ${h^{-1}}_\#$ an inverse. 

\section{Definition of $\{\vec{\bm f}\,\}^{(\star)}_{\vec{\bm m}}$}
As mentioned in the introduction, 
we will work throughout Sections 3--11 in $\mathrm{Top}^w$. 
(We can develop our consideration similarly in $\mathrm{Top}^{clw}$.)  

Let $\mathscr{S}=(X_1,\Sigma X_2,\Sigma^2X_3,\dots,\Sigma^{n-1}X_n;C_1,\dots,C_{n+1};g_1,\dots,g_n;j_1,\dots,j_n)$ be a quasi 
iterated mapping cone with a quasi-structure $\Omega=\{\omega_i:C_{i+1}\cup_{j_i} CC_i\simeq \Sigma^iX_i\,|\,1\le i\le n\}$ 
of depth $n$ in $\mathrm{Top}^w$. 
Let $\ell$ be any non negative integer. 
It follows from \cite[Corollary 2.3]{OO} that $\Sigma^\ell \Sigma^{i-1}X_i, \Sigma^\ell C_i \in \mathrm{Top}^w$  
and $\Sigma^\ell j_i$ is a free cofibration for $1\le i\le n$. 
Set 
\begin{align*}
&\widetilde{\Sigma}^\ell g_i=\Sigma^\ell g_i\circ(1_{X_i}\wedge\tau(\s^\ell,\s^{i-1})):\Sigma^{i-1}\Sigma^\ell X_i\xrightarrow{1_{X_i}\wedge\tau(\s^\ell,\s^{i-1})} \Sigma^\ell \Sigma^{i-1}X_i\xrightarrow{\Sigma^\ell g_i}\Sigma^\ell C_i,\\
&\widetilde{\Sigma}^\ell\mathscr{S}=(\Sigma^\ell X_1,\Sigma\Sigma^{\ell}X_2,\Sigma^2\Sigma^\ell X_3,\dots,\Sigma^{n-1}\Sigma^\ell X_n;\Sigma^\ell C_1,\Sigma^\ell C_2,\dots,\Sigma^\ell C_{n+1}\\
&\hspace{4cm};\widetilde{\Sigma}^\ell g_1,\dots,\widetilde{\Sigma}^\ell g_n;\Sigma^\ell j_1,\dots,\Sigma^\ell j_n),\\
&\widetilde{\Sigma}^\ell\omega_i=(1_{\Sigma^{i-1}X_i}\wedge\tau(\s^i,\s^\ell))\circ \Sigma^\ell\omega_i\circ\psi^\ell_{j_i}:\Sigma^\ell C_{i+1}\cup_{\Sigma^\ell j_i} C\Sigma^\ell C_i\xrightarrow[\approx]{\psi^\ell_{j_i}}\Sigma^\ell(C_{i+1}\cup_{j_i} CC_i)\\
&\hspace{6cm}\xrightarrow[\simeq]{\Sigma^\ell\omega_i}\Sigma^\ell \Sigma^iX_i\xrightarrow[\approx]{1_{X_i}\wedge\tau(\s^i,\s^\ell)}\Sigma^i\Sigma^\ell X_i\ (1\le i\le n),\\
&\widetilde{\Sigma}^\ell\Omega=\{\widetilde{\Sigma}^\ell\omega_i\,|\,1\le i\le n\},\quad
\omega_0=1_{C_1},\ \widetilde{\Sigma}^\ell\omega_0=\Sigma^\ell\omega_0=1_{\Sigma^\ell C_1}.
\end{align*} 
We call $\widetilde{\Sigma}^\ell\mathscr{S}$ the $\ell$-{\it fold suspension} of $\mathscr{S}$. 

\begin{prop}
Under the situation above, we have
\begin{enumerate}
\item[\rm(1)] $\widetilde{\Sigma}^\ell\mathscr{S}$ is a quasi iterated mapping cone of depth $n$ with a quasi-structure $\widetilde{\Sigma}^\ell\Omega$. 
\item[\rm(2)] If $\mathscr{S}$ is an iterated mapping cone of 
depth $n$ with $\mathscr{A}=\{a_i\,|\,1\le i\le n\}$ a structure, 
then $\widetilde{\Sigma}^\ell\mathscr{S}$ is an iterated mapping cone of depth $n$ with a structure 
$\widetilde{\Sigma}^\ell\mathscr{A}=\{\widetilde{\Sigma}^\ell a_i\,|\,1\le i\le n\}$ and 
$\Omega(\widetilde{\Sigma}^\ell\mathscr{A})=\widetilde{\Sigma}^\ell(\Omega(\mathscr{A}))$, 
where 
\begin{align*}
&\widetilde{\Sigma}^\ell a_i=(1_{\Sigma^\ell C_i}\cup C(1_{X_i}\wedge\tau(\s^{i-1},\s^\ell)))\circ(\psi^{\ell}_{g_i})^{-1}\circ \Sigma^\ell a_i\\
& : \Sigma^\ell C_{i+1}\xrightarrow[\simeq]{\Sigma^\ell a_i} \Sigma^\ell (C_i\cup_{g_i}C\Sigma^{i-1}X_i)
\xrightarrow[\approx]{(\psi^\ell_{g_i})^{-1}} \Sigma^\ell C_i\cup_{\Sigma^\ell g_i}C\Sigma^\ell \Sigma^{i-1}X_i\\
&\hspace{2cm}
\xrightarrow[\approx]{1_{\Sigma^\ell C_i}\cup C(1_{X_i}\wedge\tau(\s^{i-1},\s^\ell))} \Sigma^\ell C_i \cup_{\widetilde{\Sigma}^\ell g_i}C\Sigma^{i-1}\Sigma^\ell X_i. 
\end{align*}
\end{enumerate}
\end{prop} 
\begin{proof}
(1). 
Given a well-pointed space $Z$, we have the following commutative diagram
$$
\xymatrix{
[\Sigma^{i-1}\Sigma^\ell X_i,Z] &[\Sigma^\ell C_i,Z]\ar[l]_-{{\widetilde{\Sigma}g_i}^*} \ar[ld]^-{{\Sigma^\ell g_i}^*}\ar[dd]^-b_\cong&[\Sigma^\ell C_{i+1},Z]\ar[dd]^-a_\cong \ar[l]_-{{\Sigma^\ell j_i}^*}\\
[\Sigma^\ell \Sigma^{i-1}X_i,Z]\ar[u]^-{{1_{X_i}\wedge\tau(\s^\ell,\s^{i-1})}^*}_\cong \ar[d]^-c_\cong & &\\
[\Sigma^{i-1}X_i,\Omega^\ell Z] &[C_i,\Omega^\ell Z]\ar[l]_-{{g_i}^*} &[C_{i+1},\Omega^\ell Z]\ar[l]_-{{j_i}^*}
}
$$
where $a,b,c$ are adjoint isomorphisms. 
Since $\Omega^\ell Z$ is well-pointed by Corollary 2.2, and since $\mathscr{S}$ is a quasi iterated mapping cone, 
the lower sequence is exact so that the upper sequence is exact. 
Since $\omega_i$ is a homotopy equivalence, 
$\widetilde{\Sigma}^\ell\omega_i=(1_{\Sigma^{i-1}X_i}\wedge\tau(\s^i,\s^\ell))\circ \Sigma^\ell\omega_i\circ\psi^\ell_{j_i}$ is a homotopy equivalence for $1\le i\le n$. 
Thus $\widetilde{\Sigma}^\ell\mathscr{S}$ is a quasi iterated mapping cone of depth $n$ with a quasi-structure 
$\widetilde{\Sigma}^\ell\Omega$. 

We omit the proof of (2), since it is easy.
\end{proof}

Under the situation above, suppose that a map $f:\Sigma^\ell C_1\to Y$ is given and 
has an extension $\overline{f}:\Sigma^\ell C_{n+1}\to Y$. 
Then $(\widetilde{\Sigma}^\ell\mathscr{S})(\overline{f},\widetilde{\Sigma}^\ell \Omega)$ 
is a reduced iterated mapping cone with a reduced quasi-structure 
$\widetilde{\widetilde{\Sigma}^\ell\Omega}
=\{\widetilde{\widetilde{\Sigma}^\ell \omega_i}\,|\,0\le i\le n\}$, where 
\begin{align*}
&\widetilde{\widetilde{\Sigma}^\ell \omega_0}=q_{f}':(Y\cup_f C\Sigma^\ell C_1)\cup_{i_f} CY\to \Sigma\Sigma^\ell C_1,\\ 
&\widetilde{\widetilde{\Sigma}^\ell \omega_i}=\Sigma\widetilde{\Sigma}^\ell \omega_i\circ q_{\overline{f}^{i+1}\cup C\overline{f}^i}\circ\xi_{i+1}:
(Y\cup_{\overline{f}^{i+1}}C\Sigma^\ell C_{i+1})\cup_{1_Y\cup C\Sigma^\ell j_i} C(Y\cup_{\overline{f}^i} C\Sigma^\ell C_i)\\
&\hspace{2cm}\xrightarrow[\approx]{\xi_{i+1}} (Y\cup CY)\cup_{\overline{f}^{i+1}\cup C\overline{f}^i} C(\Sigma^\ell C_{i+1}\cup_{\Sigma^\ell j_i} C\Sigma^\ell C_i)\\
&\hspace{2cm}\xrightarrow[\simeq]{q} \Sigma(\Sigma^\ell C_{i+1}\cup_{\Sigma^\ell j_i} C\Sigma^\ell C_i)\\
&\hspace{2cm}\xrightarrow[\simeq]{\Sigma\widetilde{\Sigma}^\ell\omega_i} \Sigma\Sigma^i\Sigma^\ell X_i\qquad(1\le i\le n)\quad (\text{see \cite[(5.3)]{OO}}).
\end{align*}  
We have 
\begin{align*}
&(\widetilde{\Sigma}^\ell\mathscr{S})(\overline{f},\widetilde{\Sigma}^\ell\Omega)\\
&=
(\Sigma^\ell C_1,\Sigma\Sigma^\ell X_1,\Sigma^2\Sigma^\ell X_2,\dots,\Sigma^n\Sigma^\ell X_n; Y,Y\cup_{\overline{f}^1} C\Sigma^\ell C_1,\dots,Y\cup_{\overline{f}^{n+1}} C\Sigma^\ell C_{n+1};\\
&\hspace{2cm}\widetilde{g}_1,\widetilde{g}_2,\dots,\widetilde{g}_n;i_f,1_Y\cup C\Sigma^\ell j_1,\dots,1_Y\cup C\Sigma^\ell j_n),
\end{align*} 
where $\widetilde{g}_1=f$ and $\widetilde{g}_i=(\overline{f}^i\cup C1_{\Sigma^\ell C_{i-1}})\circ(\widetilde{\Sigma}^\ell \omega_{i-1})^{-1}:\Sigma^{i-1}\Sigma^\ell X_{i-1}\to Y\cup C\Sigma^\ell C_{i-1}$ for $i\ge 2$. 
If $\mathscr{S}$ has a structure $\mathscr{A}$, then 
$(\widetilde{\Sigma}^\ell\mathscr{S})(\overline{f},\widetilde{\Sigma}^\ell\mathscr{A})$ 
is a reduced iterated mapping cone with a reduced quasi-structure 
$\widetilde{\widetilde{\Sigma}^\ell\Omega(\mathscr{A})}$. 

Let $n\ge 3$ be an integer. 
Given a sequence of non-negative integers $\vec{\bm m}=(m_n,\dots,m_2,m_1)$, we set $|\vec{\bm m}|=m_n+\cdots+m_1$ and, for $n\ge k\ge \ell\ge 1$, we denote 
the iterated suspension $\Sigma^{m_k}\cdots \Sigma^{m_{\ell+1}}\Sigma^{m_\ell}$ by $\Sigma^{m_{[k,\ell]}}$, 
that is, $\Sigma^{m_{[k,\ell]}}X=X\wedge\s^{m_\ell}\wedge\cdots\wedge\s^{m_k}$, 
and $\s^{m_{[k,\ell]}}=\Sigma^{m_{[k,\ell]}}\s^0=\s^{m_\ell}\wedge\cdots\wedge\s^{m_k}$. 
We denote by $(\vec{\bm \alpha};\vec{\bm m})$ ($\vec{\bm \alpha}$ for short) 
a sequence of homotopy classes $\alpha_i\in[\Sigma^{m_i}X_i,X_{i+1}]\ (1\le i\le n)$, that is, 
$\vec{\bm\alpha}=(\alpha_n,\dots,\alpha_1)$. 
If $f_i:\Sigma^{m_i}X_i\to X_{i+1}$ represents $\alpha_i$, then the sequence $\vec{\bm f}=(f_n,\dots,f_1)$ is called a representative of $(\vec{\bm \alpha};\vec{\bm m})$. 
We denote by $\mathrm{Rep}(\vec{\bm\alpha};\vec{\bm m})$ ($\mathrm{Rep}(\vec{\bm\alpha})$ for short) the set of representatives of $(\vec{\bm \alpha};\vec{\bm m})$. 
Given $\vec{\bm f}\in\mathrm{Rep}(\vec{\bm \alpha};\vec{\bm m})$, 
we will define $\{\vec{\bm f}\,\}_{\vec{\bm m}}^{(\star)}\subset[\Sigma^{n-2}\Sigma^{|\vec{\bm m}|}X_1,X_{n+1}]$. 
We consider collections 
$\{\mathscr{S}_r,\overline{f_r},\Omega_r\,|\,2\le r\le n\}$, $\{\mathscr{S}_2,\overline{f_2},\Omega_2\}\cup\{\mathscr{S}_r,\overline{f_r},\mathscr{A}_r\,|\,3\le r\le n\}$, and $\{\mathscr{S}_r,\overline{f_r},\mathscr{A}_r\,|\,2\le r\le n\}$ 
(provided $\mathscr{S}_2$ is an iterated mapping cone) which satisfy (i)--(iii) below. 
\begin{enumerate}
\item[(i)] $\mathscr{S}_2$ is a quasi iterated mapping cone of depth $1$ as displayed in
$$
\begin{CD}
\Sigma^{m_1}X_1 @.\\
@Vf_1VV @.\\
X_2 @>j_{2,1}>> C_{2,2}
\end{CD}
$$
with $\Omega_2$ a quasi-structure and $\mathscr{A}_2$ a structure provided $\mathscr{S}_2$ is an iterated mapping cone. (Note that $\widetilde{\Sigma}^{m_2}\mathscr{S}_2=\Sigma^{m_2}\mathscr{S}_2$.)
\item[(ii)] $\mathscr{S}_r$ is an iterated mapping cone of depth $r-1$ for $3\le r\le n$ as displayed in 
$$
\xymatrix{
\Sigma^{m_{r-1}}X_{r-1} \ar[d]^-{f_{r-1}=g_{r,1}} & \Sigma\Sigma^{m_{[r-1,r-2]}}X_{r-2} \ar[d]^-{g_{r,2}} 
&\dots &\Sigma^{r-2}\Sigma^{m_{[r-1,1]}}X_1 \ar[d]^-{g_{r,r-1}}& \\
X_r \ar[r]^-{j_{r,1}} & C_{r,2} \ar[r]^-{j_{r,2}} &\dots \ar[r]^-{j_{r,r-2}} & C_{r,r-1} 
\ar[r]^-{j_{r,r-1}} & C_{r,r}
}
$$
with $\Omega_r=\{\omega_{r,s}\,|\,1\le s<r\}$ a quasi-structure and $\mathscr{A}_r$ a structure. 
(Notice that $\omega_{r,s}:C_{r,s+1}\cup_{j_{r,s}}CC_{r,s}\simeq \Sigma^s\Sigma^{m_{[r-1,r-s]}}X_{r-s}$.) 
\item[(iii)] $\overline{f_r}:\Sigma^{m_r}C_{r,r}\to X_{r+1}$ is an extension of 
$f_r$ to $\Sigma^{m_r}C_{r,r}$ for $2\le r\le n-1$, 
and $\overline{f_n}:\Sigma^{m_n}C_{n,n-1}\to X_{n+1}$ is an extension of $f_n$ to 
$\Sigma^{m_n}C_{n,n-1}$. 
\end{enumerate}

\begin{defi}
Various presentations of $\vec{\bm f}$ are defined as follows. 
\begin{enumerate}
\item A collection $\{\mathscr{S}_r,\overline{f_r},\Omega_r\,|\,2\le r\le n\}$ is 
\begin{enumerate}
\item a $q$-{\it presentation} if $\mathscr{S}_{r+1}=(\widetilde{\Sigma}^{m_r}\mathscr{S}_r)(\overline{f_r},\widetilde{\Sigma}^{m_r}\Omega_r)$ 
for $2\le r\le n-1$;
\item[(a$'$)] a $qs_2$-{\it presentation} if it is a $q$-presentation and $\mathscr{S}_2$ is an iterated mapping cone;
\item[(b)] a {\it $q\dot{s}_2$-presentation} if it is a $qs_2$-presentation and $\mathscr{S}_2$ is reduced;
\item[(c)] a {\it $q\ddot{s}_2$-presentation} if it is a $q\dot{s}_2$-presentation and $\Omega_2$ is reduced;
\item[(d)] an {\it $aq$-presentation} if $\mathscr{S}_{r+1}=(\widetilde{\Sigma}^{m_r}\mathscr{S}_r)(\overline{f_r},\widetilde{\Sigma}^{m_r}\Omega_r)$ and $\Omega_{r+1}=\widetilde{\widetilde{\Sigma}^{m_r}\Omega_r}$ for $2\le r< n$;
\item[(d$'$)] an {\it $aqs_2$-presentation} if it is an $aq$-presentation and $\mathscr{S}_2$ is an iterated mapping cone;
\item[(e)] an {\it $aq\dot{s}_2$-presentation} if it is an $aqs_2$-presentation and $\mathscr{S}_2$ is reduced; 
\item[(f)] an {\it $aq\ddot{s}_2$-presentation} if it is an $aq\dot{s}_2$-presentation and $\Omega_2$ is reduced.
\end{enumerate}
Notice that $\widetilde{\widetilde{\Sigma}^{m_r}\Omega_r}=\{\widetilde{\widetilde{\Sigma}^{m_r}\omega_{r,s}}\,|\,0\le s<r\}$, 
where 
\begin{align*}
&\widetilde{\widetilde{\Sigma}^{m_r}\omega_{r,s}}=\Sigma\widetilde{\Sigma}^{m_r}\omega_{r,s}\circ q\circ\xi\\
&\ : C_{r+1,s+2}\cup CC_{r+1,s+1}=(X_{r+1}\cup_{\overline{f_r}^{s+1}}C\Sigma^{m_r}C_{r,s+1})\cup C(X_{r+1}\cup_{\overline{f_r}^{s}}C\Sigma^{m_r}C_{r,s})\\
&\hspace{4cm}\xrightarrow[\approx]{\xi} (X_{r+1}\cup CX_{r+1})\cup C(\Sigma^{m_r}C_{r,s+1}\cup C\Sigma^{m_r}C_{r,s})\\
&\hspace{4cm}\xrightarrow[\simeq]{q} \Sigma(\Sigma^{m_r}C_{r,s+1}\cup C\Sigma^{m_r}C_{r,s})\\
&\hspace{4cm}\xrightarrow[\simeq]{\Sigma\widetilde{\Sigma}^{m_r}\omega_{r,s}}\Sigma\Sigma^s\Sigma^{m_{[r,
r-s]}}X_{r-s}
\end{align*}
for $0\le s<r$. 
\item[(2)] A collection 
$\{\mathscr{S}_2,\overline{f_2},\Omega_2\}\cup\{\mathscr{S}_r,\overline{f_r},\mathscr{A}_r\,|\,3\le r\le n\}$ is a 
{\it $q_2$-presentation} if $\mathscr{S}_3=(\widetilde{\Sigma}^{m_2}\mathscr{S}_2)(\overline{f_2},\widetilde{\Sigma}^{m_2}\Omega_2)$, 
$\mathscr{S}_{r+1}=(\widetilde{\Sigma}^{m_r}\mathscr{S}_r)(\overline{f_r},\widetilde{\Sigma}^{m_r}\mathscr{A}_r)\ (3\le r<n)$, and $\mathscr{A}_{r}$ is reduced for 
$3\le r\le n$. 
\item[(3)] A collection $\{\mathscr{S}_r,\,\overline{f_r},\,\mathscr{A}_r\,|\,2\le r\le n\}$ is
\begin{enumerate}
\item[(g)] an {\it $s_t$-presentation} if  
$\mathscr{S}_{r+1}=(\widetilde{\Sigma}^{m_r}\mathscr{S}_r)(\overline{f_r},\widetilde{\Sigma}^{m_r}\mathscr{A}_r)\ (2\le r<n)$ and $\mathscr{A}_r$ is reduced for $3\le r\le  n$; 
\item[(h)] an {\it $\dot{s}_t$-presentation} if it is an $s_t$-presentation and $\mathscr{S}_2$ is reduced; 
\item[(i)] an {\it $\ddot{s}_t$-presentation} if it is an $\dot{s}_t$-presentation and $\mathscr{A}_2$ is reduced. 
\end{enumerate}
\item[(4)] Let $\star$ denote one of the following: 
$$
q,\ aq,\ qs_2,\ q\dot{s}_2,\ q\ddot{s}_2,\ aqs_2,\ aq\dot{s}_2,\ aq\ddot{s}_2,\ s_t,\ \dot{s}_t,\ \ddot{s}_t,\ q_2.
$$ 
$\vec{\bm f}$ is {\it $\star$-presentable} if it has a $\star$-presentation, and 
$\vec{\bm \alpha}$ is {\it $\star$-presentable} if it has a $\star$-presentable representative. 
\end{enumerate}
\end{defi}

\begin{defi}
We denote by $\{\vec{\bm f}\,\}^{(\star)}_{\vec{\bm m}}$ the set of homotopy classes of 
\begin{align*}
&\overline{f_n}\circ\widetilde{\Sigma}^{m_n}g_{n,n-1}=\overline{f_n}\circ \Sigma^{m_n}g_{n,n-1}\circ (1_{\Sigma^{m_{[n-1,1]}}X_1}\wedge\tau(\s^{m_n},\s^{n-2}))\\
&:\Sigma^{n-2}\Sigma^{m_{[n,1]}}X_1\to 
\Sigma^{m_n}\Sigma^{n-2}\Sigma^{m_{[n-1,1]}}X_1 \xrightarrow{\Sigma^{m_n}g_{n,n-1}}\Sigma^{m_n}C_{n,n-1}\xrightarrow{\overline{f_n}} X_{n+1}
\end{align*} 
for all $\star$-presentations of 
$\vec{\bm f}$. 
It is a subset of $[\Sigma^{n-2}\Sigma^{|\vec{\bm m}|}X_1,X_{n+1}]$.
\end{defi}

\begin{rem}
\begin{enumerate}
\item We abbreviate $\{\vec{\bm f}\,\}^{(\star)}_{(0,\dots,0)}$ to 
$\{\vec{\bm f}\,\}^{(\star)}$, because they are the same by definitions. 
\item By replacing $\Sigma^{m_1}X_1$ with $X_1$, we can suppose $m_1=0$ without loosing any generality. 
\item It follows from definitions that if $\vec{\bm \alpha}$ is $q$-presentable, then $\alpha_{r+1}\circ \Sigma^{m_{r+1}}\alpha_r=0$ 
for $1\le r\le n-1$, and that we have 
the commutative diagram 
\begin{equation}
\begin{split}
\xymatrix{
\{\vec{\bm f}\,\}^{(q\ddot{s}_2)}_{\vec{\bm m}} \ar[r]^-{!!} &\{\vec{\bm f}\,\}^{(q\dot{s}_2)}_{\vec{\bm m}} \ar[drr]^-{!}& &\{\vec{\bm f}\,\}^{(\ddot{s}_t)}_{\vec{\bm m}} \ar[r]^-{!} & \{\vec{\bm f}\,\}^{(\dot{s}_t)}_{\vec{\bm m}} 
\ar[d]^-{!}\\
\{\vec{\bm f}\,\}^{(aq\ddot{s}_2)}_{\vec{\bm m}} \ar[r]  \ar[u] & \{\vec{\bm f}\,\}^{(aq\dot{s}_2)}_{\vec{\bm m}}  \ar[r]^-{!} \ar[u] & \{\vec{\bm f}\,\}^{(aqs_2)}_{\vec{\bm m}} \ar[d]^-{\#} \ar[r] & \{\vec{\bm f}\,\}^{(qs_2)}_{\vec{\bm m}} \ar[d]^-{\#} & \{\vec{\bm f}\,\}^{(s_t)}_{\vec{\bm m}} \ar[l] \ar[d]\\
& &   \{\vec{\bm f}\,\}^{(aq)}_{\vec{\bm m}} \ar[r] & \{\vec{\bm f}\,\}^{(q)}_{\vec{\bm m}} & \{\vec{\bm f}\,\}^{(q_2)}_{\vec{\bm m}} \ar[l]
}
\end{split}
\end{equation}
where arrows are inclusions, $!!$ is $=$ for $n\ge 4$, and 
four $!$'s are $=$ as will be shown in Theorem 4.1. 
Notice that if Problem 5.6 of \cite{OO} is affirmative, two $\#$'s will be $=$. 
\end{enumerate}
\end{rem}

The following two propositions are easy consequences of definitions.  

\begin{prop}
Let $\{\mathscr{S}_r,\overline{f_r},\Omega_r\,|\,2\le r\le n\}$ be a $q$-presentation of $\vec{\bm f}$. 
Then 
\begin{align*}
C_{r,2}&=X_r\cup_{\overline{f_{r-1}}^1}C\Sigma^{m_{r-1}}X_{r-1}\ (3\le r\le n),\quad C_{3,3}=X_3\cup_{\overline{f_2}}C\Sigma^{m_2}C_{2,2},\\ 
C_{r,s}&=X_r\cup_{\overline{f_{r-1}}^{s-1}}C\Sigma^{m_{r-1}}(X_{r-1}\cup_{\overline{f_{r-2}}^{s-2}}C\Sigma^{m_{r-2}}(X_{r-2}\cup\cdots\\
&\hspace{1cm} \cup_{\overline{f_{r-s+2}}^2}C\Sigma^{m_{r-s+2}}(X_{r-s+2}\cup_{\overline{f_{r-s+1}}^1}C\Sigma^{m_{r-s+1}}X_{r-s+1})\cdots))\\
&\hspace{4cm} (3\le s<r\le n),\\
C_{r,r}&=X_r\cup_{\overline{f_{r-1}}^{r-1}}C\Sigma^{m_{r-1}}(X_{r-1}\cup_{\overline{f_{r-2}}^{r-2}} C\Sigma^{m_{r-2}}\\
&\hspace{1cm}( \cdots \cup_{\overline{f_3}^3}C\Sigma^{m_3}(X_3\cup_{\overline{f_2}^2}C\Sigma^{m_2}C_{2,2})\cdots ))\ (4\le r\le n).
\end{align*}
\end{prop}

\begin{prop}
If $\{\vec{\bm f}\}^{(\star)}_{\vec{\bm m}}$ is not empty, then $\{f_k,\dots,f_\ell\}^{(\star)}_{m_k,\dots,m_{\ell}}$ contains $0$ for $1\le \ell<k\le n,\,(k,\ell)\ne(n,1)$.
\end{prop}

\begin{defi}
If there exist null-homotopies $A_i:f_{i+1}\circ \Sigma^{m_{i+1}}f_i\simeq *\ (1\le i\le n-1)$ 
such that $[f_{i+2},A_{i+1},\Sigma^{m_{i+2}}f_{i+1}]\circ(\Sigma^{m_{i+2}}f_{i+1},\widetilde{\Sigma}^{m_{i+2}}A_i,\Sigma^{m_{[i+2,i+1]}}f_i)\simeq *\ (1\le i\le n-2)$, then we call $\vec{\bm f}$ and 
$(\vec{\bm f};\vec{\bm A})$ {\it admissible}, where $\vec{\bm A}=(A_{n-1},\dots,A_1)$. 
We call $\vec{\bm \alpha}$ {\it admissible} if it has an admissible representative. 
\end{defi}

It follows from Proposition 2.11 of \cite{OO1} that if $\vec{\bm \alpha}$ 
is admissible, then every representative of $\vec{\bm \alpha}$ is admissible. 
From results in forthcoming sections, we can prove the following without difficulties: 
when $n=3$, $\vec{\bm \alpha}$ is admissible if and only if 
$\{\vec{\bm \alpha}\}^{(\star)}_{\vec{\bm m}}$ contains $0$ for some and hence 
all $\star$; when $n=4$, $\vec{\bm \alpha}$ is admissible if and only if 
$\vec{\bm \alpha}$ is $\star$-presentable for some and hence all $\star$; 
when $n\ge 5$, if $\vec{\bm \alpha}$ is $\star$-presentable for some and 
hence all $\star$, then $\vec{\bm \alpha}$ is admissible. 

\begin{prop}
$\{f_n,\dots, f_1\}^{(\star)}_{\vec{\bm m}}\subset\{f_n,\Sigma^{m_n}f_{n-1},\Sigma^{m_{[n,n-1]}}f_{n-2},\dots,\Sigma^{m_{[n,2]}}f_1\}^{(\star)}$. 
\end{prop}
\begin{proof}
Given a map $f:X\to Y$ and integers $k,\ell$ with $n\ge k\ge\ell\ge 1$, let 
$$
\psi^{m_{[k,\ell]}}_f:\Sigma^{m_{[k,\ell]}}Y\cup_{\Sigma^{m_{[k,\ell]}}f} C\Sigma^{m_{[k,\ell]}}X\approx \Sigma^{m_{[k,\ell]}}(Y\cup_f CX)
$$ 
be the canonical homeomorphism. 
We set 
\begin{gather*}
X_{n+1}^*=X_{n+1},\quad X_i^*=\Sigma^{m_{[n,i]}}X_i\ (1\le i\le n),\\
f_n^*=f_n:X_n^*\to X_{n+1}^*,\quad f_i^*=\Sigma^{m_{[n,i+1]}}f_i:X_i^*\to X_{i+1}^*\ (1\le i<n).
\end{gather*}
Then the assertion is 
$\{f_n,\dots, f_1\}^{(\star)}_{\vec{\bm m}}\subset\{f_n^*,f_{n-1}^*,\dots,f_1^*\}^{(\star)}$. 
Let $\alpha\in\{f_n,\dots, f_1\}^{(\star)}_{\vec{\bm m}}$ and 
$\{\mathscr{S}_r,\overline{f_r},\Omega_r\ (\text{or }\mathscr{A}_r)\,|\,2\le r\le n\}$ a $\star$-presentation of $\vec{\bm f}$ such that 
$\alpha=\overline{f_n}\circ \widetilde{\Sigma}^{m_n}g_{n,n-1}
$. 
We will define a $\star$-presentation $\{\mathscr{S}_r^*,\overline{f_r^*},\Omega_r^*\ (\text{or }\mathscr{A}_r^*)\,|\,2\le r\le n\}$ of $\vec{\bm f^*}$ 
and homeomorphisms 
$\theta_{r,s}:C_{r,s}^*\approx \Sigma^{m_{[n,r]}}C_{r,s}$ such that 
\begin{align}
&\theta_{r,s+1}\circ j_{r,s}^*=\Sigma^{m_{[n,r]}}j_{r,s}\circ\theta_{r,s},\ \theta_{r,1}=1_{X_r^*},\\
&\overline{f_r^*}=\begin{cases}\Sigma^{m_{[n,r+1]}}\overline{f_r}\circ\theta_{r,r}:C_{r,r}^*\to X_{r+1}^*& 2\le r<n\\ \overline{f_n}\circ\theta_{n,n-1}:C_{n,n-1}^*\to X_{n+1}^*& r=n\end{cases},\\
&\theta_{r,s}=\psi^{m_{[n,r]}}_{\overline{f_{r-1}}^{s-1}}\circ(1_{X_{r}^*}\cup C\theta_{r-1,s-1})\ (2\le s<r\le n),\\
\begin{split}
&\omega_{r,s}^*=\widetilde{\Sigma}^{m_{[n,r]}}\omega_{r,s}\circ\psi^{m_{[n,r]}}_{j_{r,s}}\circ
(\theta_{r,s+1}\cup C\theta_{r,s})\\
&\hspace{1cm}:C_{r,s+1}^*\cup CC_{r,s}^*\to \Sigma^sX_{r-s}^*=\Sigma^s\Sigma^{m_{[n,r-s]}}X_{r-s},
\end{split}
\end{align}
where $\widetilde{\Sigma}^{m_{[n,r]}}\omega_{r,s}$ is the composition of 
\begin{align*}
\Sigma^{m_{[n,r]}}(C_{r,s+1}\cup CC_{r,s})&\xrightarrow{\Sigma^{m_{[n,r]}}\omega_{r,s}}\Sigma^{m_{[n,r]}}\Sigma^s\Sigma^{m_{[r-1,r-s]}}X_{r-s}\\
&\xrightarrow{1_{\Sigma^{m_{[r-1,r-s]}}X_{r-s}}\wedge\tau(\s^s,\s^{m_{[n,r]}})} \Sigma^s\Sigma^{m_{[n,r-s]}}X_{r-s}.
\end{align*}
If this is done, then the assertion is proved as follows. 
Consider the following diagram for $1\le s<r<n$. 
$$
\xymatrix{
C^*_{r,s+1}\cup CC^*_{r,s} \ar[d]_-{\theta_{r,s+1}\cup C\theta_{r,s}} \ar[r]^-{\omega^*_{r,s}} & \Sigma^s \Sigma^{m_{[n,r-s]}}X_{r-s} \ar[dd]^-{1_{\Sigma^{m_{[r-1,r-s]}}X_{r-s}}\wedge\tau(\s^{[n,r]},\s^s)}
\\
\Sigma^{m_{[n,r]}}C_{r,s+1]}\cup C\Sigma^{m_{[n,r]}}C_{r,s} \ar[d]_-{\psi^{m_{[n,r]}}} & \\
\Sigma^{m_{[n,r]}}(C_{r,s+1}\cup CC_{r,s}) \ar[d]_-{\Sigma^{m_{[n,r+1]}}(\psi^{m_r}_{j_{r,s}})^{-1}} 
\ar[r]^-{\Sigma^{m_{[n,r]}}\omega_{r,s}} & \Sigma^{m_{[n,r]}}\Sigma^s\Sigma^{m_{[r-1,r-s]}}X_{r-s} 
\ar[d]^-{
\Sigma^{m_{[n,r+1]}}
(1_{\Sigma^{m_{[r-1,r-s]}}X_{r-s}}\wedge\tau(\s^s,\s^{m_r}))
}
\\
\Sigma^{m_{[n,r+1]}}(\Sigma^{m_r}C_{r,s+1}\cup C\Sigma^{m_r}C_{r,s}) \ar[d]^-{\Sigma^{m_{[n,r+1]}}(\overline{f}_r^{s+1}\cup C1_{\Sigma^{m_r}C_{r,s}})} 
\ar[r]^-{\Sigma^{m_{[n,r+1]}}\widetilde{\Sigma}^{m_r}\omega_{r,s}} & \Sigma^{m_{[n,r+1]}}\Sigma^s\Sigma^{m_{[r,r-s]}}X_{r-s} \ar[d] ^-{\Sigma^{m_{[n,r+1]}}g_{r+1,s+1}}\\
\Sigma^{m_{[n,r+1]}}C_{r+1,s+1} \ar[r]^-= \ar[d]^-{\theta_{r+1,s+1}^{-1}} & \Sigma^{m_{[n,r+1]}}(X_{r+1}\cup C \Sigma^{m_r}C_{r,s}) \ar[d]^-{h_{r+1,s+1}}\\
C^*_{r+1,s+1}=X_{r+1}^*\cup CC^*_{r,s} \ar[r]_-{\overline{f_{r+1}^*}^{s+1}}& Y_{r+1,s+1}
}
$$
Here $Y_{r+1,s+1}=\begin{cases} \Sigma^{m_{[n,r+2]}}X_{r+2} & r\le n-2\\ X_{n+1} & r=n-1\end{cases}$ and $h_{r+1,s+1}=\begin{cases} \Sigma^{m_{[n,r+2]}}\overline{f_{r+1}}^{s+1} & r\le n-2\\ \overline{f_n}^{s+1} & r=n-1\end{cases}$. 
The first and last squares are commutative by (3.5) and (3.3), respectively, and the second square is commutative and the third square is homotopy commutative by the definitions. 
We are going to show the following two assertions:
\begin{enumerate}
\item[(i)] the composite of the maps on the left vertical line, say $e_{r,s}$, is $\overline{f^*_r}^{s+1}\cup C1_{C^*_{r,s}}$,
\item[(ii)] $\theta_{r+1,s+1}\circ g^*_{r+1,s+1}\simeq \widetilde{\Sigma}^{m_{[n,r+1]}}g_{r+1,s+1}$.
\end{enumerate} 
To prove (i), take $x^*\in C^*_{r,s+1}$ and represent $\theta_{r,s+1}(x^*)=x\wedge s_r\wedge\cdots\wedge s_n$ where $x\in C_{r,s+1}$ and $s_i\in\s^{m_i}$. 
Then 
\begin{align*}
&\Sigma^{m_{[n,r+1]}}(\overline{f_r}^{s+1}\cup C1_{\Sigma^{m_r}C_{r,s}})\circ \Sigma^{m_{[n,r+1]}}(\psi^{m_r}_{j_{r,s}})^{-1}\circ\psi^{m_{[n,r]}}_{j_{r,s}}
\circ(\theta_{r,s+1}\cup C\theta_{r,s})(x^*)\\
&=\overline{f_r}^{s+1}(x\wedge s_r)\wedge s_{r+1}\wedge\cdots\wedge s_n.
\end{align*}
On the other hand
\begin{align*}
\overline{f^*_r}^{s+1}(x^*)&=\Sigma^{m_{[n,r+1]}}\overline{f_r}^{s+1}\circ\theta_{r,s+1}(x^*)\quad(\text{by (3.2) and (3.3)})\\
&=\Sigma^{m_{[n,r+1]}}\overline{f_r}^{s+1}(x\wedge s_r\wedge\cdots\wedge s_n)=\overline{f_r}^{s+1}(x\wedge s_r)\wedge s_{r+1}\wedge\cdots\wedge s_n.
\end{align*}
Hence $e_{r,s}(x^*)=\overline{f^*_r}^{s+1}(x^*)$. 
Take $z\in C^*_{r,s}$ and represent $\theta_{r,s}(z)=y\wedge s_r\wedge\cdots\wedge s_n$ where $y\in C_{r,s}$ and $s_i\in\s^{m_i}$. 
Let $t\in I$. 
Then 
\begin{align*}
&\Sigma^{m_{[n,r+1]}}(\overline{f_r}^{s+1}\cup C1_{\Sigma^{m_r}C_{r,s}})\circ \Sigma^{m_{[n,r+1]}}(\psi^{m_r}_{j_{r,s}})^{-1}\circ\psi^{m_{[n,r]}}_{j_{r,s}}
\circ(\theta_{r,s+1}\cup C\theta_{r,s})(z\wedge t)\\
&=y\wedge s_r\wedge t\wedge s_{r+1}\wedge\cdots\wedge s_n.
\end{align*}
On the other hand
\begin{align*}
\theta_{r+1,s+1}(z\wedge t)&=\psi^{m_{[n,r+1]}}_{\overline{f_r}^s}\circ(1_{X^*_{r+1}}\cup C\theta_{r,s})(z\wedge t)\quad(\text{by (3.4)})\\
&=\psi^{m_{[n,r+1]}}_{\overline{f_r}^s}(\theta_{r,s}(z)\wedge t)=\psi^{m_{[n,r+1]}}_{\overline{f_r}^s}(y\wedge s_r\wedge\cdots\wedge s_n\wedge t)\\
&=y\wedge s_r\wedge t\wedge s_{r+1}\wedge\cdots\wedge s_n.
\end{align*}
Hence $e_{r,s}(z)=z$. 
Therefore $e_{r,s}=\overline{f^*_r}^{s+1}\cup C1_{C^*_{r,s}}$. 
This proves (i). 

We have (ii) from (i) and the diagram above as follows: 
\begin{align*}
g^*_{r+1,s+1}&\simeq (\overline{f^*_r}^{s+1}\cup_{\overline{f^*_r}^s} C1_{C^*_{r,s}})\circ(\omega^*_{r,s})^{-1}\\
&\simeq \theta^{-1}_{r+1,s+1}\circ \Sigma^{m_{[n,r+1]}}g_{r+1,s+1}\circ \Sigma^{m_{[n,r+1]}}(1_{\Sigma^{m_{[r-1,r-s]}}X_{r-s}}\wedge\tau(\s^s,\s^{m_r}))\\
&\hspace{1cm}\circ(1_{\Sigma^{m_{[r-1,r-s]}}X_{r-s}}\wedge\tau(\s^{m_{[n,r]}},\s^s))\\
&=\theta^{-1}_{r+1,s+1}\circ \Sigma^{m_{[n,r+1]}}g_{r+1,s+1}\circ (1_{\Sigma^{m_{[r,r-s]}}X_{r-s}}\wedge\tau(\s^{m_{[n,r+1]}},\s^s))\\
&=\theta^{-1}_{r+1,s+1}\circ\widetilde{\Sigma}^{m_{[n,r+1]}}g_{r+1,s+1}.
\end{align*}
This proves (ii). 

Let $r=n-1, s=n-2$. 
Then 
\begin{align*}
\Sigma^{m_n}(1_{\Sigma^{m_{[n-2,1]}}X_1}\wedge\tau(\s^{n-2},\s^{m_{n-1}}))&\circ 
(1_{\Sigma^{m_{[n-2,1]}}X_1}\wedge\tau(\s^{m_{[n,n-1]}},\s^{n-2}))\\
&=1_{\Sigma^{[n-1,1]}X_1}\wedge\tau(\s^{m_n},\s^{n-2}).
\end{align*}
Hence 
\begin{align*}
&\alpha\circ\omega^*_{n-1,n-2}\\
&=\overline{f_n}\circ \Sigma^{m_n}g_{n,n-1}\circ(1_{\Sigma^{[n-1,1]}X_1}\wedge\tau(\s^{m_n},\s^{n-2})) \circ\omega_{n-1,n-2}^*\\
&=\overline{f_n}\circ \Sigma^{m_n}g_{n,n-1}\circ \Sigma^{m_n}(1_{\Sigma^{[n-2,1]}X_1}\wedge\tau(\s^{n-2},\s^{m_{n-1}}))\\
&\hspace{4cm} \circ(1_{\Sigma^{m_{[n-2,1]}}X_1}\wedge\tau(\s^{m_{[n,n-1]}})\circ\omega^*_{n-1,n-2}\\
&= \overline{f_n^*}\circ \Sigma^{m_n}(\overline{f_{n-1}}\cup C1_{\Sigma^{m_{n-1}}C_{n-1,n-2}})\circ \Sigma^{m_n}(\psi^{m_{n-1}}_{j_{n-1,n-2}})^{-1}\\
&\hspace{4cm} \circ\psi^{m_{n,n-1]}}\circ (\theta_{n-1,n-1}\cup C\theta_{n-1,n-2})\\
&=\overline{f_n^*}\circ(\overline{f^*_{n-1}}\cup C1_{C^*_{n-1,n-2}}).
\end{align*}
Hence 
$\alpha=\overline{f_n^*}\circ(\overline{f_{n-1}^*}\cup C1_{C_{n-1,n-2}^*})\circ (\omega_{n-1,n-2}^*)^{-1}=\overline{f_n^*}\circ g^*_{n,n-1}$. 
Thus $\alpha\in\{f_n^*,\dots,f_1^*\}^{(\star)}$ so that the assertion follows. 

First we consider the case $\star=aq\ddot{s}_2$. 
We set 
\begin{equation}
\left\{\begin{array}{@{\hspace{0.4mm}}l}
\mathscr{S}_2^*=(X_1^*;X_2^*,X_2^*\cup_{f_1^*}CX_1^*;f_1^*;i_{f_1^*})\ (\text{hence }
C_{2,2}^*=\Sigma^{m_{[n,2]}}X_2\cup_{f_1^*}C\Sigma^{m_{[n,1]}}X_1),\\
\omega_{2,1}^*=q'_{f_1^*}:C_{2,2}^*\cup CC_{2,1}^*\to \Sigma X_1^*,\ g^*_{2,1}=f_1^*,\\
\theta_{2,1}=1_{X_2^*}:C_{2,1}^*\to \Sigma^{m_{[n,2]}}C_{2,1},\\
\theta_{2,2}=\psi^{m_{[n,2]}}_{f_1}:C_{2,2}^*\to \Sigma^{m_{[n,2]}}C_{2,2},\\ 
\overline{f_2^*}=\Sigma^{m_{[n,3]}}\overline{f_2}\circ\theta_{2,2}:C_{2,2}^*\to X_3^*.
\end{array}\right.
\end{equation}
Then (3.2)--(3.5) hold for $r=2$. 
We set $\mathscr{S}_3^*=\mathscr{S}_2^*(\overline{f_2^*},\Omega_2^*),\ 
\Omega_3^*=\widetilde{\Omega_2^*}$. 
Then 
$$
\left\{\begin{array}{@{\hspace{0.4mm}}l}
C_{3,1}^*=X_3^*,\\
C_{3,2}^*=X_3^*\cup_{f_1^*} CC_{2,1}^*=\Sigma^{m_{[n,3]}}X_3\cup C\Sigma^{m_{[n,2]}}X_2,\\
C_{3,3}^*=X_3^*\cup_{\overline{f_2^*}} CC_{2,2}^*=\Sigma^{m_{[n,3]}}X_3\cup C(\Sigma^{m_{[n,2]}}X_2\cup_{f_1^*}C\Sigma^{m_{[n,1]}}X_1).
\end{array}\right.
$$
We define
$$
\left\{\begin{array}{@{\hspace{0.4mm}}l}
\theta_{3,1}=1_{X_3^*}:C_{3,1}^*\to \Sigma^{m_{[n,3]}}C_{3,1},\\
\theta_{3,2}=\psi^{m_{[n,3]}}_{f_2}\circ(1_{X_3^*}\cup C\theta_{2,1})=\psi^{m_{[n,3]}}_{f_2}:C_{3,2}^*\to \Sigma^{m_{[n,3]}}C_{3,2},\\
\theta_{3,3}=\psi^{m_{[n,3]}}_{\overline{f_2}}\circ(1_{X_3^*}\cup C\theta_{2,2}):C_{3,3}^*\to \Sigma^{m_{[n,3]}}C_{3,3},\\
\overline{f_3^*}=\begin{cases}\Sigma^{m_{[n,4]}}\overline{f_3}\circ\theta_{3,2}:C_{3,2}^*\to X_4^* & n=3\\ \Sigma^{m_{[n,4]}}\overline{f_3}\circ\theta_{3,3}:C_{3,3}^*\to X_4^* & n\ge 4\end{cases}.
\end{array}\right.
$$
Then (3.2)--(3.4) holds for $r=3$. 
We give a proof of (3.5) for $r=3, s=2$ as follows (the case of $r=3,s=1$ is easier). 
Note that $C^*_{3,3}\cup CC^*_{3,2}=(X^*_3\cup C(X^*_2\cup CX^*_1))\cup (X^*_3\cup CX^*_2)$. 
It suffices to show that both sides of (3.5) are the same map on $CCX^*_1$. 
Let $x_1\wedge s_1\wedge\cdots\wedge s_n\wedge u\wedge v$ be any element 
of $CCX_1^*=CC\Sigma^{m_{[n,1]}}X_1$, where 
$x_1\in X_1,\ s_i\in\s^{m_i}\,(1\le i\le n),\ u, v\in I$. 
Then  
\begin{align*}
&\omega^*_{3,2}(x_1\wedge s_1\wedge\cdots\wedge s_n\wedge u\wedge v)=x_1\wedge s_1\wedge\cdots\wedge s_n\wedge\overline{u}\wedge\overline{v}\\
&=(1_{\Sigma^{m_{[2,1]}}X_1}\wedge\tau(\s^2,\s^{m_{[n,3]}}))(x_1\wedge s_1\wedge s_2\wedge\overline{u}\wedge\overline{v}\wedge s_3\wedge\cdots\wedge s_n)\\
&=\widetilde{\Sigma}^{m_{[n,3]}}\omega_{3,2}\circ \psi^{m_{[n,3]}}_{j_{3,2}}\circ(\theta_{3,3}\cup C\theta_{3,2})(x_1\wedge s_1\wedge\cdots\wedge s_n\wedge u\wedge v).
\end{align*}
Hence we obtain (3.5) for $r=3,s=2$. 

When $n=3$, $\{\mathscr{S}^*_r,\overline{f_r^*},\Omega_r^*\,|\,r=2,3\}$ is a desired $aq\ddot{s}_2$-presentation of $\vec{\bm f^*}$. 
Let $n\ge 4$. 
We define $\mathscr{S}_4^*=\mathscr{S}_3^*(\overline{f_3},\Omega_3^*),\ \Omega_4^*=\widetilde{\Omega_3^*}$. 
By proceeding the construction we obtain a desired $aq\ddot{s}_2$-presentation of $\vec{\bm f^*}$. 
Similarly we have a $\star$-presentation of $\vec{\bm f^*}$ for $\star=aq, aqs_2, aq\dot{s}_2$. 

Secondly we consider the case $\star=\ddot{s}_t$. 
Let $\alpha\in\{f_n,\dots, f_1\}^{(\ddot{s}_t)}_{\vec{\bm m}}$ and $\{\mathscr{S}_r,\overline{f_r},\mathscr{A}_r\,|\,2\le r\le n\}$ an $\ddot{s}_t$-presentation of $\vec{\bm f}$ with $\alpha=\overline{f_n}\circ\widetilde{\Sigma}^{m_n}g_{n,n-1}$. 
We will define an $\ddot{s}_t$-presentation $\{\mathscr{S}_r^*,\overline{f_r^*},\mathscr{A}_r^*\,|\,2\le r\le n\}$ of $\vec{\bm f^*}$ and homeomorphisms 
$\theta_{r,s}:C_{r,s}^*\approx \Sigma^{m_{[n,r]}}C_{r,s}$ such that (3.2)--(3.4) 
and the following (3.7)--(3.8) hold ($\Omega_r^*=\Omega(\mathscr{A}^*_r)$).
\begin{align}
&\theta_{r,s}\circ g^*_{r,s}=\widetilde{\Sigma}^{m_{[n,r]}}g_{r,s},\\
&a_{r,s}^*=(\theta_{r,s}^{-1}\cup C1_{\Sigma^{s-1}\Sigma^{m_{[n,r-s]}}X_{r-s}})\circ 
\widetilde{\Sigma}^{m_{[n,r]}}a_{r,s}\circ\theta_{r,s+1}.
\end{align}
If this is done, then (3.5) holds 
so that $\alpha=\overline{f_n^*}\circ\widetilde{\Sigma}^{m_n}g^*_{n,n-1}\in\{\vec{\bm f}\}^{(\ddot{s}_t)}_{\vec{\bm m}}$ as desired. 

Suppose that, given $3\le r\le n$, there are reduced iterated mapping cones $\mathscr{S}_i^*$ with a structure $\mathscr{A}_i^*$, 
an extension $\overline{f_i^*}:C_{i,i}^*\to X_{i+1}^*$ of $f_i^*$, and a homeomorphism $\theta_{i,s}:C_{i,s}^*\approx \Sigma^{m_{[n,i]}}C_{i,s}$ 
such that (3.2)--(3.4), (3.7), (3.8) hold for $s<i<r$.  
Set $\mathscr{S}_r^*=\mathscr{S}^*_{r-1}(\overline{f_{r-1}^*},\mathscr{A}_{r-1}^*)$. 
Define $\theta_{r,s}$, $\overline{f_r^*}$, $g_{r,s}^*$ and $a_{r,s}^*$ by (3.4), (3.3), (3.7) and (3.8), respectively. 
Then (3.2) holds. 

By the consideration above it suffices to define $\mathscr{S}^*_2$, $\overline{f^*_2}$, $\mathscr{A}^*_2$, and $\theta_{2,s}$.  
This is easily done by setting (3.6) and $a^*_{2,1}=1_{C^*_{2,2}}$. 
Then (3.2)--(3.5), (3.7), and (3.8) hold for $r=2$. 
This completes the construction of a desired $\ddot{s}_2$-presentation of $\vec{\bm f^*}$. 
Similarly we have a $\star$-presentation of $\vec{\bm f^*}$ for $\star=s_t, \dot{s}_t$. 

Thirdly we consider the case $\star=qs_2$. 
We set 
\begin{align*}
&\mathscr{S}_2^*=(X_1^*;X_2^*, \Sigma^{m_{[n,2]}}C_{2,2};f^*_1;\Sigma^{m_{[n,2]}}j_{2,1}),\\
&a^*_{2,1}=(\psi^{m_{[n,2]}}_{f_1})^{-1}\circ \Sigma^{m_{[n,2]}}a_{2,1}\\
&\hspace{1cm} :C^*_{2,2}=\Sigma^{m_{[n,2]}}C_{2,2}\xrightarrow{\Sigma^{m_{[n,2]}}a_{2,1}} \Sigma^{m_{[n,2]}}(X_2\cup_{f_1}C\Sigma^{m_1}X_1)\xrightarrow{(\psi^{m_{[n,2]}}_{f_1})^{-1}} X_2^*\cup_{f^*_1} CX^*_1,\\ 
&\mathscr{A}_2^*=\{a^*_{2,1}\},\ \theta_{2,1}=1_{X_2^*}:C^*_{2,1}\to \Sigma^{m_{[n,2]}}X_2,\ \theta_{2,2}=1_{\Sigma^{m_{[n,2]}}C_{2,2}}:C^*_{2,2}\to \Sigma^{m_{[n,2]}}C_{2,2},\\
&\omega^*_{2,1}=\widetilde{\Sigma}^{m_{[n,2]}}\omega_{2,1}\circ\psi^{m_{[n,2]}}_{j_{2,1}}\\
&\hspace{1cm}:C^*_{2,2}\cup CC^*_{2,1}\xrightarrow{\psi^{m_{[n,2]}}_{j_{2,1}}}\Sigma^{m_{[n,2]}}(C_{2,2}\cup CC_{2,1})\xrightarrow{\Sigma^{m_{[ n,2]}}\omega_{2,1}}\Sigma^{m_{[n,2]}}\Sigma\Sigma^{m_1}X_1\\
&\hspace{2cm}\xrightarrow{1_{\Sigma^{m_1}X_1}\wedge\tau(\s^1,\s^{m_{[n,2]}})}\Sigma\Sigma^{m_{[n,1]}}X_1=\Sigma X^*_1,\\
& \Omega^*_2=\{\omega^*_{2,1}\},\ \overline{f_2^*}=\Sigma^{m_{[n,3]}}\overline{f_2},\\
&\theta_{2,1}=1_{\Sigma^{m_{[n,2]}}C_{2,1}}:C_{2,1}^*=X_2^*\to \Sigma^{m_{[n,2]}}C_{2,1},\ \theta_{2,2}=1_{\Sigma^{m_{[n,2]}}C_{2,2}}:C^*_{2,2}\to \Sigma^{m_{[n,2]}}C_{2,2}.
\end{align*}
Then (3.2)--(3.5) hold for $r=2$. 
Set $\mathscr{S}_3^*=\mathscr{S}^*_2(\overline{f_2^*},\Omega_2^*)$. 
Then $C^*_{3,1}=X_3^*$, $C^*_{3,2}=X^*_3\cup_{f^*_2}CX^*_2$ and $C^*_{3,3}=X^*_3\cup_{\overline{f^*_2}}CC^*_{2,2}$. 
Define 
\begin{align*}
&a^*_{2,1}=(\psi^{m_{[n,2]}}_{j_{2,1}})^{-1}\circ \Sigma^{m_{[n,2]}}a_{2,1}:C^*_{2,2}=\Sigma^{m_{[n,2]}}C_{2,2}\to X^*_2\cup_{f^*_1}CX^*_1,\ \mathscr{A}^*_2=\{a^*_{2,1}\},\\
&\theta_{3,1}=1_{X^*_3}:C^*_{3,1}\to \Sigma^{m_{[n,3]}}X_3,\ \theta_{3,2}=\psi^{m_{[n,3]}}_{f_2}:C^*_{3,2}\to \Sigma^{m_{[n,3]}}C_{3,2},\\
&\theta_{3,3}=\psi^{m_{[n,3]}}_{\overline{f_2}}:C^*_{3,3}\to \Sigma^{m_{[n,3]}}(X_3\cup_{\overline{f_2}}C\Sigma^{m_2}C_{2,2})=\Sigma^{m_{[n,3]}}C_{3,3},\\
&\omega^*_{3,1}=q'_{f^*_2}:C^*_{3,2}\cup CC^*_{3,1}\to \Sigma X_1^*,\\ 
&\omega^*_{3,2}=\widetilde{\Sigma}^{m_{[n,3]}}\omega_{3,2}\circ\psi^{m_{[n,3]}}_{j_{3,2}}\circ(\theta_{3,3}\cup C\theta_{3,1}):C^*_{3,3}\to \Sigma^2\Sigma^{m_{[n,1]}}X_1,
\ \Omega^*_3=\{\omega^*_{3,1},\omega^*_{3,2}\},\\
&\overline{f^*_3}=\begin{cases} \Sigma^{m_{[n,4]}}\overline{f_3}\circ\theta_{3,3} & n\ge 4\\ \Sigma^{m_{[n,4]}}\overline{f_3}\circ\theta_{3,2} & n=3\end{cases}.
\end{align*}
Then (3.2)--(3.5) hold for $r=3$. 
By proceeding the construction we obtain a desired $qs_2$-presentation of $\vec{\bm f^*}$. 
Similarly we have a $\star$-presentation of $\vec{\bm f^*}$ for $\star=q_2,q,q\dot{s}_2,q\ddot{s}_2$.
\end{proof}

The proposition above is a partial solution of the following problem.

\begin{prob}
Given $\vec{\bm m}$, $\vec{\bm f}$, and $\vec{\bm \ell}$ with 
$0\le\ell_i\le m_i\ (1\le i\le n)$, we set
\begin{align*}
X_{n+1}^*&=X_{n+1},\ X_i^*=\Sigma^{m_n-l_n}\cdots \Sigma^{m_i-l_i}X_i\ (1\le i\le n),\\
f_n^*&=f_n:\Sigma^{l_n}X_n^*=\Sigma^{l_n}\Sigma^{m_n-l_n}X_n=\Sigma^{m_n}X_n\to X_{n+1}=X_{n+1}^*,\\
f_i^*&=\Sigma^{m_n-l_n}\cdots \Sigma^{m_{i+1}-l_{i+1}}f_i\circ(1_{\Sigma^{m_i-l_i}X_i}\wedge\tau(\s^{m_{i+1}-l_{i+1}}\wedge\cdots\wedge\s^{m_n-l_n},\s^{l_i}))\\
& : \Sigma^{l_i}X_i^*=\Sigma^{l_i}\Sigma^{m_n-l_n}\cdots \Sigma^{m_i-l_i}X_i
\approx \Sigma^{m_n-l_n}\cdots \Sigma^{m_{i+1}-l_{i+1}}\Sigma^{l_i}\Sigma^{m_i-l_i}X_i\\
&=\Sigma^{m_n-l_n}\cdots \Sigma^{m_{i+1}-l_{i+1}}\Sigma^{m_i}X_i\to 
\Sigma^{m_n-l_n}\cdots \Sigma^{m_{i+1}-l_{i+1}}X_{i+1}=X_{i+1}^*\  (1\le i<n).
\end{align*}
Does the containment 
$
\{f_n,\dots, f_1\}^{(\star)}_{\vec{\bm m}}\subset\{f_n^*,f_{n-1}^*,\dots,f_1^*\}_{\vec{\bm \ell}}^{(\star)}
$ hold up to sign for $\star=aq\ddot{s}_2, \ddot{s}_t$? 
\end{prob}

Note that the problem above is affirmative when $n=3$ by Theorem 7.1 below. 
In fact we have 
$
\{f_3,f_2, f_1\}^{(\star)}_{\vec{\bm m}}\subset
(-1)^{\ell_1(m_2-\ell_2+m_3-\ell_3)+\ell_2(m_3-\ell_3)}\{f_3^*,f_2^*,f_1^*\}_{\vec{\bm \ell}}^{(\star)}
$
for $\star=aq\ddot{s}_2, \ddot{s}_t$. 

\section{Properties of brackets}
\begin{thm}
\begin{enumerate}
\item[\rm(1)] $\{\vec{\bm f}\,\}^{(\ddot{s}_t)}_{\vec{\bm m}}=\{\vec{\bm f}\,\}^{(\dot{s}_t)}_{\vec{\bm m}}
=\{\vec{\bm f}\,\}^{(s_t)}_{\vec{\bm m}}$. 
\item[\rm(2)] We have 
\begin{align*}
\{\vec{\bm f}\,\}^{(aqs_2)}_{\vec{\bm m}}&=\{\vec{\bm f}\,\}^{(aq\dot{s}_2)}_{\vec{\bm m}}\\
&=\{\vec{\bm f}\,\}^{(aq\ddot{s}_2)}_{\vec{\bm m}}\circ 
\Sigma^{n-3}\big((1_{\Sigma^{m_1}X_1}\wedge\tau(\s^1,\s^{m_{[n,2]}}))
\circ \Sigma^{m_{[n,2]}}\mathscr{E}(\Sigma\Sigma^{m_1}X_1)\\
&\hspace{5cm}\circ(1_{\Sigma^{m_1}X_1}\wedge\tau(\s^{m_{[n,2]}},\s^1))\big)\\
&\subset \{\vec{\bm f}\,\}^{(aq\ddot{s}_2)}_{\vec{\bm m}}\circ 
(1_{\Sigma^{m_{[n-2,1]}}X_1}\wedge\tau(\s^{n-2},\s^{m_{[n,n-1]}}))\\
&\hspace{1cm}\circ \Sigma^{m_{[n,n-1]}}\mathscr{E}(\Sigma^{n-2}\Sigma^{m_{[n-2,1]}}X_1)\circ (1_{\Sigma^{m_{[n-2,1]}}X_1}\wedge\tau(\s^{m_{[n,n-1]}},\s^{n-2}))\\
&=\{\vec{\bm f}\,\}^{(qs_2)}_{\vec{\bm m}}
=\begin{cases}\{\vec{\bm f}\,\}^{(q\dot{s}_2)}_{\vec{\bm m}}&n=3\\
\{\vec{\bm f}\,\}^{(q\dot{s}_2)}_{\vec{\bm m}}=\{\vec{\bm f}\,\}^{(q\ddot{s}_2)}_{\vec{\bm m}}&n\ge 4\end{cases};\ 
\{\vec{\bm f}\,\}^{(aq\ddot{s}_2)}_{\vec{\bm m}}=\{\vec{\bm f}\,\}^{(q\ddot{s}_2)}_{\vec{\bm m}}\ \text{when }n=3. 
\end{align*}
If the suspension 
$\Sigma^{n-3}\Sigma^{m_{[n,n-1]}}:\mathscr{E}(\Sigma\Sigma^{m_1}X_1) \to\mathscr{E}(\Sigma^{|\vec{\bm m}|+n-2}X_1)
$ 
is surjective, for example if $X_1$ is a sphere, then 
$\{\vec{\bm f}\}^{(aq\dot{s}_2)}_{\vec{\bm m}}=\{\vec{\bm f}\}^{(aq\ddot{s}_2)}_{\vec{\bm m}}\circ\mathscr{E}(\Sigma^{|\vec{\bm m}|+n-2}X_1)$. 
If the suspension $\Sigma^{m_{[n,n-1]}}:\mathscr{E}(\Sigma^{n-2}\Sigma^{m_{[n-2,1]}}X_1)\to 
\mathscr{E}(\Sigma^{|\vec{\bm m}|+n-2}X_1)$ is surjective, for example if $X_1$ is a sphere, then $\{\vec{\bm f}\}^{(qs_2)}_{\vec{\bm m}}=\{\vec{\bm f}\}^{(aq\ddot{s}_2)}_{\vec{\bm m}}\circ\mathscr{E}(\Sigma^{|\vec{\bm m}|+n-2}X_1)$. 
\item[\rm(3)] We have 
\begin{align*}
\{\vec{\bm f}\,\}^{(q)}_{\vec{\bm m}}=\{\vec{\bm f}\,\}^{(aq)}_{\vec{\bm m}}&\circ(1_{\Sigma^{m_{[n-2,1]}}X_1}\wedge\tau(\s^{n-2},\s^{m_{[n,n-1]}}))
\circ \Sigma^{m_{[n,n-1]}}\mathscr{E}(\Sigma^{n-2}\Sigma^{m_{[n-2,1]}}X_1)\\
&\circ(1_{ \Sigma ^{m_{[n-2,1]}}X_1}\wedge\tau(\s^{m_{[n,n-1]}},\s^{n-2})).
\end{align*}
If the suspension 
$ \Sigma ^{m_n+m_{n-1}}:\mathscr{E}( \Sigma ^{m_{n-2}+\cdots+m_1+n-2}X_1)\to\mathscr{E}( \Sigma ^{|\vec{\bm m}|+n-2}X_1)
$ 
is surjective, for example if $X_1$ is a sphere, then 
$
\{\vec{\bm f}\}^{(aq)}_{\vec{\bm m}}\circ\mathscr{E}(\Sigma^{|\vec{\bm m}|+n-2}X_1)
=\{\vec{\bm f}\}^{(q)}_{\vec{\bm m}}.
$
\item[\rm(4)] We have $\{\vec{\bm f}\,\}^{(\ddot{s}_t)}_{\vec{\bm m}}\circ\Gamma=
\{\vec{\bm f}\,\}^{(aq\ddot{s}_2)}_{\vec{\bm m}}\circ\Gamma$, where $\Gamma$ is 
the subgroup of $\mathscr{E}(\Sigma^{|\vec{\bm m}|+n-2}X_1)$ defined by 
\begin{align*}
\Gamma=\big(1_{\Sigma^{m_{[n-1,1]}}X_1}\wedge\tau(\s^{n-2},\s^{m_n})\big)&
\circ \Sigma^{m_n}\mathscr{E}(\Sigma^{n-2}\Sigma^{m_{[n-1,1]}}X_1)\\
&\circ 
\big(1_{\Sigma^{m_{[n-1,1]}}X_1}\wedge\tau(\s^{m_n},\s^{n-2})\big),
\end{align*}
and so $\{\vec{\bm f}\,\}^{(\ddot{s}_t)}_{\vec{\bm m}}\circ
\mathscr{E}( \Sigma ^{|\vec{\bm m}|+n-2}X_1)
=\{\vec{\bm f}\,\}^{(aq\ddot{s}_2)}_{\vec{\bm m}}\circ
\mathscr{E}( \Sigma ^{|\vec{\bm m}|+n-2}X_1)$. 
\item[\rm(5)] If $\alpha\in\{\vec{\bm f}\,\}^{(q)}_{\vec{\bm m}}$, then there are 
$\theta, \theta'\in \Gamma'$ 
such that $\alpha\circ \theta\in\{\vec{\bm f}\,\}^{(aq\ddot{s}_2)}_{\vec{\bm m}}$ and 
$\alpha\circ \theta'\in\{\vec{\bm f}\,\}^{(\ddot{s}_t)}_{\vec{\bm m}}$, where 
$\Gamma'$ is the subset (in fact subgroup) of $[\Sigma^{|\vec{\bm m}|+n-2}X_1,\Sigma^{|\vec{\bm m}|+n-2}X_1]$ defined by 
\begin{align*}
\Gamma'= (1_{\Sigma^{m_{[n-1,1]}}X_1}\wedge\tau(\s^{n-2},\s^{m_n}))
&\circ \Sigma^{m_n}[\Sigma^{n-2}\Sigma^{m_{[n-1,1]}}X_1,\Sigma^{n-2}\Sigma^{m_{[n-1,1]}}X_1]\\
&\circ (1_{\Sigma^{m_{[n-1,1]}}X_1}\wedge\tau(\s^{m_n},\s^{n-2})).
\end{align*}
\end{enumerate}
\end{thm}

Obviously $\Gamma\subset\Gamma'$, while $\Gamma$ is not necessarily 
a subgroup of $\Gamma'$. 

\begin{cor} 
\begin{enumerate}
\item[\rm(1)] $\{\vec{\bm f}\,\}^{(\star)}_{\vec{\bm m}}\circ\varepsilon=\{\vec{\bm f}\,\}^{(\star)}_{\vec{\bm m}}$ for every $\star\in\{q,qs_2,q\dot{s}_2\}$ and every homotopy equivalence
\begin{align*}
\varepsilon\in & 
(1_{ \Sigma ^{m_{[n-2,1]}}X_1}\wedge\tau(\s^{n-2},\s^{m_{[n,n-1]}}))\\
&\hspace{5mm} \circ \Sigma ^{m_{[n,n-1]}}\mathscr{E}(\Sigma^{n-2}\Sigma^{m_{[n-2,1]}}X_1)\circ (1_{ \Sigma ^{m_{[n-2,1]}}X_1}\wedge\tau(\s^{m_{[n,n-1]}},\s^{n-2})).
\end{align*}
Also $\{\vec{\bm f}\,\}^{(\star)}_{\vec{\bm m}}\circ \Sigma^{n-3}\varepsilon=\{\vec{\bm f}\,\}^{(\star)}_{\vec{\bm m}}$ for every $\star\in\{aqs_2, aq\dot{s}_2\}$ and every homotopy equivalence 
$$
\varepsilon \in (1_{\Sigma^{m_1}X_1}\wedge\tau(\s^1,\s^{m_{[n,2]}}))\circ \Sigma^{m_{[n,2]}}\mathscr{E}(\Sigma\Sigma^{m_1}X_1)\circ(1_{\Sigma^{m_1}X_1}\wedge\tau(\s^{m_{[n,2]}},\s^1)).
$$
In particular, $ -\{\vec{\bm f}\,\}^{(\star)}_{\vec{\bm m}}=\{\vec{\bm f}\,\}^{(\star)}_{\vec{\bm m}}$ for every $\star\in\{q,qs_2,q\dot{s}_2,aqs_2,aq\dot{s}_2\}$. 
\item[\rm(2)] $\{\vec{\bm f}\,\}^{(aq)}_{\vec{\bm m}}\circ F(\gamma)=\{\vec{\bm f}\,\}^{(aq)}_{\vec{\bm m}}$ for every $\gamma\in\mathscr{E}(\Sigma\Sigma^{m_1}X_1)$, 
where 
\begin{align*}F(\gamma)&=(1_{\Sigma^{m_1}X_1}\wedge\tau(\s^{n-2},\s^{m_{[n,2]}}))\\
&\hspace{1cm}
\circ \Sigma^{m_{[n,2]}}\Sigma^{n-3}\gamma 
\circ (1_{\Sigma^{m_1}X_1}\wedge\tau(\s^{m_{[n,2]}},\s^{n-2})),
\end{align*} 
and $-\{\vec{\bm f}\,\}^{(aq)}_{\vec{\bm m}}=\{\vec{\bm f}\,\}^{(aq)}_{\vec{\bm m}}$.
\item[\rm(3)] If the suspension $\Sigma^{m_{[n,2]}}\Sigma^{n-3}:\mathscr{E}(\Sigma\Sigma^{m_1}X_1)\to\mathscr{E}(\Sigma^{|\vec{\bm m}|+n-2}X_1)$ is surjective, 
for example if $X_1$ is a sphere, then 
$\{\vec{\bm f}\,\}^{(q)}_{\vec{\bm m}}=\{\vec{\bm f}\,\}^{(aq)}_{\vec{\bm m}}
\circ \mathscr{E}(\Sigma^{|\vec{\bm m}|+n-2}X_1)$. 
\item[\rm(4)] If $\{\vec{\bm f}\,\}^{(\star)}_{\vec{\bm m}}$ is not empty for some $\star$, then 
$\{\vec{\bm f}\,\}^{(\star)}_{\vec{\bm m}}$ is not empty for all $\star$. 
\item[\rm(5)] If $\{\vec{\bm f}\,\}^{(\star)}_{\vec{\bm m}}$ contains $0$ for some $\star$, then 
$\{\vec{\bm f}\,\}^{(\star)}_{\vec{\bm m}}$ contains $0$ for all $\star$. 
\end{enumerate}
\end{cor}

\begin{prop}
Given maps $f_{n+1}:\Sigma^{m_{n+1}}X_{n+1}\to X_{n+2}$ and $f_0:\Sigma^{m_0}X_0\to X_1$, 
where $m_{n+1}, m_0$ are non negative integers, we have 
\begin{gather}
\begin{split}
&f_{n+1}\circ \Sigma^{m_{n+1}}\{\vec{\bm f}\}^{(\star)}_{\vec{\bm m}}\\
&\hspace{1cm}\subset\{f_{n+1}\circ \Sigma^{m_{n+1}}f_n,f_{n-1},\dots,f_1\}^{(\star)}_{(m_{n+1}+m_n,m_{n-1},\dots,m_1)}\\
&\hspace{2cm}\circ (1_{\Sigma^{m_{[n,1]}}X_1}\wedge\tau(\s^{n-2},\s^{m_{n+1}}))\\
&\hspace{1cm}=(-1)^{n\cdot m_{n+1}}\{f_{n+1}\circ \Sigma^{m_{n+1}}f_n, f_{n-1},\dots,f_1\}^{(\star)}_{
(m_{n+1}+m_n,m_{n-1},\dots,m_1)};
\end{split}
\\
\begin{split}
&\{f_{n+1}\circ \Sigma^{m_{n+1}}f_n,f_{n-1},\dots,f_1\}^{(\star)}_{(m_{n+1}+m_n,m_{n-1},\dots,m_1)}\\
&\hspace{1cm}\subset\{f_{n+1},f_n\circ \Sigma^{m_n}f_{n-1},f_{n-2},\dots,f_1\}^{(\star)}_{(m_{n+1},m_n+m_{n-1},m_{n-2},\dots,m_1)};
\end{split}
\\
\begin{split}
&\{f_n,\dots,f_2,f_1\circ \Sigma^{m_1}f_0\}^{(\star)}_{\vec{\bm m}}\\
&\hspace{1cm}\subset\{f_n,\dots,f_3,f_2\circ \Sigma^{m_2}f_1,f_0\}^{(\star)}_{(m_n,\dots,m_3,m_2+m_1,m_0)}\quad(\star=aq\ddot{s}_2, \ddot{s}_t);
\end{split}
\\
\{f_n,\dots,f_1\}^{(\star)}_{\vec{\bm m}}\circ \Sigma^{|\vec{\bm m}|+n-2}f_0
\subset\{f_n,\dots,f_2,f_1\circ \Sigma^{m_1}f_0\}^{(\star)}_{\vec{\bm m}}\quad(\star=aq\ddot{s}_2,\ddot{s}_t).
\end{gather}
\end{prop}

\begin{proof}[Proof of Theorem 4.1(1)] 
It suffices to show that $\{\vec{\bm f}\,\}^{(s_t)}_{\vec{\bm m}}\subset \{\vec{\bm f}\,\}^{(\ddot{s}_t)}_{\vec{\bm m}}$. 
Let $\alpha\in\{\vec{\bm f}\,\}^{(s_t)}_{\vec{\bm m}}$ and $\{\mathscr{S}_r,\overline{f_r},\mathscr{A}_r\,|\,2\le r\le n\}$ an $s_t$-presentation 
of $\vec{\bm f}$ with $\alpha=\overline{f_n}\circ  \Sigma ^{m_n}g_{n,n-1}\circ(1_{ \Sigma ^{m_{[n-1,1]}}X_1}\wedge\tau(\s^{m_n},\s^{n-2}))$, where 
\begin{align*}
&\mathscr{S}_r=( \Sigma ^{m_{r-1}}X_{r-1}, \Sigma  \Sigma ^{m_{[r-1,r-2]}}X_{r-2},\dots, \Sigma ^{r-2} \Sigma ^{m_{[r-1,1]}}X_1;\\
&\hspace{5cm}C_{r,1},\dots,C_{r,r}; g_{r,1},\dots,g_{r,r-1};j_{r,1},\dots,j_{r,r-1});\\
&C_{r,1}=X_r,\ g_{r,1}=f_{r-1},\ \mathscr{A}_r=\{a_{r,s}\,|\,1\le s<r\},\  \Omega(\mathscr{A}_r)=\{\omega_{r,s}\,|\,1\le s<r\};\\
&\text{if $3\le r\le n$, then $C_{r,2}=X_r\cup_{f_{r-1}}C \Sigma ^{m_{r-1}}X_{r-1},\ j_{r,1}=i_{f_{r-1}}$, and $a_{r,1}=1_{C_{r,2}}$}. 
\end{align*}
We will construct an $\ddot{s}_t$-presentation $\{\mathscr{S}_r',\overline{f_r}',\mathscr{A}_r'\,|\,2\le r\le n\}$ of 
$\vec{\bm f}$ such that $\overline{f_n}'\circ  \Sigma ^{m_n}g_{n,n-1}'\circ(1_{ \Sigma ^{m_{n-1}}\cdots \Sigma ^{m_1}X_1}\wedge\tau(\s^{m_n},\s^{n-2}))=\alpha$. 

First we set $\mathscr{S}_2'=( \Sigma ^{m_1}X_1;X_2,X_2\cup_{f_1}C \Sigma ^{m_1}X_1;f_1;i_{f_1})$, $a_{2,1}'=1_{C_{2,2}'}$,  
$e_{2}=a_{2,1}^{-1}:C_{2,2}'\to C_{2,2}$ (see \cite[Convention 5.2]{OO}), and $\overline{f_2}'=\overline{f_2}\circ  \Sigma ^{m_2}e_{2}$. 
Then we have 
$$C_{2,1}'=C_{2,1}=X_2,\quad     
e_{2}\circ j_{2,1}'=j_{2,1},\quad  \omega_{2,1}'=q_{f_1}'\simeq \omega_{2,1}\circ(e_{2}\cup C1_{C_{2,1}}). 
$$
Secondly we set $\mathscr{S}_3'=(\widetilde{ \Sigma }^{m_2}\mathscr{S}_2')(\overline{f_2}',\widetilde{ \Sigma }^{m_2}\mathscr{A}_2')$ and
$$
e_{3}=1_{X_3}\cup C \Sigma ^{m_2}e_{2}:C_{3,3}'=X_3\cup_{\overline{f_2}'}C \Sigma ^{m_2}(X_2\cup_{f_1}C \Sigma ^{m_1}X_1)\to C_{3,3}=X_3\cup_{\overline{f_2}}C \Sigma ^{m_2}C_{2,2}.
$$
Then $C_{3,s}'=C_{3,s}\,(s=1,2),\  j_{3,1}'=j_{3,1}=i_{f_2},\  j_{3,2}'=1_{X_3}\cup C \Sigma ^{m_2}j_{2,1}',\  e_{3}\circ j_{3,2}'=j_{3,2}$, 
and 
\begin{align*}
g_{3,2}'&=(\overline{f_2}'\cup C1_{ \Sigma ^{m_2}X_2})\circ (\widetilde{ \Sigma }^{m_2}\omega_{2,1}')^{-1}\\
&=(\overline{f_2}\cup C1_{\Sigma^{m_2}X_2})\circ(\Sigma^{m_2}e_2\cup C\Sigma^{m_2}1_{X_2})
\circ (\widetilde{\Sigma}^{m_2}\omega'_{2,1})^{-1}\\
&\simeq (\overline{f_2}\cup C1_{\Sigma^{m_2}X_2})\circ(\widetilde{\Sigma}^{m_2}\omega_{2,1})^{-1}
=g_{3,2}.
\end{align*} 
As remarked in \cite[Remark 5.5(1)]{OO}, we may suppose/take $g'_{3,2}=g_{3,2}$ 
by Lemma 4.3(2) of \cite{OO}. 
We set 
\begin{align*}
&a_{3,2}'= a_{3,2}\circ e_{3}:C_{3,3}'\to (X_3\cup_{f_2}C \Sigma ^{m_2}X_2)\cup_{g_{3,2}}C \Sigma  \Sigma ^{m_{[2,1]}} X_1,\\
&a_{3,1}'=1_{C_{3,2}'},\ \mathscr{A}_3'=\{a_{3,1}', a_{3,2}'\},\ \overline{f_3}'
=\begin{cases} \overline{f_3} & n=3\\ \overline{f_3}\circ \Sigma ^{m_3} e_{3} & n\ge 4\end{cases}.
\end{align*}
Then $\mathscr{A}_3'$ is a structure on $\mathscr{S}_3'$, $\overline{f_3}'^2=\overline{f_3}^2$, and 
\begin{align*} 
\omega_{3,2}'&=q_{g_{3,2}'}'\circ (a_{3,2}'\cup C1_{C_{3,2}'})
=q_{g_{3,2}}'\circ(a_{3,2}\circ e_3\cup C1_{C_{3,2}})\\
&=q_{g_{3,2}}'\circ (a_{3,2}\cup C1_{C_{3,2}})\circ(e_3\cup C1_{C_{3,2}})\\
&=\omega_{3,2}\circ (e_3\cup C1_{C_{3,2}}).
\end{align*}
When $n=3$, $\{\mathscr{S}_r',\overline{f_r}',\mathscr{A}_r'\,|\, r=2,3\}$ is an $\ddot{s}_t$-presentation of 
$(f_3,f_2,f_1)$ such that $\overline{f_3}'\circ \Sigma ^{m_3} g_{3,2}'\circ(1_{ \Sigma ^{m_2} \Sigma ^{m_1}X_1}\wedge\tau(\s^{m_3},\s^1))=\alpha$. 
When $n\ge 4$, by repeating the process above, we have 
an $\ddot{s}_t$-presentation 
$\{\mathscr{S}_r',\overline{f_r}',\mathscr{A}_r'\,|\,2\le r\le n\}$ of $\vec{\bm f}$, and $e_{r}:C_{r,r}'\simeq C_{r,r}$ for $3\le r\le n$ such that
$$
\left\{\begin{array}{@{\hspace{0.6mm}}l}
\mathscr{S}_r'=(\widetilde{ \Sigma }^{m_{r-1}}\mathscr{S}_{r-1}')(\overline{f_{r-1}}',\widetilde{ \Sigma }^{m_{r-1}}\mathscr{A}_{r-1}), 
\ e_r=1_{X_r}\cup C \Sigma ^{m_{r-1}}e_{r-1};\\
C_{r,s}'=C_{r,s}\ (1\le s\le r-1),\ C_{r,r}'=X_r\cup_{\overline{f_{r-1}}'}C \Sigma ^{m_{r-1}}C_{r-1,r-1}';\\
j_{r,s}'=j_{r,s},\ a_{r,s}'=a_{r,s},\ g_{r,s}'=g_{r,s}\ (1\le s\le r-2);\\
 j_{r,r-1}'=1_{X_r}\cup C \Sigma ^{m_{r-1}}j_{r-1,r-2}', \ e_{r}\circ j_{r,r-1}'=j_{r,r-1};\\ 
\overline{f_r}'
=\begin{cases} \overline{f_n} : \Sigma ^{m_n}C_{n,n-1}'= \Sigma ^{m_n}C_{n,n-1}\to X_{n+1}& r=n\\
\overline{f_r}\circ \Sigma ^{m_r} e_{r} :C_{r,r}'\to X_{r+1}& r<n\end{cases};\\
\omega_{r,r-1}'\simeq \omega_{r,r-1}\circ (e_{r}\cup C1_{C_{r,r-1}}), \ g_{r,r-1}'=g_{r,r-1};\\
a_{r,r-1}'= a_{r,r-1}\circ e_r:C_{r,r}'\to C_{r,r-1}'\cup_{g_{r,r-1}'}C \Sigma ^{r-2} \Sigma ^{m_{r-1}}\cdots \Sigma ^{m_1}X_1
\end{array}\right.
$$
and $\overline{f_n}'\circ  \Sigma ^{m_n}g_{n,n-1}'\simeq\overline{f_n}\circ \Sigma ^{m_n}g_{n,n-1}$. 
Hence $\alpha\in\{\vec{\bm f}\,\}^{(\ddot{s}_t)}_{\vec{\bm m}}$. 
This proves Theorem~4.1(1). 
\end{proof}

\begin{proof}[Proof of the first equality in Theorem 4.1(2)] 
The first equality is equivalent to the relation 
$\{\vec{\bm f}\,\}^{(aqs_2)}_{\vec{\bm m}}\subset\{\vec{\bm f}\,\}^{(aq\dot{s}_2)}_{\vec{\bm m}}$. 
Let $\alpha\in\{\vec{\bm f}\,\}^{(aqs_2)}_{\vec{\bm m}}$ and $\{\mathscr{S}_r,\overline{f_r},\Omega_r\,|\,2\le r\le n\}$ 
an $aqs_2$-presentation of $\vec{\bm f}$ with $\alpha=\overline{f_n}\circ \Sigma ^{m_n}g_{n,n-1}\circ(1_{ \Sigma ^{m_{n-1}}\cdots \Sigma ^{m_1}X_1}\wedge\tau(\s^{m_n},\s^{n-2}))$. 
It suffices to construct an $aq\dot{s}_2$-presentation $\{\mathscr{S}_r',\overline{f_r}',\Omega_r'\,|\,2\le r\le n\}$ 
with $\alpha=\overline{f_n}'\circ \Sigma ^{m_n} g_{n,n-1}' \circ(1_{ \Sigma ^{m_{n-1}}\cdots \Sigma ^{m_1}X_1}\wedge\tau(\s^{m_n},\s^{n-2}))$. 
Set $\mathscr{S}_2'=(\widetilde{ \Sigma }^{m_1}X_1;X_2,X_2\cup_{f_1}C \Sigma ^{m_1}X_1;f_1;i_{f_1})$. 
Since $\mathscr{S}_2$ is an iterated mapping cone, we can take $e_2:C_{2,2}'=X_2\cup_{f_1}C \Sigma ^{m_1}X_1\simeq C_{2,2}$ such that 
$e_2\circ j_{2,1}'=j_{2,1}$. 
Set $\overline{f_2}'=\overline{f_2}\circ \Sigma ^{m_2} e_2$, 
$\omega_{2,1}'=\omega_{2,1}\circ (e_2\cup C1_{X_2}):C_{2,2}'\cup_{j_{2,1}'}CC_{2,1}'\to  \Sigma  \Sigma ^{m_1} X_1$, and 
$\Omega_2'=\{\omega_{2,1}'\}$. 
Set $\mathscr{S}_3'=( \Sigma ^{m_2}\mathscr{S}_2')(\overline{f_2}',\widetilde{ \Sigma }^{m_2}\Omega_2')$, $\Omega_3'=\widetilde{\widetilde{ \Sigma }^{m_2}\Omega_2'}$, 
$e_3=1_{X_3}\cup C \Sigma ^{m_2}e_2:C_{3,3}'\to C_{3,3}$, 
and $\overline{f_3}'=\begin{cases} \overline{f_3}: \Sigma ^{m_3}C_{3,2}'= \Sigma ^{m_3}C_{3,2}\to X_4 & n=3\\ 
\overline{f_3}\circ  \Sigma ^{m_3}e_3: \Sigma ^{m_3}C_{3,3}'\to X_4 & n\ge 4\end{cases}$. 
Then 
\begin{align*}
&\omega_{2,1}'=\omega_{2,1}\circ(e_2\cup C1_{X_2}),\ \widetilde{ \Sigma }^{m_2}\omega_{2,1}'=\widetilde{ \Sigma }^{m_2}\omega_{2,1}\circ( \Sigma ^{m_2}e_2\cup C1_{ \Sigma ^{m_2}X_2}),\\
&\omega_{3,1}'=\omega_{3,1}=q'_{f_2},\\ 
&\omega_{3,2}'=\omega_{3,2}\circ(e_3\cup C1_{C_{3,2}}),\ \widetilde{ \Sigma }^{m_3}\omega'_{3,2}=\widetilde{ \Sigma }^{m_3}\omega_{3,2}\circ( \Sigma ^{m_3}e_3\cup C1_{ \Sigma ^{m_3}C_{3,2}}),
\\
& g_{3,2}'=(\overline{f_2}'\cup C1_{ \Sigma ^{m_2}X_2})\circ (\widetilde{ \Sigma }^{m_2}\omega_{2,1}')^{-1}\\
&\hspace{5mm}\simeq (\overline{f_2}\circ \Sigma ^{m_2}e_2\cup C1_{ \Sigma ^{m_2}X_2})\circ( \Sigma ^{m_2}e_2\cup C1_{ \Sigma ^{m_2}X_2})^{-1}\circ
(\widetilde{ \Sigma }^{m_2}\omega_{2,1})^{-1}\\
&\hspace{5mm}= g_{3,2}.
\end{align*} 
By continuing the construction inductively, we obtain an $aq\dot{s}_2$-presentation 
$\{\mathscr{S}_r',\overline{f_r}',\Omega_r'\,|\,2\le r\le n\}$ and 
$e_r: C_{r,r}'\simeq C_{r,r}$ such that 
\begin{gather*}
C_{r,s}'=C_{r,s}\ (1\le s<r<n), \\ 
\omega_{r,r-1}'=\omega_{r,r-1}\circ (e_r\cup 1_{C_{r,r-1}}):C_{r,r}'\cup CC_{r,r-1}'\to C_{r,r}\cup CC_{r,r-1}\ (r< n), \\
\overline{f_r}'=\begin{cases} \overline{f_n} :  \Sigma ^{m_n}C_{n,n-1}'= \Sigma ^{m_n}C_{n,n-1}\to X_{n+1} & r=n\\ 
\overline{f_r}\circ  \Sigma ^{m_r}e_r: \Sigma ^{m_r}C_{r,r}'\to X_{r+1} & r<n\end{cases},\\ 
g_{n,n-1}'=(\overline{f_{n-1}}'\cup C1_{ \Sigma ^{m_{n-1}}C_{n-1,n-2}})\circ(\widetilde{ \Sigma }^{m_{n-1}}\omega_{n-1,n-2}')^{-1}\simeq g_{n,n-1}.
\end{gather*} 
Hence $\overline{f_n}'\circ  \Sigma ^{m_n}g_{n,n-1}'\simeq \overline{f_n}\circ \Sigma ^{m_n} g_{n,n-1}$. 
This proves the first equality in Theorem 4.1(2). 
\end{proof}

\begin{proof}[Proof of the second equality in Theorem 4.1(2)] 
First we prove ``$\subset$''. 
Let $\alpha\in\{\vec{\bm f}\,\}^{(aq\dot{s}_2)}_{\vec{\bm m}}$ and 
$\{\mathscr{S}_r,\overline{f_r},\Omega_r\,|\,2\le r\le n\}$ an $aq\dot{s}_2$-presentation of $\vec{\bm f}$ 
with $\alpha=\overline{f_n}\circ  \Sigma ^{m_n}g_{n,n-1}\circ(1_{ \Sigma ^{m_{n-1}}\cdots \Sigma ^{m_1}X_1}\wedge\tau(\s^{m_n},\s^{n-2})$. 
Define inductively an $aq\ddot{s}_2$-presentation $\{\mathscr{S}_r',\overline{f_r}',\Omega_r'\,|\,2\le r\le n\}$ of $\vec{\bm f}$ as follows:
\begin{gather*}
\mathscr{S}_2'=\mathscr{S}_2, \ \omega_{2,1}'=q_{f_1}',\ \overline{f_2}'=\overline{f_2}, \ \Omega_2'=\{\omega_{2,1}'\}, \\
\mathscr{S}_{r+1}'=(\widetilde{\Sigma}^{m_r}\mathscr{S}_r')(\overline{f_r},\widetilde{\Sigma}^{m_r}\Omega_r'),\ \Omega_{r+1}'=\widetilde{\widetilde{\Sigma}^{m_r}\Omega_r'},\ \overline{f_{r+1}}'=\overline{f_{r+1}}\ (2\le r\le n-1).
\end{gather*}
Define 
\begin{align*}
\theta_1&=\omega'_{2,1}\circ\omega_{2,1}\in\mathscr{E}(\Sigma\Sigma^{m_1}X_1),\\
\theta_2&=(1_{\Sigma^{m_1}X_1}\wedge\tau(\s^1,\s^{m_2}))\circ \Sigma^{m_2}\theta_1\circ 
(1_{\Sigma^{m_1}X_1}\wedge\tau(\s^{m_2},\s^1))\in\mathscr{E}(\Sigma\Sigma^{m_{[2,1]}}X_1),\\
\theta_k&=(1_{\Sigma^{m_{[k-1,1]}}X_1}\wedge\tau(\s^1,\s^{m_k}))\circ \Sigma^{m_k}\theta_{k-1}\\
&\hspace{1cm}\circ (1_{\Sigma^{m_{[k-1,1]}}X_1}\wedge\tau(\s^{m_k},\s^1))\in\mathscr{E}(\Sigma\Sigma^{m_{[k,1]}}X_1)\ (2\le k\le n-1).
\end{align*}
By an induction on $k$, we have easily
$$
\theta_k=(1_{\Sigma^{m_1}X_1}\wedge\tau(\s^1,\s^{m_{[k,2]}}))\circ \Sigma^{m_{[k,2]}}\theta_1
\circ (1_{\Sigma^{m_1}X_1}\wedge\tau(\s^{m_{[k,2]}},\s^1))\ (2\le k<n).
$$

\begin{lemma}
\begin{enumerate}
\item $\Sigma^{r-2}\theta_{r-1}\circ\omega_{r,r-1}\simeq \omega'_{r,r-1}\ (2\le r\le n)$.
\item $g_{r,r-1}\simeq g'_{r,r-1}\circ \Sigma^{r-3}\theta_{r-1}\ (3\le r\le n)$.
\end{enumerate}
\end{lemma}
\begin{proof}
(1) The case $r=2$ is true by definition. 
Suppose the assertion for some $r$ with $2\le r<n$, that is, 
$\Sigma^{r-2}\theta_{r-1}\circ\omega_{r,r-1}\simeq \omega'_{r,r-1}$. 
Consider the case for $r+1$. 
Consider the following diagram. 
$$
\xymatrix{
C_{r+1,r+1}\cup CC_{r+1,r} \ar[d]_-{h}& \Sigma\Sigma^{m_r}\Sigma^{r-1}\Sigma^{m_{[r-1,1]}}X_1 \ar[r]^-\approx & \Sigma\Sigma^{r-1}\Sigma^{m_{[r,1]}}X_1\\
 \Sigma\Sigma^{m_r}(C_{r,r}\cup CC_{r,r-1})\ar[ur]^-{\Sigma\Sigma^{m_r}\omega'_{r,r-1}} \ar[r]_-{\Sigma\Sigma^{m_r}\omega_{r,r-1}} 
& \Sigma\Sigma^{m_r}\Sigma^{r-1}\Sigma^{m_{[r-1,1]}}X_1\ar[r]^-\approx \ar[u]_-{\Sigma\Sigma^{m_r}\Sigma^{r-2}\theta_{r-1}} & \Sigma\Sigma^{r-1}\Sigma^{m_{[r,1]}}X_1 \ar@{.>}[u]_-{\theta^*_r}
}
$$
where two $\approx$'s are $\Sigma\big(1_{\Sigma^{m_{[r-1,1]}}X_1}\wedge\tau(\s^{r-1},\s^{m_r})\big)$, $\theta^*_r$ is defined to make the square commutative, 
and $h$ is  
the composite of the the following three maps:
\begin{align*}
C_{r+1,r+1}\cup CC_{r+1,r}&=(X_{r+1}\cup_{\overline{f_r}} C\Sigma^{m_r}C_{r,r})\cup 
(X_{r+1}\cup_{\overline{f_r}^{r-1}}C\Sigma^{m_r}C_{r,r-1})\\
&\xrightarrow[\approx]{\xi}(X_{r+1}\cup CX_{r+1})\cup C(\Sigma^{m_r}C_{r,r}\cup C\Sigma^{m_r}C_{r,r-1})\\
&\xrightarrow[\simeq]{q}\Sigma(\Sigma^{m_r}C_{r,r}\cup C\Sigma^{m_r}C_{r,r-1})
\xrightarrow[\approx]{\Sigma\psi^{m_r}_{j_{r,r-1}}} \Sigma\Sigma^{m_r}(C_{r,r}\cup CC_{r,r-1}).
\end{align*}
Then $\omega'_{r+1,r}\simeq \theta^*_r\circ\omega_{r+1,r}$. 
It suffices for completing the induction to show $\theta^*_r=\Sigma^{r-1}\theta_r$. 
Let $u\in\s^1$ and $v\in\s^{r-2}$. Then $u\wedge v\in \s^{r-1}$. 
Let $x\in X_1$ and $s_i\in\s^{m_i}\ (1\le i\le r)$. 
Write $\theta_{r-1}(x\wedge s_1\wedge\cdots\wedge s_{r-1}\wedge u)=
x'\wedge s_1'\wedge\cdots\wedge s_{r-1}'\wedge u'$, where $x'\in X_1$ and $s_i'\in \s^{m_i}\ (1\le i\le r)$. 
Then 
$$
\theta^*_r(x\wedge s_1\wedge\cdots\wedge s_r\wedge u\wedge v)=
x'\wedge s_1'\wedge\cdots\wedge s_{r-1}'\wedge u'\wedge v=
\Sigma^{r-1}\theta_r(x\wedge s_1\wedge\cdots\wedge s_r\wedge u\wedge v).
$$
This completes the induction and proves (1). 

(2) By definition of $\theta_2$, we have $\theta_2\circ\widetilde{\Sigma}^{m_2}\omega_{3,2}\simeq\widetilde{\Sigma}^{m_2}\omega_{3,2}$ so that $\theta_2\circ g'_{3,2}\simeq g_{3,2}$ and hence (2) holds for $r=3$. 
Suppose the assertion holds for some $r$ with $3\le r<n$. 
Consider the case for $r+1$. 
Consider the following diagram. 
$$
\xymatrix{
& \Sigma^{m_r}\Sigma^{r-1}\Sigma^{m_{[r-1,1]}}X_1 \ar[r]^-\approx & \Sigma^{r-1}\Sigma^{m_{[r,1]}}X_1\\
\Sigma^{m_r}(C_{r,r}\cup CC_{r,r-1}) \ar[ur]^-{\Sigma^{m_r}\omega'_{r,r-1}} \ar[r]_-{\Sigma^{m_r}\omega_{r,r-1}} 
& \Sigma^{m_r}\Sigma^{r-1}\Sigma^{m_{[r-1,1]}}X_1 \ar[u]_-{\Sigma^{m_r}\Sigma^{r-2}\theta_{r-1}} 
\ar[r]^-\approx & \Sigma^{r-1}\Sigma^{m_{[r,1]}}X_1 \ar@{.>}[u]_-{\theta_r'}
}
$$
where two $\approx$'s are $1_{\Sigma^{m_{[r-1,1]}}X_1}\wedge\tau(\s^{r-1},\s^{m_r})$ 
and $\theta_r'$ is defined to make the square commutative. 
It suffices for the purpose to show $\theta_r'=\Sigma^{r-2}\theta_r$. 
This is proved by a similar method in (1). 
\end{proof}

The following proposition shows ``$\subset$''. 

\begin{prop}
Under the notations above we have
\begin{align*}
\overline{f_n}\circ\widetilde{\Sigma}^{m_n}g_{n,n-1}\simeq\overline{f_n}\circ\widetilde{\Sigma}^{m_n}g'_{n,n-1}\circ \Sigma^{n-3}\big(&(1_{\Sigma^{m_1}X_1}\wedge\tau(\s^1,\s^{m_{[n,2]}}))\\
&\circ \Sigma^{m_{[n,2]}}\theta_1\circ (1_{\Sigma^{m_1}X_1}\wedge\tau(\s^{m_{[n,2]}}))\big).
\end{align*}
\end{prop}
\begin{proof}
We have
\begin{align*}
&\overline{f_n}\circ\widetilde{\Sigma}^{m_n}g_{n,n-1}=\overline{f_n}\circ \Sigma^{m_n}g_{n,n-1}
\circ (1_{\Sigma^{m_{[n-1,1]}}X_1}\wedge\tau(\s^{m_n},\s^{n-2}))\\
&\simeq \overline{f_n}\circ \Sigma^{m_n}(g'_{n,n-1}\circ \Sigma^{n-3}\theta_{n-1})\circ
 (1_{\Sigma^{m_{[n-1,1]}}X_1}\wedge\tau(\s^{m_n},\s^{n-2})) \ (\text{by Lemma 4.4(2)})\\
&=\overline{f_n}\circ \Sigma^{m_n}g'_{n,n-1}\circ 
(1_{\Sigma^{m_{[n-1,1]}}X_1}\wedge\tau(\s^{m_n},\s^{n-2}))\\
&\hspace{1cm}\circ (1_{\Sigma^{m_{[n-1,1]}}X_1}\wedge\tau(\s^{n-2},\s^{m_n})
\circ \Sigma^{m_n}\Sigma^{n-3}\theta_{n-1}\circ (1_{\Sigma^{m_{[n-1,1]}}X_1}\wedge\tau(\s^{m_n},\s^{n-2})).
\end{align*}
By a precise checking, we can show 
\begin{align*}
& (1_{\Sigma^{m_{[n-1,1]}}X_1}\wedge\tau(\s^{n-2},\s^{m_n})
\circ \Sigma^{m_n}\Sigma^{n-3}\theta_{n-1}\circ (1_{\Sigma^{m_{[n-1,1]}}X_1}\wedge\tau(\s^{m_n},\s^{n-2}))\\
&= \Sigma^{n-3}\Big((1_{\Sigma^{m_1}X_1}\wedge\tau(\s^1,\s^{m_{[n,2]}}))\circ \Sigma^{m_{[n,2]}}\theta_1\circ (1_{\Sigma^{m_1}X_1}\wedge\tau(\s^{m_{[n,2]}},\s^1))\Big)
\end{align*}
so that the assertion follows.
\end{proof}

Secondly we prove ``$\supset$''. 
Let $\beta'\in\{\vec{\bm f}\}^{(aq\ddot{s}_2)}$, 
$\varepsilon\in\mathscr{E}(\Sigma\Sigma^{m_1}X_1)$, and 
$\{\mathscr{S}'_r,\overline{f_r}',\Omega'_r\,|\,2\le r\le n\}$ an 
$aq\ddot{s}_2$-presentation of $\vec{\bm f}$ representing $\beta'$. 
Define 
\begin{gather*}
\mathscr{S}_2=\mathscr{S}_2',\ \omega_{2,1}=\varepsilon\circ\omega'_{2,1},\ \Omega_2=\{\omega_{2,1}\},\ \overline{f_2}=\overline{f_2}',\\
\mathscr{S}_{r+1}=(\widetilde{\Sigma}^{m_r}\mathscr{S}_r)(\overline{f_r},\widetilde{\Sigma}^{m_r}\Omega_r), \Omega_{r+1}=\widetilde{\widetilde{\Sigma}^{m_r}\Omega_r},\ \overline{f_{r+1}}=\overline{f_{r+1}}'\ (2\le r<n).
\end{gather*}
Then $\{\mathscr{S}_r,\overline{f_r},\Omega_r\,|\,2\le r\le n\}$ is an 
$aq\dot{s}_2$-presentation of $\vec{\bm f}$. 
Set $\beta=\overline{f_n}\circ\widetilde{\Sigma}^{m_n}g_{n,n-1}$ and 
$\theta_1=\varepsilon^{-1}$. 
Tracing the proof above of ``$\subset$'' we can show
$$
\beta=\beta'\circ \Sigma^{n-3}\big((1_{\Sigma^{m_1}X_1}\wedge\tau(\s^1,\s^{m_{[n,2]}}))
\circ \Sigma^{m_{[n,2]}}\theta_1\circ (1_{\Sigma^{m_1}X_1}\wedge\tau(\s^{m_{[n,2]}}))\big)
$$
so that ``$\supset$'' holds. 
This completes the proof of the second equality in Theorem 4.1(2).
\end{proof}

\begin{proof}[Proof of the first containment in Theorem 4.1(2)] 
Let $\theta\in\mathscr{E}(\Sigma\Sigma^{m_1}X_1)$. 
Then without much difficulties we can show
\begin{align*}
&\Sigma^{n-3}\big((1_{\Sigma^{m_1}X_1}\wedge\tau(\s^1,\s^{m_{[n,2]}}))\circ \Sigma^{m_{[n,2]}}\theta\circ (1_{\Sigma^{m_1}X_1}\wedge\tau(\s^{m_{[n,2]}},\s^1))\big)\\
&=(1_{\Sigma^{m_{[n-2,1]}}X_1}\wedge\tau(\s^{n-2},\s^{m_{[n,n-1]}}))\circ \Sigma^{m_{[n,n-1]}}\Sigma^{n-3}\big((1_{\Sigma^{m_1}X_1}\wedge\tau(\s^1,\s^{m_{[n-2,2]}}))\\
&\hspace{1cm}\circ \Sigma^{m_{[n-2,2]}}\theta\circ (1_{\Sigma^{m_1}X_1}\wedge\tau(\s^{m_{[n-2,2]}},\s^1))\big)
\circ (1_{\Sigma^{m_{[n-2,1]}}X_1}\wedge(\s^{m_{[n,n-1]}},\s^{n-2})).
\end{align*}
Since 
\begin{align*}
&\Sigma^{n-3}\big((1_{\Sigma^{m_1}X_1}\wedge\tau(\s^1,\s^{m_{[n-2,2]}}))
\circ \Sigma^{m_{[n-2,2]}}\theta\circ (1_{\Sigma^{m_1}X_1}\wedge\tau(\s^{m_{[n-2,2]}},\s^1))\big)\\
&\hspace{3cm}\in \mathscr{E}(\Sigma^{n-2}\Sigma^{m_{[n-2,1]}}X_1),
\end{align*}
we have the assertion.
\end{proof}

\begin{proof}[Proof of the third, fourth and last equalities in Theorem 4.1(2)]
We prove 
\begin{gather}
\begin{split}
\{\vec{\bm f}\,\}^{(qs_2)}_{\vec{\bm m}}
&\subset
\{\vec{\bm f}\,\}^{(aq\ddot{s}_2)}_{\vec{\bm m}}\circ(1_{ \Sigma ^{m_{[n-2,1]}}X_1}\wedge\tau(\s^{n-2},\s^{m_{[n,n-1]}}))\\
&\hspace{1cm}\circ \Sigma^{m_{[n,n-1]}}\mathscr{E}(\Sigma^{n-2}\Sigma^{m_{[n-2,1]}}X_1)\circ(1_{\Sigma^{m_{[n-2,1]}}X_1}\wedge\tau(\s^{m_{[n,n-1]}},\s^{n-2})),
\end{split}\\
\begin{split}
\{\vec{\bm f}\,\}^{(aq\ddot{s}_2)}_{\vec{\bm m}}&\circ(1_{ \Sigma ^{m_{[n-2,1]}}X_1}\wedge\tau(\s^{n-2},\s^{m_{[n,n-1]}}))\\
&\hspace{1cm}\circ \Sigma^{m_{[n,n-1]}}\mathscr{E}(\Sigma^{n-2}\Sigma^{m_{[n-2,1]}}X_1)\circ(1_{\Sigma^{m_{[n-2,1]}}X_1}\wedge\tau(\s^{m_{[n,n-1]}},\s^{n-2}))\\
&\subset\begin{cases}\{\vec{\bm f}\}^{(q\dot{s}_2)}_{\vec{\bm m}} & n=3\\ \{\vec{\bm f}\}^{(q\ddot{s}_2)}_{\vec{\bm m}} & n\ge 4\end{cases}.
\end{split}
\end{gather}
If these are proved, then 
\begin{align*}
\{\vec{\bm f}\,\}^{(qs_2)}_{\vec{\bm m}}
&\subset
\{\vec{\bm f}\,\}^{(aq\ddot{s}_2)}_{\vec{\bm m}}\circ(1_{ \Sigma ^{m_{[n-2,1]}}X_1}\wedge\tau(\s^{n-2},\s^{m_{[n,n-1]}}))\\
&\hspace{1cm}\circ \Sigma^{m_{[n,n-1]}}\mathscr{E}(\Sigma^{n-2}\Sigma^{m_{[n-2,1]}}X_1)\circ(1_{\Sigma^{m_{[n-2,1]}}X_1}\wedge\tau(\s^{m_{[n,n-1]}},\s^{n-2}))\\
&\subset\begin{cases}\{\vec{\bm f}\}^{(q\dot{s}_2)}_{\vec{\bm m}}\subset\{\vec{\bm f}\}^{(qs_2)}_{\vec{\bm m}} & n=3\\ \{\vec{\bm f}\}^{(q\ddot{s}_2)}_{\vec{\bm m}}\subset\{\vec{\bm f}\}^{(q\dot{s}_2)}_{\vec{\bm m}}\subset \{\vec{\bm f}\}^{(qs_2)}_{\vec{\bm m}} & n\ge 4\end{cases}.
\end{align*}
so that the third and fourth equalities in (2) and the last assertion of (2) 
follow. 

To prove (4.5), let $\alpha\in\{\vec{\bm f}\,\}^{(qs_2)}_{\vec{\bm m}}$ and $\{\mathscr{S}_r,\overline{f_r},\Omega_r\,|\,2\le r\le n\}$ a 
$qs_2$-presentation of $\vec{\bm f}$ with $\alpha=\overline{f_n}\circ  \Sigma ^{m_n}g_{n,n-1}\circ(1_{ \Sigma ^{m_{n-1}}\dots \Sigma ^{m_1}X_1}\wedge\tau(\s^{m_n},\s^{n-2}))$. 
Set 
$$
\mathscr{S}_2'=( \Sigma ^{m_1}X_1;X_2,X_2\cup_{f_1}C \Sigma ^{m_1}X_1;f_1;i_{f_1}),\quad j_{2,1}'=i_{f_1},\quad \omega_{2,1}'=q_{f_1}',\quad \Omega_2'=\{\omega_{2,1}'\}.
$$ 
Since $j_{2,1}$ is a homotopy cofibre of $f_1$ by the hypothesis, there exists a homotopy equivalence $e_2:C_{2,2}'=X_2\cup_{f_1}C \Sigma ^{m_1}X_1\to C_{2,2}$ such that $e_2\circ j_{2,1}'=j_{2,1}$. 
Set $\overline{f_2}'=\overline{f_2}\circ  \Sigma ^{m_2}e_2$. 
Then $\overline{f_2}'$ is an extension of $f_2$ to $C_{2,2}'$. 
Set 
\begin{gather*}
\mathscr{S}_3'=(\widetilde{ \Sigma }^{m_2}\mathscr{S}_2')(\overline{f_2}',\widetilde{ \Sigma }^{m_2}\Omega_2'),\ \Omega_3'=\widetilde{\widetilde{ \Sigma }^{m_2}\Omega_2'},\ 
e_3=1_{X_3}\cup C \Sigma ^{m_2}e_2:C_{3,3}'\to C_{3,3},\\
\overline{f_3}'=\begin{cases} \overline{f_3}: \Sigma ^{m_3}C_{3,2}'= \Sigma ^{m_3}C_{3,2}\to X_4 & n=3\\ \overline{f_3}\circ  \Sigma ^{m_3}e_3: \Sigma ^{m_3}C_{3,3}'\to X_4 & n\ge 4\end{cases}.
\end{gather*}
Proceeding with the construction, we have an $aq\ddot{s}_2$-presentation $\{\mathscr{S}_r',\overline{f_r}',\Omega_r'\,|\,2\le r\le n\}$ 
of $\vec{\bm f}$ and homotopy equivalences $e_r:C'_{r,r}\to C_{r,r}$ such that 
\begin{gather*}
C_{r,s}'=C_{r,s}\ (1\le s\le r-1),\ j_{r,s}'=j_{r,s}\ (1\le s\le r-2),\\
e_r=1_{X_r}\cup C \Sigma ^{m_{r-1}}e_{r-1}:C_{r,r}'\to C_{r,r}\ (3\le r\le n),\\
\overline{f_r}'=\begin{cases} \overline{f_r}\circ  \Sigma ^{m_r}e_r :  \Sigma ^{m_r}C_{r,r}'\to X_{r+1} & r<n\\ \overline{f_n} :  \Sigma ^{m_n}C_{n,n-1}'= \Sigma ^{m_n}C_{n,n-1}\to X_{n+1} & r=n\end{cases}.
\end{gather*} 
Set 
\begin{align*}
\theta_{n-1}&=\omega_{n-1,n-2}\circ(e_{n-1}\cup C1_{C_{n-1,n-2}})\circ{\omega_{n-1,n-2}'}^{-1},\\
\theta_{n-1}'&=(1_{\Sigma^{m_{[n-2,1]}}X_1}\wedge\tau(\s^{n-2},\s^{m_{n-1}})\circ \Sigma^{m_{n-1}}\theta_{n-1}\circ(1_{\Sigma^{m_{[n-2,1]}}X_1}\wedge\tau(\s^{m_{n-1}},\s^{n-2})).
\end{align*}
Then 
$$
\widetilde{\Sigma}^{m_{n-1}}\omega_{n-1,n-2}\circ(\Sigma^{m_{n-1}}e_{n-1}\cup C1_{\Sigma^{m_{n-1}}C_{n-1,n-2}})\simeq\theta'_{n-1}\circ\widetilde{\Sigma}^{m_{n-1}}\omega'_{n-1,n-2}
$$ and so 
\begin{align*}
g_{n,n-1}\circ\theta'_{n-1}&=(\overline{f_{n-1}}\cup C1_{\Sigma^{m_{n-1}}C_{n-1,n-2}})\circ(\widetilde{\Sigma}^{m_{n-1}}\omega_{n-1,n-2})^{-1}\circ\theta'_{n-1}\\
&\simeq(\overline{f'_{n-1}}\cup C1_{\Sigma^{m_{n-1}}C_{n-1,n-2}})\circ(\widetilde{\Sigma}^{m_{n-1}}
\omega'_{n-1,n-2})^{-1}=g'_{n,n-1}.
\end{align*}
Hence we have 
\begin{align*}
&\widetilde{\Sigma}^{m_n}g_{n,n-1}=\Sigma^{m_n}g_{n,n-1}\circ(1_{\Sigma^{m_{[n-1,1]}}X_1}\wedge\tau(\s^{m_n},\s^{n-2}))\\
&\simeq \Sigma^{m_n}(g'_{n,n-1}\circ{\theta'_{n-1}}^{-1})\circ(1_{\Sigma^{m_{[n-1,1]}}X_1}\wedge\tau(\s^{m_n},\s^{n-2}))\\
&\simeq \Sigma^{m_n}g'_{n,n-1}\circ (1_{\Sigma^{m_{[n-1,1]}}X_1}\wedge\tau(\s^{m_n},\s^{n-2}))\circ (1_{\Sigma^{m_{[n-1,1]}}X_1}\wedge\tau(\s^{n-2},\s^{m_n}))\\
&\hspace{2cm}\circ \Sigma^{m_n}{\theta'_{n-1}}^{-1}\circ (1_{\Sigma^{m_{[n-1,1]}}X_1}\wedge\tau(\s^{m_n},\s^{n-2}))\\
&\simeq \widetilde{\Sigma}^{m_n}g'_{n,n-1}\circ (1_{\Sigma^{m_{[n-1,1]}}X_1}\wedge\tau(\s^{n-2},\s^{m_n}))\circ \Sigma^{m_n}{\theta'_{n-1}}^{-1}\circ (1_{\Sigma^{m_{[n-1,1]}}X_1}\wedge\tau(\s^{m_n},\s^{n-2}))\\
&\simeq \widetilde{\Sigma}^{m_n}g'_{n,n-1}\circ (1_{\Sigma^{m_{[n-1,1]}}X_1}\wedge\tau(\s^{n-2},\s^{m_n}))\\
&\hspace{5mm}\circ \Sigma^{m_n}\big((1_{\Sigma^{m_{[n-2,1]}}X_1}\wedge\tau(\s^{n-2},\s^{m_{n-1}}))\circ \Sigma^{m_{n-1}}\theta_{n-1}^{-1}\circ (1_{\Sigma^{m_{[n-2,1]}}X_1}\wedge\tau(\s^{m_{n-1}},\s^{n-2}))\big)\\
&\hspace{5mm}\circ(1_{\Sigma^{m_{[n-1,1]}}X_1}\wedge\tau(\s^{m_n},\s^{n-2}))\\
&\simeq \widetilde{\Sigma}^{m_n}g'_{n,n-1}
\circ(1_{\Sigma^{m_{[n-2,1]}}X_1}\wedge\tau(\s^{n-2},\s^{m_{[n,n-1]}}))
\circ \Sigma^{m_{[n,n-1]}}\theta_{n-1}^{-1}\\
&\hspace{2cm}
\circ(1_{\Sigma^{m_{[n-2,1]}}X_1}\wedge\tau(\s^{m_{[n,n-1]}},\s^{n-2}))
\end{align*}
Hence $\alpha=\overline{f_n}\circ\widetilde{\Sigma}^{m_n}g_{n,n-1}$ is in 
the right hand term of (4.5). 
This proves (4.5). 

To prove (4.6), let $\alpha\in\{\vec{\bm f}\,\}^{(aq\ddot{s}_2)}_{\vec{\bm m}}$, $\gamma\in\mathscr{E}( \Sigma ^{n-2} \Sigma ^{m_{[n-2,1]}}X_1)$, and 
$\{\mathscr{S}_r,\overline{f_r},\Omega_r\,|\,2\le r\le n\}$ an $aq\ddot{s}_2$-presentation of $\vec{\bm f}$ such that $\alpha=\overline{f_n}\circ \widetilde{ \Sigma }^{m_n}g_{n,n-1} 
=\overline{f_n}\circ \Sigma^{m_n}g_{n,n-1}\circ(1_{ \Sigma ^{m_{n-1}}\cdots \Sigma ^{m_1}X_1}\wedge\tau(\s^{m_n},\s^{n-2}))$. 
Set $\Omega_{n-1}'=\{\omega_{n-1,s}\,|\,1\le s\le n-3\}\cup\{\gamma\circ\omega_{n-1,n-2}\}$ which is a quasi-structure on $\mathscr{S}_{n-1}$. 
Set $\mathscr{S}_n'=(\widetilde{\Sigma}^{m_{n-1}}\mathscr{S}_{n-1})(\overline{f_{n-1}},\widetilde{\Sigma}^{m_{n-1}}\Omega_{n-1}')$ and $\Omega_n'=\widetilde{\widetilde{\Sigma}^{m_{n-1}}\Omega'_{n-1}}$. 
Let $\{\mathscr{S}'_r,\overline{f_r}',\Omega'_r\,|\,2\le r\le n\}$ be obtained from 
$\{\mathscr{S}_r,\overline{f_r},\Omega_r\,|\,2\le r\le n\}$ by replacing $\Omega_{n-1},\mathscr{S}_n,\Omega_n$ with $\Omega_{n-1}', \mathscr{S}_n', 
\Omega_n'$, respectively. 
Then the new collection is a $q\dot{s}_2$-presentation if $n=3$ and an $q\ddot{s}_2$-presentation if $n\ge 4$ of $\vec{\bm f}$. 
Set 
\begin{align*}
\theta &=(1_{\Sigma^{m_{[n-2,1]}}X_1}\wedge\tau(\s^{n-2},\s^{m_{n-1}}))\circ \Sigma^{m_{n-1}}\gamma\circ (1_{\Sigma^{m_{[n-2,1]}}X_1}\wedge\tau(\s^{m_{n-1}},\s^{n-2})),\\
\theta' &=(1_{\Sigma^{m_{[n-1,1]}}X_1}\wedge\tau(\s^{n-2},\s^{m_n}))\circ \Sigma^{m_n}\theta\circ (1_{\Sigma^{m_{[n-1,1]}}X_1}\wedge\tau(\s^{m_n},\s^{n-2})).
\end{align*}
Then 
$\overline{f_n}\circ \widetilde{\Sigma}^{m_n}g_{n,n-1}\circ\theta'^{-1}
\simeq \overline{f_n}\circ \widetilde{\Sigma}^{m_n}g_{n,n-1}'$ which represents 
an element of $\begin{cases} \{\vec{\bm f}\}^{(q\dot{s}_2)}_{\vec{\bm m}} & n=3\\
\{\vec{\bm f}\}^{(q\ddot{s}_2)}_{\vec{\bm m}} & n\ge 4\end{cases}
$. 
We have 
$$
\theta'^{-1}\simeq 
(1_{\Sigma^{m_{[n-2,1]}}X_1}\wedge\tau(\s^{n-2},\s^{m_{[n,n-1]}}))\circ \Sigma^{m_{[n,n-1]}}\gamma^{-1}\circ (1_{\Sigma^{m_{[n-2,1]}}X_1}\wedge\tau(\s^{m_{[n,n-1]}},\s^{n-2})).
$$
Hence we have (4.6). 

When $n=3$, $\{\vec{\bm f}\}^{(aq\ddot{s}_2)}_{\vec{\bm m}}=
\{\vec{\bm f}\}^{(q\ddot{s}_2)}_{\vec{\bm m}}$ by definitions. 
\end{proof}

\begin{proof}[Proof of Theorem 4.1(3)] 
It suffices to prove
\begin{align*}
\{\vec{\bm f}\,\}^{(q)}_{\vec{\bm m}}&=\{\vec{\bm f}\,\}^{(aq)}_{\vec{\bm m}}\circ(1_{ \Sigma ^{m_{[n-1,1]}}X_1}\wedge\tau(\s^{n-2},\s^{m_n}))\\
&\hspace{0.5cm}
\circ \Sigma ^{m_n}\Big((1_{ \Sigma ^{m_{[n-2,1]}}X_1}\wedge\tau(\s^{n-2},\s^{m_{n-1}}))
\circ \Sigma ^{m_{n-1}}\mathscr{E}( \Sigma ^{n-2} \Sigma ^{m_{[n-2,1]}}X_1)\\
&\hspace{0.5cm}
\circ(1_{ \Sigma ^{m_{[n-2,1]}}X_1}\wedge\tau(\s^{m_{n-1}},\s^{n-2}))\Big)
\circ(1_{ \Sigma ^{m_{[n-1,1]}}X_1}\wedge\tau(\s^{m_n},\s^{n-2}))\\
&=\{\vec{\bm f}\,\}^{(aq)}_{\vec{\bm m}}\circ(1_{ \Sigma ^{m_{[n-2,1]}}X_1}\wedge\tau(\s^{n-2},\s^{m_{[n,n-1]}}))\\
&\hspace{5mm} \circ \Sigma ^{m_{[n,n-1]}}\mathscr{E}(\Sigma^{n-2}\Sigma^{m_{[n-2,1]}}X_1)
\circ(1_{ \Sigma ^{m_{[n-2,1]}}X_1}\wedge\tau(\s^{m_{[n,n-1]}},\s^{n-2})).
\end{align*}
The second equality is easily proved. 
We prove the first equality. 
First we prove ``$\,\subset\,$''. 
Let $\alpha\in \{\vec{\bm f}\,\}^{(q)}_{\vec{\bm m}}$ and $\{\mathscr{S}_r,\overline{f_r},\Omega_r\,|\,2\le r\le n\}$ 
a $q$-presentation of $\vec{\bm f}$ with $\alpha=\overline{f_n}\circ  \Sigma ^{m_n}g_{n,n-1}\circ(1_{ \Sigma ^{m_{[n-1,1]}}X_1}\wedge\tau(\s^{m_n},\s^{n-2}))$. 
We define inductively
\begin{gather*}
\mathscr{S}_2'=\mathscr{S}_2,\, \Omega_2'=\Omega_2;\ \mathscr{S}_3'
=(\widetilde{ \Sigma }^{m_2}\mathscr{S}_2')(\overline{f_2},\widetilde{ \Sigma }^{m_2}\Omega_2'),\, \Omega_3'=\widetilde{\widetilde{ \Sigma }^{m_2}\Omega_2'};\\ 
\mathscr{S}_4'=(\widetilde{ \Sigma }^{m_3}\mathscr{S}_3')(\overline{f_3},\widetilde{ \Sigma }^{m_3}\Omega_3'),\,\Omega_4'=\widetilde{\widetilde{ \Sigma }^{m_3}\Omega_3'};\ \dots .
\end{gather*}
By Remark 5.5(3) of \cite{OO}, this definition is possible and $\mathscr{S}_r', \mathscr{S}_r$ have the same edge. 
Then $\{\mathscr{S}_r',\overline{f_r},\Omega_r'\,|\,2\le r\le n\}$ is an $aq$-presentation of $\vec{\bm f}$. 
Set 
$$
\gamma=\omega_{n-1,n-2}\circ\omega_{n-1,n-2}'^{-1}\in\mathscr{E}( \Sigma ^{n-2} \Sigma ^{m_{[n-2,1]}}X_1).
$$
Then
\begin{align*}
\widetilde{ \Sigma }^{m_{n-1}}\omega_{n-1,n-2}&\simeq (1_{ \Sigma ^{m_{[n-2,1]}}X_1}\wedge\tau(\s^{n-2},\s^{m_{n-1}})\\
&\hspace{5mm}
\circ \Sigma ^{m_{n-1}}\gamma\circ(1_{ \Sigma ^{m_{[n-2,1]}}X_1}\wedge\tau(\s^{m_{n-1}},\s^{n-2}))\circ\widetilde{ \Sigma }^{m_{n-1}}\omega_{n-1,n-2}'
\end{align*} 
and
\begin{align*}
\alpha&= \overline{f_n}\circ  \Sigma ^{m_n}g_{n,n-1}\circ(1_{ \Sigma ^{m_{[n-1,1]}}X_1}\wedge\tau(\s^{m_n},\s^{n-2}))\\
&=\overline{f_n}\circ \Sigma ^{m_n}\Big((\overline{f_{n-1}}\cup C1_{ \Sigma ^{m_{n-1}}C_{n-1,n-2}})\circ(\widetilde{ \Sigma }^{m_{n-1}}\omega_{n-1,n-2})^{-1}\Big)\\
&\hspace{1cm}\circ(1_{ \Sigma ^{m_{[n-1,1]}}X_1}\wedge\tau(\s^{m_n},\s^{n-2}))\\
&=\overline{f_n}\circ \Sigma ^{m_n}\Big((\overline{f_{n-1}}\cup C1_{ \Sigma ^{m_{n-1}}C_{n-1,n-2}})
\circ(\widetilde{ \Sigma }^{m_{n-1}}\omega_{n-1,n-2}')^{-1}\\
&\hspace{1cm}\circ(1_{ \Sigma ^{m_{[n-2,1]}}X_1}\wedge\tau(\s^{n-2},\s^{m_{n-1}})\circ( \Sigma ^{m_{n-1}}\gamma)^{-1}\\
&\hspace{1cm}\circ 
(1_{ \Sigma ^{m_{[n-1,1]}}X_1}\wedge\tau(\s^{m_n},\s^{n-2}))\Big)
\circ(1_{ \Sigma ^{m_{[n-1,1]}}X_1}\wedge\tau(\s^{m_n},\s^{n-2}))\\
&=\overline{f_n}\circ \Sigma ^{m_n}\Big((\overline{f_{n-1}}\cup C1_{ \Sigma ^{m_{n-1}}C_{n-1,n-2}})
\circ(\widetilde{ \Sigma }^{m_{n-1}}\omega_{n-1,n-2}')^{-1}\Big)\\
&\hspace{1cm}\circ(1_{ \Sigma ^{m_{[n-1,1]}}X_1}\wedge\tau(\s^{m_n},\s^{n-2}))
\circ(1_{ \Sigma ^{m_{[n-1,1]}}X_1}\wedge\tau(\s^{n-2},\s^{m_n}))\\
&\hspace{1cm}\circ \Sigma ^{m_n}\Big((1_{ \Sigma ^{m_{[n-2,1]}}X_1}\wedge\tau(\s^{n-2},\s^{m_{n-1}})\circ( \Sigma ^{m_{n-1}}\gamma)^{-1}\\
&\hspace{1cm}\circ 
(1_{ \Sigma ^{m_{[n-1,1]}}X_1}\wedge\tau(\s^{m_n},\s^{n-2}))\Big)
\circ(1_{ \Sigma ^{m_{[n-1,1]}}X_1}\wedge\tau(\s^{m_n},\s^{n-2}))\\
&\in\{\vec{\bm f}\,\}^{(aq)}_{\vec{\bm m}}\circ(1_{ \Sigma ^{m_{[n-1,1]}}X_1}\wedge\tau(\s^{n-2},\s^{m_n}))\\
&\hspace{1cm}\circ \Sigma ^{m_n}\Big((1_{ \Sigma ^{m_{[n-2,1]}}X_1}\wedge\tau(\s^{n-2},\s^{m_{n-1}})\circ( \Sigma ^{m_{n-1}}\gamma)^{-1}\\
&\hspace{1cm}\circ 
(1_{ \Sigma ^{m_{[n-1,1]}}X_1}\wedge\tau(\s^{m_n},\s^{n-2}))\Big)
\circ(1_{ \Sigma ^{m_{[n-1,1]}}X_1}\wedge\tau(\s^{m_n},\s^{n-2}))\\
&\subset\{\vec{\bm f}\,\}^{(aq)}_{\vec{\bm m}}\circ 
(1_{ \Sigma ^{m_{[n-1,1]}}X_1}\wedge\tau(\s^{n-2},\s^{m_n}))\\
&\hspace{1cm}\circ \Sigma ^{m_n}\Big((1_{ \Sigma ^{m_{[n-2,1]}}X_1}\wedge\tau(\s^{n-2},\s^{m_{n-1}})\circ \Sigma ^{m_{n-1}}\mathscr{E}( \Sigma ^{n-2} \Sigma ^{m_{[n-2,1]}}X_1)\\
&\hspace{1cm}\circ 
(1_{ \Sigma ^{m_{[n-1,1]}}X_1}\wedge\tau(\s^{m_n},\s^{n-2}))\Big)
\circ(1_{ \Sigma ^{m_{[n-1,1]}}X_1}\wedge\tau(\s^{m_n},\s^{n-2}))
\end{align*}
Hence ``$\subset$'' is obtained. 

The converse containment ``$\supset$'' is obtained as follows. 
Let $\alpha\in\{\vec{\bm f}\}^{(aq)}_{\vec{\bm m}}$ which is represented by an $aq$-presentation $\{\mathscr{S}_r,\overline{f_r},\Omega_t\,|\,2\le r\le n\}$. 
Let $\gamma\in\mathscr{E}(\Sigma^{n-2}\Sigma^{m_{[n-2,1]}}X_1)$. 
Set $\Omega_{n-1}'$ be obtained from $\Omega_{n-1}$ by replacing $\omega_{n-1,n-2}$ with $\gamma^{-1}\circ\omega_{n-1,n-2}$. 
Set $\mathscr{S}_n'=(\widetilde{\Sigma}^{m_{n-1}}\mathscr{S}_{n-1})(\overline{f_{n-1}},\widetilde{\Sigma}^{m_{n-1}}\Omega_{n-1}')$ and $\Omega_n'=\widetilde{\widetilde{\Sigma}^{m_{n-1}}\Omega_{n-1}'}$. 
Let $\{\mathscr{S}_r',\overline{f_r}',\Omega_r'\,|\,2\le r\le n\}$ be obtained from 
$\{\mathscr{S}_r,\overline{f_r},\Omega_t\,|\,2\le r\le n\}$ by replacing $\Omega_{n-1}, 
\mathscr{S}_n$ with $\Omega_{n-1}',\mathscr{S}_n'$, respectively. 
The new collection is a $q$-presentation of $\vec{\bm f}$ such that 
\begin{align*}
g'_{n,n-1}\simeq g_{n,n-1}&\circ(1_{\Sigma^{m_{[n-2,1]}}X_1}\wedge\tau(\s^{n-2},\s^{m_{n-1}}))\\
&\circ \Sigma^{m_{n-1}}\gamma\circ(1_{\Sigma^{m_{[n-2,1]}}X_1}\wedge\tau(\s^{m_{n-1}},\s^{n-2})).
\end{align*}
Then
\begin{align*}
\{\vec{\bm f}\}^{(q)}_{\vec{\bm m}}\ni &\overline{f_n}\circ \Sigma^{m_n}g_{n,n-1}'\circ(1_{\Sigma^{m_{[n-1,1]}}X_1}\wedge\tau(\s^{m_n},\s^{n-2}))\\
&\quad \simeq \overline{f_n}\circ \Sigma^{m_n}g_{n,n-1}\circ(1_{\Sigma^{m_{[n-1,1]}}X_1}\wedge\tau(\s^{m_n},\s^{n-2}))\\
&\qquad \circ(1_{\Sigma^{m_{[n-1,1]}}X_1}\wedge\tau(\s^{n-2},\s^{m_n}))\circ \Sigma^{m_n}
\Big[(1_{\Sigma^{m_{[n-2,1]}}X_1}\wedge\tau(\s^{n-2},\s^{m_{n-1}}))\\
&\qquad \circ \Sigma^{m_{n-1}}\gamma\circ(1_{\Sigma^{m_{[n-2,1]}}X_1}\wedge\tau(\s^{m_{n-1}},\s^{n-2}))
\Big]\circ(1_{\Sigma^{m_{[n-1,1]}}X_1}\wedge\tau(\s^{m_n},\s^{n-2})).
\end{align*}
This proves ``$\supset$'' and ends the proof of Theorem 4.1(3). 
\end{proof}

\begin{proof}[Proof of Theorem 4.1(4)] 
We define a subgroup $\Gamma$ of $\mathscr{E}(\Sigma^{|\vec{\bm m}|+n-2}X_1)$ by  
\begin{align*}
\Gamma=\big(1_{\Sigma^{m_{[n-1,1]}}X_1}\wedge\tau(\s^{n-2},\s^{m_n})\big)
&\circ \Sigma^{m_n}\mathscr{E}(\Sigma^{n-2}\Sigma^{m_{[n-1,1]}}X_1)\\
&\circ \big(1_{\Sigma^{m_{[n-1,1]}}X_1}\wedge\tau(\s^{m_n},\s^{n-2})\big).
\end{align*}
We prove
\begin{gather}
\{\vec{\bm f}\,\}^{(\ddot{s}_t)}_{\vec{\bm m}}\subset \{\vec{\bm f}\,\}^{(aq\ddot{s}_2)}_{\vec{\bm m}}\circ\Gamma,\\ 
\{\vec{\bm f}\,\}^{(\ddot{s}_t)}_{\vec{\bm m}}\circ\Gamma
\supset \{\vec{\bm f}\,\}^{(aq\ddot{s}_2)}_{\vec{\bm m}}. 
\end{gather}
If these are done, then, by applying $\Gamma$ to them from the right, 
we have the equality.  
To prove (4.7), let 
$\alpha\in\{\vec{\bm f}\,\}^{(\ddot{s}_t)}_{\vec{\bm m}}$ and $\{\mathscr{S}_r,\overline{f_r},\mathscr{A}_r\,|\,2\le r\le n\}$ 
an $\ddot{s}_t$-presentation of $\vec{\bm f}$ with $\alpha=\overline{f_n}\circ \widetilde{ \Sigma }^{m_n}g_{n,n-1}
=\overline{f_n}\circ  \Sigma ^{m_n}g_{n,n-1}\circ(1_{\Sigma^{m_{[n-1,1]}}X_1}\wedge\tau(\s^{m_n},\s^{n-2})))$. 
We define inductively
\begin{gather*}
\mathscr{S}_2'=\mathscr{S}_2, \Omega_2'=\Omega(\mathscr{A}_2);\, 
\mathscr{S}_3'=(\widetilde{ \Sigma }^{m_2}\mathscr{S}_2')(\overline{f_2},\widetilde{ \Sigma }^{m_2}\Omega_2'), 
\Omega_3'=\widetilde{\widetilde{ \Sigma }^{m_2}\Omega'_2};\\ 
\mathscr{S}_4'=(\widetilde{ \Sigma }^{m_3}\mathscr{S}_3')(\overline{f_3},\widetilde{ \Sigma }^{m_3}\Omega_3'), \Omega_4'=\widetilde{\widetilde{ \Sigma }^{m_3}\Omega'_3};\ \dots .
\end{gather*}
By Remark 5.5(3) of \cite{OO}, this definition is possible, and $\mathscr{S}_r', \mathscr{S}_r$ have the same edge. 
Then $\{\mathscr{S}_r',\overline{f_r},\Omega_r'\,|\,2\le r\le n\}$ is an $aq\ddot{s}_2$-presentation  of 
$\vec{\bm f}$ and 
\begin{align*}
\{\vec{\bm f}\,\}^{(aq\ddot{s}_2)}_{\vec{\bm m}}\ni& \overline{f_n}\circ \widetilde{ \Sigma }^{m_n}g_{n,n-1}'
=\overline{f_n}\circ \Sigma^{m_n}g'_{n,n-1}\circ(1_{\Sigma^{m_{[n-1,1]}}X_1}\wedge\tau(\s^{m_n},\s^{n-2}))\\
&=\overline{f_n}\circ \Sigma ^{m_n}\big((\overline{f_{n-1}}\cup C1_{ \Sigma ^{m_{n-1}}C_{n-1,n-2}})\circ(\widetilde{ \Sigma }^{m_{n-1}}\omega_{n-1,n-2}')^{-1}\big)\\
&\hspace{1cm}\circ(1_{\Sigma^{m_{[n-1,1]}}X_1}\wedge\tau(\s^{m_n},\s^{n-2}))\\
&=\overline{f_n}\circ \Sigma ^{m_n}\big((\overline{f_{n-1}}\cup C1_{ \Sigma ^{m_{n-1}}C_{n-1,n-2}})\circ(\widetilde{ \Sigma }^{m_{n-1}}\omega_{n-1,n-2})^{-1}\big)\\
&\hspace{1cm}\circ \Sigma ^{m_n}\big((\widetilde{ \Sigma }^{m_{n-1}}\omega_{n-1,n-2})
\circ(\widetilde{ \Sigma }^{m_{n-1}}\omega_{n-1,n-2}')^{-1}\big)\\
&\hspace{1cm}\circ(1_{ \Sigma ^{m_{n-1}}\cdots \Sigma ^{m_1}X_1}\wedge\tau(\s^{m_n},\s^{n-2}))\\
&=\overline{f_n}\circ \Sigma ^{m_n}\big((\overline{f_{n-1}}\cup C1_{ \Sigma ^{m_{n-1}}C_{n-1,n-2}})\circ(\widetilde{ \Sigma }^{m_{n-1}}\omega_{n-1,n-2})^{-1}\big)\\
&\hspace{1cm}\circ(1_{ \Sigma ^{m_{[n-1,1]}}X_1}\wedge\tau(\s^{m_n},\s^{n-2}))\\
&\hspace{1cm}\circ(1_{ \Sigma ^{m_{[n-1,1]}}X_1}\wedge\tau(\s^{n-2},\s^{m_n}))\\
&\hspace{1cm}\circ \Sigma ^{m_n}\big((\widetilde{ \Sigma }^{m_{n-1}}\omega_{n-1,n-2})
\circ(\widetilde{ \Sigma }^{m_{n-1}}\omega_{n-1,n-2}')^{-1}\big)\\
&\hspace{1cm}\circ(1_{ \Sigma ^{m_{[n-1,1]}}X_1}\wedge\tau(\s^{m_n},\s^{n-2}))\\
&=\alpha\circ\varepsilon_0,
\end{align*}
where 
\begin{align*}
\varepsilon_0&=(1_{ \Sigma ^{m_{[n-1,1]}}X_1}\wedge\tau(\s^{n-2},\s^{m_n}))\\
&\hspace{1cm}\circ \Sigma ^{m_n}\big((\widetilde{ \Sigma }^{m_{n-1}}\omega_{n-1,n-2})
\circ(\widetilde{ \Sigma }^{m_{n-1}}\omega_{n-1,n-2}')^{-1}\big)\\
&\hspace{1cm}\circ(1_{ \Sigma ^{m_{[n-1,1]}}X_1}\wedge\tau(\s^{m_n},\s^{n-2}))\\
&\in\Gamma.
\end{align*} 
Hence 
$$
\alpha\in\{\vec{\bm f}\,\}^{(aq\ddot{s}_2)}_{\vec{\bm m}}\circ\varepsilon^{-1}_0\subset
\{\vec{\bm f}\,\}^{(aq\ddot{s}_2)}_{\vec{\bm m}}
\circ\Gamma. 
$$
This proves (4.7). 

To prove (4.8), let $\alpha\in\{\vec{\bm f}\,\}^{(aq\ddot{s}_2)}_{\vec{\bm m}}$ and $\{\mathscr{S}_r,\overline{f_r},\Omega_r\,|\,2\le r\le n\}$ 
an $aq\ddot{s}_2$-presentation of $\vec{\bm f}$ with $\alpha=\overline{f_n}\circ \Sigma ^{m_n} g_{n,n-1}\circ(1_{ \Sigma ^{m_{[n-1,1]}}X_1}\wedge\tau(\s^{m_n},\s^{n_2}))$. 
We define inductively
\begin{align*}
&\mathscr{S}_2'=\mathscr{S}_2,\quad \mathscr{A}_2'=\{1_{X_2\cup_{f_1}C\Sigma^{m_1}X_1}\},\\ 
&\mathscr{S}_3'=(\widetilde{ \Sigma }^{m_2}\mathscr{S}_2')(\overline{f_2}, \widetilde{ \Sigma }^{m_2}\mathscr{A}_2'),\quad \mathscr{A}_3'\text{ is a reduced structure on $\mathscr{S}_3'$},\\
&\mathscr{S}_4'=(\widetilde{ \Sigma }^{m_3}\mathscr{S}_3')(\overline{f_3}, \widetilde{ \Sigma }^{m_3}\mathscr{A}_3'),\quad \mathscr{A}_4'\text{ is a reduced structure on $\mathscr{S}_4'$},\\
& \dots .
\end{align*}
By Remark 5.5(3) of \cite{OO}, this definition is possible, and $\mathscr{S}_r', \mathscr{S}_r$ have the same edge. 
Then $\{\mathscr{S}_r',\overline{f_r},\mathscr{A}_r'\,|\,2\le r\le n\}$ is an 
$\ddot{s}_t$-presentation of $\vec{\bm f}$ and 
\begin{align*}
\{\vec{\bm f}\,\}^{(\ddot{s}_t)}_{\vec{\bm m}}\ni& \overline{f_n}\circ \Sigma ^{m_n}g_{n,n-1}'\circ(1_{ \Sigma ^{m_{[n-1,1]}}X_1}\wedge\tau(\s^{m_n},\s^{n-2}))\\
&=\overline{f_n}\circ \Sigma ^{m_n}\big( (\overline{f_{n-1}}\cup C1_{ \Sigma ^{m_{n-1}}C_{n-1,n-2}})\circ(\widetilde{ \Sigma }^{m_{n-1}}\omega_{n-1,n-2}')^{-1}\big)\\
&\hspace{1cm}\circ(1_{ \Sigma ^{m_{[n-1,1]}}X_1}\wedge\tau(\s^{m_n},\s^{n-2}))\\
&=\overline{f_n}\circ \Sigma ^{m_n}\big((\overline{f_{n-1}}\cup C1_{ \Sigma ^{m_{n-1}}C_{n-1,n-2}})\circ(\widetilde{ \Sigma }^{m_{n-1}}\omega_{n-1,n-2})^{-1}\big)\\
&\hspace{1cm}\circ(1_{ \Sigma ^{m_{[n-1,1]}}X_1}\wedge\tau(\s^{m_n},\s^{n-2}))\\
&\hspace{1cm}\circ(1_{ \Sigma ^{m_{[n-1,1]}}X_1}\wedge\tau(\s^{n-2},\s^{m_n}))\\
&\hspace{1cm}\circ \Sigma ^{m_n}\big((\widetilde{ \Sigma }^{m_{n-1}}\omega_{n-1,n-2})
\circ(\widetilde{ \Sigma }^{m_{n-1}}\omega_{n-1,n-2}')^{-1}\big)\\
&\hspace{1cm}\circ(1_{ \Sigma ^{m_{[n-1,1]}}X_1}\wedge\tau(\s^{m_n},\s^{n-2}))\\
&=\alpha\circ\varepsilon_0
\end{align*}
where 
$\Omega(\mathscr{A}_{n-1}')=\{\omega_{n-1,s}'\,|\,1\le s<n-1\}$ and 
\begin{align*}
\varepsilon_0&=(1_{ \Sigma ^{m_{[n-1,1]}}X_1}\wedge\tau(\s^{n-2},\s^{m_n}))\\
&\hspace{1cm}\circ \Sigma ^{m_n}\big((\widetilde{ \Sigma }^{m_{n-1}}\omega_{n-1,n-2})
\circ(\widetilde{ \Sigma }^{m_{n-1}}\omega_{n-1,n-2}')^{-1}\big)\\
&\hspace{1cm}\circ(1_{ \Sigma ^{m_{[n-1,1]}}X_1}\wedge\tau(\s^{m_n},\s^{n-2}))\\
&\in\Gamma.
\end{align*}
Hence $\alpha\in\{\vec{\bm f}\,\}^{(\ddot{s}_t)}_{\vec{\bm m}}\circ\varepsilon_0^{-1}\subset \{\vec{\bm f}\,\}^{(\ddot{s}_t)}_{\vec{\bm m}}\circ\Gamma$. 
This proves (4.8) and completes the proof of Theorem~4.1(4).
\end{proof}

\begin{proof}[Proof of Theorem 4.1(5)]
Let $\alpha\in\{\vec{\bm f}\}^{(q)}_{\vec{\bm m}}$ and $\{\mathscr{S}_r,\overline{f_r},\Omega_r\,|\,2\le r\le n\}$ a 
$q$-presentation of $\vec{\bm f}$ with $\alpha=\overline{f_n}\circ\widetilde{\Sigma} ^{m_n} g_{n,n-1}$. 
We will define an $aq\ddot{s}_2$-presentation $\{\mathscr{S}_r',\overline{f_r}',\Omega_r'\,|\,2\le r\le n\}$ of $\vec{\bm f}$ such that 
$\overline{f_n}'\circ\widetilde{\Sigma}^{m_n}g_{n,n-1}'=\alpha\circ \theta$ for some self map $\theta$ 
of $\Sigma^{|\vec{\bm m}|+n-2}X_1$ 
(notice that $\theta$ is not necessarily a homotopy equivalence). 
Now set
$$
\mathscr{S}_2'=( \Sigma ^{m_1}X_1;X_2,X_2\cup_{f_1}C \Sigma ^{m_1}X_1;f_1;i_{f_1}),\quad \Omega_2'=\{q_{f_1}'\}.
$$
Since $\mathscr{S}_2$ is a quasi iterated mapping cone and $j_{2,1}'=i_{f_1}$ is a cofibration, there exists a map (not necessarily a homotopy equivalence) $e_2:C_{2,2}'\to C_{2,2}$ 
such that $e_2\circ j_{2,1}'=j_{2,1}$. 
Set $\overline{f_2}'=\overline{f_2}\circ  \Sigma ^{m_2}e_2$. 
Then $\overline{f_2}'\circ \Sigma ^{m_2} j_{2,1}'=f_2$ and so $\overline{f_2}'$ is an 
extension of $f_2$ to $ \Sigma ^{m_2}C_{2,2}'$. 
Set 
\begin{align*}
&\mathscr{S}_3'=(\widetilde{ \Sigma }^{m_2}\mathscr{S}_2')(\overline{f_2}',\widetilde{ \Sigma }^{m_2}\Omega_2'),\ \Omega_3'=\widetilde{\widetilde{ \Sigma }^{m_2}\Omega_2'},\ 
 e_3=1_{X_3}\cup C \Sigma ^{m_2}e_2:C_{3,3}'\to C_{3,3},\\ 
&\hspace{2cm}\overline{f_3}'=\begin{cases} \overline{f_3}: \Sigma ^{m_3}C_{3,2}'= \Sigma ^{m_3}C_{3,2}\to X_4 & n=3\\ \overline{f_3}\circ  \Sigma ^{m_3}e_3: \Sigma ^{m_3}C_{3,3}'\to X_4 & n\ge 4\end{cases}.
\end{align*}
Proceeding with the construction, we have an $aq\ddot{s}_2$-presentation $\{\mathscr{S}_r',\overline{f_r}',\Omega_r'\,|\,2\le r\le n\}$ 
of $\vec{\bm f}$ and maps $e_r:C_{r,r}'\to C_{r,r}\ (2\le r\le n)$ such that  
\begin{gather*}
C_{r,s}'=C_{r,s}\ (1\le s\le r-1),\ j_{r,s}'=j_{r,s}\ (1\le s\le r-2),\\
 e_r\circ j_{r,r-1}'=j_{r,r-1}\ (2\le r\le n),\\ 
\overline{f_r}'=\begin{cases} \overline{f_n} :  \Sigma ^{m_n}C_{n,n-1}'= \Sigma ^{m_n}C_{n,n-1}\to X_{n+1} & r=n\\ \overline{f_r}\circ \Sigma ^{m_r} e_r :  \Sigma ^{m_r}C_{r,r}'\to X_{r+1} & r<n\end{cases}.
\end{gather*}
Let $\theta$ be a self map of $\Sigma^{|\vec{\bm m}|+n-2}X_1$ defined by 
\begin{align*}
\theta&=(1_{ \Sigma ^{m_{[n-1,1]}}X_1}\wedge\tau(\s^{n-2},\s^{m_n}))\\
&\hspace{5mm}\circ \Sigma ^{m_n}\Big((\widetilde{ \Sigma }^{m_{n-1}}\omega_{n-1,n-2})\circ
( \Sigma ^{m_{n-1}}e_{n-1}\cup C1_{ \Sigma ^{m_{n-1}}C_{n-1,n-2}})
\circ(\widetilde{ \Sigma }^{m_{n-1}}\omega_{n-1,n-2}')^{-1}
\Big)\\
&\hspace{5mm}\circ(1_{ \Sigma ^{m_{[n-1,1]}}X_1}\wedge\tau(\s^{m_n},\s^{n-2})).
\end{align*} 
Then $\theta\in\Gamma'$ 
and 
\begin{align*}
\{\vec{\bm f}\,\}^{(aq\ddot{s}_2)}_{\vec{\bm m}}&\ni\overline{f_n}'\circ\widetilde{\Sigma} ^{m_n} g_{n,n-1}'\\
&=\overline{f_n}'\circ \Sigma ^{m_n}\big((\overline{f_{n-1}}'\cup C1_{ \Sigma ^{m_{n-1}}C_{n-1,n-2}})\circ(\widetilde{ \Sigma }^{m_{n-1}}\omega_{n-1,n-2}')^{-1}\big)\\
&\hspace{1cm}\circ(1_{ \Sigma ^{m_{[n-1,1]}}X_1}\wedge\tau(\s^{m_n},\s^{n-2}))\\
&=\overline{f_n}\circ \Sigma ^{m_n}\big((\overline{f_{n-1}}\circ \Sigma ^{m_{n-1}}e_{n-1}\cup C1_{ \Sigma ^{m_{n-1}}C_{n-1,n-2}})\\
&\hspace{1cm}\circ(\widetilde{ \Sigma }^{m_{n-1}}\omega_{n-1,n-2}')^{-1}\big)
\circ(1_{ \Sigma ^{m_{[n-1,1]}}X_1}\wedge\tau(\s^{m_n},\s^{n-2}))\\
&=\overline{f_n}\circ \Sigma ^{m_n}\big((\overline{f_{n-1}}\cup C1_{ \Sigma _{m_{n-1}}C_{n-1,n-2}})\circ(\Sigma^{m_{n-1}}e_{n-1}\cup C_{1_{\Sigma^{m_{n-1}}C_{n-1,n-2}}})\\
&\hspace{1cm}\circ(\widetilde{\Sigma}^{m_{n-1}}\omega'_{n-1,n-2})^{-1}\big)
\circ(1_{\Sigma^{m_{[n-1,1]}}X_1}\wedge\tau(\s^{m_n},\s^{n-2}))\\
&=\overline{f_n}\circ \Sigma ^{m_n}\Big((\overline{f_{n-1}}\cup C1_{ \Sigma _{m_{n-1}}C_{n-1,n-2}})
\circ(\widetilde{\Sigma}^{m_{n-1}}\omega_{n-1,n-2})^{-1}\\
&\hspace{1cm}\circ(\widetilde{\Sigma}^{m_{n-1}}\omega_{n-1,n-2})
\circ(\Sigma^{m_{n-1}}e_{n-1}\cup C_{1_{\Sigma^{m_{n-1}}C_{n-1,n-2}}})\\
&\hspace{1cm}\circ(\widetilde{\Sigma}^{m_{n-1}}\omega'_{n-1,n-2})^{-1}\Big)
\circ(1_{\Sigma^{m_{[n-1,1]}}X_1}\wedge\tau(\s^{m_n},\s^{n-2}))\\
&=\overline{f_n}\circ \Sigma ^{m_n}g_{n,n-1}\\
&\hspace{1cm}\circ \Sigma^{m_n}\Big((\widetilde{\Sigma}^{m_{n-1}}\omega_{n-1,n-2})
\circ(\Sigma^{m_{n-1}}e_{n-1}\cup C_{1_{\Sigma^{m_{n-1}}C_{n-1,n-2}}})\\
&\hspace{1cm}\circ(\widetilde{\Sigma}^{m_{n-1}}\omega'_{n-1,n-2})^{-1}\Big)
\circ (1_{\Sigma^{m_{[n-1,1]}}X_1}\wedge\tau(\s^{m_n},\s^{n-2}))\\
&=\overline{f_n}\circ \Sigma ^{m_n}g_{n,n-1}\circ(1_{\Sigma^{m_{[n-1,1]}}X_1}\wedge\tau(\s^{m_n},\s^{n-2}))\circ(1_{\Sigma^{m_{[n-1,1]}}X_1}\wedge\tau(\s^{n-2},\s^{m_n}))\\
&\hspace{1cm}\circ \Sigma^{m_n}\Big((\widetilde{\Sigma}^{m_{n-1}}\omega_{n-1,n-2})
\circ(\Sigma^{m_{n-1}}e_{n-1}\cup C_{1_{\Sigma^{m_{n-1}}C_{n-1,n-2}}})\\
&\hspace{1cm}\circ(\widetilde{\Sigma}^{m_{n-1}}\omega'_{n-1,n-2})^{-1}\Big)
\circ (1_{\Sigma^{m_{[n-1,1]}}X_1}\wedge\tau(\s^{m_n},\s^{n-2}))\\
&=\alpha\circ \theta.
\end{align*}
Since $\{\vec{\bm f}\,\}^{(aq\ddot{s}_2)}_{\vec{\bm m}}\subset \{\vec{\bm f}\,\}^{(aq\ddot{s}_2)}_{\vec{\bm m}}\circ\Gamma =
\{\vec{\bm f}\,\}^{(\ddot{s}_t)}_{\vec{\bm m}}\circ\Gamma $ by (4), we have $\alpha\circ\theta=\beta\circ\gamma$ for some $\beta\in\{\vec{\bm f}\,\}^{(\ddot{s}_t)}_{\vec{\bm m}}$ and $\gamma\in\Gamma $. 
Set $\theta'=\theta\circ\gamma^{-1}$. 
Then $\theta'\in \Gamma'$ and $\alpha\circ\theta'=\beta\in\{\vec{\bm f}\,\}^{(\ddot{s}_t)}_{\vec{\bm m}}$. 

Finally $\Gamma'$ is a subgroup of 
$[\Sigma^{|\vec{\bm m}|+n-2}X_1,\Sigma^{|\vec{\bm m}|+n-2}X_1]$, 
since comultiplication on $\s^i$ for $i\ge 2$ is unique up to homotopy and 
the following three functions are homomorphisms: 
\begin{gather*}
\Sigma^\ell : [\Sigma^kX,\Sigma^kX]\to[\Sigma^\ell \Sigma^kX,\Sigma^\ell \Sigma^kX],\\
(1_X\wedge\tau(\s^\ell,\s^k))^*:[\Sigma^\ell \Sigma^kX,\Sigma^\ell \Sigma^kX]\to [\Sigma^k \Sigma^\ell X, \Sigma^\ell \Sigma^kX],\\
(1_X\wedge\tau(\s^k,\s^\ell))_*:[\Sigma^k \Sigma^\ell X,\Sigma^\ell \Sigma^k X]\to [\Sigma^k \Sigma^\ell X,\Sigma^k \Sigma^\ell X],
\end{gather*}
where $X=\Sigma^{m_{[n-1,1]}}X_1,\ k=n-2,\ \ell=m_n$. 
This ends the proof of Theorem~4.1(5).
\end{proof}

\begin{proof}[Proof of Corollary 4.2(1)]
This follows Theorem 4.1(2). 
\end{proof}

\begin{proof}[Proof of Corollary 4.2(2)]
Let $\alpha\in\{\vec{\bm f}\}^{(aq)}_{\vec{\bm m}}$ and $\{\mathscr{S}_r,\overline{f_r},\Omega_r\,|\,2\le r\le n\}$ an $aq$-presentation of $\vec{\bm f}$ 
with $\alpha=\overline{f_n}\circ\widetilde{\Sigma}^{m_n}g_{n,n-1}$. 
Let $\theta\in\mathscr{E}(\Sigma\Sigma^{m_1}X_1)$. 
We set 
\begin{align*}
F(\theta)&=(1_{\Sigma^{m_1}X_1}\wedge\tau(\s^{n-2},\s^{m_{[n,2]}}))\circ \Sigma^{m_{[n,2]}}\Sigma^{n-3}\theta\circ(1_{\Sigma^{m_1}X_1}\wedge\tau(\s^{m_{[n,2]}},\s^{n-2})),\\
\theta^*_r&=(1_{\Sigma^{m_1}X_1}\wedge\tau(\s^{r-1},\s^{m_{[r-1,2]}})\circ \Sigma^{m_{[r-1,2]}}\Sigma^{r-2}\theta\\
&\hspace{3cm}\circ(1_{\Sigma^{m_1}X_1}\wedge\tau(\s^{m_{[r-1,2]}},\s^{r-1})\quad (3\le r\le n),\\
\omega'_{2,1}&=\theta\circ\omega_{2,1},\quad
\omega'_{r,s}=\begin{cases}\omega_{r,s}&1\le s\le r-2\\ \theta^*_r\circ\omega_{r,s}& s=r-1\end{cases}\ (3\le r\le n).
\end{align*}
Then $\theta^*_r\in\mathscr{E}(\Sigma^{r-1}\Sigma^{m_{[r-1,1]}}X_1)$ for $3\le r\le n$, and $\Omega_r'=\{\omega'_{r,s}\,|\,1\le s<r\}$ is a quasi-structure on $\mathscr{S}_r$. 
Using the identity $S^1\wedge \s^{r-2}=\s^{r-1}$, we have easily 
\begin{equation}
\begin{split}
&\Sigma(1_{\Sigma^{m_{[r-1,1]}}X_1}\wedge\tau(\s^{r-1},\s^{m_r}))\circ \Sigma\Sigma^{m_r}\theta^*_r\\
&=
\theta^*_{r+1}\circ \Sigma(1_{\Sigma^{m_{[r-1,1]}}X_1}\wedge\tau(\s^{m_r},\s^{r-1}))\quad 
(3\le r<n).
\end{split}
\end{equation}
We define
\begin{align*}
\mathscr{S}'_2&=\mathscr{S}_2,\quad \Omega'_2=\{\omega'_{2,1}\},\quad \overline{f_2}'=\overline{f_2},\\
\mathscr{S}'_{r+1}&=(\widetilde{\Sigma}^{m_r}\mathscr{S}_r')(\overline{f_r}',\widetilde{\Sigma}^{m_r}\Omega'_r),\quad \overline{f_{r+1}}'=\overline{f_{r+1}}\quad (2\le r<n).
\end{align*}
By \cite[Remark 5.5]{OO}, the definition is possible, and 
$\mathscr{S}_r$ and $\mathscr{S}'_r$ have the same edge. 
We have 
\begin{equation}
\Omega'_{r+1}=\widetilde{\widetilde{\Sigma}^{m_r}\Omega'_r}\ (2\le r<n).
\end{equation}
The assertion (4.10) is proved as follows. 
Since $\Omega_{r+1}=\widetilde{\widetilde{\Sigma}^{m_r}\Omega_r}$, we have 
$$
\omega_{r+1,s}=\Sigma(1_{\Sigma^{m_{[r-1,r-s+1]}}X_{r-s+1}}\wedge\tau(\s^{s-1},\s^{m_r}))
\circ \Sigma\Sigma^{m_r}\omega_{r,s-1}\circ \Sigma\psi^{m_r}_{j_{r,s-1}}\circ q\circ \xi. 
$$ 
We write $\widetilde{\widetilde{\Sigma}^{m_r}\Omega'_r}=\{\omega^*_{r+1,s}\,|\,1\le s\le r\}$. 
It suffices to prove $\omega^*_{r+1,s}=\omega'_{r+1,s}$. 
This shall be done by an induction on $r$. 
We have $\omega^*_{r+1,1}=q'_{f_r}=\omega_{r+1,1}'$ by definitions. 
Suppose $2\le s\le r$. 
We have 
\begin{align*}
\omega_{3,2}^*&=\Sigma(1_{\Sigma^{m_1}X_1}\wedge\tau(\s^1,\s^{m_2}))\circ \Sigma\Sigma^{m_2}\theta\circ \Sigma\Sigma^{m_2}\omega_{2,1}\circ \Sigma\psi^{m_2}_{j_{2,1}}\circ q\circ \xi,\\
\omega_{3,2}'&=(1_{\Sigma^{m_1}X_1}\wedge\tau(\s^2,\s^{m_2}))\circ \Sigma^{m_2}\Sigma\theta\circ 
(1_{\Sigma^{m_1}X_1}\wedge\tau(\s^{m_2},\s^2))\circ \omega_{3,2}\\
&=(1_{\Sigma^{m_1}X_1}\wedge\tau(\s^2,\s^{m_2}))\circ \Sigma^{m_2}E\theta\circ 
(1_{\Sigma^{m_1}X_1}\wedge\tau(\s^{m_2},\s^2))\\
&\hspace{2cm}\circ \Sigma(1_{\Sigma^{m_1}X_1}\wedge\tau(\s^1,\s^{m_2}))\circ \Sigma\Sigma^{m_2}\omega_{2,1}\circ \Sigma\psi^{m_2}\circ q\circ\xi.
\end{align*}
As is easily seen, we have 
\begin{align*}
&\Sigma(1_{\Sigma^{m_1}X_1}\wedge\tau(\s^1,\s^{m_2}))\circ \Sigma\Sigma^{m_2}\theta\\
&=
(1_{\Sigma^{m_1}X_1}\wedge\tau(\s^2,\s^{m_2}))\circ \Sigma^{m_2}\Sigma\theta
\circ 
(1_{\Sigma^{m_1}X_1}\wedge\tau(\s^{m_2},\s^2))
\circ \Sigma(1_{\Sigma^{m_1}X_1}\wedge\tau(\s^1,\s^{m_2})).
\end{align*}
Hence $\omega_{3,2}^*=\omega'_{3,2}$. 
Suppose $\omega^*_{r+1,s}=\omega'_{r+1,s}\ (2\le s\le r)$. 
We are going to show $\omega^*_{r+2,s}=\omega'_{r+2,s}\ (2\le s\le r+1)$. 
We have $\omega^*_{r+2,s}=\omega_{r+2,s}=\omega'_{r+2,s}$ for $s\le r$, and 
\begin{align*}
&\omega'_{r+2,r+1}=\theta^*_{r+2}\circ\omega_{r+2,r+1}\\
&=\theta^*_{r+2}\circ \Sigma(1_{\Sigma^{m_{[r,1]}}X_1}\wedge\tau(\s^r,\s^{m_{r+1}}))
\circ \Sigma\Sigma^{m_{r+1}}\omega_{r+1,r}\circ \Sigma\psi^{m_{r+1}}_{j_{r+1,r}}\circ q\circ \xi\\
&=\Sigma(1_{\Sigma^{m_{[r,1]}}X_1}\wedge\tau(\s^r,\s^{m_{r+1}}))\circ \Sigma\Sigma^{m_{r+1}}\theta^*_{r+1}\circ \Sigma\Sigma^{m_{r+1}}\omega'_{r+1,r}\circ \Sigma\psi^{m_{r+1}}\circ q\circ\xi\ 
(\text{by (4.9)})\\
&=\Sigma(1_{\Sigma^{m_{[r,1]}}X_1}\wedge\tau(\s^r,\s^{m_{r+1}}))\circ \Sigma\Sigma^{m_{r+1}}\omega'_{r+1,r}\circ \Sigma\psi^{m_{r+1}}\circ q\circ \xi\\
&=\omega^*_{r+2,r+1}.
\end{align*} 
This completes the induction and ends the proof of (4.10).

By (4.10), $\{\mathscr{S}'_r,\overline{f_r}',\Omega'_r\,|\,2\le r\le n\}$ 
is an $aq$-presentation of $\vec{\bm f}$. 
Set 
$$
\theta'_{n-1}=(1_{\Sigma^{m_{[n-2,1]}}X_1}\wedge\tau(\s^{n-2},\s^{m_{n-1}}))\circ 
\Sigma^{m_{n-1}}\theta^*_{n-1}\circ (1_{\Sigma^{m_{[n-2,1]}}X_1}\wedge\tau(\s^{m_{n-1}},\s^{n-2})).
$$
Then
$\theta'_{n-1}\circ(1_{\Sigma^{m_{[n-2,1]}}X_1}\wedge\tau(\s^{n-2},\s^{m_{n-1}})) =(1_{\Sigma^{m_{[n-2,1]}}X_1}\wedge\tau(\s^{n-2},\s^{m_{n-1}}))\circ \Sigma^{m_{n-1}}\theta^*_{n-1}$ 
and, by definitions, 
$g'_{n,n-1}\circ\theta'_{n-1}\simeq g_{n,n-1}$. 
We have 
\begin{align*}
&\overline{f_n}\circ\widetilde{\Sigma}^{m_n}g_{n,n-1}=\overline{f_n}\circ \Sigma^{m_n}g_{n,n-1}
\circ (1_{\Sigma^{m_{[n-1,1]}}X_1}\wedge\tau(\s^{m_n},\s^{n-2}))\\
&\simeq \overline{f_n}\circ \Sigma^{m_n}(g'_{n,n-1}\circ\theta'_{n-1})
\circ (1_{\Sigma^{m_{[n-1,1]}}X_1}\wedge\tau(\s^{m_n},\s^{n-2}))\\
&=\overline{f_n}\circ \Sigma^{m_n}g'_{n,n-1}\circ(1_{\Sigma^{m_{[n-1,1]}}X_1}\wedge\tau(\s^{m_n},\s^{n-2}))\\
&\hspace{1cm}\circ(1_{\Sigma^{m_{[n-1,1]}}X_1}\wedge\tau(\s^{n-2},\s^{m_n}))
\circ \Sigma^{m_n}\theta'_{n-1}\circ (1_{\Sigma^{m_{[n-1,1]}}X_1}\wedge\tau(\s^{m_n},\s^{n-2}))\\
&=\overline{f_n}\circ \widetilde{\Sigma}^{m_n}g'_{n,n-1}\circ F^*(\theta),
\end{align*}
where $F^*(\theta)=(1_{\Sigma^{m_{[n-1,1]}}X_1}\wedge\tau(\s^{n-2},\s^{m_n}))
\circ \Sigma^{m_n}\theta'_{n-1}(1_{\Sigma^{m_{[n-1,1]}}X_1}\wedge\tau(\s^{m_n},\s^{n-2}))$. 
We will prove
\begin{equation}
F^*(\theta)=F(\theta),\  F(\theta^{-1})=F(\theta)^{-1}.
\end{equation}
If this is done, then $\overline{f_n}\circ\widetilde{\Sigma}^{m_n}g_{n,n-1}\simeq\overline{f_n}\circ\widetilde{\Sigma}^{m_n}g'_{n,n-1}\circ F(\theta)$ and hence $\{\vec{\bm f}\,\}^{(aq)}_{\vec{\bm m}}\subset\{\vec{\bm f}\,\}^{(aq)}_{\vec{\bm m}}\circ F(\theta)$. 
By replacing $\theta$ with $\theta^{-1}$, we have $\{\vec{\bm f}\,\}^{(aq)}_{\vec{\bm m}}\subset\{\vec{\bm f}\,\}^{(aq)}_{\vec{\bm m}}\circ F(\theta^{-1})=
\{\vec{\bm f}\,\}^{(aq)}_{\vec{\bm m}}\circ F(\theta)^{-1}$ so that 
$\{\vec{\bm f}\,\}^{(aq)}_{\vec{\bm m}}\circ F(\theta)\subset\{\vec{\bm f}\,\}^{(aq)}_{\vec{\bm m}}$. 
Hence $\{\vec{\bm f}\,\}^{(aq)}_{\vec{\bm m}}=\{\vec{\bm f}\,\}^{(aq)}_{\vec{\bm m}}\circ F(\theta)$ and we obtain Corollary~4.2(2). 

Now we prove (4.11) as follows. 
We have
\begin{align*}
&F^*(\theta)=(1_{\Sigma^{m_{[n-1,1]}}X_1}\wedge\tau(\s^{n-2},\s^{m_n}))\circ \Sigma^{m_n}\big((1_{\Sigma^{m_{[n-2,1]}}X_1}\wedge\tau(\s^{n-2},\s^{m_{n-1}}))\\
&\hspace{1cm}\circ \Sigma^{m_{n-1}}\theta^*_{n-1}\circ(1_{\Sigma^{m_{[n-2,1]}}X_1}\wedge\tau(\s^{m_{n-1}},\s^{n-2}))\big)
\circ(1_{\Sigma^{m_{[n-1,1]}}X_1}\wedge\tau(\s^{m_n},\s^{n-2}))\\
&=(1_{\Sigma^{m_{[n-1,1]}}X_1}\wedge\tau(\s^{n-2},\s^{m_n}))\circ \Sigma^{m_n}
\Big((1_{\Sigma^{m_{[n-2,1]}}X_1}\wedge\tau(\s^{n-2},\s^{m_{n-1}}))\\
&\hspace{1cm}\circ \Sigma^{m_{n-1}}
\big((1_{\Sigma^{m_1}X_1}\wedge\tau(\s^{n-2},\s^{m_{[n-2,2]}}))\circ \Sigma^{m_{[n-2,2]}}\Sigma^{n-3}\theta\\
&\hspace{1cm}\circ(1_{\Sigma^{m_1}X_1}\wedge\tau(\s^{m_{[n-2,2]}},\s^{n-2}))
\big)
\circ(1_{\Sigma^{m_{[n-2,1]}}X_1}\wedge\tau(\s^{m_{n-1}},\s^{n-2}))
\Big) \\
&\hspace{1cm}\circ (1_{\Sigma^{m_{[n-1,1]}}X_1}\wedge\tau(\s^{m_n},\s^{n-2}))\\
&=(1_{\Sigma^{m_1}X_1}\wedge\tau(\s^{n-2},\s^{m_{[n,2]}}))\circ \Sigma^{m_{[n,2]}}\Sigma^{n-3}\theta\circ(1_{\Sigma^{m_1}X_1}\wedge\tau(\s^{m_{[n,2]}},\s^{n-2}))\\
&=F(\theta).
\end{align*}
By definitions, $F(\theta^{-1})=F(\theta)^{-1}$. 
This completes the proof of (4.11) and Corollary 4.2(2).
\end{proof}

\begin{proof}[Proof of Corollary 4.2(3)] 
Suppose that $\Sigma^{m_{[n,2]}}\Sigma^{n-3}:\mathscr{E}(\Sigma\Sigma^{m_1}X_1)\to\mathscr{E}(\Sigma^{|\vec{\bm m}|+n-2}X_1)$ is surjective. 
Then the argument in the proof of Corollary 4.2(2) implies that every element of 
$\mathscr{E}(\Sigma^{|\vec{\bm m}|+n-2}X_1)$ has a form of $F(\theta)$ so that 
$\{\vec{\bm f}\}^{(aq)}_{\vec{\bm m}}=\{\vec{\bm f}\}^{(aq)}_{\vec{\bm m}}\circ\mathscr{E}(\Sigma^{|\vec{\bm m}|+n-2}X_1)$. 
As is easily seen, $\Sigma^{m_{[n,2]}}:\mathscr{E}(\Sigma^{n-3}\Sigma\Sigma^{m_1}X_1)\to
\mathscr{E}(\Sigma^{m_{[n,2]}}\Sigma^{n-3}\Sigma\Sigma^{m_1}X_1))$ 
is equal to 
$\big(1_{\Sigma^{m_1}X_1}\wedge\tau(\s^{m_{[n,2]}},\s^1\wedge\s^{n-3})\big)_\#
\circ \Sigma^{m_{[n,n-1]}}\circ \big(1_{\Sigma^{m_1}X_1}\wedge\tau(\s^1\wedge\s^{n-3})\big)_\#\circ \Sigma^{m_{[n-2,2]}}$. 
Since two functions of the form $(\ )_\#$ are isomorphisms and 
$\Sigma^{m_{[n,2]}}$ is surjective by the assumption, 
$\Sigma^{m_{[n,n-1]}}:\mathscr{E}(\Sigma^{n-3}\Sigma\Sigma^{m_{[n-2,1]}}X_1)\to \mathscr{E}(\Sigma^{m_{[n,n-1]}}\Sigma^{n-3}\Sigma\Sigma^{m_{[n-2,1]}}X_1)$ is surjective so that 
$$
\{\vec{\bm f}\}^{(q)}_{\vec{\bm m}}=\{\vec{\bm f}\}^{(aq)}_{\vec{\bm m}}\circ\mathscr{E}(\Sigma^{|\vec{\bm m}|+n-2}X_1)
$$
by Theorem 4.1(3). 
\end{proof}

\begin{proof}[Proof of Corollary 4.2(4),(5)] 
If $\{\vec{\bm f}\,\}^{(\star)}_{\vec{\bm m}}$ is not empty for some $\star$, 
then $\{\vec{\bm f}\,\}^{(q)}_{\vec{\bm m}}$ is not empty by (3.1) so that $\{\vec{\bm f}\,\}^{(aq\ddot{s}_2)}_{\vec{\bm m}}$ and 
$\{\vec{\bm f}\,\}^{(\ddot{s}_t)}_{\vec{\bm m}}$ are not empty by Theorem 4.1(5), and so $\{\vec{\bm f}\,\}^{(\star)}_{\vec{\bm m}}$ is 
not empty for every $\star$ by (3.1) and Theorem 4.1. 
This proves (4). 

If $\{\vec{\bm f}\,\}^{(\star)}_{\vec{\bm m}}$ contains $0$ for some $\star$, then $\{\vec{\bm f}\,\}^{(q)}_{\vec{\bm m}}$ contains $0$ by (3.1), 
and so $\{\vec{\bm f}\,\}^{(\star)}_{\vec{\bm m}}$ contains $0$ for every $\star$ by Theorem 4.1(5)  and (3.1). This proves (5). 
\end{proof}

\begin{proof}[Proof of Proposition 4.3(4.1)]
Let $\alpha=\overline{f_n}\circ \Sigma^{m_n}g_{n,n-1}\circ (1_{\Sigma^{m_{[n-1,1]}}X_1}\wedge\tau(\s^{m_n},\s^{n-2}))\in\{\vec{\bm f}\}^{(\star)}_{\vec{\bm m}}$. 
Then \
\begin{align*}
&f_{n+1}\circ \Sigma^{m_{n+1}}\alpha=f_{n+1}\circ \Sigma^{m_{n+1}}\overline{f_n}\circ \Sigma^{m_{n+1}}\Sigma^{m_n}g_{n,n-1}\circ \Sigma^{m_{n+1}}(1_{\Sigma^{m_{[n-1,1]}}X_1}\wedge\tau(\s^{m_n},\s^{n-2}))\\
&=(f_{n+1}\circ \Sigma^{m_{n+1}}\overline{f_n})\circ \Sigma^{m_{n+1}+m_n}g_{n,n-1}\\
&\hspace{1cm}\circ (1_{\Sigma^{m_{[n-1,1]}}X_1}\wedge\tau(\s^{m_{[n+1,n]}},\s^{n-2}))\circ (1_{m_{[n,1]}X_1}\wedge\tau(\s^{n-2},\s^{m_{n+1}}))\\
&\in\{f_{n+1}\circ \Sigma^{m_{n+1}}f_n,f_{n-1},\dots,f_1\}^{(\star)}_{(m_{n+1}+m_n,m_{n-1},\dots,m_1)}\circ (1_{\Sigma^{m_{[n,1]}}X_1}\wedge(\s^{n-2},\s^{m_{n+1}})).
\end{align*}
\end{proof}

\begin{proof}[Proof of Proposition 4.3(4.2)]
We set
\begin{gather*}
(X_{n+1}',X_n',\dots,X_1')=(X_{n+2},X_n,\dots,X_1),\ \vec{\bm m'}=(m_{n+1}+m_n,m_{n-1},\dots,m_1),\\ 
\vec{\bm f'}=(f_{n+1}\circ \Sigma^{m_{n+1}}f_n,f_{n-1},\dots,f_1),\\
(X_{n+1}^*,X_n^*,\dots,X_1^*)=(X_{n+2},X_{n+1},X_{n-1},\dots,X_1),\ \vec{\bm m^*}=(m_{n+1},m_n+m_{n-1},m_{n-2},\dots,m_1),\\ 
\vec{\bm f^*}=(f_{n+1},f_n\circ \Sigma^{m_n}f_{n-1},f_{n-2},\dots,f_1).
\end{gather*}
We should prove $\{\vec{\bm f'}\}^{(\star)}_{\vec{\bm m'}}\subset\{\vec{\bm f^*}\}^{(\star)}_{\vec{\bm m^*}}$. 
By Theorem 4.1(2),(3), it suffices to prove the assertion for $\star=aq\ddot{s}_2,\ddot{s}_t,q_2$. 

We prove the assertion only for $\star=aq\ddot{s}_2$, because the cases $\star=\ddot{s}_t,q_2$ can be treated similarly. 

Let $\alpha\in\{\vec{\bm f'}\}^{(\star)}_{\vec{\bm m'}}$ and $\{\mathscr{S}_r',\overline{f'_r},\Omega_r'\,|\,2\le r\le n\}$ an $aq\ddot{s}_2$-presentation of $\vec{\bm f'}$ such that $\alpha=\overline{f_n'}\circ \widetilde{\Sigma}^{m_n'}g'_{n,n-1}$. 
We define
\begin{gather*}
\mathscr{S}_r^*=\mathscr{S}_r',\ \Omega^*_r=\Omega'_r\ (2\le r\le n-1),\\
\overline{f^*_r}=\overline{f'_r}:\Sigma^{m_r'}C'_{r,r-1}\to X_{r+1}'=X_{r+1}=X^*_{r+1}\  (2\le r\le n-2),\\
\overline{f^*_{n-1}}=f_n\circ \Sigma^{m_n}\overline{f'_{n-1}}:\Sigma^{m^*_{n-1}}C^*_{n-1,n-1}\to X_n^*=X_{n+1},\\
\mathscr{S}^*_n=(\widetilde{\Sigma}^{m^*_{n-1}}\mathscr{S}^*_{n-1})(\overline{f^*_{n-1}},\widetilde{\Sigma}^{m^*_{n-1}}\Omega^*_{n-1}),\ \Omega^*_n=\widetilde{\widetilde{\Sigma}^{m^*_{n-1}}\Omega^*_{n-1}}.
\end{gather*}
Then
\begin{gather*}
C^*_{n,s}=X^*_n\cup_{\overline{f^*_{n-1}}^{s-1}}C\Sigma^{m^*_{s-1}}C^*_{n-1,s-1}=X_{n+1}\cup _{f_n\circ \Sigma^{m_n}\overline{f'_{n-1}}^{s-1}}C\Sigma^{m_{[n,n-1]}}C'_{n-1,s-1},\\
C^*_{n,n-1}=X_{n+1}\cup_{f_n\circ \Sigma^{m_n}\overline{f'_{n-1}}^{n-2}}C\Sigma^{m_{[n,n-1]}}C'_{n-1,n-2}.
\end{gather*}
Define $\overline{f^*_n}$ by the composite of the following maps
\begin{align*}
&\Sigma^{m^*_n}C^*_{n,n-1}=\Sigma^{m_{n+1}}(X_{n+1}\cup_{f_n\circ \Sigma^{m_n}\overline{f'_{n-1}}^{n-2}}
C\Sigma^{m_{[n,n-1]}}C'_{n-1,n-2})\\
&\xrightarrow{(\psi^{m_{n+1}}_{f_n\circ \Sigma^{m_n}\overline{f'_{n-1}}^{n-2}})^{-1}}
\Sigma^{m_{n+1}}X_{n+1}\cup_{\Sigma^{m_{n+1}}(f_n\circ \Sigma^{m_n}\overline{f'_{n-1}}^{n-2})}C\Sigma^{m_{[n+1,n-1]}}C'_{n-1,n-2}\\
&=\Sigma^{m_{n+1}}X_{n+1}\cup_{\Sigma^{m_{n+1}}f_n\circ \Sigma^{m'_n}\overline{f'_{n-1}}^{n-2}}C\Sigma^{m'_{[n,n-1]}}C'_{n-1,n-2}\\
&\xrightarrow{f_{n+1}\cup \overline{f'_n}\circ \psi^{m'_n}_{\overline{f'_{n-1}}^{n-2}}} 
X_{n+2}=X^*_{n+1}.
\end{align*}
This is a well-defined map. 
Let  $\{\mathscr{S}_r^*,\overline{f^*_r},\Omega_r^*\,|\,2\le r\le n\}$ be the collection 
obtained from $\{\mathscr{S}_r',\overline{f_r'},\Omega_r'\,|\,2\le r\le n\}$ by 
replacing $\overline{f'_{n-1}}, \mathscr{S}_n', \overline{f_n'}, \Omega_n'$ with 
$\overline{f^*_{n-1}}, \mathscr{S}^*_n, \overline{f^*_n}, \Omega_n^*$, respectively. 
Then this is an $aq\ddot{s}_2$-presentation of $\vec{\bm f^*}$ and 
$\overline{f^*_n}\circ \widetilde{\Sigma}^{m^*_n}g^*_{n,n-1}\simeq \overline{f'_n}\circ \widetilde{\Sigma}^{m'_n}g'_{n,n-1}$. 
This proves the assertion for $\star=aq\ddot{s}_2$. 
\end{proof}

For the proofs of Proposition 4.3 (4.3), (4.4), we introduce the following notations: 
\begin{gather*}
(X_{n+1}',\dots,X_1')=(X_{n+1},\dots,X_2,\Sigma^{m_0}X_0),\ \vec{\bm m'}=\vec{\bm m},\\
(f_n',\dots,f_1')=(f_n,\dots,f_2, f_1\circ \Sigma^{m_1}f_0),\\
(X_{n+1}^*,\dots,X^*_1)=(X_{n+1},\dots,X_3,X_1,X_0),\ \vec{\bm m^*}=(m_n,\dots,m_3, m_2+m_1, m_0),\\
(f_n^*,\dots,f_1^*)=(f_n,\dots,f_3, f_2\circ \Sigma^{m_2}f_1, f_0).
\end{gather*}

\begin{proof}[Proof of Proposition 4.3(4.3)]
%
First consider the case $\star=aq\ddot{s}_2$. 
We should prove $\{\vec{\bm f'}\}^{(aq\ddot{s}_2)}_{\vec{\bm m'}}\subset\{\vec{\bm f^*}\}^{(aq\ddot{s}_2)}_{\vec{\bm m^*}}$. 
Let $\alpha\in\{\vec{\bm f'}\}^{(aq\ddot{s}_2)}_{\vec{\bm m'}}$ and 
$\{\mathscr{S}_r',\overline{f'_r},\Omega_r'\,|\,2\le r\le n\}$ 
an $aq\ddot{s}_2$-presentation of $\vec{\bm f'}$ such that 
$\alpha=\overline{f'_n}\circ \widetilde{\Sigma}^{m_n'}g'_{n,n-1}$. 
Set 
\begin{align*}
&\mathscr{S}_2^*=(\Sigma^{m^*_1}X^*_1;X^*_2,X^*_2\cup_{f^*_1}C\Sigma^{m^*_1}X^*_1;f^*_1;i_{f^*_1})=(\Sigma^{m_0}X_0;X_1,X_1\cup_{f_0}C\Sigma^{m_0}X_0;f_0;i_{f_0}),\\
&\Omega^*_2=\{q'_{f^*_1}\},\ e_{2,1}=f_1:\Sigma^{m_1}C^*_{2,1}=\Sigma^{m_1}X_1\to X_2=C'_{2,1},\\
&e_{2,2}=(f_1\cup C1_{\Sigma^{m_{[1,0]}}X_0})\circ(\psi^{m_1}_{f_0})^{-1}\\
&\hspace{1cm}:\Sigma^{m_1}C^*_{2,2}=
\Sigma^{m_1}(X_1\cup_{f_0}C\Sigma^{m_0}X_0)\cong \Sigma^{m_1}X_1\cup_{\Sigma^{m_1}f_0}C\Sigma^{m_{[1,0]}}X_0\\
&\hspace{2cm}\longrightarrow X_2\cup_{f_1\circ \Sigma^{m_1}f_0}C\Sigma^{m_{[1,0]}}X_0=C'_{2,2},\\
&\overline{f^*_2}=\overline{f'_2}\circ \Sigma^{m_2}e_{2,2}:\Sigma^{m^*_2}C^*_{2,2}=\Sigma^{m_2}\Sigma^{m_1}C^*_{2,2}\to \Sigma^{m_2}C'_{2,2}\to X_3'=X_3=X^*_3.
\end{align*} 
Then $\Sigma^{m_2}e_{2,2}\circ \Sigma^{m^*_2}j^*_{2,1}=\Sigma^{m_2}j'_{2,1}\circ \Sigma^{m_2}e_{2,1}$ and 
$\overline{f^*_2}\circ \Sigma^{m^*_2}j^*_{2,1}=f^*_2$. 
We set 
\begin{align*}
&\mathscr{S}^*_3=(\widetilde{\Sigma}^{m^*_2}\mathscr{S}^*_2)(\overline{f^*_2},\widetilde{\Sigma}^{m^*_2}\Omega^*_2),\ \Omega^*_3=\widetilde{\widetilde{\Sigma}^{m^*_2}\Omega^*_2},\\
&e_{3,1}=1_{X_3}:C^*_{3,1}=X_3\to C'_{3,1}=X_3,\\
&e_{3,2}=1_{X_3}\cup C\Sigma^{m_2}e_{2,1}:C^*_{3,2}=X_3\cup_{f_2\circ \Sigma^{m_2}f_1}C\Sigma^{m_{[2,1]}}X_1\to X_3\cup_{f_2}C\Sigma^{m_2}X_2=C'_{3,2},\\
&e_{3,3}=1_{X_3}\cup C\Sigma^{m_2}e_{2,2}:C^*_{3,,3}=X_3\cup_{\overline{f'_2}\circ \Sigma^{m_2}e_{2,2}}C\Sigma^{m_{[2,1]}}C^*_{2,2}\to X_3\cup_{\overline{f'_2}}C\Sigma^{m_2}C'_{2,2}=C'_{3,3}.
\end{align*}
Then $e_{3,s+1}\circ j^*_{3,s}=j'_{3,s}\circ e_{3,s}\ (s=1,2)$ and 
$\widetilde{\Sigma}^{m^*_2}\omega^*_{2,1}=\widetilde{\Sigma}^{m_2}\omega'_{2,1}\circ(\Sigma^{m_2}e_{2,2}\cup C\Sigma^{m_2}e_{2,1})$. 
Hence $e_{3,2}\circ g^*_{3,2}\simeq g'_{3,2}$. 
We define 
$$
\overline{f^*_3}=\begin{cases} \overline{f'_3}\circ \Sigma^{m_3}e_{3,2}:\Sigma^{m^*_3}C^*_{3,2}\to X^*_4=X_4 & n=3\\
\overline{f'_3}\circ \Sigma^{m_3}e_{3,3}:\Sigma^{m^*_3}C^*_{3,3}\to X_4^*=X_4 & n\ge 4\end{cases}.
$$
Then, when $n=3$, $\{\mathscr{S}^*_r,\overline{f^*_r},\Omega^*_r\,|\,r=2,3\}$ 
is an $aq\ddot{s}_2$-presentation of $\vec{\bm f^*}$ and $\overline{f^*_3}\circ \widetilde{\Sigma}^{m^*_3}g^*_{3,2}\simeq \overline{f'_3}\circ\widetilde{\Sigma}^{m'_3}g'_{3,2}$ and 
so $\alpha\in\{\vec{\bm f^*}\}^{(aq\ddot{s}_2)}_{\vec{\bm m^*}}$ as desired. 
Suppose $n\ge 4$.  
We set
$$
\mathscr{S}^*_4=(\widetilde{\Sigma}^{m_3}\mathscr{S}^*_3)(\overline{f^*_3},\widetilde{\Sigma}^{m_3}\Omega^*_3),\ \Omega^*_4=\widetilde{\widetilde{\Sigma}^{m_3}\Omega^*_3}.
$$
Then $C^*_{4,s}=C'_{4,s}$ for $s=1,2$. 
We set
\begin{align*}
&e_{4,1}=1_{X_4}:C^*_{4,1}=X_4\to C'_{4,1}=X_4,\\
&e_{4,2}=1_{X_4}\cup C\Sigma^{m_3}e_{3,1}=1_{C^*_{4,2}}:C^*_{4,2}\to C'_{4,2},\\
&e_{4,3}=1_{X_4}\cup C\Sigma^{m_3}e_{3,2}:C^*_{4,3}=X_4\cup_{\overline{f^*_3}^2}C\Sigma^{m_3}C^*_{3,2}\to X_4\cup_{\overline{f'_3}^2}C\Sigma^{m_3}C'_{3,2}=C'_{4,3},\\
&e_{4,4}=1_{X_4}\cup C\Sigma^{m_3}e_{3,3}:C^*_{4,4}=X_4\cup_{\overline{f^*_3}}C\Sigma^{m_3}C^*_{3,3}\to X_4\cup_{\overline{f'_3}}C\Sigma^{m_3}C'_{3,3}=C'_{4,4}.
\end{align*}
Then $e_{4,s+1}\circ j^*_{4,s}=j'_{4,s}\circ e_{4,s}\ (s=1,2,3)$ and 
$\widetilde{\Sigma}^{m^*_3}\omega^*_{3,2}=\widetilde{\Sigma}^{m_3'}\omega'_{3,2}\circ (\Sigma^{m^*_3}e_{3,3}\cup C\Sigma^{m^*_3}e_{3,2})$ so that 
$e_{4,3}\circ g^*_{4,3}\simeq g'_{4,3}$. 
We define
$$
\overline{f^*_4}=\begin{cases} \overline{f'_4}\circ \Sigma^{m_4}e_{4,3}:\Sigma^{m^*_4}C^*_{4,3}\to \Sigma^{m'_4}C'_{4,3}\to X_5'=X_5^*=X_5 & n=4\\
\overline{f'_4}\circ \Sigma^{m_4}e_{4,4}:\Sigma^{m^*_4}C^*_{4,4}\to \Sigma^{m'_4}C'_{4,4}\to X_5'=X_5^*=X_5 & n\ge 5\end{cases}.
$$
When $n=4$, $\{\mathscr{S}^*_r,\overline{f^*_r},\Omega_r^*\,|\,2\le r\le 4\}$ 
is an $aq\ddot{s}_2$-presentation of $\vec{\bm f^*}$ and 
\begin{align*}
&\overline{f^*_4}\circ\widetilde{\Sigma}^{m^*_4}g^*_{4,3}=\overline{f'_4}\circ \Sigma^{m_4}e_{4,3}\circ \Sigma^{m_4}g^*_{4,3}\circ(1_{\Sigma^{m^*_{[3,1]}}X^*_1}\wedge\tau(\s^{m^*_4},\s^2))\\
&\simeq\overline{f'_4}\circ \Sigma^{m_4}g'_{4,3}\circ (1_{\Sigma^{m_{[3,0]}}X_0}\wedge\tau(\s^{m_4},\s^2))\\
&=\overline{f'_4}\circ \Sigma^{m_4}g'_{4,3}\circ(1_{\Sigma^{m'_{[3,1]}}X'_1}\wedge\tau(\s^{m_4},\s^2))=\overline{f'_4}\circ\widetilde{\Sigma}^{m'_4}g'_{4,3}
\end{align*}
so that $\alpha\in\{\vec{\bm f^*}\}^{(aq\ddot{s}_2)}_{\vec{\bm m^*}}$ as desired. 
By repeating the process, we have an $aq\ddot{s}_2$-presentation 
$\{\mathscr{S}^*_r,\overline{f^*_r},\Omega_r^*\,|\,2\le r\le n\}$ of 
$\vec{\bm f^*}$ such that $\overline{f^*_n}\circ\widetilde{\Sigma}^{m^*_n}g^*_{n,n-1}$ represents $\alpha$. 
This proves the assertion for $\star=aq\ddot{s}_2$.

Secondly we consider the case $\star=\ddot{s}_t$. 
We should prove $\{\vec{\bm f'}\}^{(\ddot{s}_t)}_{\vec{\bm m'}}\subset\{\vec{\bm f^*}\}^{(\ddot{s}_t)}_{\vec{\bm m^*}}$. 
Let $\alpha\in\{\vec{\bm f'}\}^{(\ddot{s}_t)}_{\vec{\bm m'}}$ and 
$\{\mathscr{S}_r',\overline{f_r'},\mathscr{A}_r'\,|\,2\le r\le n\}$ an 
$\ddot{s}_t$-presentation of $\vec{\bm f'}$ with 
$\alpha=\overline{f'_n}\circ \widetilde{\Sigma}^{m_n}g'_{n,n-1}$. 
Set 
\begin{gather*}
\mathscr{S}_2^*=(\Sigma^{m^*_1}X^*_1;X^*_2,X^*_2\cup_{f^*_1}C\Sigma^{m^*_1}X^*_1;f^*_1;i_{f^*_1})=(\Sigma^{m_0}X_0;X_1, X_1\cup_{f_0}C\Sigma^{m_0}X_0;f_0;i_{f_0}),\\
e_{2,2}=(f_1\cup C1_{\Sigma^{m_{[1,0]}}X_0})\circ (\psi^{m_1}_{f_0})^{-1}:\Sigma^{m_1}C^*_{2,2}\to C'_{2,2},\\
\overline{f^*_2}=\overline{f_2'}\circ \Sigma^{m_2}e_{2,2}:\Sigma^{m^*_2}C^*_{2,2}\to X_3'=X_3^*=X_3,\quad 
\mathscr{A}^*_2=\{1_{C^*_{2,2}}\},\\
\mathscr{S}_3^*=(\widetilde{\Sigma}^{m_2^*}\mathscr{S}^*_2)(\overline{f^*_2},\widetilde{\Sigma}^{m_2^*}\mathscr{A}^*_2),\quad 
e_{3,1}=1_{X_3}:C^*_{3,1}=X_3^*\to X_3'=C'_{3,1},\\
 e_{3,2}=1_{X_3}\cup C\Sigma^{m_2}f_1:C^*_{3,2}=X_3\cup_{f_2\circ \Sigma^{m_2}f_1}C\Sigma^{m_{[2,1]}}X_1\to C'_{3,2}=X_3\cup_{f_2}C\Sigma^{m_2}X_2.
\end{gather*}
By definitions the following diagram is commutative.
$$
\xymatrix{
C^*_{3,2} \ar[d]^-{e_{3,2}} & \Sigma^{m_{[2,1]}}C^*_{2,2}\cup C\Sigma^{m_{[2,1]}}X_1 \ar[d]^-{\Sigma^{m_2}e_{2,2}\cup C\Sigma^{m_2}f_1} \ar[l]_-{\overline{f^*_2}\cup C1} \ar[r]^-{\widetilde{\Sigma}^{m^*_2}\omega^*_{2,1}} & \Sigma\Sigma^{m_{[2,0]}}X_0 \ar@{=}[d]\\
C'_{3,2} & \Sigma^{m_2}C'_{2,2}\cup C\Sigma^{m_2}X_2 \ar[l]_-{\overline{f'_2}\cup C1} \ar[r]^-{\widetilde{\Sigma}^{m_2}\omega'_{2,1}} & \Sigma\Sigma^{m_{[2,0]}}X_0
}
$$
Hence $e_{3,2}\circ g^*_{3,2}\simeq g'_{3,2}$. 
Let $\mathscr{A}_3^*$ be an arbitrary reduced structure on $\mathscr{S}^*_3$. 

Let $n=3$. 
We set $\overline{f^*_3}=\overline{f'_3}\circ \Sigma^{m_3}e_{3,2}:\Sigma^{m_3}C^*_{3,2}\to X_4$. 
Then we can easily see that $\{\mathscr{S}^*_r,\overline{f^*_r},\mathscr{A}^*_r\,|\,r=2,3\}$ 
is an $\ddot{s}_t$-presentation of $\vec{\bm f^*}$ and 
$\overline{f^*_3}\circ\widetilde{\Sigma}^{m_3^*}g^*_{3,2}\simeq \overline{f'_3}\circ \widetilde{\Sigma}^{m_3}g'_{3,2}$. 
This proves the assertion for $n=3$. 

In the rest of the proof, we suppose $n\ge 4$. 
Take $J^3:e_{3,2}\circ g^*_{3,2}\simeq g'_{3,2}$ and set 
$\Phi(J^3)=\Phi(g^*_{3,2},g'_{3,2},1_{},e_{3,2};J^3):C_{g^*_{3,2}}\to C_{g'_{3,2}}$. 
Then we have the following homotopy commutative diagram by \cite[Proposition 3.3]{OO}. 
$$
\xymatrix{
&& C^*_{3,3} \ar[d]^-{a^*_{3,2}} \ar[r] & C^*_{3,3}\cup CC^*_{3,2} \ar[d]_-{a^*_{3,2}\cup C1} \ar[dr]^-{\omega^*_{3,2}} & \\
\Sigma\Sigma^{m_{[2,0]}}X_0 \ar@{=}[d] \ar[r]^-{g^*_{3,2}} & C^*_{3,2} \ar[d]_-{e_{3,2}} \ar[r] \ar[ur]^-{j^*_{3,2}} & C_{g^*_{3,2}} \ar[d]_-{\Phi(J^3)} \ar[r] & C_{g^*_{3,2}}\cup CC^*_{3,2} \ar[d]^-{\Phi(J^3)\cup Ce_{3,2}} \ar[r] & \Sigma^2 \Sigma^{m_{[2,0]}}X_0 \ar@{=}[d] \\
\Sigma\Sigma^{m_{[2,0]}}X_0 \ar[r]^-{g'_{3,2}} & C'_{3,2} \ar[r] \ar[dr]_-{j'_{3,2}} & C_{g'_{3,2}} \ar[d]^-{{a'_{3,2}}^{-1}} \ar[r] & C_{g'_{3,2}}\cup CC'_{3,2} \ar[d]_-{{a'_{3,2}}^{-1}\cup C1} \ar[r] & \Sigma^2 \Sigma^{m_{[2,0]}}X_0\\
&& C'_{3,3} \ar[r] & C'_{3,3}\cup CC'_{3,2} \ar[ur]_-{\omega'_{3,2}} &
}
$$
Set $e_{3,3}={a'_{3,2}}^{-1}\circ\Phi(J^3)\circ a^*_{3,2}:C^*_{3,3}\to C'_{3,3}$, 
$\overline{f^*_3}=\overline{f'_3}\circ \Sigma^{m_3}e_{3,3}:\Sigma^{m_3}C^*_{3,3}\to X_4$, 
$\mathscr{S}^*_4=(\widetilde{\Sigma}^{m_3}\mathscr{S}^*_3)(\overline{f^*_3},\widetilde{\Sigma}^{m_3}\mathscr{A}^*_3)$, and $e_{4,3}=1_{X_4}\cup C\Sigma^{m_3}e_{3,2}:C^*_{4,3}\to C'_{4,3}$. 
Then the following diagram is homotopy commutative. 
$$
\xymatrix{
C^*_{4,3} \ar[d]^-{e_{4,3}} & \Sigma^{m_3}C^*_{3,3}\cup C\Sigma^{m_3}C^*_{3,2} \ar[d]^-{\Sigma^{m_3}e_{3,3}\cup C\Sigma^{m_3}e_{3,2}} \ar[l]_-{\overline{f^*_3}\cup C1} \ar[r]^-{\widetilde{\Sigma}^{m^*_3}\omega^*_{3,2}} & \Sigma^2\Sigma^{m_{[3,0]}}X_0 \ar@{=}[d]\\
C'_{4,3} & \Sigma^{m_3}C'_{3,3}\cup C\Sigma^{m_3}C'_{3,2} \ar[l]_-{\overline{f'_3}\cup C1} \ar[r]^-{\widetilde{\Sigma}^{m_3}\omega'_{3,2}} & \Sigma^2\Sigma^{m_{[3,0]}}X_0
}
$$
Hence $e_{4,3}\circ g^*_{4,3}\simeq g'_{4,3}$. 
Let $\mathscr{A}_4^*$ be an arbitrary reduced structure on $\mathscr{S}^*_4$. 

Let $n=4$. 
Set $\overline{f^*_4}=\overline{f'_4}\circ \Sigma^{m_4}e_{4,3}:\Sigma^{m_4}C^*_{4,3}\to X_5$. 
Then $\{\mathscr{S}^*_r,\overline{f^*_r},\mathscr{A}^*_r\,|\,r=2,3,4\}$ is an $\ddot{s}_t$-presentation of $\vec{\bm f^*}$ and
$\overline{f^*_4}\circ \widetilde{\Sigma}^{m_4}g^*_{4,3}\simeq \overline{f'_4}\circ\widetilde{\Sigma}^{m_4}g'_{4,3}$. 
This proves the assertion for $n=4$. 

Proceeding with the process above, we have an $\ddot{s}_t$-presentation of $\vec{\bm f^*}$ with $\overline{f^*_n}\circ\widetilde{\Sigma}^{m_n}g^*_{n,n-1}\simeq \overline{f'_n}\circ\widetilde{\Sigma}^{m_n}g'_{n,n-1}$ so that the proof of the assertion for $\star=\ddot{s}_t$ is obtained. 
\end{proof}

\begin{proof}[Proof of Proposition 4.3(4.4)]
We should prove $\{\vec{\bm f}\}^{(\star)}_{\vec{\bm m}}\circ \Sigma^{|\vec{\bm m}|+n-2}f_0\subset\{\vec{\bm f'}\}^{(\star)}_{\vec{\bm m}}$ for $\star=aq\ddot{s}_2,\ddot{s}_t$. 

First we prove the assertion for $\star=aq\ddot{s}_2$. 
Let $\alpha\in\{\vec{\bm f}\}^{(aq\ddot{s}_2)}_{\vec{\bm m}}$ and 
$\{\mathscr{S}_r, \overline{f_r}, \Omega_r\,|\,2\le r\le n\}$ an 
$aq\ddot{s}_2$-presentation of $\vec{\bm f}$ such that 
$\alpha=\overline{f_n}\circ \widetilde{\Sigma}^ng_{n,n-1}$. 
We set
\begin{align*}
&\mathscr{S}'_2=(\Sigma^{m'_1}X'_1;X_2',X'_2\cup_{f'_1}C\Sigma^{m'_1}X'_1;f'_1;i_{f'_1})\\
&\hspace{0.4cm}=(\Sigma^{m_{[1,0]}}X_0;X_2,X_2\cup_{f_1\circ \Sigma^{m_1}f_0}C\Sigma^{m_{[1,0]}}X_0;f_1\circ \Sigma^{m_1}f_0;i_{f_1\circ \Sigma^{m_1}f_0}),\\
&\Omega'_2=\{q'_{f_1\circ \Sigma^{m_1}f_0}\},\\
&e_{2,1}=1_{X_2}:C'_{2,1}\to C_{2,1},\\
&e_{2.2}=1_{X_2}\cup C\Sigma^{m_1}f_0:C'_{2,2}=X_2\cup_{f_1\circ \Sigma^{m_1}f_0}C\Sigma^{m_1}\Sigma^{m_0}X_0\to X_2\cup_{f_1}C\Sigma^{m_1}X_1=C_{2,2},\\
&\overline{f'_2}=\overline{f_2}\circ \Sigma^{m_2}e_{2,2}:\Sigma^{m'_2}C'_{2,2}\to \Sigma^{m_2}C_{2,2}\to X_3=X'_3.
\end{align*}
Then $\Sigma\Sigma^{m_{[2,1]}}f_0\circ\widetilde{\Sigma}^{m'_2}\omega'_{2,1}=\widetilde{\Sigma}^{m_2}\omega_{2,1}\circ(\Sigma^{m_2}e_{2,2}\cup C \Sigma^{m_2}e_{2,1})$ so that 
$g'_{3,2}\simeq  g_{3,2}\circ \Sigma\Sigma^{m_{[2,1]}}f_0$. 
We set
\begin{align*}
&\mathscr{S}_3'=(\widetilde{\Sigma}^{m'_2}\mathscr{S}'_2)(\overline{f'_2},\widetilde{\Sigma}^{m'_2}\Omega'_2),\ \Omega'_3=\widetilde{\widetilde{\Sigma}^{m'_2}\Omega'_2},\\
&e_{3,1}=1_{X_3}:C'_{3,1}\to C_{3,1},\ e_{3,2}=1_{C_{3,2}}:C'_{3,2}\to C_{3,2},\\
&e_{3,3}=1_{X_3}\cup C\Sigma^{m_2}e_{2,2}:C'_{3,3}=X_3\cup_{\overline{f'_2}}C\Sigma^{m_2}C'_{2,2}
\to X_3\cup_{\overline{f_2}}C\Sigma^{m_2}C_{2,,2}=C_{3,3},\\
&\overline{f'_3}=\begin{cases} \overline{f_3}\circ \Sigma^{m_3}e_{3,2}=\overline{f_3}:\Sigma^{m_3}C'_{3,2}=\Sigma^{m_3}C_{3,2}\to X_4=X'_4 & n=3\\
\overline{f_3}\circ \Sigma^{m_3}e_{3,3}:\Sigma^{m_3}C'_{3,3}\to \Sigma^{m_3}C_{3,3}\to X_4=X'_4 & n\ge 4\end{cases}.
\end{align*}
When $n=3$, $\{\mathscr{S}'_r,\overline{f'_r},\Omega'_r\,|\,r=2,3\}$ is an $aq\ddot{s}_2$-presentation of $\vec{\bm f'}$ such that 
\begin{align*}
&\overline{f'_3}\circ\widetilde{\Sigma}^{m'_3}g'_{3,2}=\overline{f_3}\circ \Sigma^{m_3}e_{3,2}\circ \Sigma^{m_3}g'_{3,2}\circ(1_{\Sigma^{m'_{[2,1]}}X'_1}\wedge\tau(\s^{m_3},\s^1))\\
&\simeq\overline{f_3}\circ \Sigma^{m_3}g_{3,2}\circ \Sigma^{m_3}\Sigma\Sigma^{m_{[2,1]}}f_0\circ 
(1_{\Sigma^{m_{[2,0]}}X_0}\wedge\tau(\s^{m_3},\s^1))\\
&=\overline{f_3}\circ \Sigma^{m_3}g_{3,2}\circ(1_{\Sigma^{m_{[2,1]}}X_1}\wedge\tau(\s^{m_3},\s^1))\circ \Sigma\Sigma^{m_{[3,1]}}f_0\\
&=\overline{f_3}\circ\widetilde{\Sigma}^{m_3}g_{3,2}\circ \Sigma\Sigma^{m_{[3,1]}}f_0
\end{align*}
so that $\alpha\circ \Sigma^{|\vec{\bm m}|+1}f_0\in\{\vec{\bm f'}\}^{(aq\ddot{s}_2)}_{\vec{\bm m'}}$. 
This proves the assertion for $n=3$. 
When $n\ge 4$, we set $\mathscr{S}_4'=(\widetilde{\Sigma}^{m'_3}\mathscr{S}'_3)(\overline{f'_3},\widetilde{\Sigma}^{m'_3}\Omega'_3)$ and $\Omega'_4=\widetilde{\widetilde{\Sigma}^{m'_3}\Omega'_3}$. 
By repeating the process, we have an $aq\ddot{s}_2$-presentation 
$\{\mathscr{S}'_r,\overline{f'_r},\Omega'_r\,|\,2\le r\le n\}$ of $\vec{\bm f'}$ 
which represents $\alpha\circ \Sigma^{|\vec{\bm m}|+n-2}f_0$. 
This proves the assertion for $\star=aq\ddot{s}_2$.

Secondly we consider the case $\star=\ddot{s}_t$. 
Let $\alpha\in\{\vec{\bm f}\}^{(\ddot{s}_t)}_{\vec{\bm m}}$ and 
$\{\mathscr{S}_r, \overline{f_r}, \mathscr{A}_r\,|\,2\le r\le n\}$ an 
$\ddot{s}_t$-presentation of $\vec{\bm f}$ such that 
$\alpha=\overline{f_n}\circ \widetilde{\Sigma}^ng_{n,n-1}$. 
We set
\begin{align*}
&\mathscr{S}'_2=(\Sigma^{m'_1}X'_1;X_2',X'_2\cup_{f'_1}C\Sigma^{m'_1}X'_1;f'_1;i_{f'_1})\\
&\hspace{0.4cm}=(\Sigma^{m_{[1,0]}}X_0;X_2,X_2\cup_{f_1\circ \Sigma^{m_1}f_0}C\Sigma^{m_{[1,0]}}X_0;f_1\circ \Sigma^{m_1}f_0;i_{f_1\circ \Sigma^{m_1}f_0}),\\
&\mathscr{A}'_2=\{1_{C'_{2,2}}\},\\
&e_{2,1}=1_{X_2}:C'_{2,1}\to C_{2,1},\\
&e_{2.2}=1_{X_2}\cup C\Sigma^{m_1}f_0:C'_{2,2}=X_2\cup_{f_1\circ \Sigma^{m_1}f_0}C\Sigma^{m_1}\Sigma^{m_0}X_0\to X_2\cup_{f_1}C\Sigma^{m_1}X_1=C_{2,2},\\
&\overline{f'_2}=\overline{f_2}\circ \Sigma^{m_2}e_{2,2}:\Sigma^{m'_2}C'_{2,2}\to \Sigma^{m_2}C_{2,2}\to X_3=X'_3,\\
&\mathscr{S}_3'=(\widetilde{\Sigma}^{m'_2}\mathscr{S}'_2)(\overline{f'_2},\widetilde{\Sigma}^{m'_2}\mathscr{A}'_2),\ \mathscr{A}'_3\text{ is an arbitrary reduced structure on $\mathscr{S}_3'$}.
\end{align*}
Then $C'_{3,i}=C_{3,i}$ for $i=1,2$, and 
$\Sigma\Sigma^{m_{[1,0]}}f_0\circ\omega'_{2,1}=\omega_{2,1}\circ (e_{2,2}\cup Ce_{2,1})$ 
so that $g'_{3,2}\simeq g_{3,2}\circ \Sigma\Sigma^{m_{[2,1]}}f_0$. 
When $n=3$, we set $\overline{f'_3}=\overline{f_3}:\Sigma^{m_3}C'_{3,2}=\Sigma^{m_3}C_{3,2}\to X_4$ so that $\{\mathscr{S}_r',\overline{f_r'},\mathscr{A}_r'\,|\,r=2,3\}$ is an $\ddot{s}_t$-presentation of $\vec{\bm f'}$ with $\overline{f_3'}\circ\widetilde{\Sigma}^{m_3'}g'_{3,2}=\alpha\circ \Sigma\Sigma^{m_{[3,1]}}f_0$. 
Indeed, 
\begin{align*}
&\overline{f'_3}\circ\widetilde{\Sigma}^{m'_3}g'_{3,2}=\overline{f_3}\circ \Sigma^{m_3}g'_{3,2}\circ(1_{\Sigma^{m'_{[2,1]}}X'_1}\wedge\tau(\s^{m_3},\s^1))\\
&\simeq\overline{f_3}\circ \Sigma^{m_3}g_{3,2}\circ \Sigma^{m_3}\Sigma\Sigma^{m_{[2,1]}}f_0\circ 
(1_{\Sigma^{m_{[2,0]}}X_0}\wedge\tau(\s^{m_3},\s^1))\\
&=\overline{f_3}\circ \Sigma^{m_3}g_{3,2}\circ(1_{\Sigma^{m_{[2,1]}}X_1}\wedge\tau(\s^{m_3},\s^1))\circ \Sigma\Sigma^{m_{[3,1]}}f_0\\
&=\overline{f_3}\circ\widetilde{\Sigma}^{m_3}g_{3,2}\circ \Sigma\Sigma^{m_{[3,1]}}f_0.
\end{align*}
This proves the assertion for $n=3$. 

In the rest of the proof, suppose $n\ge 4$. 
Take $J^3:g'_{3,2}\simeq g_{3,2}\circ \Sigma\Sigma^{m_{[2,1]}}f_0$ arbitrarily and set 
$\Phi(J^3)=\Phi(g'_{3,2},g_{3,2},\Sigma\Sigma^{m_{[2,1]}}f_0,1_{C_{3,2}};J^3):C_{g'_{3,2}}\to C_{g_{3,2}}$. 
We have the following homotopy commutative diagram. 
$$
\xymatrix{
&&C'_{3,3} \ar[d]^-{a'_{3,2}} \ar[r] &C'_{3,3}\cup CC'_{3,2}\ar[d]_-{a'_{3,2}\cup C1} \ar[dr]^-{\omega'_{3,2}}&\\
\Sigma\Sigma^{m_{[2,0]}}X_0 \ar[d]_-{\Sigma\Sigma^{m_{[2,1]}}f_0} \ar[r]^-{g'_{3,2}}&C'_{3,2} \ar@{=}[d] \ar[ur] \ar[r] &C_{g'_{3,2}}\ar[d]^-{\Phi(J^3)} \ar[r] & C_{g'_{3,2}}\cup CC'_{3,2} \ar[d]^-{\Phi(J^3)\cup C1} \ar[r] & \Sigma^2\Sigma^{m_{[2,0]}}X_0 \ar[d]^-{\Sigma^2\Sigma^{m_{[2,1]}}f_0} \\
\Sigma\Sigma^{m_{[2,1]}}X_1 \ar[r]^-{g_{3,2}} & C_{3,2} \ar[dr] \ar[r] & C_{g_{3,2}} \ar[d]^-{a_{3,2}^{-1}} \ar[r] & C_{g_{3,2}}\cup CC_{3,2} \ar[d]_-{a_{3,2}^{-1}\cup C1} \ar[r] & \Sigma^2 \Sigma^{m_{[2,1]}}X_1\\
&& C_{3,3} \ar[r] & C_{3,3}\cup CC_{3,2} \ar[ur]_-{\omega_{3,2}} &
}
$$
Set $e_{3,i}=1_{C_{3,i}}:C'_{3,i}\to C_{3,i}$ for $i=1,2$, and $e_{3,3}=a_{3,2}^{-1}\circ \Phi(J^3)\circ a'_{3,2}:C'_{3,3}\to C_{3,3}$. 
Also set $\overline{f'_3}=\overline{f_3}\circ \Sigma^{m_3}e_{3,3}:\Sigma^{m_3}C'_{3,3}\to X_4=X_4'$ and $\mathscr{S}_4'=(\widetilde{\Sigma}^{m_3}\mathscr{S}'_3)(\overline{f_3'},\widetilde{\Sigma}^{m_3}\mathscr{A}_3')$. 
Then $C'_{4,i}=C_{4,i}$ for $1\le i\le 3$. 
Let $\mathscr{A}_4'$ be an arbitrary regular structure on $\mathscr{S}'_4$. 
Then $C_{4,i}'=C_{4,i}$ for $i=1,2,3$. 
We have $\Sigma^2\Sigma^{m_{[2,1]}}f_0\circ\omega'_{3,2}=\omega_{3,2}\circ (e_{3,3}\cup Ce_{3,2})$ so that $g'_{4,3}\simeq g_{4,3}\circ \Sigma^2\Sigma^{m_{[3,1]}}f_0$. 
When $n=4$, set $\overline{f_4'}=\overline{f_4}:\Sigma^{m_4}C'_{4,3}\to X_5$ so that 
$\{\mathscr{S}'_r,\overline{f_r'},\mathscr{A}'_r\,|\,r=2,3,4\}$ is an $\ddot{s}_t$-presentation of $\vec{\bm f'}$ which represents $\alpha\circ \Sigma^{|\vec{\bm m}|+2}f_0$. 

When $n\ge 5$, take $J^4:g'_{4,3}\simeq g_{4,3}\circ \Sigma^2\Sigma^{m_{[3,1]}}f_0$, set 
$\Phi(J^4)=\Phi(g'_{4,3},g_{4,3},\Sigma^2\Sigma^{m_{[3,1]}}f_0,1_{C_{4,3}};J^4):C_{g'_{4,3}}\to C_{g_{4,3}}$, and consider the following homotopy commutative diagram. 
$$
\xymatrix{
&&C'_{4,4} \ar[d]^-{a'_{4,3}} \ar[r] &C'_{4,4}\cup CC'_{4,3}\ar[d]_-{a'_{4,3}\cup C1} \ar[dr]^-{\omega'_{4,3}}&\\
\Sigma^2\Sigma^{m_{[3,0]}}X_0 \ar[d]_-{\Sigma^2\Sigma^{m_{[3,1]}}f_0} \ar[r]^-{g'_{4,3}}&C'_{4,3} \ar@{=}[d] \ar[ur] \ar[r] &C_{g'_{4,3}}\ar[d]^-{\Phi(J^4)} \ar[r] & C_{g'_{4,3}}\cup CC'_{4,3} \ar[d]^-{\Phi(J^4)\cup C1} \ar[r] & \Sigma^3\Sigma^{m_{[3,0]}}X_0 \ar[d]^-{\Sigma^3\Sigma^{m_{[3,1]}}f_0} \\
\Sigma^2\Sigma^{m_{[3,1]}}X_1 \ar[r]^-{g_{4,3}} & C_{4,3} \ar[dr] \ar[r] & C_{g_{4,3}} \ar[d]^-{a_{4,3}^{-1}} \ar[r] & C_{g_{4,3}}\cup CC_{4,3} \ar[d]_-{a_{4,3}^{-1}\cup C1} \ar[r] & \Sigma^3 \Sigma^{m_{[3,1]}}X_1\\
&& C_{4,4} \ar[r] & C_{4,4}\cup CC_{4,3} \ar[ur]_-{\omega_{4,3}} &
}
$$ 
Set $e_{4,i}=1_{C_{4,i}}:C'_{4,i}\to C_{4,i}$ for $1\le i\le 3$ and $e_{4,4}=a_{4,3}^{-1}\circ\Phi(J^4)\circ a'_{4,3}:C'_{4,4}\to C_{4,4}$.

Suppose $n=5$.  
We set $\overline{f_4'} =\overline{f_4}\circ \Sigma^{m_4}e_{4,4}:\Sigma^{m_4}C_{4,4}'\to X_5=X_5'$, $\mathscr{S}_5'=(\widetilde{\Sigma}^{m_4}\mathscr{S}_4')(\overline{f_4'},\widetilde{\Sigma}^{m_4}\mathscr{A}_4')$, and let $\mathscr{A}_5'$ be an arbitrary regular structure on $\mathscr{S}_5'$. 
Then $C'_{5,i}=C_{5,i}$ for $1\le i\le 4$ and $g'_{5,4}\simeq g_{5,4}\circ \Sigma^3\Sigma^{m_{[4,1]}}f_0$ so that $\{\mathscr{S}_r',\overline{f_r'},\mathscr{A}_r'\,|\,2\le r\le 5\}$ is 
an $\ddot{s}_t$-presentation of $\vec{f'}$ representing $\alpha\circ \Sigma^{|\vec{\bm m}|+3}f_0$. 
 
By repeating the process, we have an $\ddot{s}_t$-presentation 
$\{\mathscr{S}'_r,\overline{f'_r},\mathscr{A}'_r\,|\,2\le r\le n\}$ of $\vec{\bm f'}$ 
which represents $\alpha\circ \Sigma^{|\vec{\bm m}|+n-2}f_0$. 
This proves the assertion for $\star=\ddot{s}_t$.
\end{proof}

The two proofs above of (4.3) and (4.4) contain executing the notice 
in \cite[Remark~6.2.4]{OO}. 

\section{Suspension}
Given a sequence of non-negative integers $\vec{\bm m}=(m_n,\dots,m_1)$, maps $f_i:\Sigma^{m_i}X_i\to X_{i+1}\ (1\le i\le n)$, and an integer $\ell\ge 0$, 
we set
\begin{align*}
\widetilde{\Sigma}^\ell f_i&=\Sigma^\ell f_i\circ(1_{X_i}\wedge\tau(\s^\ell,\s^{m_i})):\Sigma^{m_i}\Sigma^\ell X_i\to \Sigma^\ell X_{i+1},\\
\widetilde{\Sigma}^\ell\vec{\bm f}&=(\widetilde{\Sigma}^\ell f_n,\dots,\widetilde{\Sigma}^\ell f_1).
\end{align*}

\begin{thm}
Under the situation above, we have 
\begin{equation}
\Sigma^\ell \{\vec{\bm f}\,\}^{(\star)}_{\vec{\bm m}}\subset
\{\widetilde{\Sigma}^\ell\vec{\bm f}\,\}^{(\star)}_{\vec{\bm m}}\circ(1_{X_1}\wedge\tau(\s^{|\vec{\bm m}|+n-2},\s^\ell))
=(-1)^{(|\vec{\bm m}|+n)\ell}\{\widetilde{\Sigma}^\ell\vec{\bm f}\,\}^{(\star)}_{\vec{\bm m}}.
\end{equation}
\end{thm}
\begin{proof}
By Theorem 4.1, it suffices to prove (5.1) for $\star= aq\ddot{s}_2, aq, s_t, q_2$. 
We will prove the assertion for only $\star=aq\ddot{s}_2$, because other three 
cases can be proved similarly or more easily.  
Take $\alpha\in\{\vec{\bm f}\,\}^{(aq\ddot{s}_2)}_{\vec{\bm m}}$ and let 
$\{\mathscr{S}_r,\overline{f_r},\Omega_r\,|\,2\le r\le n\}$ be an $aq\ddot{s}_2$-presentation of $\vec{\bm f}$ with $\alpha=\overline{f_n}\circ \widetilde{\Sigma}^{m_n}g_{n,n-1}=\overline{f_n}\circ \Sigma^{m_n}g_{n,n-1}\circ (1_{\Sigma^{m_{[n-1,1]}}}\wedge\tau(\s^{m_n},\s^{n-2}))$. 
Set $X_i^*=\Sigma^\ell X_i$. 
Then $\widetilde{\Sigma}^\ell f_i:\Sigma^\ell X_i^*\to X_{i+1}^*$. 
We construct an $aq\ddot{s}_2$-presentation $\{\mathscr{S}^*_r,\overline{\widetilde{\Sigma}^\ell f_r},\Omega_r^*\,|\,2\le r\le n\}$ of $\widetilde{\Sigma}^\ell\vec{\bm f}$ 
such that 
\begin{equation}
\Sigma^\ell(\overline{f_n}\circ\widetilde{\Sigma}^{m_n}g_{n,n-1})\simeq 
\overline{\widetilde{\Sigma}^\ell f_n}\circ\widetilde{\Sigma}^{m_n}g^*_{n,n-1}\circ(1_{X_1}\wedge\tau(\s^{m_{[n,1]}}\wedge\s^{n-2},\s^\ell))
\end{equation}
which implies (5.1). 
We construct  $\{\mathscr{S}^*_r,\overline{\widetilde{\Sigma}^\ell f_r},\Omega_r^*\,|\,2\le r\le n\}$ 
and  homeomorphisms $e_{r,s}:C^*_{r,s}\to \Sigma^\ell C_{r,s}\ (1\le s\le r\le n,\, 2\le r)$ as follows. 
Set 
\begin{gather*}
\mathscr{S}^*_2=(\Sigma^{m_1}X_1^*;X_2^*, X_2^*\cup_{\widetilde{\Sigma}^\ell f_1}C\Sigma^{m_1}X_1^*;\widetilde{\Sigma}^\ell f_1;i_{\widetilde{\Sigma}^\ell f_1}),\\
\omega^*_{2,1}=q'_{\widetilde{\Sigma}^\ell f_1}:C^*_{2,2}\cup CC^*_{2,1}\to \Sigma\Sigma^{m_1}X^*_1,\quad e_{2,1}=1_{\Sigma^\ell X_2}:C_{2,1}^*=\Sigma^\ell C_{2,1},\\
e_{2,2}=\psi^\ell_{f_1}\circ(1_{X^*_2}\cup C(1_{X_1}\wedge\tau(\s^\ell,\s^{m_1})))
:C_{2,2}^*\approx \Sigma^\ell C_{2,2},\\
\overline{\widetilde{\Sigma}^\ell f_2}=\Sigma^\ell \overline{f_2}\circ (1_{C_{2,2}}\wedge\tau(\s^\ell,\s^{m_2}))
\circ \Sigma^{m_2}e_{2,2} : \Sigma^{m_2}C^*_{2,2}\to X^*_3,\\
\mathscr{S}^*_3=(\widetilde{\Sigma}^{m_2}\mathscr{S}^*_2)(\overline{\widetilde{\Sigma}^\ell f_2},\widetilde{\Sigma}^{m_2}\Omega^*_2),\ \Omega_3^*=\widetilde{\widetilde{\Sigma}^{m_2}\Omega^*_2},\ 
e_{3,1}=1_{X^*_3}:C^*_{3,1}\to \Sigma^\ell X_3,\\
e_{3,2}=\psi^\ell_{f_2}\circ (1_{X^*_3}\cup C(1_{X_2}\wedge\tau(\s^\ell,\s^{m_2}))\circ(1_{X^*_3}\cup C\Sigma^{m_2}e_{2,1}) : C^*_{3,2}\approx \Sigma^\ell C_{3,2},\\
e_{3,3}=\psi^\ell_{\overline{f_2}}\circ (1_{X^*_3}\cup C(1_{C_{2,2}}\wedge\tau(\s^\ell,\s^{m_2}))\circ(1_{X^*_3}\cup C\Sigma^{m_2}e_{2,2}) : C^*_{3,3}\approx \Sigma^\ell C_{3,3},\\
\overline{\widetilde{\Sigma}^\ell f_3}=\begin{cases} \Sigma^\ell\overline{f_3}\circ(1_{C_{3,2}}\wedge\tau(\s^\ell,\s^{m_3}))\circ \Sigma^{m_3}e_{3,2}  : \Sigma^{m_3} C^*_{3,2}\to X^*_4 & 3=n\\
\Sigma^\ell\overline{f_3}\circ(1_{C_{3,3}}\wedge\tau(\s^\ell,\s^{m_3}))\circ \Sigma^{m_3}e_{3,3}  : \Sigma^{m_3} C^*_{3,3}\to X^*_4 & 3<n\end{cases}.
\end{gather*}
By repeating the process, we have an $aq\ddot{s}_2$-presentation 
$\{\mathscr{S}^*_r,\overline{\widetilde{\Sigma}^\ell f_r},\Omega^*_r\,|\,2\le r\le n\}$ of $\widetilde{\Sigma}^\ell\vec{\bm f}$ 
and homeomorphisms $e_{r,s}:C^*_{r,s}\to \Sigma^\ell C_{r,s}$ such that 
\begin{gather*}
e_{r,1}=1_{\Sigma^\ell X_r},\ \Sigma^\ell j_{r,s}\circ e_{r,s}=e_{r,s+1}\circ j^*_{r,s},\\
\overline{\widetilde{\Sigma}^\ell f_r}=\begin{cases} \Sigma^\ell\overline{f_n}\circ(1_{C_{n,n-1}}\wedge\tau(\s^\ell,\s^{m_n}))\circ \Sigma^{m_n}e_{n,n-1}:\Sigma^{m_n}C^*_{n,n-1}\to X^*_{n+1} & r=n\\
\Sigma^\ell\overline{f_r}\circ(1_{C_{r,r}}\wedge\tau(\s^\ell,\s^{m_r}))\circ \Sigma^{m_r}e_{r,r}:\Sigma^{m_r}C^*_{r,r}\to X^*_{r+1} & r<n\end{cases}.
\end{gather*}
We will prove 
\begin{gather}
(1_{X_{r-s}}\wedge\tau(\s^\ell,\s^{m_{[r-1,r-s]}}\wedge\s^s))\circ \omega^*_{r,s}
=\Sigma^\ell\omega_{r,s}\circ\psi^\ell_{j_{r,s}}\circ(e_{r,s+1}\cup Ce_{r,s}),\\
e_{r,s}\circ g^*_{r,s}\simeq \Sigma^\ell g_{r,s}\circ(1_{X_{r-s}}\wedge\tau(\s^\ell,\s^{m_{[r-1,r-s]}}\wedge\s^{s-1})).
\end{gather}
If (5.4) holds, then (5.2) and hence (5.1) are obtained as follows. 
\begin{align*}
&\overline{\widetilde{\Sigma}^\ell f_n}\circ\widetilde{\Sigma}^{m_n}g^*_{n,n-1}\circ(1_{X_1}\wedge\tau(\s^{m_{[n,1]}}\wedge\s^{n-2},\s^\ell))\\
&\simeq \Sigma^\ell\overline{f_n}\circ(1_{C_{n,n-1}}\wedge\tau(\s^\ell,\s^{m_n}))\circ \Sigma^{m_n}\big(\Sigma^\ell g_{n,n-1}\circ(1_{X_1}\wedge\tau(\s^\ell,\s^{m_{[n-1,1]}}\wedge\s^{n-2}))\big)\\
&\qquad \circ (1_{\Sigma^{m_{[n-1,1]}}\Sigma^\ell X_1}\wedge\tau(\s^{m_n},\s^{n-2}))
\circ (1_{X_1}\wedge\tau(\s^{m_{[n,1]}}\wedge\s^{n-2},\s^\ell))\\
&=\Sigma^\ell\big(\overline{f_n}\circ \Sigma^{m_n}g_{n,n-1}\circ(1_{\Sigma^{m_{[n-1,1]}}X_1}\wedge\tau(\s^{m_n},\s^{n-2})\big)
=\Sigma^\ell(\overline{f_n}\circ\widetilde{\Sigma}^{m_n}g_{n,n-1}).
\end{align*}

In the rest of the proof, we prove (5.3) and (5.4). 
First we show that if (5.3) holds then (5.4) holds. 
Since 
\begin{align*}
e_{r,s}\circ(\overline{\widetilde{\Sigma}^\ell f_{r-1}}^s\cup C1_{\Sigma^{m_r}C^*_{r-1,s-1}})
&=\Sigma^\ell(\overline{f_{r-1}}^s\cup C1_{\Sigma^{m_{r-1}}C_{r-1,s-1}})\circ 
\psi^\ell_{\Sigma^{m_{r-1}}j_{r-1,s-1}}\\
&\quad\circ\big((1_{C_{r-1,s}}\wedge\tau(\s^\ell,\s^{m_{r-1}}))\cup C(1_{C_{r-1,s-1}}\wedge\tau(\s^\ell,\s^{m_{r-1}}))\big)\\
&\quad\circ (\Sigma^{m_{r-1}}e_{r-1,s}\cup C\Sigma^{m_{r-1}}e_{r-1,s-1}),
\end{align*}
we have 
$e_{r,s}\circ g^*_{r,s}\simeq \Sigma^\ell g_{r,s}\circ(1_{X_{r-s}}\wedge\tau(\s^\ell,\s^{m_{[r-1,r-s]}}\wedge\s^{s-1}))$ as desired. 

Secondly we prove (5.3) for $r=2$. 
Let $x_1\in X_1,\ u\in \s^\ell,\ s_1\in \s^{m_1},\ t\in I$. 
The map $\omega^*_{2,1}:C^*_{2,2}\cup C^*_{2,1}
\to \Sigma\Sigma^{m_1}\Sigma^\ell X_1$ maps $x_1\wedge u\wedge s_1\wedge\overline{ t}$ 
to $x_1\wedge u\wedge s_1\wedge\overline{t}$, and the map 
$\Sigma^\ell\omega_{2,1}\circ\psi^\ell_{j_{2,1}}:
\Sigma^\ell C_{2,2}\cup C\Sigma^\ell C_{2,1}\to \Sigma^\ell \Sigma \Sigma^{m_1}X_1$ maps 
$x_1\wedge s_1\wedge \overline{t}\wedge u$ to 
$x_1\wedge s_1\wedge\overline{t}\wedge u$. 
Hence 
$
(1_{X_1}\wedge\tau(\s^\ell,\s^{m_{1}}\wedge \s^1))\circ\omega^*_{2,1}
=\Sigma^\ell\omega_{2,1}\circ\psi^\ell_{j_{2,1}} \circ (e_{2,2}\cup Ce_{2,1}).
$ 
This proves (5.3) for $r=2$. 

By the induction, we can prove (5.3) for $2\le r\le n$. 
We omit the details. 
This completes the proof of Theorem 5.1 for $\star=aq\ddot{s}_2$. 
\end{proof}

\section{Homotopy invariance}

We will prove the following theorem which is (1.11) and allows us to use the notation 
$\{\vec{\bm \alpha}\,\}^{(\star)}_{\vec{\bm m}}$ instead of $\{\vec{\bm f}\,\}^{(\star)}_{\vec{\bm m}}$. 

\begin{thm}
If $\vec{\bm f},\vec{\bm f}'\in\mathrm{Rep}(\vec{\bm \alpha})$, then 
$\{\vec{\bm f}\,\}^{(\star)}_{\vec{\bm m}}=\{\vec{\bm f}'\,\}^{(\star)}_{\vec{\bm m}}$.
\end{thm}

For $\vec{\bm f}=(f_n,\dots,f_1)\in\mathrm{Rep}(\vec{\bm \alpha})$ and 
$i\in\{1,2,\dots,n\}$, 
let $\vec{\bm f}_i\in\mathrm{Rep}(\vec{\bm \alpha})$ denote a sequence 
obtained from $\vec{\bm f}$ by replacing $f_i$ with $f_i'$ such that $f_i\simeq f_i'$, 
for example $\vec{\bm f}_2=(f_n,\dots,f_3,f'_2,f_1)$ with $f'_2\simeq f_2$. 

\begin{lemma}
If $\vec{\bm f}\in\mathrm{Rep}(\vec{\bm \alpha})$, then $\{\vec{\bm f}\,\}^{(\star)}_{\vec{\bm m}}=\{\vec{\bm f}_i\,\}^{(\star)}_{\vec{\bm m}}$ for all $\star$ and $i$. 
\end{lemma}

From the lemma, the theorem is proved as follows: 
let $\vec{\bm f}'=(f_n',\dots,f_1')$ then 
$$
\{\vec{\bm f}\,\}^{(\star)}_{\vec{\bm m}}
=\{f_n,\dots,f_2,f_1'\,\}^{(\star)}_{\vec{\bm m}}=
\{f_n,\dots,f_3,f_2',f_1'\,\}^{(\star)}_{\vec{\bm m}}=\cdots
=\{\vec{\bm f}'\,\}^{(\star)}_{\vec{\bm m}}.
$$ 

\begin{proof}[Proof of Lemma 6.2] 
By Theorem 4.1, it suffices to prove the lemma for the cases 
$\star=\ddot{s}_t, aq\ddot{s}_2, aq, q_2$. 
We consider only the case $\star=aq\ddot{s}_2$, the other cases can be treated similarly or more easily. 
For simplicity we abbreviate $\{\ \}^{(aq\ddot{s}_2)}_{\vec{\bm m}}$ as $\{\ \}$. 
Let $\alpha\in\{\vec{\bm f}\,\}$, $\{\mathscr{S}_r,\overline{f_r},\Omega_r\,|\,2\le r\le n\}$ 
an $aq\ddot{s}_2$-presentation of $\vec{\bm f}=(f_n,\dots,f_1)$ 
with $\alpha=\overline{f_n}\circ  \Sigma ^{m_n}g_{n,n-1}\circ(1_{ \Sigma ^{m_{[n-1,1]}}X_1}\wedge\tau(\s^{m_n},\s^{n-2}))$, $\Omega_r=\{\omega_{r,s}\,|\,1\le s<r\}$ with $\omega_{r,1}=q'_{f_r}$, 
and $\Omega_{r+1}=\widetilde{\widetilde{ \Sigma }^{m_r}\Omega_r}$, 
that is, $\omega_{r+1,s+1}=\widetilde{\widetilde{ \Sigma }^{m_r}\omega_{r,s}}$. 
We are going to construct an $aq\ddot{s}_2$-presentation $\{\mathscr{S}_r',\overline{f_r}',\Omega_r'\,|\,2\le r\le n\}$ of 
$\vec{\bm f}_i$ with $\overline{f_n}'\circ  \Sigma ^{m_n}g_{n,n-1}'\circ(1_{ \Sigma ^{m_{[n-1,1]}}X_1}\wedge\tau(\s^{m_n},\s^{n-2}))=\alpha$. 
If this is done, then $\{\vec{\bm f}\,\}\subset\{\vec{\bm f}_i\}$, and by interchanging $\vec{\bm f}$ with $\vec{\bm f}_i$ 
each other we have $\{\vec{\bm f}_i\}\subset \{\vec{\bm f}\,\}$ so that $\{\vec{\bm f}\,\}=\{\vec{\bm f}_i\}$. 

We divide the proof into three cases: $i=n$; $i=1$; $2\le i\le n-1$. 

First we consider the case: $i=n$. 
Let $\vec{\bm f}_n=(f_n',f_{n-1},\dots,f_1)$ with $J^n:f_n\simeq f_n'$ and set $j=j_{n,n-2}\circ \cdots\circ j_{n,2}\circ j_{n,1}$. 
Since $j$ is a cofibration, there exists a map $H: \Sigma ^{m_n}C_{n,n-1}\times I\to X_{n+1}$ 
which makes the following diagram commutative. 
$$
\xymatrix{
&   \Sigma ^{m_n}X_n\times I \ar[dr]_-{ \Sigma ^{m_n}j\times 1_I} \ar@/^3mm/[drr]^-{J^n} & &\\
 \Sigma ^{m_n}X_n \ar[ur]^-{i_0^{ \Sigma ^{m_n}X_n}} \ar[dr]_-{ \Sigma ^{m_n}j} & &  \Sigma ^{m_n}C_{n,n-1}\times I \ar[r]^-H & X_{n+1}\\
&   \Sigma ^{m_n}C_{n,n-1} \ar[ur]^-{i_0^{ \Sigma ^{m_n}C_{n,n-1}}} \ar@/_/[urr]_-{\overline{f_n}} & &
}
$$
Let $\{\mathscr{S}_r',\overline{f_r}',\Omega_r'\,|\,2\le r\le n\}$ be the 
family obtained from $\{\mathscr{S}_r,\overline{f_r},\Omega_r\,|\,2\le r\le n\}$ by replacing 
$\overline{f_n}$ with $H_1$. 
Then the new family is an $aq\ddot{s}_2$-presentation of $\vec{\bm f}_n$ such that it represents 
$\alpha=H_1\circ  \Sigma ^{m_n}g_{n,n-1}\circ(1_{ \Sigma ^{m_{[n-1,1]}}X_1}\wedge\tau(\s^{m_n},\s^{n-2}))\in\{\vec{\bm f}_n\}$. 
Hence $\{\vec{\bm f}\,\}=\{\vec{\bm f}_n\}$. 

Secondly we consider the case: $i=1$. 
Let $\vec{\bm f}_1=(f_n,\dots,f_2,f_1')$ with $J^1:f_1\simeq f_1'$. 
We define $\mathscr{S}_r',\overline{f_r}',\Omega_r'$, and $e_{r,s}:C_{r,s}'\simeq C_{r,s}$ inductively for $r=2,3,\dots,n$ as follows. 
Set 
\begin{gather*}
\mathscr{S}_2'=( \Sigma ^{m_1}X_1;X_2,X_2\cup_{f_1'}C \Sigma ^{m_1}X_1;f_1';i_{f_1'}),\\
e_{2,2}=\Phi(f_1',f_1,1_{ \Sigma ^{m_1}X_1},1_{X_2};-J^1):C_{2,2}'=X_2\cup_{f_1'}C \Sigma ^{m_1}X_1\to C_{2,2}=X_2\cup_{f_1}C \Sigma ^{m_1}X_1, \\
e_{2,1}=1_{X_2},\quad \overline{f_2}'=\overline{f_2}\circ  \Sigma ^{m_2}e_{2,2}: \Sigma ^{m_2}C_{2,2}'\to X_3, \quad \omega_{2,1}'=q'_{f_1'},\ \Omega_2'=\{\omega_{2,1}'\}. 
\end{gather*}
Then $e_{2,2}\circ j_{2,1}'=j_{2,1}$, and $\omega_{2,1}'\simeq \omega_{2,1}\circ(e_{2,2}\cup Ce_{2,1})$ 
by Proposition~3.3(2) of \cite{OO} and so $(e_{2,2}\cup Ce_{2,1})\circ\omega_{2,1}'^{-1}\simeq \omega_{2,1}^{-1}$. 
Set $\mathscr{S}_3'=(\widetilde{ \Sigma }^{m_2}\mathscr{S}_2')(\overline{f_2}',\widetilde{ \Sigma }^{m_2}\Omega_2')$ and $\Omega_3'=\widetilde{\widetilde{ \Sigma }^{m_2}\Omega_2'}$. 
Then $C_{3,s}'=C_{3,s}$ for $s=1,2$. 
Set 
\begin{align*}
e_{3,3}&=1_{X_3}\cup C \Sigma ^{m_2}e_{2,2}: C_{3,3}'=X_3\cup_{\overline{f_2}'}C \Sigma ^{m_2}C_{2,2}'\to C_{3,3}=X_3\cup_{\overline{f_2}}C \Sigma ^{m_2}C_{2,2},\\
e_{3,s}&=1_{C_{3,s}}\ (s=1,2),\quad \overline{f_3}'=\begin{cases} \overline{f_3}\circ  \Sigma ^{m_3}e_{3,2} & n=3\\ \overline{f_3}\circ  \Sigma ^{m_3}e_{3,3} & n\ge 4\end{cases}.
\end{align*}
Then we have 
\begin{align*}
g_{3,2}'&=(\overline{f_2}'\cup C1_{ \Sigma ^{m_2}X_2})\circ(\widetilde{ \Sigma }^{m_2}\omega_{2,1}')^{-1}\\
&=(\overline{f_2}\cup C1_{ \Sigma ^{m_2}X_2})\circ ( \Sigma ^{m_2}e_{2,2}\cup C1_{ \Sigma ^{m_2}X_2})\circ(\widetilde{ \Sigma }^{m_2}\omega_{2,1}')^{-1}\\
&\simeq (\overline{f_2}\cup C1_{ \Sigma ^{m_2}X_2})\circ(\widetilde{ \Sigma }^{m_2}\omega_{2,1})^{-1}=g_{3,2},\\
\omega_{3,1}&\circ(e_{3,2}\cup Ce_{3,1})=\omega_{3,1}=\omega_{3,1}',\quad 
\omega_{3,2}\circ(e_{3,3}\cup Ce_{3,2})\simeq\omega_{3,2}'.
\end{align*}
When $n=3$, $\{\mathscr{S}_r',\overline{f_r}',\Omega_r'\,|\,r=2,3\}$ is an $aq\ddot{s}_2$-presentation of $(f_3,f_2,f_1')$ such that 
$\overline{f_3}'\circ \Sigma ^{m_3} g_{3,2}'\circ(1_{ \Sigma ^{m_2} \Sigma ^{m_1}X_1}\wedge\tau(\s^{m_3},\s^1))=\alpha$. 
When $n\ge 4$, set $\mathscr{S}_4'=(\widetilde{ \Sigma }^{m_3}\mathscr{S}_3')(\overline{f_3}',\widetilde{ \Sigma }^{m_3}\Omega_3')$ and $\Omega_4'=\widetilde{\widetilde{ \Sigma }^{m_3}\Omega_3'}$. 
Then $C_{4,s}'=C_{4,s}$ for $1\le s\le 3$. 
Set 
$$
e_{4,s}=\begin{cases} 1_{C_{4,s}} & 1\le s\le 3\\ 1_{X_4}\cup C \Sigma ^{m_3}e_{3,3} & s=4\end{cases},\quad 
\overline{f_4}'=\begin{cases} \overline{f_4}\circ \Sigma ^{m_4} e_{4,3} & n=4\\ \overline{f_4}\circ  \Sigma ^{m_4}e_{4,4} & n\ge 5\end{cases}.
$$
We have $g_{4,1}'=g_{4,1}$, $g_{4,2}'=g_{4,2}$, and 
\begin{align*}
g_{4,3}'&=(\overline{f_3}'\cup C1_{ \Sigma ^{m_3}C_{3,2}'})\circ(\widetilde{ \Sigma }^{m_3}\omega_{3,2}')^{-1}\\
&=(\overline{f_3}\cup C1_{ \Sigma ^{m_3}C_{3,2}})\circ( \Sigma ^{m_3}e_{3,3}\cup C \Sigma ^{m_3}e_{3,2})\circ(\widetilde{ \Sigma }^{m_3}\omega_{3,2}')^{-1}\\
&\simeq(\overline{f_3}\cup C1_{ \Sigma ^{m_3}C_{3,2}})
\circ(\widetilde{ \Sigma }^{m_3}\omega_{3,2})^{-1}=g_{4,3}
\end{align*}
If $n=4$, then 
$$
\overline{f_4}'\circ \Sigma ^{m_4} g_{4,3}'\circ(1_{ \Sigma ^{m_{[3,1]}}X_1}\wedge\tau(\s^{m_3},\s^2))\simeq \overline{f_4}\circ  \Sigma ^{m_4}g_{4,3}\circ(1_{ \Sigma ^{m_{[3,1]}}X_1}\wedge\tau(\s^{m_3},\s^2))
$$ and so 
$\{\vec{\bm f}\,\}^{(aq\ddot{s}_2)}=\{\vec{\bm f_1}\,\}^{(aq\ddot{s}_2)}$. 
Continuing the process, we obtain a desired $aq\ddot{s}_2$-presentation of $\vec{\bm f}_1$. 

Thirdly we consider the case: $2\le i\le n-1$. 
We prove only the case $i=2$, because other cases can be proved similarly. 
Let $\vec{\bm f}_2=(f_n,\dots,f_3,f_2',f_1)$ with $J^2:f_2\simeq f_2'$. 
Set $\mathscr{S}_2'=\mathscr{S}_2$ and $\Omega_2'=\Omega_2$. 
Then $C_{2,s}'=C_{2,s}$, $j_{2,1}'=j_{2,1}=i_{f_1}$, and $\omega_{2,1}'=\omega_{2,1}=q_{f_1}'$. 
Set $e_{2,s}=1_{C_{2,s}}$ for $1\le s\le 2$. 
Since $ \Sigma ^{m_2}j_{2,1}= \Sigma ^{m_2}i_{f_1}: \Sigma ^{m_2}X_2\to \Sigma ^{m_2} C_{2,2}= \Sigma ^{m_2}(X_2\cup_{f_1}C \Sigma ^{m_1}X_1)$ is a free cofibration, there exists $H^2: \Sigma ^{m_2}C_{2,2}\times I\to X_3$ 
such that $H^2\circ i_0^{ \Sigma ^{m_2}C_{2,2}}=\overline{f_2}$ and $H\circ ( \Sigma ^{m_2}i_{f_1}\times 1_I)=J^2$. 
Then the following diagram is commutative. 
$$
\xymatrix{
&  \Sigma ^{m_2}X_2\times I \ar[dr]_-{ \Sigma ^{m_2}i_{f_1}\times 1_I} \ar@/^6mm/[drr]_-{J^2} & & \\
 \Sigma ^{m_2}X_2 \ar[dr]_-{ \Sigma ^{m_2}i_{f_1}} \ar[ur]^-{i_0} & &  \Sigma ^{m_2}(X_2\cup_{f_1}CX_1)\times I \ar[r]^-{H^2} & X_3\\
&  \Sigma ^{m_2}(X_2\cup_{f_1}C \Sigma ^{m_1}X_1) \ar[ur]^-{i_0} \ar@/_/[urr]_-{\overline{f_2}} & &
}
$$
Set $\overline{f_2}'=H^2_1: \Sigma ^{m_2}C_{2,2}'= \Sigma ^{m_2}C_{2,2}\to X_3$. 
Set 
\begin{align*}
&\mathscr{S}_3'=(\widetilde{ \Sigma }^{m_2}\mathscr{S}_2')(\overline{f_2}',\widetilde{ \Sigma }^{m_2}\Omega_2'),\quad \Omega_3'=\widetilde{\widetilde{ \Sigma }^{m_2}\Omega_2'},\\
&e_{3,3}=\Phi(\overline{f_2}',\overline{f_2},1_{ \Sigma ^{m_2}C_{2,2}},1_{X_3};-H^2)\\
&\hspace{2cm}:C_{3,3}'=X_3\cup_{\overline{f_2}'}C \Sigma ^{m_2}C_{2,2}'
\to C_{3,3}=X_3\cup_{\overline{f_2}}C \Sigma ^{m_2}C_{2,2},\\
&e_{3,2}=\Phi(f_2',f_2,1_{ \Sigma ^{m_2}X_2},1_{X_3};-J^2)\\
&\hspace{2cm}:C_{3,2}'=X_3\cup_{f_2'}C \Sigma ^{m_2}X_2\to C_{3,2}=X_3\cup_{f_2}C \Sigma ^{m_2}X_2,\\
&e_{3,1}=1_{X_3}:C_{3,1}'\to C_{3,1},\quad 
\overline{f_3}'=\begin{cases} \overline{f_3}\circ  \Sigma ^{m_3}e_{3,2} & n=3\\ \overline{f_3}\circ  \Sigma ^{m_3}e_{3,3} & n\ge 4\end{cases}.
\end{align*}
Then $e_{3,3}\circ j_{3,2}'=j_{3,2}\circ e_{3,2}$, $e_{3,2}\circ j_{3,1}'=j_{3,2}\circ e_{3,1}$, and $e_{3,1}\circ g_{3,1}'=f_2'\simeq f_2=g_{3,1}$. 
We will prove 
$$
e_{3,2}\circ g_{3,2}'\simeq g_{3,2}\quad  i.e.\quad e_{3,2}\circ(\overline{f_2}'\cup C1_{ \Sigma ^{m_2}X_2})\simeq \overline{f_2}\cup C1_{ \Sigma ^{m_2}X_2}.
$$
We have
\begin{align*}
&e_{3,2}\circ (\overline{f_2}'\cup C1_{ \Sigma ^{m_2}X_2})\\
&\simeq\Phi (f_2',f_2,1_{ \Sigma ^{m_2}X_2},1_{X_3};-J^2)\circ \Phi ( \Sigma ^{m_2}i_{f_1},f_2',1_{ \Sigma ^{m_2}X_2},\overline{f_2}';1_{f_2'})\\
&\hspace{6cm} (\text{by \cite[Proposition 3.3(3)]{OO}})\\
&\simeq\Phi ( \Sigma ^{m_2}i_{f_1},f_2,1_{ \Sigma ^{m_2}X_2},\overline{f_2}';((-J^2)\bar\circ 1_{1_{ \Sigma ^{m_2}X_2}})\bullet(1_{1_{X_3}}\bar\circ 1_{f_2'}))\\ 
&\hspace{6cm}(\text{by \cite[Proposition 3.3(1)(d)]{OO}})\\
&\simeq \Phi( \Sigma ^{m_2}i_{f_1},f_2,1_{ \Sigma ^{m_2}X_2},\overline{f_2}';-J^2)\\
&\hspace{2cm} (\text{by \cite[Proposition 3.3(5)]{OO} 
and $((-J^2)\bar\circ 1_{1_{ \Sigma ^{m_2}X_2}})\bullet(1_{1_{X_3}}\bar\circ 1_{f_2'})\overset{ \Sigma ^{m_2}X_2}{\simeq}-J^2$}).
\end{align*}
We define $\widetilde{J}^t:  \Sigma ^{m_2}X_2\times I\to X_3$ for $t\in I$ by $\widetilde{J}^t(z,u)=(-J^2)(z,t+u-tu)$. 
Then $\widetilde{J}^t:(-H^2)_t\circ \Sigma ^{m_2} i_{f_1}\simeq f_2\circ 1_{ \Sigma ^{m_2}X_2}$ and so 
\begin{align*}
&\Phi( \Sigma ^{m_2}i_{f_1},f_2,1_{ \Sigma ^{m_2}X_2},\overline{f_2}';-J^2)\\
&\qquad\simeq\Phi( \Sigma ^{m_2}i_{f_1},f_2,1_{ \Sigma ^{m_2}X_2},\overline{f_2};1_{f_2})\ (\text{by \cite[Proposition 3.3(4)]{OO}})\\
&\qquad\simeq \overline{f_2}\cup C1_{ \Sigma ^{m_2}X_2}\ (\text{by \cite[Proposition 3.3(3)]{OO}}).
\end{align*}
Therefore $e_{3,2}\circ(\overline{f_2}'\cup C1_{ \Sigma ^{m_2}X_2})\simeq \overline{f_2}\cup C1_{ \Sigma ^{m_2}X_2}$ as desired. 
When $n=3$, $\{\mathscr{S}_r',\overline{f_r}',\Omega_r'\,|\,r=2,3\}$ is an $aq\ddot{s}_2$-presentation of $(f_3,f_2',f_1)$ such that 
$\overline{f_3}'\circ g_{3,2}'=\overline{f}_3\circ e_{3,2}\circ g_{3,2}'\simeq \overline{f_3}\circ g_{3,2}$ 
so that $\{\vec{\bm f}\,\}^{(aq\ddot{s}_2)}=\{\vec{\bm f_2}\,\}^{(aq\ddot{s}_2)}$. 
When $n\ge 4$, by repeating the process above, we have $\mathscr{S}_r'$, $\overline{f_r}'$, $\Omega_r'$, and 
$e_{r,s}:C_{r,s}'\simeq C_{r,s}$ for $3\le r\le n$ such that 
\begin{gather*}
\mathscr{S}_r'=(\widetilde{ \Sigma }^{m_{r-1}}\mathscr{S}_{r-1}')(\overline{f_{r-1}}',\widetilde{ \Sigma }^{m_{r-1}}\Omega_{r-1}'),\ \Omega_r'=\widetilde{\widetilde{ \Sigma }^{m_{r-1}}\Omega_{r-1}'},\\
C_{r,s}'=C_{r,s},\ e_{r,s}=1_{C_{r,s}}\ (1\le s\le r-2),\quad j_{r,s}'=j_{r,s}\ (1\le s\le r-3),\\
e_{r,s+1}\circ j_{r,s}'=j_{r,s}\circ e_{r,s}\ (1\le s\le r-1),\\
e_{r,s}\circ g_{r,s}'\simeq g_{r,s}\ (1\le s\le r-1),\\
\overline{f_r}'=\begin{cases} \overline{f_n}\circ  \Sigma ^{m_n}e_{n,n-1} & r=n\\
\overline{f_r}\circ  \Sigma ^{m_r}e_{r,r} & r<n\end{cases}.
\end{gather*}
Therefore $\{\mathscr{S}_r',\overline{f_r}',\Omega_r'\,|\,2\le r\le n\}$ is an $aq\ddot{s}_2$-presentation of $\vec{{\bm f}_2}$ 
such that $\overline{f_n}'\circ  \Sigma ^{m_n}g_{n,n-1}'=\overline{f_n}\circ \Sigma ^{m_n}e_{n,n-1}\circ \Sigma ^{m_n}g_{n,n-1}'\simeq\overline{f_n}\circ \Sigma ^{m_n}g_{n,n-1}$. 
This ends the proof for the case of $i=2$ and completes the proof of Lemma 6.2 (for the case of $\star=aq\ddot{s}_2$). 
\end{proof}

\section{3-fold brackets}
In this section we prove (1.12). 
\begin{thm}
$\{f_3,f_2,f_1\}_{(m_3,m_2,m_1)}^{(\ddot{s}_t)}=\{f_3,f_2,f_1\}_{(m_3,m_2,m_1)}^{(aq\ddot{s}_2)}=\{f_3,f_2, \Sigma ^{m_2}f_1\}_{m_3}$. 
\end{thm}

\begin{cor}
Let $f_i:\Sigma^{m_i}X_i\to X_{i+1}\ (0\le i\le 3)$ be maps such that 
$f_{i+1}\circ \Sigma^{m_{i+1}}f_i\simeq *\ (0\le i\le 2)$. 
Then for $\star=aq\ddot{s}_2,\ddot{s}_t$ we have
\begin{align*}
\{f_3,f_2,f_1\}^{(\star)}_{(m_3,m_2,m_1)}\circ \Sigma^{m_3+m_2+m_1+1}f_0
&=(-1)^{m_3+1}\big(f_3\circ \Sigma^{m_3}\{f_2,\Sigma^{m_2}f_1,\Sigma^{m_{[2,1]}}f_0\}\big)\\
&\supset (-1)^{m_3+1}\big(f_3\circ \Sigma^{m_3}\{f_2,f_1,f_0\}^{(\star)}_{(m_2,m_1,m_0)}\big),
\end{align*}
where the inclusion is the identity when $m_2=0$. 
\end{cor}

\begin{prob}
Is there a sequence $(f_3,f_2,f_1,f_0)$ such that it satisfies the assumption of 
Corollary 7.2 and the inclusion in 7.2 is not the identity? 
\end{prob}

\begin{proof}[Proof of Theorem 7.1]
The first equality holds from definitions. 
Let $\alpha\in\{f_3,f_2,f_1\}^{(aq\ddot{s}_2)}_{(m_3,m_2,m_1)}$ and 
$\{\mathscr{S}_r,\overline{f_r},\Omega_r\,|\,r=2,3\}$ an $aq\ddot{s}_2$-presentation of $\vec{\bm f}$ with $\alpha=\overline{f_3}\circ \widetilde{\Sigma}^{m_3}g_{3,2}$. 
Then 
\begin{gather*}
\mathscr{S}_2=( \Sigma ^{m_1}X_1;X_2,X_2\cup_{f_1}C \Sigma ^{m_1}X_1;f_1;i_{f_1}),\ \overline{f_2}: \Sigma ^{m_2}(X_2\cup_{f_1}C \Sigma ^{m_1}X_1)\to X_3,\\
\Omega_2=\{\omega_{2,1}\},\ \omega_{2,1}=q_{f_1}':X_2\cup_{f_1}C \Sigma ^{m_1}X_1\cup CX_2\simeq \Sigma  \Sigma ^{m_1}X_1,\\
\mathscr{S}_3=( \Sigma ^{m_2}\mathscr{S}_2)(\overline{f_2},\widetilde{ \Sigma }^{m_2}\Omega_2),\ \overline{f_3}: \Sigma ^{m_3}(X_3\cup_{f_2}C \Sigma ^{m_2}X_2)\to X_4,\ 
\Omega_3=\widetilde{ \Sigma }^{m_2}\Omega_2.
\end{gather*}
Take $A_1:f_2\circ \Sigma ^{m_2}f_1\simeq *$ and $A_2:f_3\circ \Sigma ^{m_3}f_2\simeq *$ such that 
$\overline{f_2}\circ\psi^{m_2}_{f_1}=[f_2,A_1, \Sigma ^{m_2}f_1]$ and $\overline{f_3}\circ\psi^{m_3}_{f_2}=[f_3,A_2, \Sigma ^{m_3}f_2]$. 
Then $g_{3,2}\simeq(f_2,A_1, \Sigma ^{m_1}f_1)$ by \cite[(4.2)]{OO}. 
Therefore \begin{align*}
\overline{f_3}\circ\widetilde{\Sigma}^{m_3}g_{3,2}&=\overline{f_3}\circ  \Sigma ^{m_3}g_{3,2}\circ (1_{ \Sigma ^{m_{[2,1]}}X_1}\wedge\tau(\s^{m_3},\s^1))\\
&\simeq [f_3,A_2, \Sigma ^{m_3}f_2]\circ(\psi^{m_3}_{f_2})^{-1}\circ  \Sigma ^{m_3}(f_2,A_1, \Sigma ^{m_2}f_1)\circ 
(1_{ \Sigma ^{m_{[2,1]}}X_1}\wedge\tau(\s^{m_3},\s^1))\\
&=[f_3,A_2, \Sigma ^{m_3}f_2]\circ( \Sigma ^{m_3}f_2,\widetilde{ \Sigma }^{m_3}A_1, \Sigma ^{m_{[3,2]}}f_1)\quad(\text{by \cite[Lemma 2.4]{OO1}})\\
&\in\{f_3,f_2, \Sigma ^{m_2}f_1\}_{m_3}.
\end{align*}
Hence $\{f_3,f_2,f_1\}^{(aq\ddot{s}_2)}_{(m_3,m_2,m_1)}\subset\{f_3,f_2, \Sigma ^{m_2}f_1\}_{m_3}$. 

Next we prove $\{f_3,f_2, \Sigma ^{m_2}f_1\}_{m_3}\subset\{f_3,f_2,f_1\}^{(aq\ddot{s}_2)}_{(m_3,m_2,m_1)}$. 
Let $\alpha\in\{f_3,f_2, \Sigma ^{m_2}f_1\}_{m_3}$. 
Then $\alpha=[f_3,A_2, \Sigma ^{m_3}f_2]\circ( \Sigma ^{m_3}f_2,\widetilde{ \Sigma }^{m_3}A_1, \Sigma ^{m_{[3, 2]}}f_1)$ for some 
$A_2:f_3\circ \Sigma ^{m_3}f_2\simeq *$ and $A_1:f_2\circ \Sigma ^{m_2}f_1\simeq *$. 
By \cite[Lemma 2.4]{OO1}, $\alpha=[f_3,A_2, \Sigma ^{m_3}f_2]\circ(\psi^{m_3}_{f_2})^{-1}\circ \Sigma ^{m_3}(f_2,A_1, \Sigma ^{m_2}f_1)\circ(1_{ \Sigma ^{m_{[2,1]}}X_1}\wedge\tau(\s^{m_3},\s^1))$. 
Set 
\begin{align*}
\overline{f_2}&=[f_2,A_1, \Sigma ^{m_2}f_1]\circ(\psi^{m_2}_{f_1})^{-1} :  \Sigma ^{m_2}(X_2\cup_{f_1}C \Sigma ^{m_1}X_1)\to X_3,\\
g_{3,2}&=(f_2,A_1, \Sigma ^{m_2}f_1): \Sigma  \Sigma ^{m_{[2,1]}}X_1\to X_3\cup_{f_2}C \Sigma ^{m_2}X_2,\\
\overline{f_3}&=[f_3,A_2, \Sigma ^{m_3}f_2]\circ(\psi^{m_3}_{f_2})^{-1}: \Sigma ^{m_3}(X_3\cup_{f_2}C \Sigma ^{m_2}X_2)\to X_4,\\
\mathscr{S}_2&=( \Sigma ^{m_1}X_1;X_2,X_2\cup_{f_1}C \Sigma ^{m_1}X_1;f_1;i_{f_1}),\quad 
\omega_{2,1}=q'_{f_1},\quad \Omega_2=\{\omega_{2,1}\},\\
\mathscr{S}_3&=( \Sigma ^{m_2}\mathscr{S}_2)(\overline{f_2},\widetilde{ \Sigma }^{m_2}\Omega_2)\\
&=( \Sigma ^{m_2}X_2, \Sigma  \Sigma ^{m_{[2,1]}}X_1;X_3,X_3\cup_{f_2}C \Sigma ^{m_2}X_2,X_3\cup_{\overline{f_2}}C \Sigma ^{m_2}(X_2\cup_{f_1}C \Sigma ^{m_1}X_1);f_2,g_{3,2};\\
&\hspace{4cm}i_{f_2},1_{X_3}\cup C \Sigma ^{m_2}i_{f_1}),\quad \Omega_3=\widetilde{\widetilde{ \Sigma }^{m_2}\Omega_2}.
\end{align*}
Then $\{\mathscr{S}_r,\overline{f_r},\Omega_r\,|\,2\le r\le 3\}$ is an $aq\ddot{s}_2$-presentation of $\vec{\bm f}$ and it represents $\alpha$. 
Hence $\{f_3,f_2, \Sigma ^{m_2}f_1\}_{m_3}\subset\{f_3,f_2,f_1\}^{(aq\ddot{s}_2)}_{(m_3,m_2,m_1)}$. 
Thus  $\{f_3,f_2, \Sigma ^{m_2}f_1\}_{m_3}=\{f_3,f_2,f_1\}^{(aq\ddot{s}_2)}_{(m_3,m_2,m_1)}$. 
\end{proof}

\begin{proof}[Proof of Corollary 7.2]
Let $\star=aq\ddot{s}_2, \ddot{s}_t$. Then we have
\begin{align*}
&\{f_3,f_2,f_1\}^{(\star)}_{(m_3,m_2,m_1)}\circ \Sigma^{m_3+m_2+m_1+1}f_0\\
&\ =\{f_3,f_2,\Sigma^{m_2}f_1\}_{m_3}\circ \Sigma^{m_3+1}\Sigma^{m_{[2,1]}}f_0\quad(\text{by Theorem 7.1})\\
&\ =(-1)^{m_3+1}\big(f_3\circ \Sigma^{m_3}\{f_2,\Sigma^{m_2}f_1,\Sigma^{m_{[2,1]}}f_0\}\big)\quad(\text{by \cite[Proposition 1.4]{T}})\\
&\ \supset (-1)^{m_3+1}\big(f_3\circ \Sigma^{m_3}\{f_2,f_1,f_0\}^{(\star)}_{(m_2,m_1,m_0)}\big)
\quad(\text{by Proposition 3.8}),
\end{align*}
where the inclusion is the identity when $m_2=0$ by Theorem 7.1. 
This ends the proof. 
\end{proof}

\section{4-fold brackets}
In this section we prove (1.13). 

\begin{thm}
When $n=4$, we have
\begin{align*}
\{\vec{\bm f}\}^{(\ddot{s}_t)}_{\vec{\bm m}}&=\bigcup\{f_4,[f_3,A_2,\Sigma^{m_3}f_2],(\Sigma^{m_3}f_2,\widetilde{\Sigma}^{m_3}A_1,\Sigma^{m_{[3,2]}}f_1)\}_{m_4}
\circ \Sigma(1_{\Sigma^{m_{[3,1]}}X_1}\wedge\tau(\s^{m_4},\s^1))\\
&=\bigcup (-1)^{m_4}\{f_4,[f_3,A_2,\Sigma^{m_3}f_2],(\Sigma^{m_3}f_2,\widetilde{\Sigma}^{m_3}A_1,\Sigma^{m_{[3,2]}}f_1)\}_{m_4}
\end{align*}
where the unions $\bigcup$ are taken 
over all triples $\vec{\bm A}=(A_3,A_2,A_1)$ such that $(\vec{\bm f};\vec{\bm A})$ is admissible. 
\end{thm}

\begin{cor}
If in addition $m_0$ is a non negative integer and 
$f_0:\Sigma^{m_0}X_0\to X_1$ is a map such that $\{f_2,f_1,f_0\}^{(\ddot{s}_t)}_{(m_2,m_1,m_0)}=\{0\}$, 
then 
$$
\{f_4,f_3,f_2,f_1\}^{(\ddot{s}_t)}_{(m_4,m_3,m_2,m_1)}\circ \Sigma^{m_4+\cdots+m_1+2}f_0\subset 
f_4\circ \Sigma^{m_4}\{f_3,\Sigma^{m_3}f_2,\Sigma^{m_{[3,2]}}f_1,\Sigma^{m_{[3,1]}}f_0\}^{(\ddot{s}_t)}.
$$
\end{cor}

The theorem consists of the following two relations. 
\begin{gather}
\begin{split}
&\{\vec{\bm f}\}^{(\ddot{s}_t)}_{\vec{\bm m}}\subset \bigcup\{f_4,[f_3,A_2,\Sigma^{m_3}f_2],(\Sigma^{m_3}f_2,\widetilde{\Sigma}^{m_3}A_1,\Sigma^{m_{[3,2]}}f_1)\}_{m_4}\\
&\hspace{6cm}\circ \Sigma(1_{\Sigma^{m_{[3,1]}}X_1}\wedge\tau(\s^{m_4},\s^1)),
\end{split}\\
\begin{split}
&\{\vec{\bm f}\}^{(\ddot{s}_t)}_{\vec{\bm m}}\supset \bigcup\{f_4,[f_3,A_2,\Sigma^{m_3}f_2],(\Sigma^{m_3}f_2,\widetilde{\Sigma}^{m_3}A_1,\Sigma^{m_{[3,2]}}f_1)\}_{m_4}\\
&\hspace{6cm}\circ \Sigma(1_{\Sigma^{m_{[3,1]}}X_1}\wedge\tau(\s^{m_4},\s^1)).
\end{split}
\end{gather}

\begin{proof}[Proof of (8.1)]
Let $\alpha\in\{\vec{\bm f}\}^{(\ddot{s}_t)}_{\vec{\bm m}}$ and $\{\mathscr{S}_r,\overline{f_r},\mathscr{A}_r\,|\,r=2,3,4\}$ an $\ddot{s}_t$-presentation of $\vec{\bm f}$ 
with $\alpha=\overline{f_4}\circ\widetilde{\Sigma}^{m_4}g_{4,3}$. 
Set $\overline{f_2}'=\overline{f_2}\circ \psi^{m_2}_{f_1}:\Sigma^{m_2}X_2\cup_{\Sigma^{m_2}f_1}C\Sigma^{m_{[2,1]}}X_1\to X_3$. 
We take $A_1:f_2\circ \Sigma^{m_2}f_1\simeq *$ satisfying $\overline{f_2}'=[f_2,A_1,\Sigma^{m_2}f_1]$. 
Since $\widetilde{\Sigma}^{m_2}a_{2,1}=(\psi^{m_2}_{f_1})^{-1}$ by definition, we have 
$\overline{f_2}=\overline{f_2}'\circ \widetilde{\Sigma}^{m_2}a_{2,1}$. 
Since $\widetilde{\Sigma}^{m_2}\omega_{2,1}=q'_{\Sigma^{m_2}f_1}\circ(\widetilde{\Sigma}^{m_2}a_{2,1}\cup C1_{\Sigma^{m_2}X_2})$, we have
\begin{align*}
g_{3,2}&=(\overline{f_2}\cup C1_{\Sigma^{m_2}X_2})\circ(\widetilde{\Sigma}^{m_2}\omega_{2,1})^{-1}=(\overline{f_2}'\cup C1_{\Sigma^{m_2}X_2})\circ(q'_{\Sigma^{m_2}f_1})^{-1}\\
&\simeq (f_2,A_1,\Sigma^{m_2}f_1)\quad(\text{by \cite[(4.2)]{OO}}).
\end{align*} 
Set ${\overline{f_3}^2}'=\overline{f_3}^2\circ\psi^{m_3}_{f_2}(=\overline{f_3}^2\circ(\widetilde{\Sigma}^{m_3}a_{3,1})^{-1})$. 
We take $A_2:f_3\circ \Sigma^{m_3}f_2\simeq *$ such that ${\overline{f_3}^2}'=[f_3,A_2,\Sigma^{m_3}f_2]$. 
We have 
\begin{align*}
* &\simeq \overline{f_3}\circ \Sigma^{m_3}j_{3,2}\circ \Sigma^{m_3}g_{3,2}\circ(1_{\Sigma^{m_{[2,1]}}X_1}\wedge\tau(\s^{m_3},\s^1))\\
&={\overline{f_3}^2}'\circ (\psi^{m_3}_{f_2})^{-1}\circ \Sigma^{m_3}g_{3,2}\circ(1_{\Sigma^{m_{[2,1]}}X_1}\wedge\tau(\s^{m_3},\s^1))\\
&\simeq {\overline{f_3}^2}'\circ (\psi^{m_3}_{f_2})^{-1}\circ \Sigma^{m_3}(f_2,A_1,\Sigma^{m_2}f_1)\circ (1_{\Sigma^{m_{[2,1]}}X_1}\wedge\tau(\s^{m_3},\s^1))\\
&= {\overline{f_3}^2}'\circ(\Sigma^{m_3}f_2,\widetilde{\Sigma}^{m_3}A_1,\Sigma^{m_{[3,2]}}f_1)\circ(1_{\Sigma^{m_{[2,1]}}X_1}\wedge\tau(\s^1,\s^{m_3}))\circ(1_{\Sigma^{m_{[2,1]}}X_1}\wedge\tau(\s^{m_3},\s^1))\\
&={\overline{f_3}^2}'\circ(\Sigma^{m_3}f_2,\widetilde{\Sigma}^{m_3}A_1,\Sigma^{m_{[3,2]}}f_1)\\
&=[f_3,A_2,\Sigma^{m_3}f_2]\circ(\Sigma^{m_3}f_2,\widetilde{\Sigma}^{m_3}A_1,\Sigma^{m_{[3,2]}}f_1).
\end{align*}
We have $g_{4,2}=(\overline{f_3}^2\cup C1_{\Sigma^{m_3}X_3})\circ(\widetilde{\Sigma}^{m_3}\omega_{3,1})^{-1}\simeq(f_3,A_2,\Sigma^{m_3}f_2)$ by \cite[(4.2)]{OO}. 
Since $\overline{f_4}^2\circ\psi^{m_4}_{f_3}:\Sigma^{m_4}X_4\cup_{\Sigma^{m_4}f_3}C\Sigma^{m_{[4,3]}}X_3\to X_5$ is an extension of $f_4$, there exists $A_3:f_4\circ \Sigma^{m_4}f_3\simeq *$ 
such that $[f_4,A_3,\Sigma^{m_4}f_3]\circ(\psi^{m_4}_{f_3})^{-1}=\overline{f_4}^2$. 
We have
\begin{align*}
*&\simeq \overline{f_4}\circ \Sigma^{m_4}j_{4,2}\circ \Sigma^{m_4}g_{4,2}\circ
(1_{\Sigma^{m_{[3,2]}}X_2}\wedge\tau(\s^{m_4},\s^1))\\
&=\overline{f_4}^2\circ \Sigma^{m_4}g_{4,2}\circ(1_{\Sigma^{m_{[3,2]}}X_2}\wedge\tau(\s^{m_4},\s^1))\\
&=[f_4,A_3,\Sigma^{m_4}f_3]\circ(\psi^{m_4}_{f_3})^{-1}\circ \Sigma^{m_4}g_{4,2}\circ (1_{\Sigma^{m_{[3,2]}}X_2}\wedge\tau(\s^{m_4},\s^1))\\
&\simeq [f_4,A_3,\Sigma^{m_4}f_3]\circ(\psi^{m_4}_{f_3})^{-1}\circ \Sigma^{m_4}(f_3,A_2,\Sigma^{m_3}f_2)\circ (1_{\Sigma^{m_{[3,2]}}X_2}\wedge\tau(\s^{m_4},\s^1))\\
&= [f_4,A_3,\Sigma^{m_4}f_3]\circ(\Sigma^{m_4}f_3,\widetilde{\Sigma}^{m_4}A_2,\Sigma^{m_{[4,3]}}f_2).
\end{align*}
Hence $(\vec{\bm f};A_3,A_2,A_1)$ is admissible and so 
\begin{equation}
f_4\circ \Sigma^{m_4}{\overline{f_3}^2}'\simeq *
\end{equation}
since 
\begin{align*}
&f_4\circ \Sigma^{m_4}{\overline{f_3}^2}'=f_4\circ \Sigma^{m_4}[f_3,A_2,\Sigma^{m_3}f_2]=
[f_4,A_3,\Sigma^{m_4}f_3]\circ i_{\Sigma^{m_4}f_3}\circ \Sigma^{m_4}[f_3,A_2,\Sigma^{m_3}f_2]\\
&=[f_4,A_3,\Sigma^{m_4}f_3]\circ i_{\Sigma^{m_4}f_3}\circ [\Sigma^{m_4}f_3,\widetilde{\Sigma}^{m_4}A_2,\Sigma^{m_{[4,3]}}f_2]\circ(\psi^{m_4}_{\Sigma^{m_3}f_2})^{-1}\ 
 (\text{by \cite[Lemma 2.4]{OO1}})\\
&\simeq [f_4,A_3,\Sigma^{m_4}f_3]\circ (\Sigma^{m_4}f_3,\widetilde{\Sigma}^{m_4}A_2,\Sigma^{m_{[4,3]}}f_2)
\circ q_{\Sigma^{m_{[4,3]}}f_2}\circ (\psi^{m_4}_{\Sigma^{m_3}f_2})^{-1}\\
&\hspace{6cm}(\text{by \cite[Proposition 5.11]{Og} or \cite[Lemma 3.6]{OO1}})\\
&\simeq *.
\end{align*}

Let $(\widetilde{\Sigma}^{m_3}a_{3,2})^{-1}\in\mathrm{TOP}^{\Sigma^{m_3}C_{3,2}}(i_{\widetilde{\Sigma}^{m_3}g_{3,2}},\Sigma^{m_3}j_{3,2})$ be a homotopy inverse of the homotopy equivalence 
$\widetilde{\Sigma}^{m_3}a_{3,2}\in\mathrm{TOP}^{\Sigma^{m_3}C_{3,2}}(\Sigma^{m_3}j_{3,2},i_{\widetilde{\Sigma}^{m_3}g_{3,2}})$ in the category $\mathrm{TOP}^{\Sigma^{m_3}C_{3,2}}$. 
Set $\overline{f_3}'=\overline{f_3}\circ(\widetilde{\Sigma}^{m_3}a_{3,2})^{-1}$. 
Then $\overline{f_3}'$ is an extension of $\overline{f_3}^2$ and so there is a null homotopy 
$B:\overline{f_3}^2\circ\widetilde{\Sigma}^{m_3}g_{3,2}\simeq *$ such that 
$\overline{f_3}'=[\overline{f_3}^2,B,\widetilde{\Sigma}^{m_3}g_{3,2}]$. 
We have $\overline{f_3}'\circ\widetilde{\Sigma}^{m_3}a_{3,2}\overset{\Sigma^{m_3}C_{3,2}}{\simeq}\overline{f_3}$. 
The following diagram is homotopy commutative so that 
$g_{4,3}\simeq (\overline{f_3}^2,B,\widetilde{\Sigma}^{m_3}g_{3,2})$ by \cite[(4.2)]{OO}. 
$$
\xymatrix{
\Sigma^{m_3}C_{3,3}\cup C\Sigma^{m_3}C_{3,2} \ar@/^5mm/[rrd]_-{\widetilde{\Sigma}^{m_3}\omega_{3,2}}
\ar[rd]_-{\widetilde{\Sigma}^{m_3}a_{3,2}\cup C1_{\Sigma^{m_3}C_{3,2}}} \ar[dd]_-{\overline{f_3}\cup C1}& &\\
& (\Sigma^{m_3}C_{3,2}\cup_{}C\Sigma\Sigma^{m_{[3,1]}}X_1)\cup C\Sigma^{m_3}C_{3,2} \ar[r]^-{q'} 
\ar[ld]_-{\overline{f_3}'\cup C1} & \Sigma^2\Sigma^{m_{[3,1]}}X_1\ar@/^5mm/[lld]_-{g_{4,3}}\\
X_4\cup_{\overline{f_3}^2}C\Sigma^{m_3}C_{3,2} & &
}
$$
Set $\overline{f_4}'=\overline{f_4}\circ\psi^{m_4}_{\overline{f_3}^2}:\Sigma^{m_4}X_4\cup_{\Sigma^{m_4}\overline{f_3}^2}C\Sigma^{m_{[4,3]}}C_{3,2}$. 
Then $\overline{f_4}'$ is an extension of $\Sigma^{m_4}\overline{f_3}^2$. 
Hence there is a homotopy $D:f_4\circ \Sigma^{m_4}\overline{f_3}^2\simeq *$ such 
that $\overline{f_4}'=[f_4,D,\Sigma^{m_4}\overline{f_3}^2]$. 
We have
\begin{align*}
&\overline{f_4}\circ\widetilde{\Sigma}^{m_4}g_{4,3}
=\overline{f_4}\circ \Sigma^{m_4}g_{4,3}\circ(1_{\Sigma^{m_{[3,1]}}X_1}\wedge\tau(\s^{m_4},\s^2))\\
&\simeq \overline{f_4}'\circ(\psi^{m_4}_{\overline{f_3}^2})^{-1}\circ \Sigma^{m_4}(\overline{f_3}^2,B,\widetilde{\Sigma}^{m_3}g_{3,2})\circ(1_{\Sigma^{m_{[3,1]}}X_1}\wedge\tau(\s^{m_4},\s^2))\\
&=\overline{f_4}'\circ(\psi^{m_4}_{\overline{f_3}^2})^{-1}\circ(\psi^{m_4}_{\overline{f_3}^2})
\circ(\Sigma^{m_4}\overline{f_3}^2,\widetilde{\Sigma}^{m_4}B,\Sigma^{m_4}\widetilde{\Sigma}^{m_3}g_{3,2})\\
&\hspace{1cm}\circ 
(1_{\Sigma\Sigma^{m_{[3,1]}}X_1}\wedge\tau(\s^1,\s^{m_4})\circ(1_{\Sigma^{m_{[3,1]}}X_1}\wedge\tau(\s^{m-4},\s^2))\quad(\text{by \cite[Lemma 2.4]{OO1}})\\
&=\overline{f_4}'\circ(\Sigma^{m_4}\overline{f_3}^2,\widetilde{\Sigma}^{m_4}B,\Sigma^{m_4}\widetilde{\Sigma}^{m_3}g_{3,2})\circ \Sigma(1_{\Sigma^{m_{[3,1]}}X_1}\wedge\tau(\s^{m_4},\s^1))\\
&\in\{f_4,\overline{f_3}^2,\widetilde{\Sigma}^{m_3}g_{3,2}\}_{m_4}\circ 
\Sigma(1_{\Sigma^{m_{[3,1]}}X_1}\wedge\tau(\s^{m_4},\s^1))\\
&\overset{\#}{=}\{f_4,{\overline{f_3}^2}',(\psi^{m_3}_{f_2})^{-1}\circ\widetilde{\Sigma}^{m_3}g_{3,2}\}_{m_4}\circ 
\Sigma(1_{\Sigma^{m_{[3,1]}}X_1}\wedge\tau(\s^{m_4},\s^1))\\
&=\{f_4,{\overline{f_3}^2}',(\psi^{m_3}_{f_2})^{-1}\circ \Sigma^{m_3}g_{3,2}\circ (1_{\Sigma^{m_{[2,1]}}X_1}\wedge\tau(\s^{m_3},\s^1))\}_{m_4}
\circ \Sigma(1_{\Sigma^{m_{[3,1]}}X_1}\wedge\tau(\s^{m_4},\s^1))\\
&=\{f_4,{\overline{f_3}^2}',(\psi^{m_3}_{f_2})^{-1}\circ \Sigma^{m_3}(f_2,A_1,\Sigma^{m_2}f_1)\circ (1_{\Sigma^{m_{[2,1]}}X_1}\wedge\tau(\s^{m_3},\s^1))\}_{m_4}\\
&\hspace{6cm}\circ \Sigma(1_{\Sigma^{m_{[3,1]}}X_1}\wedge\tau(\s^{m_4},\s^1))\\
&=\{f_4,[f_3,A_2,\Sigma^{m_3}f_2],(\Sigma^{m_3}f_2,\widetilde{\Sigma}^{m_3}A_1,\Sigma^{m_{[3,2]}}f_1)\}_{m_4}
\circ \Sigma(1_{\Sigma^{m_{[3,1]}}X_1}\wedge\tau(\s^{m_4},\s^1))\\
&\hspace{6cm} (\text{by \cite[Lemma 2.4]{OO1}}),
\end{align*}
where the equality $\overset{\#}{=}$ holds because of the following reasons: 
$\overline{f_3}^2={\overline{f_3}^2}'\circ(\psi^{m_3}_{f_2})^{-1}$; (8.3); $\supset$ holds 
by Proposition 1.2(ii) of \cite{T}; indeterminacies of two brackets are the same.  
This proves (8.1).
\end{proof}

In order to prove (8.2) we need Lemma 8.3 below. 
We omit its proof, because it is an easy consequence of definitions. 

\begin{lemma}
If maps $X\overset{f}{\to}Y\overset{g}{\to}Z$ with a null homotopy 
$H:g\circ f\simeq *$ and a homeomorphism $h:Y\approx Y'$ 
are given, then we have 
\begin{gather}
[g,H,f]=[g\circ h^{-1},H,h\circ f]\circ(h\cup C1_X):Y\cup_f CX\to Z,\nonumber\\ 
(1_Z\cup Ch)\circ (g,H,f)=(g\circ h^{-1},H,h\circ f):EX\to Z\cup_{g\circ h^{-1}}CY'.
\end{gather}
\end{lemma}

\begin{proof}[Proof of (8.2)]
Suppose that $(\vec{\bm f};A_3,A_2,A_1)$ is admissible. 
Then 
\begin{align*}
&f_4\circ \Sigma^{m_4}[f_3,A_2,\Sigma^{m_3}f_2]\\
&=f_4\circ[\Sigma^{m_4}f_3,\widetilde{\Sigma}^{m_4}A_2, \Sigma^{m_{[4,3]}}f_2]
\circ(\psi^{m_4}_{\Sigma^{m_3}f_2})^{-1}\quad(\text{by \cite[Lemma 2.4]{OO1}})\\
&=[f_4,A_3,\Sigma^{m_4}f_3]\circ i_{\Sigma^{m_4}f_3}\circ [\Sigma^{m_4}f_3,\widetilde{\Sigma}^{m_4}A_2, \Sigma^{m_{[4,3]}}f_2]\circ(\psi^{m_4}_{\Sigma^{m_3}f_2})^{-1}\\
&\simeq [f_4,A_3,\Sigma^{m_4}f_3]\circ (\Sigma^{m_4}f_3,\widetilde{\Sigma}^{m_4}A_2, \Sigma^{m_{[4,3]}}f_2)
\circ q_{\Sigma^{m_{[4,3]}}f_2}\circ (\psi^{m_4}_{\Sigma^{m_3}f_2})^{-1}\\
&\hspace{6cm} (\text{by \cite[Proposition 5.11]{Og} or \cite[Lemma 3.6]{OO1}})\\
&\simeq *
\end{align*}
and so $\{f_4,[f_3,A_2,\Sigma^{m_3}f_2],(\Sigma^{m_3}f_2,\widetilde{\Sigma}^{m_3}A_1,\Sigma^{m_{[3,2]}}f_1)\}_{m_4}$ is not empty. 
Let $\alpha$ be any element of the last Toda bracket. 
Then there are null homotopies $F:f_4\circ \Sigma^{m_4}[f_3,A_2,\Sigma^{m_3}f_2]\simeq *$ 
and 
$B:[f_3,A_2,\Sigma^{m_3}f_2]\circ(\Sigma^{m_3}f_2,\widetilde{\Sigma}^{m_3}A_1,\Sigma^{m_{[3,2]}}f_1)\simeq *$ 
such that 
$\alpha=[f_4,F,\Sigma^{m_4}[f_3,A_2,\Sigma^{m_3}f_2]]\circ(\Sigma^{m_4}[f_3,A_2,\Sigma^{m_3}f_2],\widetilde{\Sigma}^{m_4}B,\Sigma^{m_4}(\Sigma^{m_3}f_2,\widetilde{\Sigma}^{m_3}A_1,\Sigma^{m_{[3,2]}}f_1))$. 
We will prove 
$$\alpha\circ \Sigma(1_{\Sigma^{m_{[3,1]}}X_1}\wedge\tau(\s^{m_4},\s^1))\in\{\vec{\bm f}\}^{(\ddot{s}_t)}_{\vec{\bm m}}. 
$$
Set 
\begin{gather*}
\mathscr{S}_2=(\Sigma^{m_1}X_1;X_2,X_2\cup_{f_1}C\Sigma^{m_1}X_1;f_1;i_{f_1}),\ 
a_{2,1}=1_{C_{2,2}},\ \mathscr{A}_2=\{a_{2,1}\},\ \omega_{2,1}=q'_{f_1},\\
\overline{f_2}'=[f_2,A_1,\Sigma^{m_2}f_1]:\Sigma^{m_2}X_2\cup_{\Sigma^{m_2}f_1}C\Sigma^{m_{[2,1]}}X_1\to X_3.
\end{gather*}
Then $\widetilde{\Sigma}^{m_2}a_{2,1}=(\psi^{m_2}_{f_1})^{-1}:\Sigma^{m_2}C_{2,2}\to \Sigma^{m_2}X_2\cup_{\Sigma^{m_2}f_1}C\Sigma^{m_{[2,1]}}X_1$ by definition. 
Set 
\begin{gather*}
\overline{f_2}=\overline{f_2}'\circ(\psi^{m_2}_{f_1})^{-1}=\overline{f_2}'\circ\widetilde{\Sigma}^{m_2}a_{2,1}:\Sigma^{m_2}C_{2,2}\to X_3,\\
\mathscr{S}_3=(\widetilde{\Sigma}^{m_2}\mathscr{S}_2)(\overline{f_2},\widetilde{\Sigma}^{m_2}\mathscr{A}_2),\ a_{3,1}=1_{C_{3,2}}.
\end{gather*}
Then $g_{3,2}\simeq (f_2,A_1,\Sigma^{m_2}f_1)$ by \cite[(4.2)]{OO}. 
Since $j_{3,2}:C_{3,2}=X_3\cup_{f_2}C\Sigma^{m_2}X_2\to C_{3,3}=X_3\cup_{\overline{f_2}}C\Sigma^{m_2}C_{2,2}$ is a homotopy cofibre of $g_{3,2}$, it follows from \cite[Lemma 4.3(2)]{OO} that $j_{3,2}$ is a homotopy cofibre of $(f_2,A_1,\Sigma^{m_2}f_1)$ so that there is a homotopy equivalence 
$a_{3,2}\in\mathrm{TOP}^{C_{3,2}}(j_{3,2},i_{(f_2,A_1,\Sigma^{m_2}f_1)})$ such that 
$\mathscr{A}_3=\{a_{3,1}, a_{3,2}\}$ is a structure on $\mathscr{S}_3$. 
We may suppose/take $g_{3,2}=(f_2,A_1,\Sigma^{m_2}f_1)$ by \cite[Lemma 4.3(2), Remark 5.5(1)]{OO}. 
We have
\begin{align*}
\widetilde{\Sigma}^{m_3}g_{3,2}&=\Sigma^{m_3}g_{3,2}\circ(1_{\Sigma^{m_{[2,1]}}X_1}\wedge\tau(\s^{m_3},\s^1))\\
&= \Sigma^{m_3}(f_2,A_1,\Sigma^{m_2}f_1)\circ(1_{\Sigma^{m_{[2,1]}}X_1}\wedge\tau(\s^{m_3},\s^1))\\
&=\psi^{m_3}_{f_2}\circ (\Sigma^{m_3}f_2,\widetilde{\Sigma}^{m_3}A_1,\Sigma^{m_{[3,2]}}f_1)\quad 
(\text{by \cite[Lemma 2.4]{OO1}}).
\end{align*}
Set
\begin{gather*}
{\overline{f_3}^2}'=[f_3,A_2,\Sigma^{m_3}f_2]
:\Sigma^{m_3}X_3\cup_{\Sigma^{m_3}f_2}C\Sigma^{m_{[3,2]}}X_2\to X_4,\\
\overline{f_3}^2=[f_3,A_2,\Sigma^{m_3}f_2]\circ
(\psi^{m_3}_{f_2})^{-1}:\Sigma^{m_3}C_{3,2}\to X_4.
\end{gather*}
We have 
\begin{align*}
\overline{f_3}^2\circ\widetilde{\Sigma}^{m_3}g_{3,2}&={\overline{f_3}^2}'\circ(\psi^{m_3}_{f_2})^{-1}\circ\widetilde{\Sigma}^{m_3}g_{3,2}\\
&= [f_3,A_2,\Sigma^{m_3}f_2]\circ (\Sigma^{m_3}f_2,\widetilde{\Sigma}^{m_3}A_1,\Sigma^{m_{[3,2]}}f_1)\\
&\simeq *.
\end{align*}
Take $B:\overline{f_3}^2\circ\widetilde{\Sigma}^{m_3}g_{3,2}\simeq *$ and set 
$$
\overline{f_3}=[\overline{f_3}^2,B,\widetilde{\Sigma}^{m_3}g_{3,2}]:\Sigma^{m_3}C_{3,3}\to X_4.
$$
Then $\overline{f_3}\circ \Sigma^{m_3}j_{3,2}=\overline{f_3}^2$. 
Set 
$$
\mathscr{S}_4=(\widetilde{\Sigma}^{m_3}\mathscr{S}_3)(\overline{f_3},\widetilde{\Sigma}^{m_3}\mathscr{A}_3),\quad a_{4,1}=1_{C_{4,2}}.
$$ 
Then $g_{4,2}\simeq (f_3,A_2,\Sigma^{m_3}f_2)$ and $g_{4,3}\simeq (\overline{f_3}^2,B,\widetilde{\Sigma}^{m_3}g_{3,2})$ by \cite[(4.2)]{OO1}. 
Since $j_{4,2}$ is a homotopy cofibre of $g_{4,2}$, it follows from \cite[Lemma 4.3(2)]{OO} that $j_{4,2}$ is a homotopy cofibre of $(f_3,A_2,\Sigma^{m_3}f_2)$ so that there is a homotopy equivalence $a_{4,2}\in\mathrm{TOP}^{C_{4,2}}(j_{4,2},i_{(f_3,A_2,\Sigma^{m_3}f_2)})$ 
and we may suppose/take $g_{4,2}=(f_3,A_2,\Sigma^{m_3}f_2)$ by \cite[Lemma 4.3(2), Remark 5.5(1)]{OO}. 
Let $a_{4,3}:C_{4,4}\to C_{4,3}\cup_{g_{4,3}}C\Sigma^2\Sigma^{m_{[3,1]}}X_1$ be any homotopy equivalence such that $a_{4,3}\circ j_{4,3}=i_{g_{4,3}}$. 
Then the set $\mathscr{A}_4=\{a_{4,i}\,|\,i=1,2,3\}$ is a structure on $\mathscr{S}_4$.  
We have
\begin{align*}
f_4\circ \Sigma^{m_4}\overline{f_3}^2&=[f_4,A_3,\Sigma^{m_4}f_3]\circ i_{\Sigma^{m_4}f_3}\circ \Sigma^{m_4}[f_3,A_2,\Sigma^{m_3}f_2]\\
&=[f_4,A_3,\Sigma^{m_4}f_3]\circ i_{\Sigma^{m_4}f_3}\circ [\Sigma^{m_4}f_3,\widetilde{\Sigma}^{m_4}A_2,\Sigma^{m_{[4,3]}}f_2]\circ(\psi^{m_4}_{\Sigma^{m_3}f_2})^{-1}\\
&\simeq [f_4,A_3,\Sigma^{m_4}f_3]\circ (\Sigma^{m_4}f_3,\widetilde{\Sigma}^{m_4}A_2,\Sigma^{m_{[4,3]}}f_2)
\circ(\psi^{m_4}_{\Sigma^{m_3}f_2})^{-1}\\
&\simeq *.
\end{align*}
Then the map
$$
\overline{f_4}=[f_4,F,\Sigma^{m_4}[f_3,A_2,\Sigma^{m_3}f_2]]\circ (1_{\Sigma^{m_4}X_4}\cup C\Sigma^{m_4}(\psi^{m_3}_{f_2})^{-1})\circ(\psi^{m_4}_{\overline{f_3}^2})^{-1}:\Sigma^{m_4}C_{4,3}\to X_5
$$
is an extension of $f_4$ to $\Sigma^{m_4}C_{4,3}$. 
Hence we have an $\ddot{s}_t$-presentation 
$\{\mathscr{S}_r,\overline{f_r},\mathscr{A}_r\,|\,2\le r\le 4\}$ of $\vec{\bm f}$. 
We have 
$$
(\overline{f_3}^2,B,\widetilde{\Sigma}^{m_3}g_{3,2})=(1_{X_4}\cup C\psi^{m_3}_{f_2})
\circ({\overline{f_3}^2}',B,(\Sigma^{m_3}f_2,\widetilde{m_3}A_1,\Sigma^{m_{[3,2]}}f_1))
$$
by (8.4). 
We have 
\begin{align*}
&\overline{f_4}\circ\widetilde{\Sigma}^{m_4}g_{4,3}\\
&=[f_4,F,\Sigma^{m_4}[f_3,A_2,\Sigma^{m_3}f_2]]\circ (1_{\Sigma^{m_4}X_4}\cup C\Sigma^{m_4}(\psi^{m_3}_{f_2})^{-1})\circ(\psi^{m_4}_{\overline{f_3}^2})^{-1}\\
&\hspace{1cm}\circ \Sigma^{m_4}g_{4,3}\circ(1_{\Sigma^{m_{[3,1]}}X_1}\wedge\tau(\s^{m_4},\s^2))\\
&\simeq [f_4,F,\Sigma^{m_4}[f_3,A_2,\Sigma^{m_3}f_2]]\circ (1_{\Sigma^{m_4}X_4}\cup C\Sigma^{m_4}(\psi^{m_3}_{f_2})^{-1})\circ(\psi^{m_4}_{\overline{f_3}^2})^{-1}\\
&\hspace{1cm}\circ \Sigma^{m_4}(\overline{f_3}^2,B,\widetilde{\Sigma}^{m_3}g_{3,2})\circ 
(1_{\Sigma^{m_{[3,1]}}X_1}\wedge\tau(\s^{m_4},\s^2))\\
&=[f_4,F,\Sigma^{m_4}[f_3,A_2,\Sigma^{m_3}f_2]]\circ (1_{\Sigma^{m_4}X_4}\cup C\Sigma^{m_4}(\psi^{m_3}_{f_2})^{-1})\circ(\psi^{m_4}_{\overline{f_3}^2})^{-1}\\
&\hspace{1cm}\circ \Sigma^{m_4}(1_{X_4}\cup C\psi^{m_3}_{f_2})\circ \Sigma^{m_4}
({\overline{f_3}^2}',B,(\Sigma^{m_3}f_2,\widetilde{\Sigma}^{m_3}A_1,\Sigma^{m_{[3,2]}}f_1))\\
&\hspace{1cm}\circ(1_{\Sigma^{m_{[3,1]}}X_1}\wedge\tau(\s^{m_4},\s^2))\\
&=[f_4,F,\Sigma^{m_4}[f_3,A_2,\Sigma^{m_3}f_2]]\circ (1_{\Sigma^{m_4}X_4}\cup C\Sigma^{m_4}(\psi^{m_3}_{f_2})^{-1})\circ(\psi^{m_4}_{\overline{f_3}^2})^{-1}\\
&\hspace{1cm}\circ \Sigma^{m_4}(1_{X_4}\cup C\psi^{m_3}_{f_2})
\circ\psi^{m_4}_{{\overline{f_3}^2}'}\\
&\hspace{1cm}\circ(\Sigma^{m_4}{\overline{f_3}^2}',\widetilde{\Sigma}^{m_4}B,\Sigma^{m_4}(\Sigma^{m_3}f_2,\widetilde{\Sigma}^{m_3}A_1,\Sigma^{m_{[3,2]}}f_1))\circ(1_{\Sigma\Sigma^{m_{[3,1]}}X_1}\wedge\tau(\s^1,\s^{m_4}))
\\
&\hspace{1cm}\circ (1_{\Sigma^{m_{[3,1]}}X_1}\wedge\tau(\s^{m_4},\s^2))\\
&=[f_4,F,\Sigma^{m_4}[f_3,A_2,\Sigma^{m_3}f_2]]\circ 
(\Sigma^{m_4}{\overline{f_3}^2}',\widetilde{\Sigma}^{m_4}B,\Sigma^{m_4}(\Sigma^{m_3}f_2,\widetilde{\Sigma}^{m_3}A_1,\Sigma^{m_{[3,2]}}f_1))\\
&\hspace{1cm}\circ \Sigma(1_{\Sigma^{m_{[3,1]}}X_1}\wedge\tau(\s^{m_4},\s^1)).
\end{align*}
This completes the proof of (8.2).
\end{proof}

\begin{proof}[Proof of Corollary 8.2]
Take $A_0:f_1\circ \Sigma^{m_1}f_0\simeq *$ arbitrarily. 
Then $ q_{f_1}\circ(f_1,A_0,\Sigma^{m_1}f_0)\simeq -\Sigma\Sigma^{m_1}f_0$ by definitions, and so 
$\Sigma\Sigma^{m_1}f_0\simeq q_{f_1}\circ (-(f_1,A_0,\Sigma^{m_1}f_0))$. 
By Theorem~8.1 we have
\begin{align*}
&\{f_4,f_3,f_2,f_1\}^{(\ddot{s}_t)}_{\vec{\bm m}}\circ \Sigma(1_{\Sigma^{m_1}X_1}\wedge\tau(\s^1,\s^{m_{[4,2]}}))\circ \Sigma\Sigma^{m_{[4,2]}}\Sigma\Sigma^{m_1} f_0\\
&=\bigcup \{f_4,[f_3,A_2,\Sigma^{m_3}f_2],(\Sigma^{m_3}f_2,\widetilde{\Sigma}^{m_3}A_1,\Sigma^{m_{[3,2]}}f_1)\}_{m_4}\circ \Sigma(1_{\Sigma^{m_{[3,1]}}X_1}\wedge\tau(\s^{m_4},\s^1))\\
&\hspace{2cm}\circ \Sigma(1_{\Sigma^{m_1}X_1}\wedge\tau(\s^1,\s^{m_{[4,2]}}))\circ \Sigma\Sigma^{m_{[4,2]}}\Sigma\Sigma^{m_1} f_0
\end{align*}
where $\bigcup$ is taken over all $(A_3,A_2,A_1)$ such that $(f_4,f_3,f_2,f_1;A_3,A_2,A_1)$ is  admissible. 
Since 
\begin{align*}
&\Sigma(1_{\Sigma^{m_{[3,1]}}X_1}\wedge\tau(\s^{m_4},\s^1))
\circ \Sigma(1_{\Sigma^{m_1}X_1}\wedge\tau(\s^1,\s^{m_{[4,2]}}))\circ \Sigma\Sigma^{m_{[4,2]}}\Sigma\Sigma^{m_1} f_0\\
&\simeq \Sigma\Sigma^{m_4}\big(q_{\Sigma^{m_{[3,2]}}f_1}\circ(\psi^{m_{[3,2]}}_{f_1})^{-1}\big)\circ 
\Sigma\Sigma^{m_{[4,2]}}(-(f_1,A_0,\Sigma^{m_1}f_0))
\end{align*}
we have
\begin{align*}
&\{f_4,[f_3,A_2,\Sigma^{m_3}f_2],(\Sigma^{m_3}f_2,\widetilde{\Sigma}^{m_3}A_1,\Sigma^{m_{[3,2]}}f_1)\}_{m_4}\circ \Sigma(1_{\Sigma^{m_{[3,1]}}X_1}\wedge\tau(\s^{m_4},\s^1))\\
&\hspace{2cm}\circ \Sigma(1_{\Sigma^{m_1}X_1}\wedge\tau(\s^1,\s^{m_{[4,2]}}))\circ \Sigma\Sigma^{m_{[4,2]}}\Sigma\Sigma^{m_1} f_0\\
&=\{f_4,[f_3,A_2,\Sigma^{m_3}f_2],(\Sigma^{m_3}f_2,\widetilde{\Sigma}^{m_3}A_1,\Sigma^{m_{[3,2]}}f_1)\}_{m_4}\\
&\hspace{2cm}\circ \Sigma\Sigma^{m_4}\big(q_{\Sigma^{m_{[3,2]}}f_1}\circ(\psi^{m_{[3,2]}}_{f_1})^{-1}\big)\circ 
\Sigma\Sigma^{m_{[4,2]}}(-(f_1,A_0,\Sigma^{m_1}f_0))\\
&\subset \{f_4,[f_3,A_2,\Sigma^{m_3}f_2],(\Sigma^{m_3}f_2,\widetilde{\Sigma}^{m_3}A_1,\Sigma^{m_{[3,2]}}f_1)
\circ q_{\Sigma^{m_{[3,2]}}f_1}\}_{m_4}\\
&\hspace{2cm}\circ \Sigma\Sigma^{m_4}(\psi^{m_{[3,2]}}_{f_1})^{-1}\circ \Sigma\Sigma^{m_{[4,2]}}(-(f_1,A_0,\Sigma^{m_1}f_0))\\
&\hspace{3cm} (\text{by \cite[Proposition 1.2(i)]{T}})\\
&=\{f_4,[f_3,A_2,\Sigma^{m_3}f_2],i_{\Sigma^{m_3}f_2}\circ[\Sigma^{m_3}f_2,\widetilde{\Sigma}^{m_3}A_1,\Sigma^{m_{[3,2]}}f_1]\}_{m_4}\\
&\hspace{1cm}\circ \Sigma\Sigma^{m_4}\big((\psi^{m_{[3,2]}}_{f_1})^{-1}\circ \Sigma^{m_{[3,2]}}(f_1,A_0,\Sigma^{m_1}f_0)\big)\circ \Sigma\Sigma^{m_{[4,2]}}(-1_{\Sigma\Sigma^{m_{[1,0]}}X_0})\\
&\hspace{3cm}(\text{by \cite[Proposition(5.11)]{Og} or \cite[Lemma 3.6]{OO1}})\\
&\subset\{f_4,[f_3,A_2,\Sigma^{m_3}f_2]\circ i_{\Sigma^{m_3}f_2},[\Sigma^{m_3}f_2,\widetilde{\Sigma}^{m_3}A_1,\Sigma^{m_{[3,2]}}f_1]\}_{m_4}\\
&\hspace{1cm}\circ \Sigma\Sigma^{m_4}\big((\Sigma^{m_{[3,2]}}f_1,\widetilde{\Sigma}^{m_{[3,2]}}A_0,\Sigma^{m_{[3,1]}}f_0)\circ (1_{\Sigma^{m_{[1,0]}}X_0}\wedge\tau(\s^1,\s^{m_{[3,2]}})\big)\\
&\hspace{1cm}\circ \Sigma\Sigma^{m_{[4,2]}}(-1_{\Sigma\Sigma^{m_{[1,0]}}X_0})\\
&\hspace{3cm}(\text{by \cite[Proposition 1.2(ii)]{T} and \cite[Lemma 2.4]{OO1}})\\
&=\{f_4,f_3,[\Sigma^{m_3}f_2,\widetilde{\Sigma}^{m_3}A_1,\Sigma^{m_{[3,2]}}f_1]\}_{m_4}
\circ \Sigma\Sigma^{m_4}(\Sigma^{m_{[3,2]}}f_1,\widetilde{\Sigma}^{m_{[3,2]}}A_0,\Sigma^{m_{[3,1]}}f_0)\\
&\hspace{1cm}\circ \Sigma\Sigma^{m_4}\big((1_{\Sigma^{m_{[1,0]}}X_0}\wedge\tau(\s^1,\s^{m_{[3,2]}}))\circ \Sigma^{m_{[3,2]}}(-1_{\Sigma\Sigma^{m_{[1,0]}}X_0})\big)\\
&=(-1)^{m_4+1}\Big(f_4\circ \Sigma^{m_4}\{f_3,[\Sigma^{m_3}f_2,\widetilde{\Sigma}^{m_3}A_1,\Sigma^{m_{[3,2]}}f_1],(\Sigma^{m_{[3,2]}}f_1,\widetilde{\Sigma}^{m_{[3,2]}}A_0,\Sigma^{m_{[3,1]}}f_0)\}\Big)\\
&\hspace{1cm}\circ \Sigma\Sigma^{m_4}\big((1_{\Sigma^{m_{[1,0]}}X_0}\wedge\tau(\s^1,\s^{m_{[3,2]}}))\circ \Sigma^{m_{[3,2]}}(-1_{\Sigma\Sigma^{m_{[1,0]}}X_0})\big)\\
&(\text{by the assumption $\{f_2,f_1,f_0\}^{(\ddot{s}_t)}_{(m_2,m_1,m_0)}=\{0\}$, \cite[Lemma 2.4]{OO1}, and \cite[Proposition 1.4]{T}})\\
&=(-1)^{m_4+m_3+m_2}\Big(f_4\circ \Sigma^{m_4}\{f_3,[\Sigma^{m_3}f_2,\widetilde{\Sigma}^{m_3}A_1,\Sigma^{m_{[3,2]}}f_1],(\Sigma^{m_{[3,2]}}f_1,\widetilde{\Sigma}^{m_{[3,2]}}A_0,\Sigma^{m_{[3,1]}}f_0)\}\Big)\\
&\hspace{2cm}(\text{the equality holds under suitable identifications of respective spaces})\\
&\subset(-1)^{m_4+m_3+m_2}\big(f_4\circ \Sigma^{m_4}\{f_3,\Sigma^{m_3}f_2,\Sigma^{m_{[3,2]}}f_1,\Sigma^{m_{[3,1]}}f_0\}^{(\ddot{s}_t)}\big),
\end{align*} 
where the last containment follows from Theorem 8.1 and the fact that 
$(f_3,f_2,f_1,f_0;A_2,A_1,A_0)$ is admissible. 
Hence
\begin{align*}
&\{\vec{\bm f}\,\}^{(\ddot{s}_t)}_{\vec{\bm m}}\circ \Sigma(1_{\Sigma^{m_1}X_1}\wedge\tau(\s^1,\s^{m_{[4,2]}}))\circ \Sigma\Sigma^{m_{[4,2]}}\Sigma\Sigma^{m_1}f_0\\
&\qquad\subset(-1)^{m_4+m_3+m_2}\big(f_4\circ \Sigma^{m_4}\{f_3,\Sigma^{m_3}f_2,\Sigma^{m_{[3,2]}}f_1,\Sigma^{m_{[3,1]}}f_0\}^{(\ddot{s}_t)}\big).
\end{align*}
By the usual identifications of spaces \cite[Section 2]{OO}, the map $\Sigma(1_{\Sigma^{m_1}X_1}\wedge\tau(\s^1,\s^{m_{[4,2]}}))$ is a self homeomorphism of $\Sigma^{|\vec{\bm m}|+2}X_1$ of the degree $(-1)^{m_4+m_3+m_2}$. 
Hence, if we denote the map $\Sigma\Sigma^{m_{[4,2]}}\Sigma\Sigma^{m_1}f_0$ by $\Sigma^{|\vec{\bm m}|+2}f_0$, then we have 
$$
\{\vec{\bm f}\}^{(\ddot{s}_t)}_{\vec{\bm m}}\circ \Sigma^{|\vec{\bm m}|+2}f_0
\subset f_4\circ \Sigma^{m_4}\{f_3,\Sigma^{m_3}f_2,\Sigma^{m_{[3,2]}}f_1, \Sigma^{m_{[3,1]}}f_0\}^{(\ddot{s}_t)}.
$$
This completes the proof of Corollary 8.2. 
\end{proof}

\section{Non-empty brackets}
Similar assertions to the following proposition can be seen in \cite{M} (for $n=4$) 
and \cite{W} (for the non subscripted case). 

\begin{prop}
\begin{enumerate}
\item If $\{f_{n-1},\dots,f_1\}^{(q)}_{(m_{n-1},\dots,m_1)}\ni 0$ and $\{f_n,\dots,f_k\}^{(aq\ddot{s}_2)}_{(m_n,\dots,m_k)}=\{0\}$ for all $k$ with $2\le k<n$, then 
$\{\vec{\bm f}\}^{(\star)}_{\vec{\bm m}}$ is not empty for all $\star$.
\item If $\{f_n,\dots,f_2\}^{(q)}_{(m_n,\dots,m_2)}\ni 0$ and $\{f_k,\dots,f_1\}^{(aq\ddot{s}_2)}_{(m_k,\dots,m_1)}=\{0\}$ for all $k$ with $2\le k<n$, then 
$\{\vec{\bm f}\}^{(\star)}_{\vec{\bm m}}$ is not empty for all $\star$.
\end{enumerate}
\end{prop}
\begin{proof}
A proof can be obtained easily by modifying the proof of \cite[(1.10)]{OO}. 
We omit the details. 
\end{proof}

\section{Subscripted stable higher Toda brackets $\langle\theta_n,\dots,\theta_1\rangle^{(\star)}_{\vec{\bm m}}$}
The purpose of this section is to define the {\it subscripted stable higher Toda bracket}  and prove a stable analogy to Proposition 3.8 (i.e.(1.1)). 

For any space $X$ and any integer $k\ge 1$, we regard $\Sigma^kX$ as 
a co-H-space with a comultiplication 
$\Sigma^kX=X\wedge \s^k\overset{1_X\wedge\mu}{\longrightarrow}X\wedge (\s^k\vee\s^k)=\Sigma^kX\vee \Sigma^kX$, 
where $\mu$ is a co-H-multiplication on $\s^k$. 
Note that $\mu$ is unique up to homotopy 
for $k\ge 2$ (cf. Proposition 2 and Remark 3 of \cite{AG}). 
We denote by $\langle X,Y\rangle$ the group of all stable maps 
$X\to Y$, that is, the limit of the sequence
$$
\begin{CD}
[X,Y]@>\Sigma>>[\Sigma X,\Sigma Y]@>\Sigma>>[\Sigma^2X,\Sigma^2Y]@>\Sigma>>\cdots .
\end{CD}
$$

Let $\vec{\bm \theta}=(\theta_n,\dots,\theta_1)$ be a sequence of stable elements 
$\theta_i\in\langle \Sigma^{m_i}X_i,X_{i+1}\rangle$ $(1\le i\le n)$. 
A {\it representative} of $\vec{\bm \theta}$ is a sequence 
$\overrightarrow{\bm f^\ell}=(f^\ell_n,\dots,f^\ell_1)$ 
of maps $f^\ell_i:\Sigma^\ell \Sigma^{m_i}X_i\to \Sigma^\ell X_{i+1}$ 
such that $f^\ell_i$ represents $\theta_i$ for all $i$. 
We denote by $\mathrm{Rep}(\vec{\bm \theta})$ the set of representatives of $\vec{\bm \theta}$. 

Given $\vec{\bm \theta}$, we will define $\langle \Sigma^{\vec{\bm m}}\vec{\bm \theta}\,\rangle^{(\star)}$ and 
$\langle\vec{\bm \theta}\,\rangle^{(\star)}_{\vec{\bm m}}$ which satisfy
\begin{equation}
\langle \Sigma^{n-2}\Sigma^{m_{[n,1]}}X_1,X_{n+1}\rangle\supset\langle \Sigma^{\vec{\bm m}}\vec{\bm \theta}\,\rangle^{(\star)}\supset \langle\vec{\bm \theta}\,\rangle^{(\star)}_{\vec{\bm m}},
\end{equation}
where the second containment is a stable analogy to Proposition 3.8 
and it is identity when $\vec{\bm m}=(0,\dots,0)$. 

Let $\overrightarrow{\bm f^{\ell}}\in
\mathrm{Rep}(\vec{\bm \theta})$. 
We set 
$$
f_i^{(\ell)}=f_i^\ell\circ(1_{X_i}\wedge\tau(\s^\ell,\s^{m_i})):\Sigma^{m_i}\Sigma^\ell X_i\to \Sigma^\ell X_{i+1}.
$$
By definitions, we have 
\begin{align*}
\widetilde{\Sigma}^r f^{(\ell)}_i&=
\Sigma^r f^{(\ell)}_i\circ(1_{\Sigma^\ell X_i}\wedge\tau(\s^r,\s^{m_i}))\\
&=\Sigma^r f^\ell_i \circ (1_{X_i}\wedge\tau(\s^\ell\wedge\s^r,\s^{m_i})):\Sigma^{m_i}\Sigma^r\Sigma^\ell X_i\to \Sigma^r\Sigma^\ell X_{i+1},
\end{align*}
and $\widetilde{\Sigma}^s\widetilde{\Sigma}^rf^{(\ell)}_i=\widetilde{\Sigma}^{s+r}f^{(\ell)}_i$ under the identification $\Sigma^s\Sigma^r=\Sigma^{s+r}$, that is, $\s^r\wedge\s^s=\s^{s+r}$ (see \cite[\S2]{OO}). 
We set 
$$\widetilde{\Sigma}^r\overrightarrow{\bm f^{(\ell)}}=(\widetilde{\Sigma}^rf^{(\ell)}_n,\dots,\widetilde{\Sigma}^rf^{(\ell)}_1).
$$ 
Then $\widetilde{\Sigma}^s\widetilde{\Sigma}^r \overrightarrow{\bm f^{(\ell)}}=\widetilde{\Sigma}^{s+r}\overrightarrow{\bm f^{(\ell)}}$ for $r,s\ge 0$. 
As defined in the introduction we set 
$$
m_{[n,n+1]}=0.
$$ 
We define
\begin{align*}
\widehat{\Sigma}^rf^{(\ell)}_i&=\Sigma^{m_{[n,i+1]}}\widetilde{\Sigma}^r f^{(\ell)}_i:\Sigma^{m_{[n,i]}}\Sigma^r\Sigma^\ell X_i\to \Sigma^{m_{[n,i+1]}}\Sigma^r\Sigma^\ell X_{i+1},\\
\widehat{\Sigma}^r \overrightarrow{{\bm f^{(\ell)}}}&=(\widehat{\Sigma}^rf^{(\ell)}_n,\dots,\widehat{\Sigma}^rf^{(\ell)}_1).
\end{align*}
For $r\ge 0$, we set 
\begin{align*}
\Gamma(r,\ell)&=[\Sigma^r\Sigma^\ell \Sigma^{n-2}\Sigma^{m_{[n,1]}}X_1,\Sigma^r \Sigma^\ell X_{n+1}],\\
A(r,\ell)&=\{\widehat{\Sigma}^r \overrightarrow{{\bm f^{(\ell)}}}\}^{(\star)}\circ(1_{X_1}\wedge\tau(\s^{m_{[n,1]}}\wedge\s^{n-2},\s^\ell\wedge\s^r)),\\
B(r,\ell)&=\{\widetilde{\Sigma}^r\overrightarrow{\bm f^{(\ell)}}\}^{(\star)}_{\vec{\bm m}}
\circ(1_{X_1}\wedge\tau(\s^{m_{[n,1]}}\wedge\s^{n-2},\s^\ell\wedge\s^r)).
\end{align*}

We are going to proceed in the following order. 

\begin{lemma}
\begin{enumerate}
\item $\Gamma(r,\ell)\supset A(r,\ell)\supset B(r,\ell)$. 
\item $\Sigma^s\Gamma(r,\ell)\subset\Gamma(s+r,\ell)$, $\Sigma^s A(r,\ell)\subset A(s+r,\ell)$, and $\Sigma^s B(r,\ell)\subset B(s+r,\ell)$. 
\item $\displaystyle{\lim_{r\to\infty}\Gamma(r,\ell)\supset\lim_{r\to\infty}A(r,\ell)\supset\lim_{r\to\infty}B(r,\ell)}$.
\end{enumerate}
\end{lemma}

\begin{lemma}
Each limit in Lemma 10.1(3) does not depend on $\ell$. 
In particular $\displaystyle{\lim_{r\to\infty}\Gamma(r,\ell)}
=\langle \Sigma^{n-2}\Sigma^{m_{[n,1]}}X_1, X_{n+1}\rangle$. 
\end{lemma}

\begin{defi}
$\langle \Sigma^{\vec{\bm m}}\vec{\bm \theta}\,\rangle^{(\star)}=\lim_{r\to\infty}A(r,\ell)$ and 
$\langle\vec{\bm \theta}\,\rangle^{(\star)}_{\vec{\bm m}}=\lim_{r\to\infty}B(r,\ell)$ 
which is called the {\it subscripted stable higher Toda bracket}. 
\end{defi}

\begin{prop} (10.1) and the last assertion just after (10.1) hold. 
\end{prop}

\begin{proof}[Proof of Lemma 10.1]
(1) We have $\Gamma(r,\ell)\supset A(r,\ell)$ by definitions, 
and $A(r,\ell)\supset B(r,\ell)$ by Proposition 3.8. 

(2) We have 
\begin{align*}
\Sigma^s B(r,\ell)\subset\{\widetilde{\Sigma}^s\widetilde{\Sigma}^r\overrightarrow{\bm f^{(\ell)}}\}^{(\star)}_{\vec{\bm m}}&\circ(1_{\Sigma^r\Sigma^\ell X_1}\wedge\tau(\s^{m_{[n,1]}}\wedge\s^{n-2},\s^s))\\
&\circ \Sigma^s(1_{X_1}\wedge\tau(\s^{m_{[n,1]}}\wedge\s^{n-2},\s^\ell\wedge\s^r)) \quad (\text{by (5.1)})\\
=\{\widetilde{\Sigma}^{s+r}\overrightarrow{\bm f^{(\ell)}}\}^{(\star)}_{\vec{\bm m}}&\circ (1_{X_1}\wedge\tau(\s^{m_{[n,1]}}\wedge\s^{n-2},\s^\ell\wedge\s^r\wedge\s^s))=B(s+r,\ell).
\end{align*}
By \cite[(6.3.1)]{OO}, we have 
\begin{align*}
\Sigma^s A(r,\ell)\subset\{\Sigma^s \widehat{\Sigma}^r\overrightarrow{{\bm f^{(\ell)}}}\}^{(\star)}&\circ(1_{\Sigma^{m_{[n,1]}}\Sigma^r\Sigma^\ell X_1}\wedge\tau(\s^{n-2},\s^s))\\
&\circ \Sigma^s(1_{X_1}\wedge\tau(\s^{m_{[n,1]}}\wedge\s^{n-2},\s^\ell\wedge\s^r)).
\end{align*}
We will prove that the right hand term of the above relation equals to $A(s+r,\ell)$. 
It suffices to prove
\begin{equation}
\{\widehat{\Sigma}^{s+r}\overrightarrow{{\bm f^{(\ell)}}}\}^{(\star)}
=\{\Sigma^s\widehat{\Sigma}^r\overrightarrow{{\bm f^{(\ell)}}}\}^{(\star)}
\circ \Sigma^{n-2}(1_{\Sigma^r\Sigma^\ell X_1}\wedge\tau(\s^s,\s^{m_{[n,1]}}))
\end{equation}
because if (10.2) holds then 
\begin{align*}
A(s+r,\ell)&=\{\widehat{\Sigma}^{s+r}\overrightarrow{{\bm f^{(\ell)}}}\}^{(\star)}
\circ (1_{X_i}\wedge\tau(\s^{m_{[n,1]}}\wedge\s^{n-2},\s^\ell\wedge\s^r\wedge\s^s))\\
&=\{\Sigma^s\widehat{\Sigma}^r\overrightarrow{{\bm f^{(\ell)}}}\}^{(\star)}\circ
 \Sigma^{n-2}(1_{\Sigma^r\Sigma^\ell X_1}\wedge\tau(\s^s,\s^{m_{[n,1]}}))\\
&\qquad\circ(1_{X_1}\wedge\tau(\s^{m_{[n,1]}}\wedge\s^{n-2},\s^\ell\wedge \s^r\wedge\s^s))\\
&=\{\Sigma^s \widehat{\Sigma}^r\overrightarrow{{\bm f^{(\ell)}}}\}^{(\star)}
\circ(1_{\Sigma^{m_{[n,1]}}\Sigma^r\Sigma^\ell X_1}\wedge\tau(\s^{n-2},\s^s))\\
&\qquad\circ \Sigma^s(1_{X_1}\wedge\tau(\s^{m_{[n,1]}}\wedge\s^{n-2},\s^\ell\wedge\s^r)).
\end{align*}
Therefore we will prove (10.2). 
Set $b_i=1_{\Sigma^r\Sigma^\ell X_i}\wedge\tau(\s^{m_{[n,i]}},\s^s)$. 
Then, as is easily seen, the following equality holds: 
$$
\widehat{\Sigma}^{s+r}f^{(\ell)}_i\circ b_i=b_{i+1}\circ \Sigma^s \widehat{\Sigma}^rf^{(\ell)}_i.
$$ 
Hence it follows from Lemma A.2 that 
$$
\{\widehat{\Sigma}^{s+r}\overrightarrow{{\bm f^{(\ell)}}}\}^{(\star)}\circ \Sigma^{n-2}b_1=b_{n+1}\circ\{\Sigma^s\widehat{\Sigma}^r\overrightarrow{{\bm f^{(\ell)}}}\}^{(\star)}
$$
that is, 
$$
\{\widehat{\Sigma}^{s+r}\overrightarrow{{\bm f^{(\ell)}}}\}^{(\star)}\circ 
\Sigma^{n-2}(1_{\Sigma^r\Sigma^\ell X_1}\wedge(\s^{m_{[n,1]}},\s^s))
=\{\Sigma^s\widehat{\Sigma}^r\overrightarrow{{\bm  f^{(\ell)}}}\}^{(\star)}
$$
so that 
$$
\{ \widehat{\Sigma}^{s+r}\overrightarrow{{\bm f^{(\ell)}}}\}^{(\star)}
=\{\Sigma^s\widehat{\Sigma}^r\overrightarrow{{\bm  f^{(\ell)}}} \}^{(\star)}
\circ \Sigma^{n-2}(1_{\Sigma^r\Sigma^\ell X_1}\wedge\tau(\s^s,\s^{m_{[n,1]}})).
$$
This proves (10.2) and ends the proof of (2). 

(3) This follows immediately from (1) and (2).  
\end{proof}

\begin{proof}[Proof of Lemma 10.2]
Obviously $\lim_{r\to\infty}\Gamma(r,\ell)=\langle \Sigma^{n-2}\Sigma^{m_{[n,1]}} X_1,X_{n+1}\rangle$ which does not depend on $\ell$. 
Take $(f^k_n,\dots,f^k_1)\in\mathrm{Rep}(\vec{\bm \theta})$. 
Then there exist natural numbers $L, K$ such that $L+\ell=K+k$ and $\Sigma^Lf^\ell_i\simeq \Sigma^Kf^k_i$ for all $i$, where we used the identification 
$\s^\ell\wedge \s^L=\s^{\ell+L}=\s^{k+K}=\s^k\wedge\s^K$ in 
\cite[\S2]{OO}. In the rest of the proof, we will use freely such identifications. 

Set 
\begin{align*}
h&=\Sigma^{K+r}(1_{X_i}\wedge\tau(\s^{m_i},\s^k))\circ \Sigma^{L+r}(1_{X_i}\wedge\tau(\s^\ell,\s^{m_i}))\\
&: \Sigma^{L+r}\Sigma^{m_i}\Sigma^\ell X_i\to \Sigma^{K+r}\Sigma^{m_i}\Sigma^k X_i
\end{align*}
and consider the following diagram:
$$
\begin{CD}
\Sigma^{m_i}\Sigma^{L+r}\Sigma^\ell X_i @>1_{\Sigma^\ell X_i}\wedge\tau(\s^{L+r},\s^{m_i})>> 
 \Sigma^{L+r}\Sigma^{m_i}\Sigma^\ell X_i @>\Sigma^{L+r}f^{(\ell)}_i>> \Sigma^{L+r}\Sigma^\ell X_{i+1}\\
@| @VV h V @|\\
\Sigma^{m_i}\Sigma^{K+r}\Sigma^k X_i @>>1_{\Sigma^k X_i}\wedge\tau(\s^{K+r},\s^{m_i})> 
\Sigma^{K+r}\Sigma^{m_i}\Sigma^k X_i @>>\Sigma^{K+r}f^{(k)}_i>  \Sigma^{K+r}\Sigma^k X_{i+1}
\end{CD}
$$
The compositions of horizontal two maps are $\widetilde{\Sigma}^{L+r}f^{(\ell)}_i$ and 
$\widetilde{\Sigma}^{K+r}f^{(k)}_i$, respectively. 
The first square is strictly commutative. 
The second square is homotopy commutative, since 
\begin{align*}
\Sigma^{L+r}f^{(\ell)}_i&=\Sigma^{L+r}f^\ell_i\circ \Sigma^{L+r}(1_{X_i}\wedge\tau(\s^\ell,\s^{m_i}))\\
&\simeq \Sigma^{K+r}f^k_i\circ \Sigma^{L+r}(1_{X_i}\wedge\tau(\s^\ell,\s^{m_i}))\\
&=\Sigma^{K+r}f^{(k)}_i\circ \Sigma^{K+r}(1_{X_i}\wedge\tau(\s^{m_i},\s^k))\circ \Sigma^{L+r}(1_{X_i}\wedge\tau(\s^\ell,\s^{m_i}))\\
&=\Sigma^{K+r}f^{(k)}_i\circ h.
\end{align*}
Hence 
\begin{equation}
\widetilde{\Sigma}^{L+r} f^{(\ell)}_i\simeq\widetilde{\Sigma}^{K+r}f^{(k)}_i. 
\end{equation}
Therefore $\{\widetilde{\Sigma}^{L+r}\overrightarrow{\bm f^{(\ell)}}\}^{(\star)}_{\vec{\bm m}}=\{\widetilde{\Sigma}^{K+r}\overrightarrow{\bm f^{(k)}}\}^{(\star)}_{\vec{\bm m}}$ 
by Theorem 6.1 so that $B(L+r,\ell)=B(K+r,k)$. 
Hence $\lim_{r\to\infty}B(r,\ell)$ does not depend on $\ell$. 

By definitions and (10.3), we have
\begin{equation} 
\widehat{\Sigma}^{L+r}f^{(\ell)}_i\simeq \widehat{\Sigma}^{K+r}f^{(k)}_i.
\end{equation}
We have
\begin{align*}
A(L+r,\ell)&=\{\widehat{\Sigma}^{L+r}\overrightarrow{{\bm f^{(\ell)}}}\}^{(\star)}\circ(1_{X_1}\wedge\tau(\s^{m_{[n,1]}}\wedge\s^{n-2},\s^\ell\wedge \s^{L+r}))\\
&=\{\widehat{\Sigma}^{K+r}\overrightarrow{{\bm f^{(k)}}}\}^{(\star)}\circ(1_{X_1}\wedge\tau(\s^{m_{[n,1]}}\wedge\s^{n-2},\s^k\wedge \s^{K+r}))\\
&\hspace{3cm} (\text{by (10.4) and Theorem 6.4.1 of \cite{OO}})\\
&=A(K+r,k).
\end{align*}
Hence $\Sigma^LA(r,\ell)\subset A(L+r,\ell)=A(K+r,k)\supset \Sigma^KA(r,k)$. 
Therefore $\lim_{r\to\infty}A(r,\ell)=\lim_{r\to\infty}A(r,k)$. 
This completes the proof of Lemma 10.2.
\end{proof}

Thus we can set Definition 10.3 here by Lemmas 10.1 and 10.2. 

\begin{proof}[Proof of Proposition 10.4] 
(10.1) follows from Lemma 10.1, Lemma 10.2, and Definition 10.3. 
If $\vec{\bm m}=(0,\dots,0)$ then $A(r,\ell)=B(r,\ell)$ and so 
$\langle \Sigma^{(0,\dots,0)}\vec{\bm \theta}\,\rangle^{(\star)}=\langle\vec{\bm \theta}\,\rangle^{(\star)}_{(0,\dots,0)}$, where $\langle \vec{\bm \theta}\,\rangle^{(\star)}_{(0,\dots,0)}$ is equal to 
$\langle\vec{\bm \theta}\,\rangle^{(\star)}$ which was defined and denoted by $\{\vec{\bm \theta}\,\}^{(\star)}$ in \cite[\S6.9]{OO}. 
This proves the last assertion just after (10.1) and completes the proof of Proposition 10.4. 
\end{proof}

We expect that the stable brackets are unstable under suitable conditions. 

\begin{prob}
If we work in the category of finite CW complexes with vertexes as base points, 
then are the stabilizations
$\Sigma^\infty : A(r,\ell)\to \langle \Sigma^{\vec{\bm m}}\vec{\bm \theta}\,\rangle^{(\star)}$ and 
$\Sigma^\infty : B(r,\ell)\to \langle \vec{\bm \theta}\,\rangle^{(\star)}_{\vec{\bm m}}$ 
bijections for sufficiently large $r$? 
\end{prob}

\begin{prop}
If $\langle \Sigma^{|\vec{\bm m}|+n-2}X_1,X_{n+1}\rangle$ is finite, then 
Problem 10.5 is affirmative.
\end{prop}
\begin{proof}
Set $r_0=|\vec{\bm m}|+n+\dim X_1-(2\cdot\mathrm{conn}X_{n+1}+\ell+2)$, where 
$\mathrm{conn}X_{n+1}$ is the connectivity of $X_{n+1}$. 
Then $\Sigma:\Gamma(r,\ell)\to \Gamma(r+1,\ell)$ is an isomorphism for $r\ge r_0$ 
by the homotopy suspension isomorphism theorem. 

Set $G_r=\Sigma^\infty B(r,\ell)$. 
Then $G_0\subset\cdots\subset G_r\subset G_{r+1}\subset\cdots$ is a sequence of subsets of $\langle \Sigma^{|\vec{\bm m}|+n-2}X_1,X_{n+1}\rangle$. 
It follows from the assumption that the sequence is stable, that is, 
there is an integer $r_0'$ such that $G_r=G_{r_0'}$ for all $r\ge r_0'$. 
Hence $\Sigma:B(r,\ell)\to B(r+1,\ell)$ is a bijection for $r\ge\max\{r_0,r_0'\}$. 
Therefore $\Sigma^\infty:B(r,\ell)\to\langle\vec{\bm \theta}\,\rangle^{(\star)}_{\vec{\bm m}}$ is a bijection for $r\ge \max\{r_0,r_0'\}$. 
Similarly there is an integer $r_0''$ such that $\Sigma^\infty A(r,\ell)=\Sigma^\infty A(r_0'',\ell)$ 
for all $r\ge r_0''$. 
Hence $\Sigma^\infty :A(r,\ell)\to\langle \Sigma^{\vec{\bm m}}\vec{\bm \theta}\,\rangle^{(\star)}$ 
is a bijection for $r\ge\max\{r_0,r_0''\}$. 
This ends the proof.
\end{proof}

The following may be a reason why Toda \cite{T3,T} did not consider 
the subscripted stable $3$-fold bracket. 

\begin{lemma}
Suppose the following conditions. 
\begin{enumerate}
\item $n=3$ and $\star=aq\ddot{s}_2, \ddot{s}_t$. 
\item $X_1$ and $X_3$ are finite CW complexes with vertexes as base points. 
\end{enumerate}
Then $\langle \Sigma^{\overrightarrow{\bm m}}\overrightarrow{\bm \theta}\rangle^{(\star)}
=\langle\overrightarrow{\bm \theta}\rangle_{\overrightarrow{\bm m}}^{(\star)}$. 
\end{lemma}
\begin{proof}
By definitions and Theorem 7.1, it suffices to prove that 
$\{\widehat{\Sigma}^r\overrightarrow{\bm f^{(\ell)}}\}^{(\star)}=\{\widetilde{\Sigma}^r\overrightarrow{\bm f^{(\ell)}}\}^{(\star)}_{\overrightarrow{\bm m}}$  
$$ 
i.e.\ \{\widetilde{\Sigma}^rf^{(\ell)}_3, \Sigma^{m_3}\widetilde{\Sigma}^r f^{(\ell)}_2, \Sigma^{m_{[3,2]}}\widetilde{\Sigma}^rf^{(\ell)}_1\}=\{\widetilde{\Sigma}^rf^{(\ell)}_3, \widetilde{\Sigma}^r f^{(\ell)}_2, \Sigma^{m_2}\widetilde{\Sigma}^rf^{(\ell)}_1\}_{m_3}
$$ for sufficiently large $r$. 
The last equality holds because two brackets have the same indeterminacies for 
sufficiently large $r$, for example, for $r\ge 3+m_2+m_1+\dim X_1-\ell-2\cdot\mathrm{conn}X_3$. 
\end{proof}

\begin{prob}
Is there an example such that $\langle \Sigma^{\overrightarrow{\bm m}}\overrightarrow{\bm \theta}\rangle^{(\star)}
\supsetneqq \langle\overrightarrow{\bm \theta}\rangle_{\overrightarrow{\bm m}}^{(\star)}$ for $n\ge 4$?
\end{prob}

\section{The $p$-fold stable bracket 
$\langle\alpha_1,\dots,\alpha_1\rangle^{(\star)}_{(2p-3,\dots,2p-3)}$}

In this section we frequently use the same notation for a map and its homotopy class. 

Let $p$ be an odd prime. 
Set $\displaystyle{\pi_k(\mathfrak{S})=\lim_{n\to\infty}\pi_{n+k}(\s^n)}$ and let $\pi_k(\mathfrak{S};p)$ be 
the $p$-primary part of $\pi_k(\mathfrak{S})$. 
Let $\alpha_1$ be a generator of $\pi_{2p-3}(\mathfrak{S};p)\cong\bZ_p$ 
(the cyclic group of order $p$). 
Let $\alpha_1(3):\s^{2p}=\Sigma^{2p-3}\s^3\to\s^3$ be a map which represents 
$\alpha_1$. 
Toda (see \cite{T2,T3,T}) showed that $\pi_{2p(p-1)-2}(\mathfrak{S};p)\cong \bZ_p$ and 
\begin{equation}
\pi_i(\mathfrak{S};p)\cong \begin{cases} \bZ_p & i=2j(p-1)-1 \ \text{for }1\le j<p\\
0 & \text{otherwise for }i<2p(p-1)-2\end{cases}.
\end{equation}
Let $\mathscr{P}^k:H^i(X;\bZ_p)\to H^{i+2k(p-1)}(X;\bZ_p)$ denote the Steenrod's 
power operation. 
We use the following Adem relation:
\begin{equation}
(\mathscr{P}^1)^k:=\overbrace{\mathscr{P}^1\cdots\mathscr{P}^1}^k=k!\mathscr{P}^k\ (k\ge 1).
\end{equation}

We will follow the argument of the previous section. 
For $i,k\ge 1$, set 
\begin{gather*}
X_i=\s^3, \ m_i=2p-3, \ \theta_i=\alpha_1\in\langle \Sigma^{m_i}X_i,X_{i+1}\rangle=\pi_{2p-3}(\mathfrak{S}),\\ 
\overrightarrow{\bm m}=(m_p,\dots,m_1),\ \overrightarrow{\bm \alpha_1}=(\theta_p,\dots,\theta_1),\\
\overrightarrow{\bm m}(k)=(m_k,\dots, m_1), \  
\overrightarrow{\bm \alpha_1}(k)=(\theta_k,\dots,\theta_1).
\end{gather*}
Then $(\overbrace{\alpha_1(3),\dots,\alpha_1(3)}^k)\in\mathrm{Rep}(\overrightarrow{\bm \alpha_1}(k))$. 
Therefore we set $\ell=0$ and $f^\ell_i=\alpha_1(3):\Sigma^{m_i}X_i\to X_{i+1}$ for all $i\ge 1$. 
Since $f^{(0)}_i=f^0_i=\alpha_1(3)$, we abbreviate $f^{(0)}_i$ and $f^0_i$ to $f_i$. 
Then  
\begin{align*}
\widetilde{\Sigma}^rf_i&=\Sigma^rf_i\circ(1_{\s^3}\wedge\tau(\s^r,\s^{2p-3})):\Sigma^{2p-3}\Sigma^r\s^3\to \Sigma^r\s^3,\\
\widehat{\Sigma}^rf_i&=\Sigma^{m_{[p,i+1]}}\widetilde{\Sigma}^rf_i : \Sigma^{m_{[p,i]}}\Sigma^r\s^3\to \Sigma^{m_{[p,i+1]}}\Sigma^r\s^3
\end{align*}
for $r\ge 0$ and $p\ge i\ge 1$, where $m_{[p,p+1]}=0$ by definition.  
The orders of $\widetilde{\Sigma}^rf_i$ and $\widehat{\Sigma}^rf_i$ are $p$ and 
\begin{align*}
[\Sigma^r\Sigma^{p-2}\Sigma^{m_{[p,1]}}\s^3,\Sigma^r\s^3]&=
[\s^3\wedge\s^{m_{[p,1]}}\wedge\s^{p-2}\wedge\s^r,\s^3\wedge\s^r]\\
&\supset 
A(r,0)=\{\widehat{\Sigma}^r\overrightarrow{\bm f}\}^{(\star)}
\circ(1_{\s^3}\wedge\tau(\s^{m_{[p,1]}}\wedge\s^{p-2},\s^r))\\
&\supset 
B(r,0)=\{\widetilde{\Sigma}^r\overrightarrow{\bm f}\}^{(\star)}_{\vec{\bm m}}\circ(1_{\s^3}\wedge\tau(\s^{m_{[p,1]}}\wedge \s^{p-2},\s^r)).
\end{align*} 
By Lemma 10.1, we have $\Sigma A(r,0)\subset A(r+1,0)$ and $\Sigma B(r,0)\subset B(r+1,0)$. 
By Definition~10.3, $\langle \Sigma^{\overrightarrow{\bm m}}\overrightarrow{\bm \alpha_1}\rangle^{(\star)}=\lim_{r\to\infty}A(r,0)$ and $\langle\overrightarrow{\bm \alpha_1}\rangle^{(\star)}_{\vec{\bm m}}=\lim_{r\to\infty}B(r,0)$. 
Hence 
$$
\pi_{2p(p-1)-2}(\mathfrak{S})=\langle \Sigma^{p-2}\Sigma^{m_{[p,1]}}\s^3,\s^3\rangle\supset
\langle \Sigma^{\overrightarrow{\bm m}}\overrightarrow{\bm \alpha_1}\rangle^{(\star)}
\supset \langle\overrightarrow{\bm \alpha_1}\rangle^{(\star)}_{\vec{\bm m}}.
$$

We will prove the following proposition (i.e. (1.16)) 
which contains \cite[Example 6.2.5]{OO}. 

\begin{prop}
For all $\star$,  
\begin{enumerate}
\item $\langle\overrightarrow{\bm \alpha_1}\rangle^{(\star)}_{\overrightarrow{\bm m}}$ contains an element of order $p$, and
\item every element of 
$\langle \Sigma^{\overrightarrow{\bm m}}\overrightarrow{\bm \alpha_1}\rangle^{(\star)}$ is of order a multiple of $p$.
\end{enumerate}
\end{prop}

If this holds, then the above two brackets do not contain $0$ so that 
the next corollary holds by Proposition 3.6 and \cite[Proposition 6.1.5]{OO}. 

\begin{cor}
The brackets 
$\langle \Sigma^{\overrightarrow{\bm m}(k)}\overrightarrow{\bm \alpha_1}(k)\rangle$ and $\langle \overrightarrow{\bm \alpha_1}(k)\rangle^{(\star)}_{\overrightarrow{\bm m}(k)}$ are empty for $k\ge p+1$ and all $\star$.
\end{cor}

By Proposition 10.6, $\langle \Sigma^{\overrightarrow{\bm m}}\overrightarrow{\bm \alpha_1}\rangle^{(\star)}$ and $\langle\overrightarrow{\bm \alpha_1}\rangle^{(\star)}_{\overrightarrow{\bm m}}$ are unstable, that is, the stabilizations 
$\Sigma^\infty : A(r,0)\to \langle \Sigma^{\overrightarrow{\bm m}}\overrightarrow{\bm \alpha_1}\rangle^{(\star)}$ and 
$\Sigma^\infty : B(r,0)\to \langle\overrightarrow{\bm \alpha_1}\rangle^{(\star)}_{\overrightarrow{\bm m}}$ are bijections for sufficiently large $r$. 
Therefore in order to prove Proposition~11.1, it suffices to examine 
$\{\widehat{\Sigma}^r\overrightarrow{\bm f}\}^{(\star)}$ and 
$\{\widetilde{\Sigma}^r\overrightarrow{\bm f}\}^{(\star)}_{\overrightarrow{\bm m}}$ 
for sufficiently large $r$. 
After doing it, we can easily prove the following: 
When $p=3$ and $\star=aq\ddot{s}_2,\ddot{s}_t$, 
it follows from \cite[Chapter 13]{T} and 
Theorem 7.1 that $\langle\Sigma^{\overrightarrow{\bm m}}\overrightarrow{\bm \alpha_1}\rangle^{(\star)}$ and 
$\langle\overrightarrow{\bm \alpha_1}\rangle^{(\star)}_{\overrightarrow{\bm m}}$ consist of a single element 
so that they are equal. 
It is unclear if the similar assertion holds for $p>3$. 

\begin{lemma}
Let $r\ge 2p(p-1)-3$. 
\begin{enumerate}
\item There is an $\ddot{s}_t$-presentation $\{\mathscr{S}_k,\overline{\widehat{\Sigma}^r f_k},\mathscr{A}_k\,|\,2\le k\le p\}$ of $\widehat{\Sigma}^r\overrightarrow{\bm f}$ such that the order of $\overline{\widehat{\Sigma}^r f_k}$ is a non trivial power of $p$, that is, $p^s$ with $s\ge 1$. 
\item For any $\ddot{s}_t$-presentation $\{\mathscr{S}_k,\overline{\widehat{\Sigma}^rf_k},\mathscr{A}_k\,|\,2\le k\le p\}$ of $\widehat{\Sigma}^r\overrightarrow{\bm f}$, if we set 
$t_k=3+r+(p-k+1)(2p-3)$, then we have 
\begin{enumerate}
\item $\widetilde{H}^m(C_{k,k};\bZ_p)\cong\begin{cases} \bZ_p & m=t_k+2i(p-1)\ (0\le i<k\le p)\\
0 & \text{otherwise}
\end{cases}$. 
\item $\mathscr{P}^1:H^{t_k+2i(p-1)}(C_{k,k};\bZ_p)\cong H^{t_k+2(i+1)(p-1)}(C_{k,k};\bZ_p)$ for $0\le i\le k-2$ with $2\le k\le p$. 
\end{enumerate}
\item $\{\widehat{\Sigma}^r\overrightarrow{\bm f}\}^{(\star)}$ is not empty for all $\star$.
\end{enumerate}
\end{lemma}
\begin{proof}
To prove (1), for simplicity, we set $X=\Sigma^r\s^3$. 
Set 
$$
\mathscr{S}_2=(\Sigma^{m_{[p,1]}}X;\Sigma^{m_{[p,2]}}X, \Sigma^{m_{[p,2]}}X\cup_{\widehat{\Sigma}^rf_1}C\Sigma^{m_{[p,1]}}X;\widehat{\Sigma}^rf_1;i_{\widehat{\Sigma}^rf_1})
$$
which is a reduced iterated mapping cone of depth $1$ with a structure $\mathscr{A}_2=\{1_{C_{2,2}}\}$. 
Since $\widehat{\Sigma}^rf_2\circ \widehat{\Sigma}^rf_1\simeq *$, the map $\widehat{\Sigma}^rf_2$ can be extended to $C_{2,2}$. 
Let $e_{2,2}:C_{2,2}\to \Sigma^{m_{[p,3]}}X$ be such an extension. 
We regard $[C_{2,2},\Sigma^{m_{[p,3]}}X]$ as a group such that the bijection 
$$(\psi^{m_{[p,2]}}_{\widetilde{\Sigma}^rf_1})^*:[\Sigma^{m_{[p,2]}}(X\cup_{\widetilde{\Sigma}^rf_1}C\Sigma^{m_1}X),\Sigma^{m_{m[p,3]}}X]\to[C_{2,2},\Sigma^{m_{[p,3]}}X]$$ is an isomorphism. 
Then $[C_{2,2},\Sigma^{m_{[p,3]}}X]$ is a finite abelian group. 
Let $i_1:X\to X\cup_{\widetilde{\Sigma}^rf_1}C\Sigma^{m_1}X$ be the inclusion. 
Then $\psi^{m_{[p,2]}}_{\widetilde{\Sigma}^rf_1}\circ j_{2,1}=\Sigma^{m_{[p,2]}}i_1$. 
Let $e_{2,2}'$ be a map representing the $p$-primary component of $e_{2,2}$. 
Set $x_{2,2}=e_{2,2}-e'_{2,2}$. 
Then $e'_{2,2}+x_{2,2}\simeq e_{2,2}$. 
We have 
\begin{align*}
e'_{2,2}\circ j_{2,1}+x_{2,2}\circ j_{2,1}&=e'_{2,2}\circ\psi^{-1}\circ \Sigma^{m_{[p,2]}}i_1
+x_{2,2}\circ\psi^{-1}\circ \Sigma^{m_{[p,2]}}i_1\\
&=(e'_{2,2}\circ\psi^{-1}+x_{2,2}\circ\psi^{-1})\circ \Sigma^{m_{[p,2]}}i_1\\
&\simeq (e'_{2,2}+x_{2,2})\circ\psi^{-1}\circ \Sigma^{m_{[p,2]}}i_1\quad(\text{since $(\psi^{-1})^*$ is an isomorphism})\\
&=(e'_{2,2}+x_{2,2})\circ j_{2,1}\simeq e_{2,1}\circ j_{2,1}=\widehat{\Sigma}^r f_2,
\end{align*}
where $\psi=\psi^{m_{[p,2]}}_{\widetilde{\Sigma}^rf_1}$. 
Since orders of $e'_{2,2}\circ j_{2,1}$ and $\widehat{\Sigma}^rf_2$ are powers of $p$, 
it follows that $x_{2,2}\circ j_{2,1}\simeq *$ so that there is a homotopy 
$G:e'_{2,2}\circ j_{2,1}\simeq \widehat{\Sigma}^rf_2$. 
Since $j_{2,1}$ is a cofibration, there is a homotopy $H:C_{2,2}\times I\to \Sigma^{m_{[p,3]}}X$ such that $H\circ (j_{2,1}\times 1_I)=G$ and $H_0=e'_{2,2}$. 
Set $\overline{\widehat{\Sigma}^rf_2}=H_1$. 
Then $\overline{\widehat{\Sigma}^rf_2}\simeq e'_{2,2}$ and $\overline{\widehat{\Sigma}^rf_2}\circ j_{2,1}=\widehat{\Sigma}^rf_2$. 
In particular the order of $\overline{\widehat{\Sigma}^rf_2}$ is a non trivial power of $p$. 
Set $\mathscr{S}_3=\mathscr{S}_2(\overline{\widehat{\Sigma}^rf_2},\mathscr{A}_2)$ and let 
$\mathscr{A}_3$ be a reduced structure on $\mathscr{S}_3$. 
Then $C_{3,s}=\Sigma^{m_{[p,3]}}X\cup_{\overline{\widehat{\Sigma}^rf_2}^{s-1}}CC_{2,s-1}\ (1\le s\le 3)$, 
$g_{2,1}=\widehat{\Sigma}^rf_2$, and $g_{3,2}:\Sigma\Sigma^{m_{[p,1]}}X\to C_{3,2}$. 
Since $j_{3,1}$ is a homotopy cofibre of $\widehat{\Sigma}^rf_2$ and $\widehat{\Sigma}^rf_3\circ \widehat{\Sigma}^rf_2\simeq *$, 
$\widehat{\Sigma}^rf_3$ can be extended to $C_{3,2}$. 
Let $e_{3,2}:C_{3,2}\to \Sigma^{m_{[p,4]}}X$ be such an extension. 
We regard $[C_{3,2},\Sigma^{m_{[p,4]}}X]$ as a group such that the bijection 
$$
(\psi^{m_{[p,3]}}_{\widetilde{\Sigma}^rf_2})^*:[\Sigma^{m_{[p,3]}}(X\cup_{\widetilde{\Sigma}^rf_2}C\Sigma^{m_2}X),\Sigma^{m_{[p,4]}}X]\to [C_{3,2},\Sigma^{m_{[p,4]}}X]
$$
is an isomorphism. 
Let $i_2:X\to X\cup_{\widetilde{\Sigma}^rf_2}C\Sigma^{m_2}X$ be the inclusion. 
Then $\psi^{m_{[p,3]}}_{\widetilde{\Sigma}^rf_2}\circ j_{3,1}=\Sigma^{m_{[p,3]}}i_2$. 
Let $e'_{3,2}$ be a map 
representing the $p$-primary part of $e_{3,2}$ in $[C_{3,2},\Sigma^{m_{[p,4]}}X]$. 
Set $x_{3,2}=e_{3,2}-e'_{3,2}$. 
Then $e'_{3,2}+x_{3,2}\simeq e_{3,2}$ and $e'_{3,2}\circ j_{3,1}+x_{3,2}\circ j_{3,1}=(e'_{3,2}+x_{3,2})\circ j_{3,1}\simeq e_{3,2}\circ j_{3,1}=\widehat{\Sigma}^rf_3$ by 
the same argument as above. 
Since orders of $e'_{3,2}\circ j_{3,1}$ and $\widehat{\Sigma}^rf_3$ are powers of $p$, it follows that $x_{3,2}\circ j_{3,1}\simeq *$ so that there is a homotopy $G:e'_{3,2}\circ j_{3,1}\simeq \widehat{\Sigma}^rf_3$. 
Since $j_{3,1}$ is a cofibration, there is a homotopy $H:C_{3,2}\times I\to \Sigma^{m_{[p,4]}}X$ such that $H\circ (j_{3,1}\times 1_I)=G$ and $H_0=e'_{3,2}$. 
Set $\overline{\widehat{\Sigma}^rf_3}^2=H_1$. 
Then $\overline{\widehat{\Sigma}^rf_3}^2\simeq e'_{3,2}$ and $\overline{\widehat{\Sigma}^rf_3}^2\circ j_{3,1}=\widehat{\Sigma}^rf_3$. 
In particular the order of $\overline{\widehat{\Sigma}^rf_3}^2$ is a non trivial power of $p$. 
If $p=3$, then, by setting $\overline{\widehat{\Sigma}^rf_3}=\overline{\widehat{\Sigma}^rf_3}^2$, 
$\{\mathscr{S}_k,\overline{\widehat{\Sigma}^rf_k},\mathscr{A}_k\,|\,k=2,3\}$ is a desired 
$\ddot{s}_t$-presentation of $\widehat{\Sigma}^r\overrightarrow{\bm f}$. 

Let $p>3$. 
Then $\overline{\widehat{\Sigma}^rf_3}^2\circ g_{3,2}$ is in the $p$-primary component of 
$[\Sigma\Sigma^{m_{[p,1]}}X,\Sigma^{m_{[p,4]}}X]$, that is, in $\pi_{2\cdot 3(p-1)-2}(\mathfrak{S};p)$. 
Hence $\overline{\widehat{\Sigma}^rf_3}^2\circ g_{3,2}\simeq *$ by (11.1). 
Since $j_{3,2}$ is a homotopy cofibre of $g_{3,2}$, it follows from \cite[Lemma 4.3(7)]{OO} that $\overline{\widehat{\Sigma}^rf_3}^2$ can be extended to $C_{3,3}$. 
Let $e_{3,3}$ be such an extension. 
By the definition of $\ddot{s}_t$-presentation, there is a homeomorphism 
$h:C_{3,3}\approx \Sigma^{m_{[p,3]}}(X\cup C\Sigma^{m_2}(X\cup_{\widetilde{\Sigma}^rf_1}C\Sigma^{m_1}X))$ 
such that $h\circ j_{3,2}\circ j_{3,1}=\Sigma^{m_{[p,3]}}i_3$, where $i_3:X\to X\cup C\Sigma^{m_2}(X\cup C\Sigma^{m_1}X)$ is the inclusion. 
We regard $[C_{3,3},\Sigma^{m_{[p,4]}}X]$ as a group such that the bijection 
$$
h^*:[\Sigma^{m_{[p,3]}}(X\cup C\Sigma^{m_2}(X\cup_{\widetilde{\Sigma}^rf_1}C\Sigma^{m_1}X)),\Sigma^{m_{[p,4]}}X]
\to [C_{3,3}, \Sigma^{m_{[p,4]}}X]
$$
is an isomorphism. 
Then $[C_{3,3},\Sigma^{m_{[p,4]}}X]$ is a finite abelian group. 
Let $e'_{3,3}$ be a map representing the $p$-primary component of $e_{3,3}$ in $[C_{3,3},\Sigma^{m_{[p,4]}}X]$. 
Set $x_{3,3}=e_{3,3}-e'_{3,3}$. 
Then $e'_{3,3}+x_{3,3}\simeq e_{3,3}$ and $e'_{3,3}\circ j_{3,2}+x_{3,3}\circ j_{3,2}\simeq (e'_{3,3}+x_{3,3})\circ j_{3,2}
\simeq e_{3,3}\circ j_{3,2}=\overline{\widehat{\Sigma}^rf_3}^2$ by the same argument as above. 
Comparing the orders, we have $x_{3,3}\circ j_{3,2}\simeq *$ so that 
there is a homotopy $G:e'_{3,3}\circ j_{3,2}\simeq \overline{\widehat{\Sigma}^rf_3}^2$. 
Since $j_{3,2}$ is a cofibration, there is a homotopy $H:C_{3,3}\times I\to \Sigma^{m_{[p,4]}}X$ such that $H\circ(j_{3,2}\times 1_I)=G$ and $H_0=e'_{3,3}$. 
Set $\overline{\widehat{\Sigma}^rf_3}=H_1$. 
Then $\overline{\widehat{\Sigma}^rf_3}\circ j_{3,2}=\overline{\widehat{\Sigma}^rf_3}^2$ and $\overline{\widehat{\Sigma}^rf_3}\simeq e'_{3,3}$. 
Hence the order of $\overline{\widehat{\Sigma}^rf_3}$ is a non trivial power of $p$. 
Set $\mathscr{S}_4=\mathscr{S}_3(\overline{\widehat{\Sigma}^rf_3},\mathscr{A}_3)$ and let 
$\mathscr{A}_4$ be a reduced structure on $\mathscr{S}_4$. 
By proceeding the process, we obtain a desired $\ddot{s}_t$-presentation 
$\{\mathscr{S}_k,\overline{\widehat{\Sigma}^rf_k},\mathscr{A}_k\,|\,2\le k\le p\}$ of 
$\widehat{\Sigma}^r\overrightarrow{\bm f}$, where $\overline{\widehat{\Sigma}^rf_k}:\begin{cases} C_{k,k}\to \Sigma^{m_{[p,k+1]}}X & 2\le k< p\\ C_{p,p-1}\to X & k=p\end{cases}$. 
By our construction, $C_{k,s}=\begin{cases} \Sigma^{m_{[p,k]}}X\cup CC_{k-1,s-1} & 2\le s\le k\le p\\ \Sigma^{m_{[p,k]}}X & 2\le k\le p\ \text{and } s=1\end{cases}$. 
This completes the proof of (1). 

(2)(a) is an easy consequence of the definition of $\ddot{s}_t$-presentation. 
Also (2)(b) follows easily from the well-known fact that if $e:\s^{n+2p-3}\to \s^n\ (n\ge 3)$ represents a non zero element of $\pi_{2p-3}(\mathfrak{S};p)$, then 
$\mathscr{P}^1:H^n(\s^n\cup_e C\s^{n+2p-3};\bZ_p)\cong H^{n+2(p-1)}(\s^n\cup_e C\s^{n+2p-3};\bZ_p)$. 

(3) follows from (1) and Corollary 4.2(4). 
\end{proof}

\begin{lemma}
Let $r\ge 2p(p-1)-3$. 
\begin{enumerate}
\item There is an $\ddot{s}_t$-presentation $\{\mathscr{S}_k,\overline{\widetilde{\Sigma}^rf_k}, 
\mathscr{A}_k\,|\,2\le k\le p\}$ of $\widetilde{\Sigma}^r\overrightarrow{\bm f}$ such that 
the order of $\overline{\widetilde{\Sigma}^rf_k}$ is a non-trivial power of $p$, that is, $p^s$ with $s\ge 1$.
\item For any $\ddot{s}_t$-presentation of $\widetilde{\Sigma}^r\overrightarrow{\bm f}$, 
we have
\begin{enumerate}
\item $\widetilde{H}^m(C_{k,k};\bZ_p)\cong\begin{cases}\bZ_p & m=3+r+2i(p-1)\, (0\le i< k\le p)\\ 0 &\text{otherwise}\end{cases}$. 
\item $\mathscr{P}^1:H^{3+r+2i(p-1)}(C_{k,k};\bZ_p)\cong H^{3+r+2(i+1)(p-1)}(C_{k,k};\bZ_p)$ for $0\le i\le k-2$ with $2\le k\le p$.
\end{enumerate}
\item $\{\widetilde{\Sigma}^r\overrightarrow{\bm f}\}^{(\star)}_{\vec{\bm m}}$ is not empty for 
all $\star$. 
\end{enumerate}
\end{lemma}
\begin{proof}
(1). 
Set 
$\mathscr{S}_2=(\Sigma^{2p-3}\s^{3+r};\s^{3+r},\s^{3+r}\cup_{\widetilde{\Sigma}^rf_1}C\Sigma^{2p-3}\s^{3+r};\widetilde{\Sigma}^rf_1;i_{\widetilde{\Sigma}^rf_1})$. 
Then it is an iterated reduced mapping cone of depth $1$ with a structure 
$\mathscr{A}_2=\{1_{C_{2,2}}\}$. 
Since $\widetilde{\Sigma}^rf_2\circ \Sigma^{2p-3}\widetilde{\Sigma}^rf_1\simeq *$ by Toda \cite{T} and since 
$\Sigma^{2p-3}j_{2,1}$ is a homotopy cofibre of $\Sigma^{2p-3}\widetilde{\Sigma}^rf_1$, the map $\widetilde{\Sigma}^rf_2$ can be extended to $\Sigma^{2p-3}C_{2,2}$ by \cite[Lemma 4.3(7)]{OO}. 
Let $e_{2,2}$ be such an extension and 
$e'_{2,2}:\Sigma^{2p-3}C_{2,2}\to\s^{3+r}$ a map representing the 
$p$-primary component of $e_{2,2}$ in the finite stable group 
$[\Sigma^{2p-3}C_{2,2},\s^{3+r}]$. 
Set $x=e_{2,2}-e'_{2,2}$. 
Then $e'_{2,2}+x\simeq e_{2,2}$ and 
$e'_{2,2}\circ \Sigma^{2p-3}j_{2,1}+x\circ \Sigma^{2p-3}j_{2,1}\simeq e_{2,2}\circ \Sigma^{2p-3}j_{2,1}=\widetilde{\Sigma}^rf_2$. 
Since orders of $e'_{2,2}\circ \Sigma^{2p-3}j_{2,1}$ and $\widetilde{\Sigma}^rf_2$ are powers of $p$, 
it follows that $x\circ \Sigma^{2p-3}j_{2,1}\simeq *$ so that there is a homotopy 
$G:e'_{2,2}\circ \Sigma^{2p-3}j_{2,1}\simeq \widetilde{\Sigma}^rf_2$. 
Since $\Sigma^{2p-3}j_{2,1}$ is a cofibration, there is a homotopy $H:\Sigma^{2p-3}C_{2,2}\times I\to \s^{3+r}$ such that $H\circ (\Sigma^{2p-3}j_{2,1}\times 1_I)=G$ and $H_0=e'_{2,2}$. 
Set $\overline{\widetilde{\Sigma}^rf_2}=H_1$. 
Then $\overline{\widetilde{\Sigma}^rf_2}\simeq e'_{2,2}$ and $\overline{\widetilde{\Sigma}^rf_2}\circ \Sigma^{2p-3}j_{2,1}=\widetilde{\Sigma}^rf_2$. 
In particular the order of $\overline{\widetilde{\Sigma}^rf_2}$ is a non-trivial power of $p$. 
Set $\mathscr{S}_3=(\widetilde{\Sigma}^{2p-3}\mathscr{S}_2)(\overline{\widetilde{\Sigma}^rf_2},\widetilde{\Sigma}^{2p-3}\mathscr{A}_2)$ and let $\mathscr{A}_3$ be a reduced structure on $\mathscr{S}_3$. 
Then $C_{3,s}=\s^{3+r}\cup_{}C\Sigma^{2p-3}C_{2,s-1}\ (1\le s\le 3)$, $g_{3,1}=\widetilde{\Sigma}^rf_2$, and $g_{3,2}=(\overline{\widetilde{\Sigma}^rf_2}\cup C1_{\Sigma^{2p-3}\s^{3+r}})\circ(\widetilde{\Sigma}^{2p-3}\omega_{2,1})^{-1}:\Sigma\Sigma^{2(2p-3)}\s^{3+r}\to C_{3,2}$. 
Since $\Sigma^{2p-3}j_{3,1}$ is a homotopy cofibre of $\Sigma^{2p-3}\widetilde{\Sigma}^rf_2$ by \cite[Lemma 4.3(5)]{OO} and $\widetilde{\Sigma}^rf_3\circ \Sigma^{2p-3}\widetilde{\Sigma}^rf_2\simeq *$, $\widetilde{\Sigma}^rf_3$ can be extended to $\Sigma^{2p-3}C_{3,2}$ by \cite[Lemma 4.3(7)]{OO}. 
Let $e_{3,2}$ be such an extension and $e'_{3,2}:\Sigma^{2p-3}C_{3,2}\to \s^{3+r}$ a map representing the $p$-primary component of $e_{3,2}$ in the finite stable group $[\Sigma^{2p-3}C_{3,2},\s^{3+r}]$. 
By the same method discussed above, there is a map $\overline{\widetilde{\Sigma}^rf_3}^2:\Sigma^{2p-3}C_{3,2}\to \s^{3+r}$ such that $\overline{\widetilde{\Sigma}^rf_3}^2\simeq e'_{3,2}$ and $\overline{\widetilde{\Sigma}^rf_3}^2\circ \Sigma^{2p-3}j_{3,1}=\widetilde{\Sigma}^rf_3$. 
In particular the order of $\overline{\widetilde{\Sigma}^rf_3}^2$ is a non-trivial power of $p$. 
If $p=3$, then $\{\mathscr{S}_k,\overline{\widetilde{\Sigma}^rf_k},\mathscr{A}_k\,|\,k=2,3\}$ is 
a desired $\ddot{s}_t$-presentation of $\widetilde{\Sigma}^r\overrightarrow{\bm f}$, where 
$\overline{\widetilde{\Sigma}^rf_3}=\overline{\widetilde{\Sigma}^rf_3}^2$. 

Let $p>3$. 
Then the homotopy class of $\overline{\widetilde{\Sigma}^rf_3}^2\circ\widetilde{\Sigma}^{2p-3}g_{3,2}$ 
is in the $p$-primary component of $[\Sigma\Sigma^{3(2p-3)}\s^{3+r},\s^{3+r}]$, that is, in 
$\pi_{2\cdot 3(p-1)-2}(\mathfrak{S};p)$. 
Hence $\overline{\widetilde{\Sigma}^rf_3}^2\circ\widetilde{\Sigma}^{2p-3}g_{3,2}\simeq *$ by (11.1). 
Since $\Sigma^{2p-3}j_{3,2}$ is a homotopy cofibre of $\widetilde{\Sigma}^{2p-3}g_{3,2}$ 
by Lemma A.1 and \cite[Lemma 4.3(5)]{OO}, it follows from \cite[Lemma 4.3(7)]{OO} that $\overline{\widetilde{\Sigma}^rf_3}^2$ can be extended to $\Sigma^{2p-3}C_{3,3}$. 
Let $e_{3,3}$ be such an extension and $e'_{3,3}$ a map representing the $p$-primary component of $e_{3,3}$ in the stable finite group 
$[\Sigma^{2p-3}C_{3,3},\s^{3+r}]$. 
Set $x=e_{3,3}-e'_{3,3}$. 
Then $e'_{3,3}\circ \Sigma^{2p-3}j_{3,2}+x\circ \Sigma^{2p-3}j_{3,2}\simeq e_{3,3}\circ \Sigma^{2p-3}j_{3,2}=\overline{\widetilde{\Sigma}^rf_3}^2$. 
Comparing orders, we have $x\circ \Sigma^{2p-3}j_{3,2}\simeq *$ so that there is a homotopy 
$G:e'_{3,3}\circ \Sigma^{2p-3}j_{3,2}\simeq \overline{\widetilde{\Sigma}^rf_3}^2$. 
Since $\Sigma^{2p-3}j_{3,2}$ is a cofibration, there is a homotopy 
$H:\Sigma^{2p-3}C_{3,3}\times I\to\s^{3+r}$ such that $H\circ (\Sigma^{2p-3}j_{3,2}\times 1_I)=G$ and $H_0=e'_{3,3}$. 
Set $\overline{\widetilde{\Sigma}^rf_3}=H_1$. 
Then $\overline{\widetilde{\Sigma}^rf_3}\circ \Sigma^{2p-3}j_{3,2}=\overline{\widetilde{\Sigma}^rf_3}^2$ and $\overline{\widetilde{\Sigma}^rf_3}\simeq e'_{3,3}$. 
Hence the order of $\overline{\widetilde{\Sigma}^rf_3}$ is a non-trivial power of $p$. 
Set $\mathscr{S}_4=(\widetilde{\Sigma}^{2p-3}\mathscr{S}_3)(\overline{\widetilde{\Sigma}^rf_3},\widetilde{\Sigma}^{2p-3}\mathscr{A}_3)$ and let $\mathscr{A}_4$ be a reduced structure on $\mathscr{S}_4$. 
By proceeding the process, we obtain  a desired $\ddot{s}_t$-presentation 
$\{\mathscr{S}_k,\overline{\widetilde{\Sigma}^rf_k},\mathscr{A}_k\,|\,2\le k\le p\}$ of $\widetilde{\Sigma}^r\overrightarrow{\bm f}$, where $\overline{\widetilde{\Sigma}^rf_k}:\begin{cases} \Sigma^{2p-3}C_{k,k}\to \s^{3+r} & 2\le k<p\\ \Sigma^{2p-3}C_{p,p-1}\to\s^{3+r} & k=p\end{cases}$. 
By our constructions, $\mathscr{S}_k$ is displayed as follows:
$$
\tiny{
\xymatrix{
\Sigma^{2p-3}\s^{3+r} \ar[d]^-{\widetilde{\Sigma}^rf_{k-1}=g_{k,1}}  & \cdots &\Sigma^{s-1}\Sigma^{s(2p-3)}\s^{3+r}\ar[d]^-{g_{k,s}} &\cdots & \Sigma^{k-2}\Sigma^{(k-1)(2p-3)}\s^{3+r} \ar[d]^-{g_{k,k-1}} & \\
\s^{3+r}\ar[r]^-{j_{k,1}}  &\cdots \ar[r] & C_{k,s} \ar[r]^-{j_{k,s}} &\cdots \ar[r] & C_{k,k-1}\ar[r]^-{j_{k,k-1}} & C_{k,k}
}
},
$$
where $C_{k,s}=\begin{cases}\s^{3+r}\cup_{\overline{\widetilde{E}^rf_{k-1}}^{s-1}}CE^{2p-3}C_{k-1,s-1} & p\ge k\ge s\ge 2\\ \s^{3+r} &p\ge k\ge 2\ \text{and }s=1\end{cases}$. 
This proves (1). 

(2)(a) is an easy consequence of the definition of $\ddot{s}_t$-presentation.  
Also (2)(b) follows easily from the well-known fact that if 
$e:\s^{n+2p-3}\to \s^n$ represents a non-zero element of $\pi_{2p-3}(\mathfrak{S};p)$ for $n\ge 3$, then $\mathscr{P}^1:H^n(\s^n\cup_e C\s^{n+2p-3};\bZ_p)\cong 
H^{n+2(p-1)}(\s^n\cup_e C\s^{n+2p-3};\bZ_p)$. 

(3) follows from (1) and Corollary 4.2(4). 
This completes the proof of Lemma 11.4. 
\end{proof}

\begin{proof}[Proof of Proposition 11.1(1)] 
In the proof, we assume $r\ge 2p(p-1)-3$. 

First we prove the assertion for $\star=\ddot{s}_t$. 
It suffices to prove that 
$\{\widetilde{\Sigma}^r\overrightarrow{\bm f}\}^{(\ddot{s}_t)}_{\vec{\bm m}}$ contains 
an element of order $p$. 
By Lemma 11.4(1), there is an $\ddot{s}_t$-presentation 
$\{\mathscr{S}_k,\overline{\widetilde{\Sigma}^rf_k},\mathscr{A}_k\, |\,2\le k\le p\}$ 
of $\widetilde{\Sigma}^r\overrightarrow{\bm f}$ such that the order of 
$\overline{\widetilde{\Sigma}^r f_k}$ is a nontrivial power of $p$ for all $k$.  
Hence the order of $\overline{\widetilde{\Sigma}^rf_p}\circ\widetilde{\Sigma}^{2p-3}g_{p.p-1}$ is a power of $p$. 
To induce a contradiction, suppose that $\overline{\widetilde{\Sigma}^r f_p}\circ\widetilde{\Sigma}^{2p-3}g_{p.p-1}\simeq *$. 
Since $\Sigma^{2p-3}j_{p,p-1}$ is a homotopy cofibre of $\widetilde{\Sigma}^{2p-3}g_{p,p-1}$ by Lemma~A.1 and \cite[Lemma~4.3(5)]{OO}, it follows from \cite[Lemma~4.3(7)]{OO} that there is a map $\widetilde{\widetilde{\Sigma}^rf_p}:\Sigma^{2p-3}C_{p,p}\to \s^{3+r}$ such that $\widetilde{\widetilde{\Sigma}^r f_p}\circ \Sigma^{2p-3}j_{p,p-1}=\overline{\widetilde{\Sigma}^rf_p}$. 
Hence $\widetilde{\widetilde{\Sigma}^r f_p}\circ \Sigma^{2p-3}(j_{p,p-1}\circ \cdots\circ j_{p,1})=\widetilde{\Sigma}^r f_p$. 
Set $\mathscr{S}_{p+1}=(\widetilde{\Sigma}^{2p-3}\mathscr{S}_p)(\widetilde{\widetilde{\Sigma}^rf_p},\widetilde{\Sigma}^{2p-3}\mathscr{A}_p)$. 
Then $C_{p+1,s+1}=\s^{3+r}\cup C\Sigma^{2p-3}C_{p,s}$ for $0\le s\le p$ so that 
\begin{equation}
C_{p+1,p+1}/C_{p+1,1}=\Sigma^{2(p-1)}C_{p,p}
\end{equation}
and 
$$
\widetilde{H}^m(C_{p+1,p+1};\bZ_p)\cong
\begin{cases} \bZ_p & m=3+r+2i(p-1)\ (0\le i\le p)\\
0 & \text{otherwise}
\end{cases}
$$ 
by the construction of $\mathscr{S}_{p+1}$. 
We have the following commutative square
$$
\begin{CD}
H^{3+r}(C_{p+1,2};\bZ_p)@>\mathscr{P}^1>>H^{3+r+2(p-1)}(C_{p+1,2};\bZ_p)\\
@A j^*_{p+1}A\cong A @A\cong A j^*_{p+1}A\\
H^{3+r}(C_{p+1,p+1};\bZ_p)@>\mathscr{P}^1>>H^{3+r+2(p-1)}(C_{p+1,p+1};\bZ_p)
\end{CD}
$$
where $j_{p+1}=j_{p+1,p}\circ \cdots\circ j_{p+1,3}\circ j_{p+1,2}$. 
Since $C_{p+1,2}=\s^{3+r}\cup_{\widetilde{\Sigma}^rf_p}C\Sigma^{2p-3}\s^{3+r}$, the first 
$\mathscr{P}^1$ of the square above is an isomorphism so that the second 
$\mathscr{P}^1$ is also an isomorphism. 
By Lemma 11.4(2)(b) we have $(\mathscr{P}^1)^{p-1}:H^{3+r}(C_{p,p};\bZ_p)\cong 
H^{3+r+2(p-1)^2}(C_{p,p};\bZ_p)$. 
Hence $(\mathscr{P}^1)^{p-1}:H^{3+r+2(p-1)}(C_{p+1,p+1};\bZ_p)\cong H^{3+r+2p(p-1)}(C_{p+1,p+1};\bZ_p)$ by (11.3) so that 
$(\mathscr{P}^1)^p:H^{3+r}(C_{p+1,p+1};\bZ_p)\cong H^{3+r+2p(p-1)}(C_{p+1,p+1};\bZ_p)$. 
This is a contradiction, since $(\mathscr{P}^1)^p=0$ by (11.2). 
Hence  $\overline{\widetilde{\Sigma}^r f_p}\circ\widetilde{\Sigma}^{2p-3}g_{p.p-1}$ 
is not null homotopic. 
Therefore the order of $\overline{\widetilde{\Sigma}^rf_p}\circ\widetilde{\Sigma}^{2p-3}g_{p,p-1}
\in\{\widetilde{\Sigma}^r\overrightarrow{\bm f}\}^{(\ddot{s}_t)}$ is a non trivial power of $p$, that is, just $p$, since $\{\widetilde{\Sigma}^r\overrightarrow{\bm f}\}^{(\ddot{s}_t)}_{\vec{\bm m}}
\subset[\Sigma^{p-2}\Sigma^{m_{[p,1]}}\Sigma^r\s^3,\Sigma^r\s^3]=\pi_{2p(p-1)-2}(\mathfrak{S})$ and 
the $p$ component of the last stable group is $\bZ_p$. 
This proves Proposition~11.1(1) for $\star=\ddot{s}_t$. 

For any $\star\ne\ddot{s}_t$, Proposition~11.1(1) follows from (3.1) and Theorem 4.1. 
\end{proof}

\begin{proof}[Proof of Proposition 11.1(2)] 
In the proof we assume $r\ge 2p(p-1)-3$. 

First we prove the assertion for $\star=\ddot{s}_t$. 
It suffices to to prove that every element of 
$\{\widehat{\Sigma}^r\overrightarrow{\bm f}\}^{(\ddot{s}_t)}$ is of order a multiple of $p$. 
Let 
$$
\alpha\in\{\widehat{\Sigma}^r\overrightarrow{\bm f}\}^{(\ddot{s}_t)}\subset[\Sigma^{p-2}\Sigma^{m_{[p,1]}}\Sigma^r\s^3,\Sigma^r\s^3]=\pi_{2p(p-1)-2}(\mathfrak{S}).
$$
Let $\{\mathscr{S}_k,\overline{\widehat{\Sigma}^rf_k},\mathscr{A}_k\, |\,2\le k\le p\}$ be an 
$\ddot{s}_t$-presentation of $\widehat{\Sigma}^r\overrightarrow{\bm f}$ such that 
$\alpha=\overline{\widehat{\Sigma}^rf_p}\circ g_{p,p-1}$. 
Set $X=\Sigma^r\s^3$ and $j_p'=j_{p,p-2}\circ\cdots\circ j_{p,2}\circ j_{p,1}:\Sigma^{2p-3}X
\to C_{p,p-1}$. 
By the definition of an $\ddot{s}_t$-presentation, there is a homeomorphism 
\begin{equation}
C_{p,p-1}\overset{h}{\leftarrow} \begin{cases} \Sigma^{m_3}(X\cup_{\widetilde{\Sigma}^rf_2}C\Sigma^{m_2}X) & p=3\\
\Sigma^{m_p}(X\cup C\Sigma^{m_{p-1}}(X\cup C\Sigma^{m_{p-2}}(\cdots\cup C\Sigma^{m_3}(X\cup_{\widetilde{\Sigma}^rf_2}C\Sigma^{m_2}X)\cdots))) & p>3\end{cases}.
\end{equation}
We denote by $\Sigma^{m_p}Y$ the space of the right hand side of (11.4). 
We can assume that $h$ satisfies $j'_p=h\circ \Sigma^{m_p}i$, where $i:X\to Y$ is the inclusion. 
We regard $[C_{p,p-1},X]$ as a group such that $h^*:[C_{p,p-1},X]\to[\Sigma^{m_p}Y,X]$ is an isomorphism. 
Then $[C_{p,p-1},X]$ is a finite abelian group. 
Let $e_p$ be a map representing the $p$-primary component of 
$\overline{\widehat{\Sigma}^rf_p}$ in $[C_{p,p-1},X]$. 
Set $x=\overline{\widehat{\Sigma}^rf_p}-e_p$. 
The order of $x$ in $[C_{p,p-1},X]$ is prime with $p$. 
We have 
\begin{align*}
x\circ j'_p+e_p\circ j'_p&=x\circ h\circ \Sigma^{m_p}i+e_p\circ h\circ \Sigma^{m_p}i=
(x\circ h+e_p\circ h)\circ \Sigma^{m_p}i\\
&\simeq (x+e_p)\circ h\circ \Sigma^{m_p}i\quad(\text{since $h^*$ is an isomorphism})\\
&=(x+e_p)\circ h\circ h^{-1}\circ j'_p=(x+e_p)\circ j'_p\\
&\simeq \overline{\widehat{\Sigma}^r f_p}\circ j'_p=\widehat{\Sigma}^r f_p.
\end{align*}
Comparing the orders, we have $x\circ j'_p\simeq *$ so that $e_p\circ j'_p\simeq \widehat{\Sigma}^rf_p$ and the order of $e_p$ is a non trivial power of $p$. 
Take a homotopy $G:e_p\circ j'_p\simeq \widehat{\Sigma}^rf_p$. 
Since $j'_p$ is a cofibration, there is a homotopy $H:C_{p,p-1}\times I\to X$ 
such that $H\circ (j'_p\times 1_I)=G$ and $H_0=e_p$. 
Then $H_1\circ j'_p=\widehat{\Sigma}^rf_p$ and $H_1\simeq e_p$ so that the order of $H_1$ is a non trivial power of $p$. 
Since the composite of 
$[\Sigma^{p-3}\Sigma^{m_{[p,1]}}X,\Sigma^{m_p-1}Y]\overset{\Sigma}{\to}[\Sigma^{p-2}\Sigma^{m_{[p,1]}}X,\Sigma^{m_p}Y]\overset{h_*}{\cong}[\Sigma^{p-2}\Sigma^{m_{[p,1]}}X, C_{p.p-1}]$ is an isomorphism by the assumption on $r$ 
so that $g_{p,p-1}\simeq h\circ \Sigma g'$ for some map $g':\Sigma^{p-3}\Sigma^{m_{[p,1]}}X\to \Sigma^{m_p-1}Y$. 
Then 
\begin{align*}
\overline{\widehat{\Sigma}^rf_p}\circ g_{p,p-1}&\simeq \overline{\widehat{\Sigma}^rf_p}\circ h\circ \Sigma g'\simeq (H_1+x)\circ h\circ \Sigma g'\\
&\simeq (H_1\circ h+x\circ h)\circ \Sigma g'\quad(\text{since $h^*$ is an isomorphism})\\
&=H_1\circ h\circ \Sigma g'+x\circ h\circ \Sigma g'\simeq H_1\circ g_{p,p-1}+x\circ g_{p,p-1}.
\end{align*}
That is, we have 
\begin{equation}
\overline{\widehat{\Sigma}^rf_p}\circ g_{p,p-1}\simeq H_1\circ g_{p,p-1}+x\circ g_{p,p-1}.
\end{equation}
To induce a contradiction, suppose that $H_1\circ g_{p,p-1}\simeq *$. 
Since $j_{p,p-1}$ is a homotopy cofibre of $g_{p,p-1}$, it follows from \cite[Lemma 4.3(7)]{OO} that there is a map $\widetilde{\widehat{\Sigma}^rf_p}:C_{p,p}\to X$ such that 
$\widetilde{\widehat{\Sigma}^rf_p}\circ j_{p,p-1}=H_1$. 
Hence $\widetilde{\widehat{\Sigma}^rf_p}\circ j_{p,p-1}\circ j'_p=H_1\circ j'_p=\widehat{\Sigma}^rf_p$. 
Set $\mathscr{S}_{p+1}=\mathscr{S}_p(\widetilde{\widehat{\Sigma}^rf_p},\mathscr{A}_p)$. 
Then $C_{p+1,s+1}=X\cup CC_{p,s}$ for $0\le s\le p$ so that 
$$
C_{p+1,p+1}/C_{p+1,1}=\Sigma C_{p,p}
$$
and 
$$
\widetilde{H}^m(C_{p+1,p+1};\bZ_p)\cong\begin{cases} \bZ_p & m=3+r+2i(p-1)\ (0\le i<p+1)\\ 0 & \text{otherwise}\end{cases}
$$
by the construction of $\mathscr{S}_{p+1}$. 
By the method in the proof of Proposition 11.1(1), we have 
$(\mathscr{P}^1)^p:H^{3+r}(C_{p+1,p+1};\bZ_p)\cong H^{3+r+2p(p-1)}(C_{p+1,p+1};\bZ_p)$. 
This contradicts to (11.2). 
Hence $H_1\circ g_{p,p-1}$ is not null homotopic and so its order is 
a non trivial power of $p$. 
Therefore the order of $\alpha=\overline{\widehat{\Sigma}^rf_p}\circ g_{p,p-1}$ is a multiple of $p$ by (11.5). 

For any $\star\ne\ddot{s}_t$, Proposition~11.1(2) follows from (3.1) and 
Theorem 4.1. 
\end{proof}

\begin{appendix}
\section{Addenda to our previous paper \cite{OO}}
The following should be contained in Lemma 4.3 of \cite{OO}. 

\begin{lemma}
Let us work in $\mathrm{TOP}^*$. 
If a map $j:Y\to Z$ is a homotopy cofibre of $f:X\to Y$ and $h:X'\to X$ is a homotopy equivalence, then $j$ is a homotopy cofibre of $f\circ h$. 
\end{lemma}
\begin{proof}
Let $a:Z\to Y\cup_f CX$ be a homotopy equivalence with $a\circ j=i_f$. 
Let $h^{-1}$ be a homotopy inverse of $h$. 
Take $J:f\simeq f\circ h\circ h^{-1}$. 
Then $\Phi(f,f\circ h,h^{-1},1_Y;J):Y\cup_f CX\to Y\cup_{f\circ h}CX'$ is a homotopy equivalence with $\Phi(f,f\circ h,h^{-1},1_Y;J)\circ i_f=i_{f\circ h}$. 
Hence $\Phi(f,f\circ h,h^{-1},1_Y;J)\circ a:Z\to Y\cup_{f\circ h}CX'$ is a homotopy equivalence with $\Phi(f,f\circ h,h^{-1},1_Y;J)\circ a\circ j=i_{f\circ h}$. 
This proves that $j$ is a homotopy cofibre of $f\circ h$. 
\end{proof}

We define a new system of unstable higher Toda brackets in $\mathrm{TOP}^*$. 
Let $\vec{\bm f}=(f_n,\dots,f_1)$ be a composable sequence of pointed maps 
$f_i:X_i\to X_{i+1}$ with $n\ge 2$.  
We define the unstable $n$-fold Toda bracket $\{\vec{\bm f}\}$ 
inductively as follows. 
$$
\{\vec{\bm f}\}=\begin{cases} \{f_2,f_1\}_{(0,0)} & n=2\\
\bigcup\{[f_3,A_2,f_2],(f_2,A_1,f_1)\}_{(0,0)} & n=3\\
\bigcup \{f_n,\dots,f_4,[f_3,A_2,f_2],(f_2,A_1,f_1)\} & n\ge 4\end{cases},
$$ 
where $\{g,f\}_{(0,0)}$ denotes the one point set of the homotopy class of $g\circ f$ for any maps $Z\overset{g}{\leftarrow}Y\overset{f}{\leftarrow}X$, and the union $\bigcup$ is taken over all pairs $(A_2,A_1)$ with $A_2:f_3\circ f_2\simeq *$ 
and $A_1:f_2\circ f_1\simeq *$ (of course if $A_2$ or $A_1$ does not exist, then it is the empty set). 
Notice that it is the classical Toda bracket for $n=3$. 
The new $n$-fold Toda bracket is a subset of $[\Sigma^{n-2}X_1,X_{n+1}]$. 
We will study this new system and its imitations elsewhere. 

In the rest of the appendix, we work in $\mathrm{TOP}^w$. 

The commutative diagram below shall be denoted by 
$\vec{\bm b}:\vec{\bm f}\to\vec{\bm f'}$, where $\vec{\bm b}=(b_{n+1},\dots,b_1)$ is 
a sequence of homeomorphisms. 
$$
\begin{CD}
X_{n+1}@<f_n<<\cdots@<<< X_{i+1}@<f_i<<X_i@<<<\cdots@<f_1<<X_1\\
@V\approx V b_{n+1}V @. @V\approx V b_{i+1}V @V\approx V b_i V @. @V\approx V b_1 V\\
X_{n+1}'@<f'_n<<\cdots@<<< X_{i+1}'@<f_i'<<X_i'@<<<\cdots@<f_1'<<X_1'
\end{CD}
$$

\begin{lemma}
For all $\vec{\bm b}:\vec{\bm f}\to\vec{\bm f'}$, we have 
$\{\vec{\bm f'}\}^{(\star)}\circ \Sigma^{n-2}b_1=b_{n+1}\circ\{\vec{\bm f}\}^{(\star)}$.
\end{lemma}
\begin{proof}
We will prove $\{\vec{\bm f'}\}^{(\star)}\circ \Sigma^{n-2}b_1\subset b_{n+1}\circ\{\vec{\bm f}\}^{(\star)}$. 
If this is done, then by considering $\overrightarrow{\bm b^{-1}}:\vec{\bm f'}\to\vec{\bm f}$ 
we have the opposite inclusion and so the assertion. 
Let $\alpha\in\{\vec{\bm f'}\}^{(\star)}$ and $\{\mathscr{S}_r',\overline{f_r'},\Omega_r'\ (\text{or }\mathscr{A}_r')\,|\,2\le r\le n\}$ a $\star$-presentation of $\vec{\bm f'}$ such that $\alpha=\overline{f'_n}\circ g'_{n,n-1}$. 
We will construct a $\star$-presentation $\{ \mathscr{S}_r,\overline{f_r},\Omega_r\ (\text{or }\mathscr{A}_r)\,|\,2\le r\le n\}$ of $\vec{\bm f}$ and homeomorphisms $e_{r,s}:C_{r,s}\to C'_{r,s}$ such that 
\begin{enumerate}
\item $j'_{r,s}\circ e_{r,s}=e_{r,s+1}\circ j_{r,s}$ for $1\le s<r\le n$,
\item $g'_{r,s}\circ \Sigma^{s-1}b_{r-s}\simeq e_{r,s}\circ g_{r,s}$ for $1\le s<r\le n$, 
\item $\overline{f'_r}^s\circ e_{r,s}=b_{r+1}\circ\overline{f_r}^s$ for $1\le s\le r<n$ or $1\le s<r=n$.
\end{enumerate}
If this is done, then 
$b_{n+1}\circ\overline{f_n}\circ g_{n,n-1}\simeq\overline{f_n'}\circ g'_{n,n-1}\circ \Sigma^{n-2}b_1$ 
so that $\alpha\circ \Sigma^{n-2}b_1\in b_{n+1}\circ\{\vec{\bm f}\}^{(\star)}$ and 
$\{\vec{\bm f'}\}^{(\star)}\circ \Sigma^{n-2}b_1\subset b_{n+1}\circ\{\vec{\bm f}\}^{(\star)}$ 
as desired. 

We will give a construction for only $\star=\ddot{s}_t$, because other cases 
can be done similarly or more easily. 
Then $\mathscr{A}'_r=\{a'_{r,s}\,|\,1\le s<r\}$ is a reduced structure on $\mathscr{S}_r'$. 
Set 
\begin{gather*}
\mathscr{S}_2=(X_1;X_2,X_2\cup_{f_1}CX_1;f_1;i_{f_1}),\ \mathscr{A}_2=\{1_{C_{2,2}}\},\ 
e_{2,1}=b_2:C_{2,1}=X_2\to X_2'=C_{2,1}',\\ 
e_{2,2}=b_2\cup Cb_1:C_{2,2}=X_2\cup_{f_1}CX_1\to X_2'\cup_{f'_1}CX_1'=C'_{2,2},\\ \overline{f_2}=b_3^{-1}\circ\overline{f_2'}\circ e_{2,2}:C_{2,2}\to X_3.
\end{gather*} 
Then $\overline{f_2}\circ j_{2,1}=f_2$. 
Hence we have $\mathscr{S}_2,\overline{f_2},\mathscr{A}_2$. 
Set $\mathscr{S}_3=\mathscr{S}_2(\overline{f_2},\mathscr{A}_2)$ and  
\begin{gather*}
e_{3,1}=b_3,\ e_{3,2}=b_3\cup Cb_2:X_3\cup CX_2=C_{3,2}\to X_3'\cup CX_2'=C'_{3,2},\\
e_{3,3}=b_3\cup Ce_{2,2}:C_{3,3}=X_3\cup_{\overline{f_2}} CC_{2,2}\to X_3'\cup_{\overline{f'_2}} CC_{2,2}'=C_{3,3}'.
\end{gather*}
We have (1) and (3) for $r=2$ by definitions, and also (2) for $r=2$ as follows
\begin{align*}
e_{3,2}\circ g_{3,2}&=e_{3,2}\circ(\overline{f_2}\cup C1_{X_2})\circ(q'_{f_1})^{-1}
=(\overline{f'_2}\cup C1_{X_2'})\circ(e_{2,2}\cup Cb_2)\circ (q'_{f_1})^{-1}\\
&\simeq (\overline{f_2'}\cup C1_{X_2'})\circ (q'_{f_1'})^{-1}\circ \Sigma b_1=g'_{3,2}\circ \Sigma b_1.
\end{align*}
Take $J:e_{3,2}^{-1}\circ g'_{3,2}\simeq g_{3,2}\circ \Sigma b_1^{-1}$ and set 
\begin{align*}
\Phi&=\Phi(g'_{3,2},g_{3,2},\Sigma b_1^{-1},e_{3,2}^{-1};J):C_{3,2}''\cup_{g'_{3,2}}C\Sigma X_1'\to C_{3,2}\cup_{g_{3,2}}C\Sigma X_1,\\
a_{3,2}&=\Phi\circ a'_{3,2}\circ e_{3,3}:C_{3,3}\to C_{3,2}\cup_{g_{3,2}}C\Sigma X_1.
\end{align*}
Then $a_{3,2}$ is a homotopy equivalence and 
\begin{align*}
a_{3,2}\circ j_{3,2}&=\Phi\circ a'_{3,2}\circ e_{3,2}\circ j_{3,2}=\Phi\circ a'_{3,2}\circ j'_{3,2}\circ e_{3,2}=\Phi\circ i_{g'_{3,2}}\circ e_{3,2}\\
&=i_{g_{3,2}}\circ e_{3,2}^{-1}\circ e_{3,2}=i_{g_{3,2}}.
\end{align*}
Hence $\mathscr{A}_3:=\{1_{C_{3,2}}, a_{3,2}\}$ is a structure on $\mathscr{S}_3$. 
We set 
$$
\overline{f_3}=\begin{cases} b_4^{-1}\circ \overline{f_3'}\circ e_{3,2}:C_{3,2}\to X_4 & n=3\\
b_4^{-1}\circ \overline{f'_3}\circ e_{3,3} : C_{3,3}\to X_4 & n\ge 4\end{cases}.
$$
Then $\overline{f_3}^2\circ j_{3,1}=f_3$ for $n\ge 3$, and $\overline{f_3}\circ j_{3,2}=\overline{f_3}^2$ for $n\ge 4$. 
Hence we obtain a desired $\ddot{s}_t$-presentation of $\vec{\bm f}$ when $n=3$. 
When $n\ge 4$, by repeating the above process, we have a desired $\ddot{s}_t$-presentation $\{\mathscr{S}_r,\overline{f_r},\mathscr{A}_r\,|\,2\le r\le n\}$ of $\vec{\bm f}$. 
We omit details. 
\end{proof}
\end{appendix}

\end{document}